\documentclass[letterpaper, reqno,11pt]{article}
\usepackage[margin=1.0in]{geometry}
\usepackage{color,latexsym,amsmath,amssymb, amsthm, graphicx,subfigure,revsymb, enumerate,xcolor,mathtools, todonotes,bm}
\usepackage{stackengine}
\usepackage{longtable}
\usepackage{array}
\usepackage[percent]{overpic}
\usepackage{chngcntr}
\counterwithout{figure}{section}

\numberwithin{equation}{section}

\newcommand{\itemizeEqnVSpacing}{\rule{0pt}{1pt}\vspace*{-12pt}}

\newcommand{\RR}{\mathbb{R}}

\newcommand{\eps}{\varepsilon}

\renewcommand{\dim}{\mathrm{dim}\,}
\newcommand{\diam}{\mathrm{diam}}

\newcommand{\dir}{\mathrm{dir}}
\newcommand{\supp}{\operatorname{supp}}

\newcommand{\tubes}{\mathbb{T}}

\newcommand{\cD}{\mathcal{D}}
\newcommand{\cE}{\mathcal{E}}
\newcommand{\TcE}{\tilde{\mathcal{E}}}
\newcommand{\cF}{\mathcal{F}}

\newcommand{\CKT}{C_{\scriptscriptstyle{KT\textrm{-}CW}}}
\newcommand{\CFC}{C_{\scriptscriptstyle{F\textrm{-}CW}}}
\newcommand{\FS}{C_{\scriptscriptstyle{F\textrm{-}SW}}}

\newtheorem{thm}{Theorem}
\numberwithin{thm}{section}
\newtheorem{conj}[thm]{Conjecture}
\newtheorem{prop}[thm]{Proposition}

\newtheorem{lem}[thm]{Lemma}

\newtheorem{cor}[thm]{Corollary}

\newtheorem*{defnConvexPrime}{Definition \ref{KatzTaoAndFrostmanTubesDefn}$^\prime$}
\newtheorem*{PropWolffHairbrushExpanded}{Proposition \ref{WolffHairbrush}, expanded}

\newtheorem*{GuthGrainsPropRepeat}{Proposition \ref{GuthGrainsProp}}

\newtheorem*{weakerPropEquivDELemInformal}{Lemma \ref{weakerPropEquivDE}, informal version}

\newtheorem*{tubeDoublingTheoremR3Again}{Theorem \ref{tubeDoublingTheoremR3}$'$}

\theoremstyle{remark}
\newtheorem{rem}[thm]{Remark}

\newtheorem{defn}[thm]{Definition}

\title{Volume estimates for unions of convex sets, and the Kakeya set conjecture in three dimensions}
\author{Hong Wang\thanks{Courant Institute of Mathematical Sciences, New York University. New York, NY, USA.} \and Joshua Zahl \thanks{Department of Mathematics, The University of British Columbia. Vancouver, BC, Canada.}}
\date{\today}
\begin{document}

\maketitle

\begin{abstract}
We study sets of $\delta$ tubes in $\RR^3$, with the property that not too many tubes can be contained inside a common convex set $V$. We show that the union of tubes from such a set must have almost maximal volume. As a consequence, we prove that every Kakeya set in $\mathbb{R}^3$ has Minkowski and Hausdorff dimension 3.
\end{abstract}

\tableofcontents


\section{Introduction}\label{introductionSection}
A Kakeya set is a compact subset of $\RR^n$ that contains a unit line segment pointing in every direction. The Kakeya set conjecture asserts that every Kakeya set in $\RR^n$ has Minkowski and Hausdorff dimension $n$. This conjecture was proved by Davies \cite{Dav71} when $n=2$, and is open in three and higher dimensions. See \cite{KT02,Wol96} for an introduction to the Kakeya conjecture and a survey of historical progress on the problem. See \cite{HRZ22, KLT00, KT02b, KaZa19, KaZa21, Wol95, Za21} for current progress towards the conjecture in three and higher dimensions.

The purpose of this paper is to obtain lower bounds on the volume of unions of $\delta$-tubes (i.e.~the $\delta$ neighbourhoods of unit line segments) in $\RR^3$ that satisfy certain non-clustering conditions. As a consequence, we resolve the Kakeya set conjecture in three dimensions. 

\begin{thm}\label{kakeyaSetThm}
Every Kakeya set in $\RR^3$ has Minkowski and Hausdorff dimension 3.
\end{thm}

Theorem \ref{kakeyaSetThm} is a corollary of the following slightly more technical result.

\begin{thm}\label{maximalFnBdpThm}
For all $\eps>0$, there exists $K>1$ so that the following holds for all $\delta>0$ sufficiently small.  Let $\tubes$ be a set of $\delta$-tubes contained in the unit ball in $\RR^3$, and suppose that every rectangular prism of dimensions $a\times b\times 2$ contains at most $100 ab\delta^{-2}$ tubes from $\tubes$ (this is true, for example, if the tubes in $\tubes$ point in $\delta$-separated directions). For each $T\in\tubes,$ let $Y(T)\subset T$ be a measurable set with $|Y(T)|\geq\lambda |T|$. Then
\begin{equation}\label{maximalFnBdpLamba}
\Big|\bigcup_{T\in\tubes}Y(T)\Big|\geq  \delta^\eps \lambda^K \sum_{T\in\tubes}|T|.
\end{equation}
\end{thm}

The Kakeya maximal function conjecture asserts that for each $\eps>0$, Inequality \eqref{maximalFnBdpLamba} is true for $K=3$. The Kakeya maximal function conjecture in $\RR^2$ was proved by Cordoba \cite{Cor77}. While we do not resolve the Kakeya maximal function conjecture in $\RR^3$, the weaker statement given in Theorem \ref{maximalFnBdpThm} is nonetheless sufficient to obtain Theorem \ref{kakeyaSetThm}.

The hypothesis that each $a\times b\times 2$ rectangular prism contains at most $100ab\delta^{-2}$ tubes from $\tubes$ is a type of non-clustering condition. A close variant of this hypothesis was first introduced by Wolff in \cite{Wol95}, and sets of tubes that satisfy this hypothesis are said to satisfy the Wolff axioms. 


\subsection{Theorem \ref{maximalFnBdpThm} and multi-scale analysis}\label{inductionOnScaleSection}
In \cite{WZ22, WZ23}, the authors showed that Theorem \ref{maximalFnBdpThm} is true when the set $\tubes$ has a property called stickiness (see Figure \ref{stickyVsWellSeparated} (left)). Roughly speaking, $\tubes$ is \emph{sticky} if it satisfies the non-clustering condition from Theorem \ref{maximalFnBdpThm}; has cardinality roughly $\delta^{-2}$; and for every intermediate scale $\delta\leq\rho\leq 1$, the tubes in $\tubes$ can be covered by a set of $\rho$ tubes that satisfy the non-clustering condition from Theorem \ref{maximalFnBdpThm} with $\rho$ in place of $\delta$.

Unfortunately, not every set of tubes is sticky --- see Figure \ref{stickyVsWellSeparated} (right) for an example. The arrangement illustrated in Figure \ref{stickyVsWellSeparated} (right) is challenging to analyze, because the $\rho$ tubes intersect with large multiplicity (i.e.~many $\rho$ tubes pass through a typical point), but the arrangement of $\delta$ tubes inside each $\rho$ tube is sparse (i.e.~the union of $\delta$ tubes inside each $\rho$ tube only fill out a small fraction of that $\rho$ tube). To help us analyze this type of arrangement, in Section \ref{unionsConvexSetsSec} we will introduce two variants of the non-clustering hypothesis from Theorem \ref{maximalFnBdpThm}, and two variants of the volume estimate \eqref{maximalFnBdpLamba}.

\begin{figure}
 \includegraphics[width=.42\linewidth]{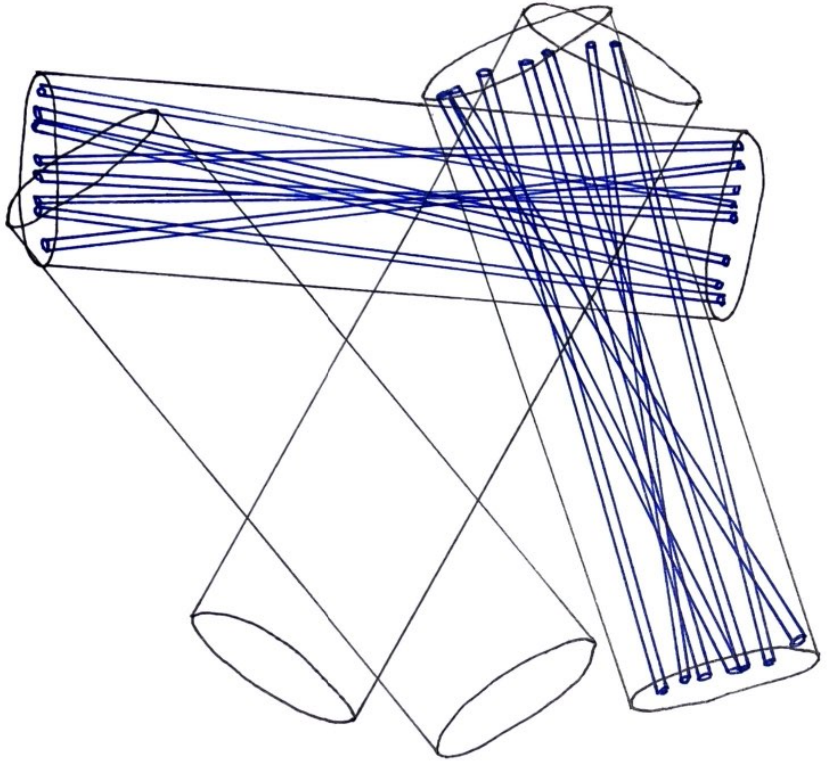}\hfill
 \includegraphics[width=.42\linewidth]{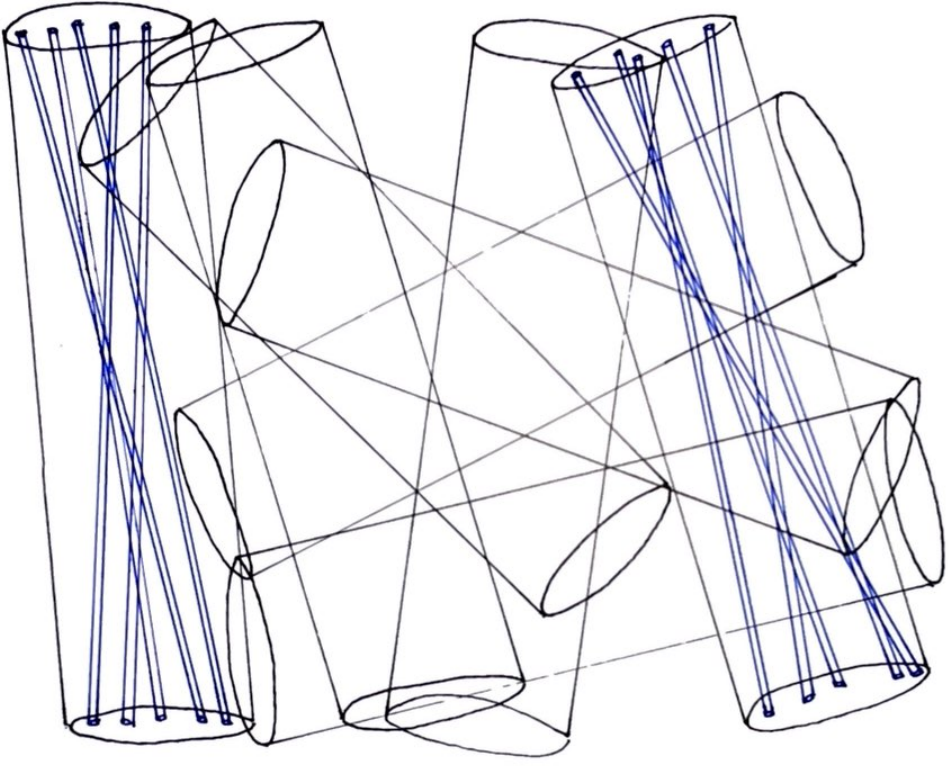}
\caption{ Left: The tubes at scale $\rho$ (black) satisfy the non-concentration hypothesis of Theorem \ref{maximalFnBdpThm}, as do the (rescaled) $\delta$ tubes (blue) inside each $\rho$ tube. Multi-scale analysis is straightforward in this setting. This is sometimes called the ``sticky'' case. For clarity, not all $\delta$ tubes have been drawn.
\\
Right: The tubes at scale $\rho$ do not satisfy the non-concentration hypothesis of Theorem \ref{maximalFnBdpThm}. The tubes at scale $\rho$ intersect with high multiplicity, while the $\delta$ tubes inside each $\rho$ tube are sparse.}
\label{stickyVsWellSeparated}
\end{figure}


\subsection{Unions of convex sets, and non-clustering}\label{unionsConvexSetsSec}
In what follows, we say a pair of sets $U,V\subset\RR^n$ are \emph{essentially distinct} if $|U\cap V|\leq\frac{1}{2}\max(|U|,|V|)$. $\tubes$ will denote a set of essentially distinct $\delta$-tubes contained in the unit ball in $\RR^3$, and $|T|$ will denote the volume of a $\delta$-tube, i.e. $|T|$ has size about $\delta^2$.

\begin{defn}\label{KatzTaoAndFrostmanTubesDefn} 
Let $\tubes$ be a set of $\delta$-tubes in $\RR^3$. \\
(A) We define $\CKT(\tubes)$ to be the infimum of all $C>0$ such that
\[
\#\{T\in\tubes\colon T\subset W\} \leq C |W| |T|^{-1}\quad\textrm{for all convex sets}\ W\subset\RR^3.
\]
We say that $\tubes$ obeys the Katz-Tao Convex Wolff Axioms with error $\CKT(\tubes)$ .

\medskip

\noindent (B) We define $\FS(\tubes)$ to be the infimum of all $C>0$ such that
\[
\#\{T\in\tubes\colon T\subset W\} \leq C |W| (\#\tubes)\quad\textrm{for all slabs}\ W\subset\RR^3,
\]
where a ``slab'' is the intersection of the unit ball with the thickened neighbourhood of a (hyper) plane.  We say that $\tubes$ obeys the Frostman Slab Wolff Axioms with error $\FS(\tubes)$.
\end{defn}

\begin{rem}$\phantom{1}$\\
(A) A note on etymology. The terms ``Katz-Tao'' and ``Frostman'' refer to different types of non-concentration conditions; they are the analogues of the well-studied non-concentration conditions $|E \cap B|\leq (r/\delta)^2$ and $|E\cap B|\leq r^2|E|$, where $E\subset\RR^n$ is a $\delta$-separated set and $B$ is a ball of radius $r$. An arrangement of tubes arising from a Kakeya set, i.e.~a set of $\delta$-tubes with one tube pointing in each $\delta$-separated direction, obeys both the Katz-Tao Convex Wolff Axiom and Frostman Slab Wolff Axiom with error $\lesssim 1$. The terms ``convex'' and ``slab'' refer to the class of sets for which the non-clustering condition is imposed. The term ``Wolff axioms'' suggests that the above definition is an analogue of the Wolff axioms from \cite{Wol95}.

\medskip
\noindent (B) The above definitions are two special cases of a non-clustering condition (Definition \ref{KatzTaoAndFrostmanTubesDefn}$^\prime$) that will be defined in Section \ref{factorConvexSetsSec}. In Definition \ref{KatzTaoAndFrostmanTubesDefn}$^\prime$, both tubes and convex sets (resp.~slabs) are replaced by more general objects. 

\medskip
\noindent (C) If $\tubes$ is non empty, then by taking $W$ to be a $\delta \times 1\times 1$-slab containing a tube of $\tubes$,  we can see $\FS(\tubes)\leq C$ implies  $\#\tubes \geq C^{-1} \delta^{-1}$. 
\end{rem}

Next, we introduce two Kakeya-type volume estimates for unions of tubes in $\RR^3$. These are analogues of Inequality \eqref{maximalFnBdpLamba} that are carefully formulated to be amenable to induction on scale. In what follows, we use the notation $(\tubes,Y)_\delta$ to denote a collection $\tubes$ of essentially distinct $\delta$-tubes in $\RR^3$, and a \emph{shading} of these tubes, i.e.~for each $T\in\tubes$, $Y(T)$ is a subset of $T$. For $\lambda>0$, we say $(\tubes,Y)_\delta$ is $\lambda$ \emph{dense} if $\sum_{T\in\tubes}|Y(T)|\geq \lambda\sum_{T\in\tubes}|T|$.

\medskip

\begin{defn}\label{defnCDE}
Let $\sigma, \omega \geq 0$. 
\begin{itemize}
\item We say that \emph{Assertion $\cD(\sigma,\omega)$ is true} if the following holds:\\
For all $\eps>0$, there exists $\kappa ,\eta>0$ such that the following holds for all $\delta>0$. Let $(\tubes,Y)_\delta$ be $\delta^{\eta}$ dense and obey the Katz-Tao Convex Wolff Axioms and Frostman Slab Wolff Axioms, both with error at most $\delta^{-\eta}$. Then
\begin{equation}\label{defnCDEqn}
\Big|\bigcup_{T\in\tubes}Y(T)\Big|\geq \kappa \delta^{\omega+\eps}(\#\tubes)|T|\big((\#\tubes)|T|^{1/2}\big)^{-\sigma}.
\end{equation}

\item  We say that \emph{Assertion $\cE(\sigma,\omega)$ is true} if the following holds:\\
For all $\eps>0$, there exists $\kappa,\eta>0$ such that the following holds for all $\delta>0$. Let $(\tubes,Y)_\delta$ be $\delta^{\eta}$ dense. Then
\begin{equation}\label{defnCEEstimate}
\Big|\bigcup_{T\in\tubes}Y(T)\Big|\geq \kappa \delta^{\omega+\eps}m^{-1}(\#\tubes)|T|\big(m^{-3/2}\ell (\#\tubes)|T|^{1/2}\big)^{-\sigma},
\end{equation}
where  $m = \CKT(\tubes)$ and $\ell = \FS(\tubes)$.
\end{itemize}
\end{defn}

\medskip

Let us examine the numerology in the estimates \eqref{defnCDEqn} and \eqref{defnCEEstimate}. First, in the special case $\sigma=\omega$, Assertion $\cD(\sigma,\sigma)$ yields the estimate
\[
\Big|\bigcup_{T\in\tubes}Y(T)\Big|\geq \kappa \delta^{\eps} (\#\tubes)^{-\sigma}\sum_{T\in\tubes}|T|,
\]
i.e.~it says that there are $\lesssim \delta^{-\eps}(\#\tubes)^{\sigma}$ tubes passing through a typical point of the union $\bigcup Y(T)$ (for general $\sigma$ and $\omega$, this quantity is about $\delta^{\sigma-\omega-\eps}(\#\tubes)^{\sigma}$). For $\sigma>0$ small, this means that the tubes in the union $\bigcup Y(T)$ are almost disjoint. In the arguments that follow, it will be helpful to consider situations where $\sigma$ and $\omega$ are not necessarily equal. 

The shape of the estimate \eqref{defnCDEqn} is motivated in part by the following consideration. To begin our induction on scale argument, we would like to prove that $\cE(\sigma,0)$ holds for some $\sigma\in(0, 2/3]$. When $\sigma = 1/2$ and $\omega = 0$, Inequality \eqref{defnCDEqn} becomes the estimate
\[
\Big|\bigcup_{T\in\tubes}Y(T)\Big| \geq \kappa \delta^{1/2+\eps}(\delta^2\#\tubes)^{1/2}.
\]
This is essentially Wolff's hairbrush bound from \cite{Wol95} (here we make use of the fact that $\tubes$ obeys the Frostman Slab Wolff Axioms with small error; see Appendix \ref{WolffHairbrushSec} for details).

\medskip

Assertion $\cD(\sigma,\omega)$ is a special case of Assertion $\cE(\sigma,\omega)$. We will explain the shape of the final bracketed term of Inequality \eqref{defnCEEstimate}. To understand the term $\ell$, it is helpful to consider the following scenario. Suppose we know that Assertion $\cD(\sigma,\omega)$ is true. Let $\delta < \rho < 1$, and let $\tubes$ be a set of $\delta$ tubes of cardinality $(\rho/\delta)^2$ that are contained inside a common $\rho$ tube, which we will denote by $T_\rho$. Suppose that the tubes in $\tubes$ obey the Katz-Tao Convex Wolff Axioms with error roughly 1. This implies that $\FS(\tubes)\sim\rho^{-1}$. 

For $T\in\tubes$ (and hence $T\subset T_\rho$), we will write $T^{T_\rho}$ to denote the image of $T$ under the affine transformation that anisotropically  dilates $T_\rho$ by a factor of $\rho^{-1}$ in its two ``short'' directions, and translates the image to the unit ball. After this rescaling and translation, the tubes in $\tubes$ become $\delta/\rho$ tubes that satisfy the Katz-Tao Convex Wolff Axioms and Frostman Slab Wolff Axioms, both with error roughly 1. Applying Assertion $\cD(\sigma,\omega)$ to this rescaled collection of tubes, we obtain the volume bound
\[
\Big| \bigcup_{T\in\tubes}T^{T_\rho}\Big| \geq \kappa (\delta/\rho)^{\eps} (\#\tubes)|T^{T_\rho}|\big( (\#\tubes)|T^\rho|^{1/2}\big)^{-\sigma}.
\]
Undoing the anisotropic rescaling and translation (which distorted volumes by a factor of $\rho^2$) and noting that $|T^{T_\rho}|\sim\rho^2|T|$, we can rewrite this as 
\[
\Big| \bigcup_{T\in\tubes}T\Big| \gtrsim \kappa \delta^{\eps} (\#\tubes)|T|\big(\ell (\#\tubes)|T|^{1/2}\big)^{-\sigma},\quad\textrm{where}\ \ell = \FS(\tubes) \sim \rho^{-1}.
\]

As a second justification for the term $\ell$, note that for every set $\tubes$ of $\delta$ tubes, we must always have $\FS(\tubes) (\#\tubes)|T|^{1/2}\geq 1$. This is because we can always select a slab $W$ of thickness $|T|^{1/2}$ that contains at least one tube from $\tubes$. This observation also explains the choice to write $|T|^{1/2}$ rather than $\delta$; any convex set $S\subset \RR^3$ of diameter 1 can be contained in a slab of thickness $|S|^{1/2}$. Later in the proof we will consider generalizations of Assertion $\cE(\sigma,\omega)$ in which tubes are replaced by more general families of convex sets.

\medskip

To understand the terms $m^{-1}$ and $m^{-3/2}$ in Inequality \eqref{defnCEEstimate}, it is helpful to consider the following scenario.  Suppose we know that Assertion $\cD(\sigma,\omega)$ is true. Let $\tubes$ be a set of $\delta$ tubes that obey the Frostman Slab Wolff Axioms with error roughly 1, and the Katz-Tao Convex Wolff Axioms with error $m>\!\!>1$. Let $\rho = m^{1/2}\delta$, and suppose that there exists a set $\tubes_\rho$ of $\rho$ tubes, each of which contains $m|T_\rho| |T|^{-1} = m^2$ tubes from $\tubes$. Observe that this is the maximum number of essentially distinct $\delta$ tubes that can fit inside a $\rho$ tube. In particular, the union of the $\delta$ tubes inside each $\rho$ tube fill out essentially all of the $\rho$ tube.
We have $\#\tubes_\rho = m^{-2}(\#\tubes) = m^{-1}(|T|/|T_\rho|)(\#\tubes)$, i.e. $(\#\tubes_\rho)|T_\rho| = m^{-1}(\#\tubes)|T|$. It is straightforward to compute that $\CKT(\tubes_\rho)\lesssim 1$. Applying Assertion $\cD(\sigma,\omega)$ and using the fact that the union of $\delta$ tubes inside each $\rho$ tube fill out most of the $\rho$ tube, we obtain the volume bound 
\begin{align*}
	\Big|\bigcup_{T\in \tubes} T \Big| \sim \Big| \bigcup_{T_{\rho}\in \tubes_\rho} T_{\rho}\Big| & \geq  \kappa \rho^{\omega+\eps}  (\#\tubes_\rho) |T_{\rho}| \big( \#\tubes_\rho ) |T_{\rho}|^{1/2}\big)^{-\sigma} \\
	&=  \kappa \rho^{\omega+\eps}  m^{-1}(\#\tubes)|T|\big(m^{-3/2} (\#\tubes)|T|^{1/2}\big)^{-\sigma}.
\end{align*}




\subsection{From Assertions $\cD$ and $\cE$ to the Kakeya set conjecture}

Clearly $\cE(\sigma,\omega)\implies \cD(\sigma,\omega)$. In Section \ref{cEIffcDSec}, we will show that the reverse implication also holds:

\begin{prop}\label{equivDE}
Let $0\leq\sigma\leq 2/3,\ \omega\geq 0$. Then $\cE(\sigma,\omega) \Longleftrightarrow \cD(\sigma,\omega).$
\end{prop}

As mentioned above, our proof uses induction on scale. In brief, if $\cE(\sigma,\omega)$ is true, then we will use this fact at many locations and scales to prove that $\cD(\sigma,\omega')$ is true for some $\omega'<\omega$ (observe that smaller values of $\omega$ are better). The precise statement is as follows.

\begin{prop}\label{improvingProp}
There exists a function $g\colon [0,2/3]\times(0,1]\to(0,1]$ so that the following is true. Let $0\leq\sigma\leq 2/3,\ \omega>0$. Then $\cE(\sigma,\omega)\implies \cD(\sigma,\omega-g(\sigma,\omega))$.
\end{prop}

Propositions \ref{equivDE} and \ref{improvingProp} lead to a self-improving property for $\cE(\sigma,\omega)$ (or equivalently, for $\cD(\sigma,\omega)$). Since the collections of tubes in the definitions of $\cE$ and $\cD$ are essentially distinct and are contained in the unit ball, we always have $\#\tubes\lesssim\delta^{-4}$, and thus we can ``trade'' an improvement in $\omega$ for an improvement in $\sigma$. In particular, Proposition \ref{improvingProp} tells us that $\cE(\sigma,\omega)\implies\cD(\sigma-g(\sigma,\omega)/4,\omega)$.

By applying Propositions \ref{equivDE} and \ref{improvingProp}, we can upgrade an initial estimate $\cD(\sigma,\omega)$ to the improved estimate $\cD(\sigma-g(\sigma,\omega)/4,\omega)$. We can then iterate this process. In order to begin the iteration, we must prove that $\cD(\sigma,\omega)$ is true for some $\omega>0$ and $0\leq\sigma\leq 2/3$. In \cite{Wol95}, Wolff proved that every Kakeya set in $\RR^n$ has Hausdorff dimension at least $\frac{n+2}{2}$. In Appendix \ref{WolffHairbrushSec}, we will use a similar argument to show that $\cD(1/2,0)$ is true:

\begin{prop}\label{WolffHairbrush}
$\cD(1/2,0)$ is true.
\end{prop}

Beginning with Proposition \ref{WolffHairbrush} and then iterating Propositions \ref{equivDE} and \ref{improvingProp}, we conclude the following.
\begin{thm}\label{cDAndcEAreTrue}
The statements $\cD(0,0)$ and $\cE(0,0)$ are true.
\end{thm}
\begin{proof}
Fix $\omega>0$. By Proposition \ref{WolffHairbrush}, we have that $\cD(1/2,0)$ and hence $\cD(1/2,\omega)$ is true. If $\cD(\sigma,\omega)$ is true for some $\sigma\in[0,2/3]$, then so is $\cD(\sigma',\omega)$ for all $\sigma'\in[\sigma,2/3]$. Using Propositions \ref{equivDE} and \ref{improvingProp}, we conclude that the set $\{\sigma\in[0,2/3]\colon \cD(\sigma,\omega)\ \textrm{is true}\}$ is relatively open in the metric space $[0,2/3]$. On the other hand, it is straightforward to verify from Definition \ref{defnCDE} that this set is also relatively closed in $[0,2/3]$. We conclude that $\cD(\sigma,\omega)$ is true for all $\sigma\in[0,2/3]$, so in particular $\cD(0,\omega)$ is true. 

A similar argument shows that $\cD(0,0)$ is true; we have shown that $\cD(0,\omega)$ is true for every $\omega>0$. On the other hand, the set $\{\omega\geq 0\colon \cD(0,\omega)\ \textrm{is true}\}$ is relatively closed in the metric space $[0,\infty)$. We conclude that $\cD(0,0)$ is true. By Proposition \ref{equivDE} we have that $\cE(0,0)$ is true.
\end{proof}

\noindent The conclusion of Theorem  \ref{cDAndcEAreTrue} can be rephrased as follows
\begin{cor}\label{maximalFnCor}
For all $\eps>0$, there exists $K$ so that the following holds for all $\delta>0$ sufficiently small. Let $(\tubes,Y)_\delta$ be $\lambda$-dense. Then 
\begin{equation}
\Big| \bigcup_{T\in\tubes}Y(T)\Big| \geq \delta^{\eps}\lambda^K m^{-1}(\#\tubes)|T|,\qquad \textrm{where}\ m = \CKT(\tubes).
\end{equation}
\end{cor}

Theorem \ref{maximalFnBdpThm} is now a special case of Corollary \ref{maximalFnCor} --- the hypotheses of Theorem \ref{maximalFnBdpThm} ensure that $\CKT(\tubes)\leq 1000$. 


\subsection{Proof philosophy, and previous work on the Kakeya set conjecture in $\RR^3$}
In \cite{KLT00}, Katz, \L{}aba, and Tao proved that every Kakeya set in $\RR^3$ has upper Minkowski dimension at least $5/2+c$ for a (small) absolute constant $c>0$. To do this, they analyzed the structure of a (hypothetical) Kakeya set in $\RR^3$ that has upper Minkowski dimension close to $5/2$. They proved that such a Kakeya set, or more precisely, the set of $\delta$ tubes arising from such a Kakeya set, must have three structural properties that they named ``planiness,'' ``graininess,'' and ``stickiness.''  Katz, \L{}aba, and Tao then showed that a Kakeya set possessing these structural properties must have dimension at least $5/2+c$.

In a talk and accompanying blog post \cite{TaoBlog} in 2014, Tao described a potential approach developed by Katz and Tao for solving the Kakeya problem. The Katz-Tao program proceeds as follows. First, one must show that a (hypothetical) counter-example to the Kakeya conjecture in $\RR^3$, i.e. a Kakeya set with dimension strictly less than 3, must have the structural properties planiness, graininess, and stickiness. Second, these properties are used to obtain increasingly precise statements about the multi-scale structure of the Kakeya set. Third, results from discretized sum-product theory, in the spirit of Bourgain's discretized sum-product theorem \cite{Bour03}, are used to show that a Kakeya set with this type of multi-scale structure cannot exist.

When Tao shared the Katz-Tao program for solving the Kakeya conjecture in $\RR^3$, some progress had already been made towards the first step described above. The Bennett-Carbery-Tao multilinear Kakeya theorem \cite{BCT06} implied that every (hypothetical) counter-example to the Kakeya conjecture in $\RR^3$ must be plany. In \cite{Gut14}, Guth proved that every (hypothetical) counter-example to the Kakeya conjecture in $\RR^3$ must be grainy. Stickiness, however, appeared to be more challenging.

The trilogy of papers \cite{WZ22,WZ23}, and the present paper, can be thought of as a realization of the Katz-Tao program. In \cite{WZ22}, the authors sidestepped the First step of the Katz-Tao program, and tackled the Second and Third steps. More precisely, the authors showed that every sticky Kakeya set in $\RR^3$ (i.e.~a Kakeya set possessing the structural property of stickiness) must have Hausdorff dimension 3. This result is called the Sticky Kakeya Theorem. See \cite[\S 1.1]{WZ22} for a discussion of the proof of this theorem, and how this proof compares to the strategy outlined in the Katz-Tao program. 

In \cite{WZ23}, the authors showed that every (hypothetical) Kakeya set in $\RR^3$ with \emph{Assouad} dimension strictly less than 3 must be sticky. More precisely, they showed that if there exists a Kakeya set $K$ with $\dim_A(K)<3$, then there must also exist a Kakeya set $K'$ with $\dim_A(K')<3$ that possesses a multi-scale self-similarity property similar to stickiness. The authors then used (a mild generalization of) the Sticky Kakeya Theorem to conclude that such a Kakeya set cannot exist. In particular, the Sticky Kakeya theorem from \cite{WZ22} assumed that the tubes from a Kakeya set point in different directions; in \cite{WZ23} the authors generalized this theorem to the weaker assumption that the tubes satisfy the Wolff axioms at every scale (a precise definition is given in Section \ref{cEIffcDSec}). Note that since the Assouad dimension of a set can be larger than its Minkowski or Hausdorff dimension, the results in \cite{WZ23} did not resolve the Kakeya set conjecture in $\RR^3$. 

In the present paper, we take this line of reasoning to its conclusion. We show that if $\tubes$ is a set of $\delta$ tubes that makes the estimate \eqref{defnCDEqn} from Assertion $\cD(\sigma,\omega)$  tight for some $\sigma$ and $\omega$, then $\tubes$ must have a multi-scale self-similarity property similar to stickiness. Specifically, at many scales $\rho$ between $\delta$ and $1$, it is possible to cover $\tubes$ by a family of $\rho$ tubes that obey Katz-Tao Convex Wolff Axioms (recall Definition \ref{KatzTaoAndFrostmanTubesDefn}) with small error. We then use a generalization of the Sticky Kakeya Theorem to show that the estimate \eqref{defnCDEqn} from Assertion $\cD(\sigma,\omega)$ can only be tight for such a set $\tubes$ if $\sigma$ and $\omega$ are both 0. As we have already seen in Section \ref{unionsConvexSetsSec}, this implies that every Kakeya set in $\RR^3$ (and indeed, every set satisfying the Wolff axioms) must have Minkowski and Hausdorff dimension 3.

%


\subsection{A vignette of the proof}\label{vignetteOfProofSection}
Proposition \ref{improvingProp} is the most important step in the proof of Theorem \ref{cDAndcEAreTrue} (which in turn implies Theorems \ref{kakeyaSetThm} and \ref{maximalFnBdpThm}). In this section we will discuss some of the ideas used to prove this proposition in the key special case where the tubes are arranged as in Figure \ref{stickyVsWellSeparated} (right). In Section \ref{proofSketchSection} we will give a more thorough proof sketch that mirrors the structure of the actual proof.

To simply our exposition, we will disregard factors of the form $\delta^{\eps}$ or $\delta^{-\eps}$, and we will (somewhat informally) write $A \lessapprox B$ to mean that $A \leq C \delta^{-\eps}B$, for some constant $C$ that is independent of $\delta$ and some small parameter $\eps>0$ that we will ignore for the purposes of this sketch. 

Fix a choice of $\sigma>0$ and $\omega>0$, and suppose that Assertions $\cD(\sigma,\omega)$ and $\cE(\sigma,\omega)$ are true (roughly speaking, this says that the union of tubes has ``dimension'' at least $3 - \sigma-\omega$).  Let $\tubes$ be a set of $\delta$ tubes of cardinality roughly $\delta^{-2}$ that obeys the hypotheses of Assertion $\cD(\sigma,\omega)$. Our goal is to prove that $\bigcup_{\tubes}T$ has volume substantially larger than what is guaranteed by the estimate \eqref{defnCDEqn}, i.e.~we wish to obtain an inequality of the form
\begin{equation}\label{desiredVolumeBdVignette}
\Big|\bigcup_{T\in\tubes}T\Big| \gtrapprox \delta^{\sigma+\omega-\alpha},
\end{equation}
for some $\alpha = \alpha(\sigma,\omega)>0$.

Let us suppose that there exists a multiplicity $\mu$ with the property that there are about $\mu$ tubes from $\tubes$ passing through each point of $\bigcup_{\tubes}T$. One way to obtain our desired volume bound \eqref{desiredVolumeBdVignette} is to instead prove the multiplicity bound
\begin{equation}\label{desiredMuBd}
\mu\lessapprox\delta^{-\sigma-\omega+\alpha}.
\end{equation}

A second way to obtain  \eqref{desiredVolumeBdVignette} is to show there exists some scale $\tau >\!\!> \delta$ such that the union $\bigcup_{\tubes} T$ has larger than expected density at scale $\tau$. More specifically, to obtain \eqref{desiredVolumeBdVignette} it suffices to show that for a typical ball $B_\tau$ of radius $\tau$ that intersects $\bigcup_{\tubes} T$, we have a density estimate of the form
\begin{equation}\label{eq: denseball}
\Big |B_\tau\ \cap\  \bigcup_{T\in \tubes} T \Big| \gtrapprox \delta^{-\alpha}  (\delta/\tau)^{\sigma+\omega} |B_\tau|. 
\end{equation}
This will be discussed in greater detail in ``Step 2, Case 2'' below. 


If $\tubes$ is sticky, then for each scale $\delta<\rho<1$, it is possible to find a set $\tubes_\rho$ consisting of about $\rho^{-2}$ essentially distinct $\rho$ tubes, each of which contain about $(\delta/\rho)^2$ tubes from $\tubes$. We will suppose instead that $\tubes$ is \emph{not} sticky, i.e.~$\tubes$ resembles the arrangement in Figure \ref{stickyVsWellSeparated} (right). We will call this \emph{Simplifying Assumption A}.  More precisely, there exists a scale $\delta <\!\!<\rho<\!\!<1$, and a set of essentially distinct $\rho$ tubes $\tubes_\rho$ so that each $T\in\tubes$ is contained in at least one tube from $\tubes_\rho$, and each $T_\rho\in\tubes_\rho$ contains about $\delta^{\nu} (\rho/\delta)^2$ tubes from $\tubes$, for some (small) $\nu>0$. We will try to establish Inequality \eqref{desiredMuBd} with some small improvement $\alpha>0$.

\medskip

\noindent {\bf A fine-scale estimate.}\\ 
For each $T_\rho\in\tubes_\rho$, define
\begin{equation}\label{tubesTRhoDefn}
\tubes[T_\rho]=\{T\in\tubes\colon T\subset T_\rho\}\quad \textrm{and}\quad \tubes^{T_\rho} = \{T^{T_\rho}\colon T\in\tubes[T_\rho]\}.
\end{equation}
(Recall that $T^{T_\rho}$ is defined in the discussion following Definition \ref{defnCDE}). Suppose that for each $T_\rho\in\tubes_\rho$, the tubes in $\tubes^{T_\rho}$ satisfy the hypotheses of Assertion $\cD(\sigma,\omega)$; we will call this \emph{Simplifying Assumption B}. We define $\mu_{\operatorname{fine}}$ to be the number of tubes from $\tubes^{T_\rho}$ passing through a typical point of $\bigcup_{\tubes^{T_\rho}}T^{T_\rho}$ (it is harmless to suppose that this number is the same for each $\rho$ tube in $\tubes_\rho$). Applying Assertion $\cD(\sigma,\omega)$ to each set $\tubes^{T_\rho}$ and recalling the discussion following Definition \ref{defnCDE}, we conclude that 
\begin{equation}\label{muFineEstimateHeartSketch}
\mu_{\operatorname{fine}} 
\lessapprox \Big(\frac{\delta}{\rho}\Big)^{\sigma-\omega}(\#\tubes[T_\rho])^{\sigma}
\leq\Big(\frac{\delta}{\rho}\Big)^{\sigma-\omega}\Big(\delta^{\nu}\frac{\rho^2}{\delta^2}\Big)^{\sigma}= \delta^{\nu\sigma} \Big(\frac{\rho}{\delta}\Big)^{\sigma+\omega},
\end{equation}
where the second inequality used our assumption that $\#\tubes[T_\rho]\leq \delta^{\nu}(\rho/\delta)^2$.

Inequality \eqref{muFineEstimateHeartSketch} bounds the typical intersection multiplicity of the $\delta$ tubes inside a common $\rho$ tube. Next, we define the quantity $\mu_{\operatorname{coarse}}$ as follows: for a typical point $x\in\bigcup_{\tubes}T$, there are about $\mu_{\operatorname{coarse}}$ distinct $\rho$ tubes $T_\rho\in\tubes_\rho$ with the property that $x\in\bigcup_{\tubes[T_\rho]}T$. With this definition, we have
\begin{equation}\label{eq: mu}
\mu \sim \mu_{\operatorname{fine}}\mu_{\operatorname{coarse}}.
\end{equation}

In the past, researchers have considered a weaker variant of \eqref{eq: mu} of the form 
$ \mu \lesssim \mu_{\operatorname{fine}}  \mu_{\tubes_\rho}$,
where $\mu_{\tubes_\rho}$ is the number of tubes from $\tubes_\rho$ passing through a typical point of $\cup_{\tubes_\rho} T_\rho$. Our use of the more refined estimate \eqref{eq: mu} is a key new ingredient in the proof.

In light of \eqref{muFineEstimateHeartSketch}, our desired multiplicity bound \eqref{desiredMuBd} will follow if we can establish the estimate
\begin{equation}\label{desiredBoundMuCoarse}
\mu_{\operatorname{coarse}}\lessapprox \rho^{-\sigma-\omega}.
\end{equation}

Naively, we might attempt to obtain \eqref{desiredBoundMuCoarse} by observing that $\mu_{\operatorname{coarse}}\leq \mu_{\tubes_\rho}$, and then bounding the latter using  Assertion $\cE(\sigma, \omega)$. However, this approach does not yield \eqref{desiredBoundMuCoarse} because the cardinality of $\tubes_\rho$ (in this proof vignette) is substantially larger than $\rho^{-2}$.

\medskip

\noindent {\bf A coarse-scale estimate Step 1: a grains decomposition.}\\
Fix a tube $T_\rho\in\tubes_\rho$. Using a variant of Guth's grains decomposition from \cite{Gut14}, we can suppose that the $\delta/\rho$ tubes in $\tubes^{T_\rho}$ arrange themselves into ``grains,'' i.e.~rectangular prisms of dimensions $\delta/\rho\times c\times c,$ with $c \geq  \frac{\rho}{\delta} (\#\tubes[T_\rho])^{-1}$ (Note that our hypotheses on the size of $\#\tubes[T_\rho]$ guarantees that $c>\!\!> \delta/\rho$). Here and throughout, we will adopt the convention that when referring to a rectangular prism of dimensions $a\times b\times c$, we will always have $a\leq b \leq c$.  

This means that we can cover $E_{T_\rho} = \bigcup_{\tubes^{T_\rho}}T^{T_\rho}$ by a set of (mostly) disjoint rectangular prisms of dimensions $\delta/\rho \times c\times c$, each of which have large intersection with $E_{T_\rho}$, in the sense that $|G \cap E_{T_\rho}|\gtrapprox |G|$, for each such prism $G$; see Figure \ref{grainsDecompHeartOfMatter} (left).

\begin{figure}[h!]
 \includegraphics[width=.42\linewidth]{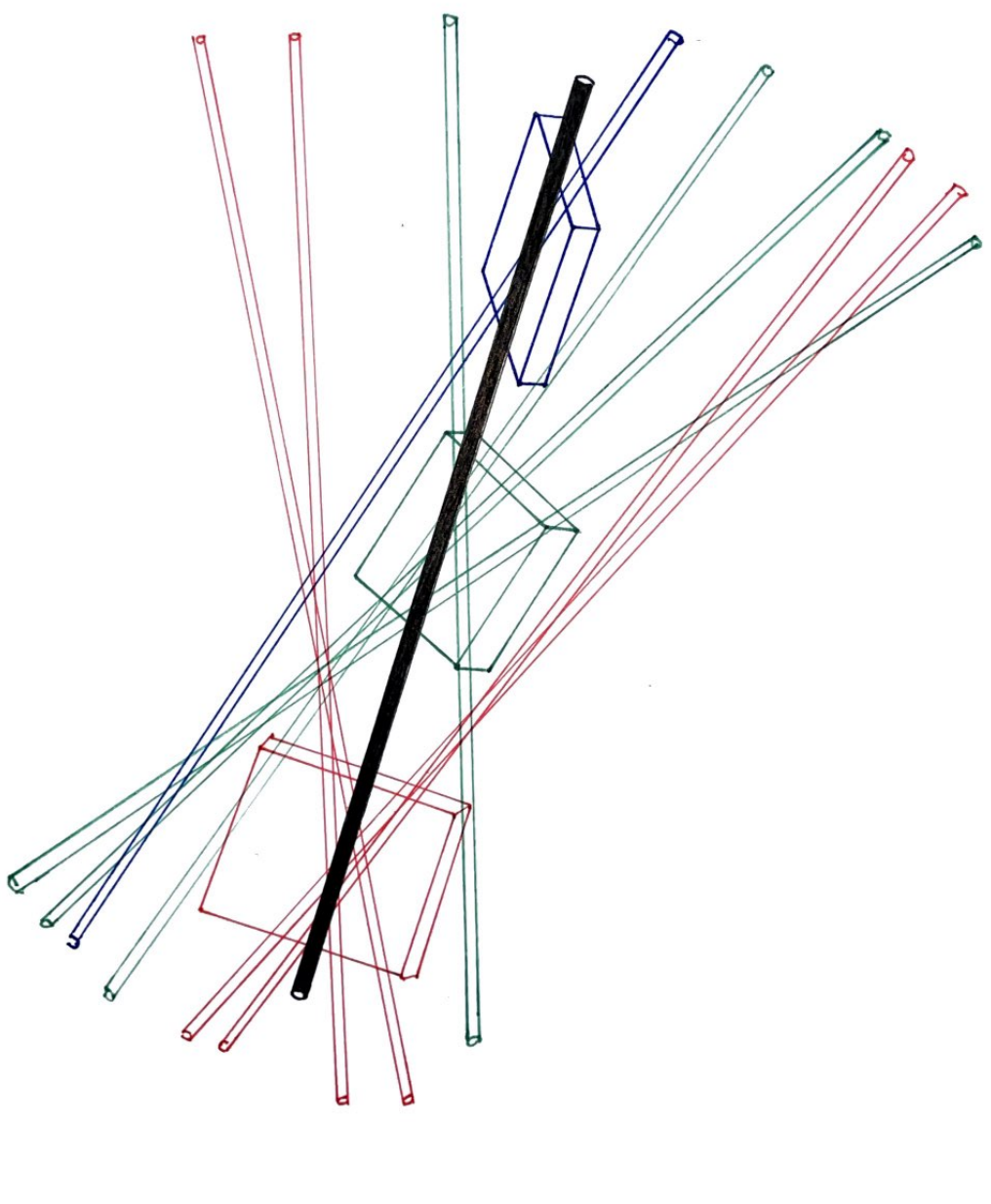}\hfill
 \includegraphics[width=.42\linewidth]{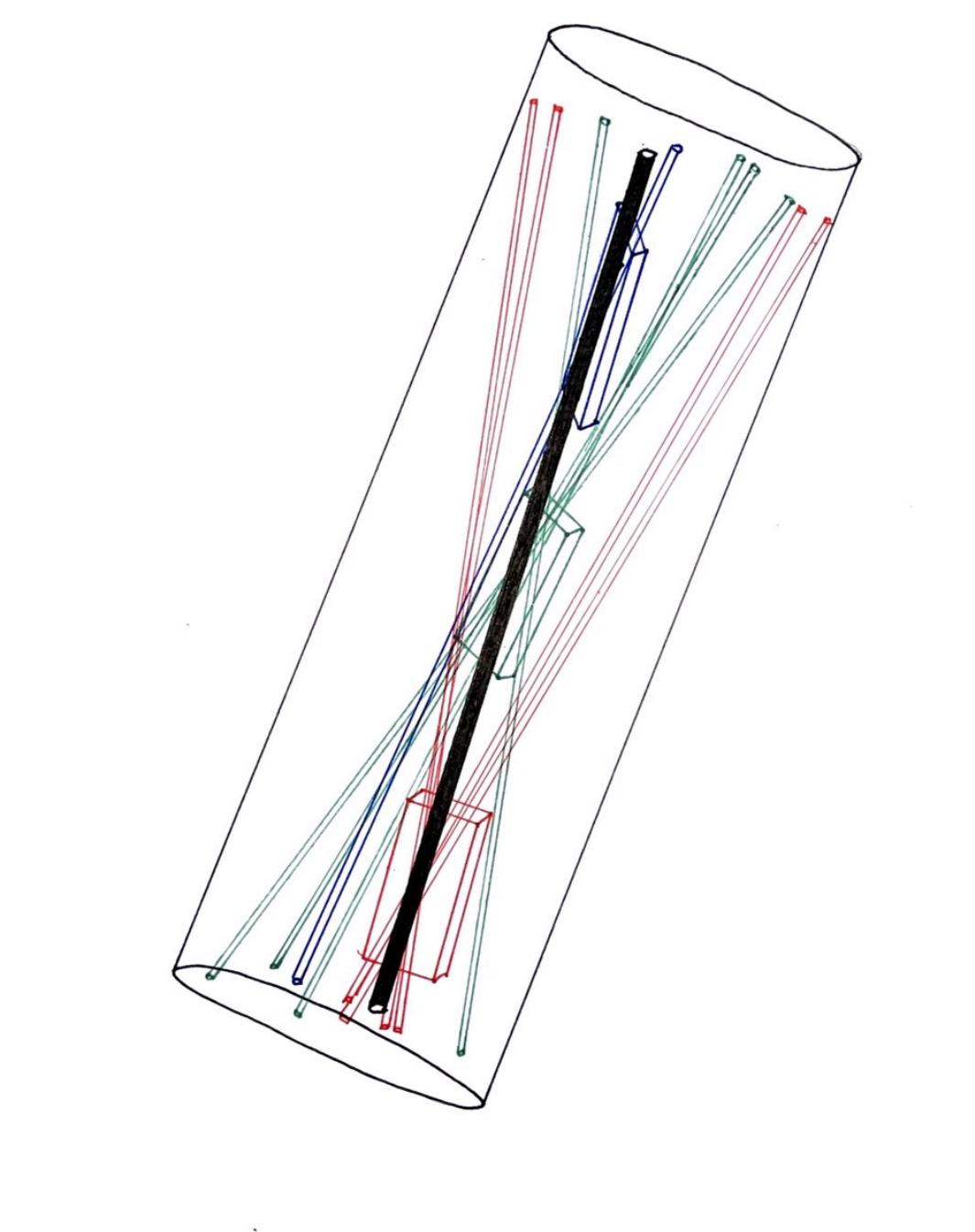}
\caption{Left: The set of tubes $\tubes^{T_\rho}$ and the grains $\{G\}.$ For clarity, we have only drawn the grains and tubes that intersect the black tube (and even most of these have been omitted; the set of red tubes passing through the red grain fill out a substantial portion of the red grain, and similarly for the other grains); the situation is similar for each tube in $\tubes^{T_\rho}$.\\
Right: The image of Figure \ref{grainsDecompHeartOfMatter} (left) after undoing the anisotropic rescaling associated to $T_\rho$. The dimensions of each grain have changed from $\delta/\rho\times c\times c$ to $\delta\times\rho c\times c$. }
\label{grainsDecompHeartOfMatter}
\end{figure}

Undoing the anisotropic rescaling associated to $T_\rho$ that was described above, we have that $\bigcup_{\tubes[T_\rho]}T$ can be covered by a set of (mostly) disjoint rectangular prisms of dimensions $\delta \times \rho c \times c$; see Figure \ref{grainsDecompHeartOfMatter} (right). The same statement is true for each $T_\rho\in\tubes_\rho$. Let $\mathcal{P}$ denote the set of all such $\delta \times \rho c \times c$ prisms, from all $\rho$ tubes in $\tubes_\rho$. In order to bound $\mu_{\operatorname{coarse}}$, it suffices to bound the typical intersection multiplicity of the prisms in $\mathcal{P}$.

\medskip

\noindent {\bf A coarse-scale estimate Step 2: intersection multiplicity of the grains.}\\
Each $\delta \times \rho c \times c$ prism in $\mathcal{P}$ has an associated tangent plane, which is well-defined up to accuracy $\delta/(\rho c)$. Suppose that the prisms in $\mathcal{P}$ intersect ``tangentially,'' in the sense that whenever two prisms $P,P'\in\mathcal{P}$ intersect, their corresponding tangent planes agree up to accuracy $\delta/(\rho c)$. We will call this \emph{Simplifying Assumption C}. This means that for each point $x$, the set of prisms from $\mathcal{P}$ containing $x$ are contained in a common prism of dimensions roughly $\delta/\rho\times c \times c$. Thus we can partition $\mathcal{P}$ into sets, $\mathcal{P} = \bigcup \mathcal{P}_i$, with the property that if two prisms intersect then they are contained in a common set, and the $\delta \times \rho c \times c$ prisms in each set $\mathcal{P}_i$ are contained in a common prism $\square_i$ of dimensions roughly $\delta/\rho \times c \times c$; see Figure \ref{rescaledPrismsHeartOfMatter} (left).

\begin{figure}[h!]
 \includegraphics[width=.42\linewidth]{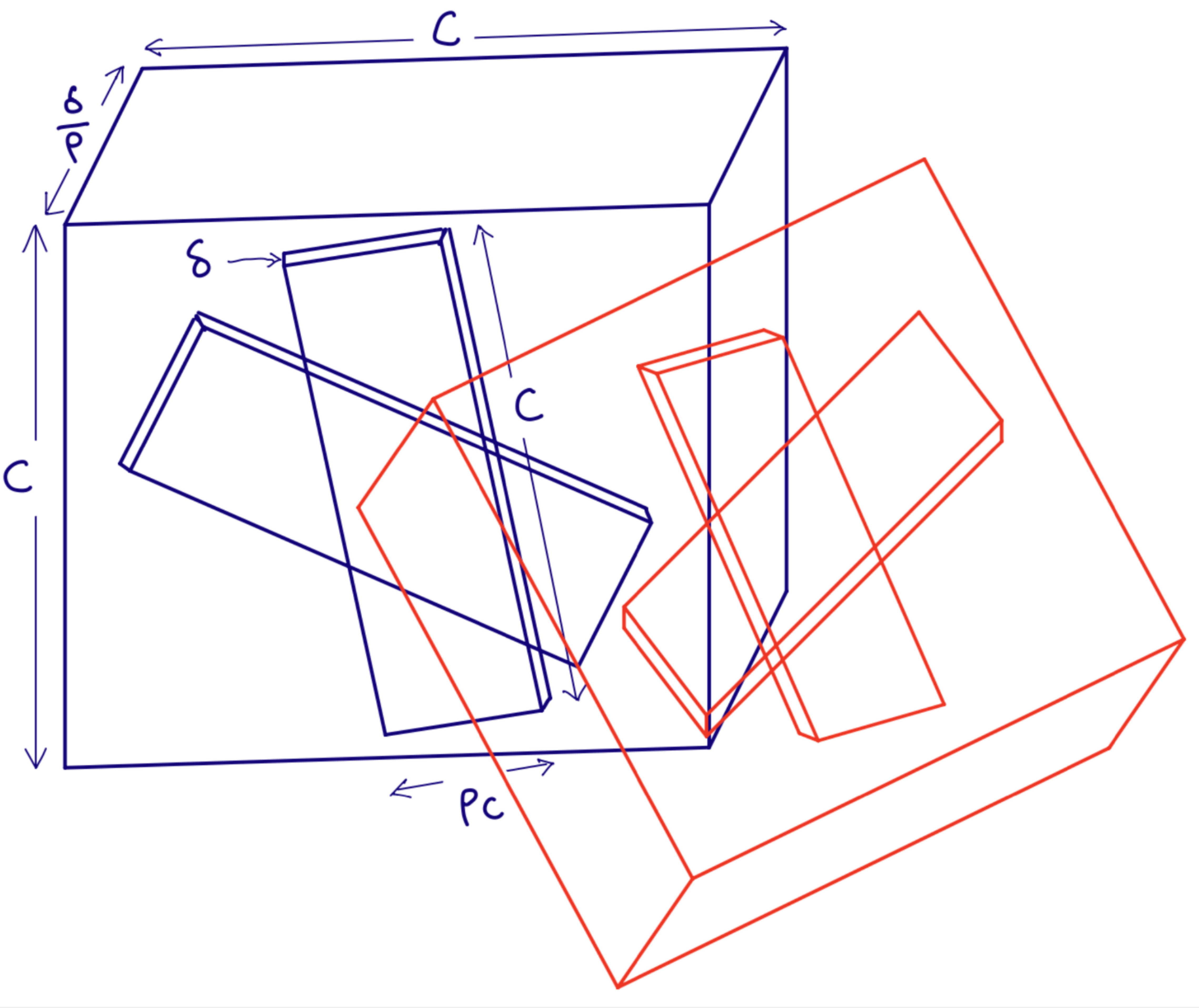}\hfill
 \includegraphics[width=.42\linewidth]{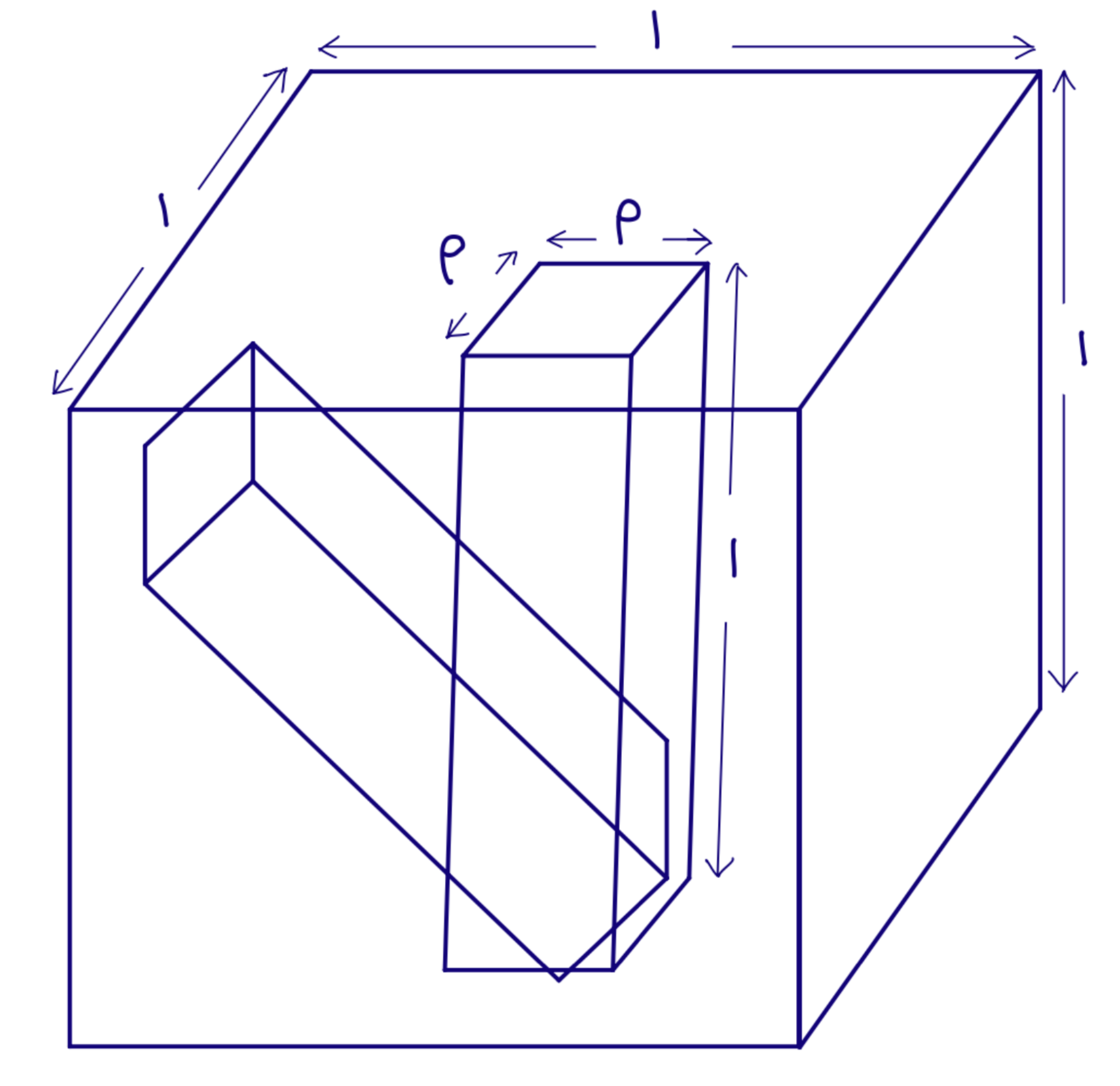}
\caption{Left: two sets of $\delta\times \rho c\times c$ prisms from the partition of $\mathcal{P}$ (blue and red, respectively), and the associated $\delta/\rho \times c \times c$ prisms ${\color{blue}\square}$ and ${\color{red}\square}$ that contain them.\\
Right: The anisotropic rescaling that maps the blue $\delta/\rho \times c \times c$ prism ${\color{blue}\square}$ to the unit cube maps each blue $\delta\times \rho c \times c$ prism to a $\rho\times\rho\times 1$ prism (this is comparable to a $\rho$ tube).}
\label{rescaledPrismsHeartOfMatter}
\end{figure}

Fix a set $\mathcal{P}'$ from the partition of $\mathcal{P}$ described above, and let $\square$ be the associated $\delta/\rho\times c\times c$ prism. The image of each $P\in\mathcal{P}'$ under the anisotropic rescaling sending $\square$ to the unit cube will be a prism of dimensions roughly $\rho\times\rho\times 1$ (see Figure \ref{rescaledPrismsHeartOfMatter} (right)). Since a $\rho\times\rho\times 1$ prism is comparable to a $\rho$ tube, we will abuse notation slightly and pretend that this set of prisms is actually a set of $\rho$ tubes; we will call this set $\tilde\tubes$. Our task of estimating $\mu_{\operatorname{coarse}}$ now reduces to estimating the typical intersection multiplicity of the tubes in  $\tilde\tubes$. 

A priori, we do not know anything about the structure of the set $\tilde\tubes$. A key new idea of our paper is a structure theorem that finds a set $\mathcal{W}$ of convex sets such that $\mathcal{W}$ obeys (a suitable analogue of) the Katz-Tao Convex Wolff Axioms with error $\lessapprox 1,$ and for each $W\in \mathcal{W}$, the set 
\[
\tilde\tubes[W]=\{\tilde T\in\tilde\tubes\colon \tilde T\subset W\}
\]
satisfies the following key properties:
\begin{enumerate}
	\item  \label{it: cardi} The cardinality estimate $\#\tilde\tubes[W] \approx \CKT(\tilde\tubes) \cdot |W|/|\tilde T|$  (here $|\tilde T|\sim\rho^2$ denotes the volume of a tube from $\tilde\tubes$). 
	\item \label{it: Frostman}  For every convex set $U\subset W$, we have $\#\tilde\tubes[U] \lessapprox \frac{|U|}{|W|} \#\tilde\tubes[W]. $
\end{enumerate}
See Figure \ref{factoringInsideBoxesFig} for an illustration of this process, and Proposition~\ref{factoringConvexSetsProp} for a precise statement.

Let's analyze a special case to see what these two properties mean. Suppose for a moment that $W$ is a $\tau$ tube for some $\rho < \tau< 1$, then Item~\ref{it: cardi} says that after rescaling $W$ to a unit cube,  $\tilde\tubes[W]$ becomes a set of $\rho/\tau$-tubes of cardinality $\gtrapprox  \CKT(\tilde\tubes) (\tau/\rho)^2$.  Item~\ref{it: Frostman} is a non-concentration condition on these tubes that was first introduced in \cite{WZ23}; families of tubes obeying this non-concentration condition are said to satisfy the Frostman Convex Wolff Axioms. For example, Items~\ref{it: cardi} and \ref{it: Frostman} are satisfied if the following holds: in each $\rho/\tau$-separated direction, we have roughly $\CKT(\tilde\tubes)$ many parallel $\rho/\tau$-tubes. This type of tube arrangement was previously considered by Wolff \cite{Wol98}, and volume estimates for unions of tubes satisfying these properties are called X-ray estimates. The Assertion $\mathcal{E}(\sigma, \omega)$, in particular $\mathcal{E}(1/2, 0)$,  is a generalization of Wolff's X-ray estimate from \cite{Wol98}. As a consequence, we should expect $\bigcup_{\tilde\tubes[W]} \tilde T$ to have a large volume if $\CKT(\tilde \tubes)$ is substantially greater than $1$. See Case 2 below for more details.

Our argument now splits into three cases.
\medskip

\noindent \emph{Case 1: $\CKT(\tilde \tubes) \lessapprox 1$.} In this case, $\mathcal{W}$ consists of a single convex set, which is comparable to the unit ball. To simplify this proof vignette, we will suppose that $\tilde\tubes$ satisfies the Frostman Slab Wolff Axioms with error $\lessapprox 1$, and thus $\tilde\tubes$ satisfies the hypothesis of Assertion $D(\sigma, \omega)$; this simplification can be justified using certain rescaling arguments that we will not detail here. In particular, this means that $\#\tilde \tubes \lessapprox \rho^{-2}$, and thus we can apply Assertion $D(\sigma, \omega)$ to obtain the desired estimate 
	\[
	\mu_{\operatorname{coarse}} \lessapprox \rho^{\sigma-\omega}(\#\tilde\tubes)^{\sigma} \lessapprox \rho^{-\sigma-\omega}. 
	\]

\medskip

\noindent \emph{Case 2: $\CKT(\tilde \tubes) \gg 1$, and each $W\in\mathcal{W}$ has thickness $t\gg\delta$}. To handle this case, we will consider the following analogy. Suppose that $\tubes$ is a set of $\delta$ tubes of cardinality $m\delta^{-2}$, for some $m\gg 1$. Suppose furthermore that $\tubes$ satisfies the Katz-Tao Convex Wolff Axioms with error $m$, and the Frostman Slab Wolff Axioms with error $\sim 1$. Then Assertion $\cE(\sigma,\omega)$ says that $\bigcup_{\tubes}T$ has volume $\gtrapprox m^{\sigma/2}\delta^{\sigma+\omega},$ which is substantially larger than $\delta^{\sigma+\omega}$. We apply a similar argument to the set of tubes $\tilde\tubes[W]$ to conclude that for each $W\in\mathcal{W}$, the union $\bigcup_{\tilde\tubes[W]}\tilde T$ has large volume (see also the discussion of the two properties above). Undoing the re-scaling described in the previous step (and illustrated in Figure \ref{rescaledPrismsHeartOfMatter}), we obtain a scale $\delta <\!\!< \tau<\!\!<\delta/\rho\geq$  (here $\tau$ depends on $t$ and the orientation of $W$ with respect to $\square$) with the property that for a typical point $x\in \bigcup_{\tubes}T$, the ball $B_\tau = B(x,\tau)$ has a large intersection with $\bigcup_{\tubes}T$. This means that we obtain an inequality of the following form:
\begin{equation}\label{eq: dense ball}
	\Big| B_\tau \bigcap \big( \bigcup_{T\in \tubes} T\big) \Big| \gtrapprox \CKT(\tilde\tubes)^{\sigma/2}  (\delta/\tau)^{\sigma+\omega} |B_\tau|. 
\end{equation}
This is precisely \eqref{eq: denseball}, provided $\CKT(\tilde\tubes)\geq \delta^{-2\alpha/\sigma}$ (this  is what we mean by $\CKT(\tubes)\gg 1$). 

Next, let $\tubes_\tau$ be a set of essentially distinct $\tau$ tubes with the property that each $T\in\tubes$ is contained in some tube from $\tubes_\tau$, and suppose that each $T_\tau\in\tubes_\tau$ contains about $(\#\tubes)/(\#\tubes_\tau)$ tubes from $\tubes$. It is straightforward to compute that $\FS(\tubes_\tau)\lessapprox 1$ (indeed, this is inherited from $\tubes$), and that $\CKT(\tubes_\tau)\approx (\#\tubes_\tau)|T_\tau|$ (this latter quantity is $\geq 1$, since $\CKT(\tubes)\lessapprox 1$ and thus at least $|T_\tau|^{-1}$ essentially distinct $\tau$ tubes are needed to cover the tubes in $\tubes$). Applying the estimate $\cE(\sigma,\omega)$ to $\tubes_\tau$, we conclude that
\[
\Big|\bigcup_{T_\tau\in\tubes_{\tau}}T_\tau\Big|\gtrapprox \tau^{\omega}\Big((\#\tubes_\tau)|T_\tau|^2\Big)^{\sigma/2} \gtrapprox \tau^{\omega+\sigma}.
\]
For the last inequality, we used the estimate $\#\tubes_\tau\gtrapprox |T_\tau|^{-1}$, which follows from the hypotheses $\CKT(\tubes)\lessapprox 1$ and $\#\tubes \sim \delta^{-2}$. 
Pairing this scale$-\tau$ estimate with our previously discussed estimate \eqref{eq: dense ball} inside balls of radius $\tau$, we obtain \eqref{desiredVolumeBdVignette}:
\begin{align*}
	\Big|\bigcup_{T\in \tubes} Y(T)\Big|  & \gtrapprox \Big|\bigcup_{T_\tau\in\tubes_{\tau}}T_\tau\Big| \cdot \CKT(\tilde\tubes)^{\sigma/2}  \left( \frac{\delta}{\tau}\right)^{\sigma+\omega}  \\
	&\gtrapprox   \delta^{\omega+\sigma} \CKT(\tilde\tubes)^{\sigma/2}. 
\end{align*}

\noindent \emph{Case 3: $\CKT(\tilde \tubes) \gg 1$, and each $W\in\mathcal{W}$ has thickness $\approx \delta$}. In this case, the grains in $\mathcal{P}$ can be replaced by larger prisms---these are the (rescaled) convex sets coming from $\mathcal{W}$. 
This process may change $\rho$ and also change the dimensions of the grains.   We iterate the argument described above with our new $\rho$ and larger grains. If we repeatedly find ourselves in Case 3 with each iteration, then the grains become wider and wider. Suppose for the moment that after a sufficient number of iterations, both $\rho$ and $c$ have size $\approx 1$. Then $\bigcup_{\tubes} T$ is organized into a union of $\delta\times 1\times 1$-slabs. From here, a straightforward geometric argument (analogous to Cordoba's proof of the Kakeya maximal function conjecture in the plane) shows that $\big| \bigcup_{\tubes} T\big|\approx 1$. If instead $c\ll 1$, then a different Cordoba type geometric argument and  the assumption that $\mu\gg 1$ (if this assumption fails, then we are done) allows us to enlarge $c$, and we iterate the argument again.

\medskip

\noindent {\bf Justifying the simplifying assumptions.}\\
We will briefly justify Simplifying Assumptions A -- C. First, if Simplifying Assumption A fails, then we can directly prove \eqref{desiredVolumeBdVignette} by using the sticky Kakeya theorem; see Section \ref{stickyKakeyaEveryScaleSec} for details. In general, Simplifying Assumption B might not hold, but if it fails, then either it is possible to directly prove \eqref{desiredVolumeBdVignette}, or else it is possible to find an intermediate scale between $\delta$ and $\rho$ at which the assumption holds; this introduces additional steps and complexity to the argument, but does not fundamentally change the flavor of the proof. 

If Simplifying Assumption C fails, then we can use a straightforward Cordoba-type geometric argument to show that for a typical prism $P_0\in\mathcal{P}$, the union of prisms $P\in\mathcal{P}$ that intersect $P_0$ fill out (most of) a thickened neighbourhood of $P_0$. This in turn means that for a typical tube $T_0\in\tubes$, the union $\bigcup_{\tubes}T$ fills out (most of) a thickened neighbourhood of $T_0$. We can then argue as in Case 2 (described above) to obtain \eqref{desiredMuBd}.

In the table below, we summarize some of the geometric objects that appeared in the arguments from Section~\ref{vignetteOfProofSection}. 

\begin{center}
	\begin{tabular}{ | m{0.8cm} | m{1.6cm} | m{1.7cm}| m{1.6cm}|  m{1.3cm}| m{1.7cm}| m{2.6cm}|m{2cm}|} 
		\hline
	   {$\!$object}   & {$\!$cardinality} & {$\!$dimensions} & {$\!$bounding box} & {$\!$union size} & {desired union size} &{multiplicity} &  {desired multiplicity } \\ 
		\hline
	 $\tubes$  & $\delta^{-2}$ &  $\delta \times \delta\times 1$ & $1\times1 \times 1$ &  $\gtrapprox\delta^{\sigma+\omega}$ & $\gtrapprox\delta^{\sigma+\omega-\alpha}$ &  $\lessapprox \delta^{-\sigma-\omega}$ & $\lessapprox\delta^{-\sigma-\omega+\alpha} $\\
		\hline
 $\tubes[T_\rho]$ & $\delta^{\nu} (\delta/\rho)^{-2} $ &$\delta\times \delta\times 1$ & $\rho\times \rho \times 1$  &  &   &  $\lessapprox\delta^{\nu \sigma} \left(\frac{\delta}{\rho}\right)^{-\sigma-\omega} $ &    \\
		\hline
		 $\mathcal{P}$  &  & $\delta\times c\rho \times c$  &  $\frac{\delta}{\rho}\times c\times c\ \ \ $ {\tiny (if tangential)}&   &   & $\mu_{\operatorname{coarse}}$ & $\lessapprox\rho^{-\sigma-\omega}$ \\
		\hline
		 $\tilde{\tubes}$  &  & $\rho \times \rho \times 1$  & $1\times 1\times 1 $ &   &   & $\mu_{\operatorname{coarse}}$ & $\lessapprox\rho^{-\sigma-\omega}$ \\
		 \hline 
	\end{tabular}
\end{center}


\subsection{Tube doubling and Keleti's line segment extension conjecture }\label{furtherConsequencesSec}
In this section we will discuss further consequences of Theorem  \ref{cDAndcEAreTrue}. We begin by introducing the Tube Doubling Conjecture (see e.g. \cite[Conjecture 15.19]{Gut16}). In what follows, if $T$ is a $\delta$ tube in $\RR^n$, then $\tilde T$ denotes the 2-fold dilate of $T$. Besicovitch constructed a set $\tubes$ of roughly $\delta^{-1}$ tubes in $\RR^2$ for which 
\begin{equation}\label{bigGainDilatedTubes}
\Big|\bigcup_{T\in\tubes}\tilde T\Big|\gtrsim \frac{\log(1/\delta)}{\log\log(1/\delta)}\Big|\bigcup_{T\in\tubes} T\Big|.
\end{equation}
This construction was adapted by Fefferman \cite{Fef71} to show that the ball multiplier is unbounded on $L^p$ for $p\neq 2$. 
The Tube Doubling Conjecture asserts that up to sub-polynomial factors, Inequality \eqref{bigGainDilatedTubes} is tight. One formulation is as follows.
\begin{conj}\label{tubeDoubling}
Let $n\geq 2$ and $\eps>0$. Then the following is true for all $\delta>0$ sufficiently small. Let $\tubes$ be a set of $\delta$ tubes in $\RR^n$. Then
\begin{equation}\label{doublingConjIneq}
\Big|\bigcup_{T\in\tubes}\tilde T\Big| \leq \delta^{-\eps}\Big|\bigcup_{T\in\tubes}T\Big|.
\end{equation}
\end{conj}
Conjecture \ref{tubeDoubling} is known in dimension two, and open in three and higher dimensions. As a consequence of Theorem  \ref{cDAndcEAreTrue}, we resolve Conjecture \ref{tubeDoubling} in $\RR^3$.

\begin{thm}\label{tubeDoublingTheoremR3}
The Tube Doubling Conjecture is true in $\RR^3$.
\end{thm}

We will discuss the proof of Theorem \ref{tubeDoublingTheoremR3} in Section \ref{extensionAndKeletiSec}. The Tube Doubling Conjecture is closely related to Keleti's Line Segment Extension Conjecture \cite{Kel16}. In the statement that follows, if $\ell$ is a line segment (by definition, line segments have positive length), then $\tilde\ell$ denotes the line containing $\ell$. 

\begin{conj}\label{lineExtension}
Let $L$ be a set of line segments in $\RR^n$. Then
\[
\dim \Big(\bigcup_{\ell \in L}\tilde\ell\phantom{.}\Big) = \dim \Big(\bigcup_{\ell\in L}\ell\Big).
\]
\end{conj}
In \cite{KelMat22}, Keleti and M\'ath\'e proved that the Kakeya set conjecture in $\RR^n$ implies Conjecture \ref{lineExtension} in $\RR^n$. As a consequence, Theorem \ref{kakeyaSetThm} has the following corollary.

\begin{thm}\label{lineSegmentExtensionThm}
Conjecture \ref{lineExtension} is true in $\RR^3$. 
\end{thm}


\subsection{Thanks}
The authors would like to thank Ciprian Demeter, Larry Guth, Nets Katz, Izabella Laba, Tuomas Orponen, Keith Rogers, Pablo Shmerkin, and Terence Tao 
for comments, suggestions, and corrections to an earlier version of this manuscript. Hong Wang would like to thank Guido de Philippis for interesting conversations. Hong Wang is supported by NSF CAREER DMS-2238818 and NSF DMS-2055544. Joshua Zahl is supported by a NSERC Discovery Grant and a NSERC Alliance Grant.


\section{A sketch of the proof}\label{proofSketchSection}
Our goal in this section is to briefly outline the major steps in the proofs of Propositions \ref{equivDE} and \ref{improvingProp}. To simplify the exposition in this proof sketch, we will gloss over many technical details and make a number of white lies. For example, we will pretend that every shading $Y(T)\subset T$ is just the trivial shading $Y(T) = T$. At the same time, we will pretend that each point $x\in\bigcup_{T\in\tubes}T$ is always contained in the same number of tubes from $\tubes$, and similarly for other collections of tubes, rectangular prisms, etc. In the same spirit as in Section \ref{vignetteOfProofSection}, we will disregard factors of the form $\delta^{\eps}$ or $\delta^{-\eps}$, and we will (somewhat informally) write $A \lessapprox B$ to mean that $A \leq C \delta^{-\eps}B$, for some constant $C$ that is independent of $\delta$ and some small parameter $\eps>0$ that we will ignore for the purposes of this sketch (in Section \ref{notationSection} we will give a precise definition of the relation $\lessapprox$, which will be used for the remainder of the proof). In the actual proof there are myriad parameters (of which $\eps$ is an example), and navigating the precise interplay between these parameters is a major technical challenge in the paper. This issue will be entirely ignored in the proof sketch. 

Finally, in this proof sketch it will be helpful to introduce ``informal versions'' of certain definitions and theorems that occur later in the paper. These informal versions are intentionally imprecise, and often are not literally true. These informal statements will be superseded by their formal counterparts that occur later in the paper. With these caveats, we now proceed as follows. 


\subsection{Proposition \ref{equivDE}: Assertions $\cD$ and $\cE$ are equivalent}\label{proofOfPropEquivDESecIntro}
Our first goal is to prove Proposition \ref{equivDE}. To do this, we will iterate the following lemma:
\begin{weakerPropEquivDELemInformal}
Let $0<\omega<\omega',$ and suppose that both $\cD(\sigma,\omega)$ and $\cE(\sigma,\omega')$ are true. Then $\cE(\sigma,\omega' - \alpha)$ is true, where $\alpha>0$ depends only on the quantities $\omega$ and $\omega'-\omega$. 
\end{weakerPropEquivDELemInformal}
To prove Proposition \ref{equivDE}, we fix $\omega$ and $\sigma$ and suppose that $\cD(\sigma,\omega)$ is true. The statement $\cE(\sigma, 2)$ is trivially true, since the volume of $\bigcup_\tubes T$ is bounded below by the volume of a single tube. We then iterate Lemma \ref{weakerPropEquivDE} multiple times to conclude that $\cE(\sigma,\omega+\eps)$ is true for every $\eps>0$, and thus $\cE(\sigma,\omega)$ is true. 

The idea behind Lemma \ref{weakerPropEquivDE} is as follows. Given a set $\tubes$ of $\delta$ tubes, our goal is to establish the estimate
\begin{equation}\label{desiredTubesEstimateSketch}
\Big|\bigcup_{T\in\tubes}T\Big| \gtrapprox \delta^{\omega'-\alpha}m^{-1}(\#\tubes)|T|\Big( m^{-3/2}\ell (\#\tubes)|T|^{1/2}\Big)^{-\sigma},
\end{equation}
with $m=\CKT(\tubes)$ and $\ell=\FS(\tubes)$. For simplicity we will pretend that every collection of tubes always satisfies $\FS(\tubes)\lessapprox 1$. Removing this assumption introduces a few additional difficulties that we will not discuss here.

If $\CKT(\tubes) \lessapprox 1$, then $\tubes$ satisfies the hypotheses of $\cD(\sigma,\omega)$, and thus we can apply the estimate $\cD(\sigma,\omega)$ to $\tubes$ and immediately obtain \eqref{desiredTubesEstimateSketch}. Suppose instead that $\CKT(\tubes)=m>\!\!> 1$. This means that there is a convex set $W$ that contains at least $m|W|\delta^{-2}$ tubes from $\tubes$. The convex set $W$ must have diameter $\geq 1$ (since it contains at least one tube), and wlog we can suppose that it has diameter $\sim 1$ (since the tubes in $\tubes$ are contained in the unit ball). Thus we may suppose that $W$ is comparable to a rectangular prism of dimensions $a\times b\times 1$, for some $\delta\leq a\leq b\leq 1$. We will focus on the most interesting case, which is when $a$ and $b$ have similar size, i.e.~$W$ is comparable to a $\rho$ tube for some $\delta\leq \rho\leq 1$. 

Motivated by the above discussion, let us explore what happens when $\CKT(\tubes) = m >\!\!> 1$; there is a scale $\delta<\!\!<\rho<\!\!< 1$; and a set $\tubes_\rho$ of $\rho$ tubes, each of which contains about $m(\rho/\delta)^2$ tubes from $\tubes$. It is straightforward to verify that $\CKT(\tubes_\rho)=O(1)$: if a convex set $W$ contains $N$ tubes from $\tubes_\rho$, then it contains about $Nm(\rho/\delta)^2$ tubes from $\tubes$. On the other hand, $W$ can contain at most $m|W|/\delta^2$ tubes from $\tubes$; see Figure \ref{denseInsideRhoTube}. Note that this situation is in some sense the opposite of the problematic situation described in Section \ref{inductionOnScaleSection} (and illustrated in Figure \ref{stickyVsWellSeparated} (right)); in that Section, we considered the scenario where there are many (i.e. far more than $\rho^{-2}$) $\rho$ tubes, each of which contains few (i.e. far fewer than $(\rho/\delta)^2$) $\delta$ tubes.

We have just shown that $\tubes_\rho$ satisfies the hypotheses of $\cD(\sigma,\omega)$, and thus
\begin{equation}\label{unionRhoTubes}
\Big|\bigcup_{T_\rho\in\tubes_\rho}T_\rho\Big| \gtrapprox \rho^{\omega}(\#\tubes_\rho)|T_\rho|\Big((\#\tubes_\rho)|T_\rho|^{1/2}\Big)^{-\sigma}.
\end{equation}
(In the above, we write $|T_\rho|\sim\rho^2$ to denote the volume of a $\rho$ tube). On the other hand, for each $T_\rho\in\tubes_\rho$, the (re-scaled) $\delta$ tubes inside $T_\rho$ will satisfy the Katz-Tao Convex Wolff Axioms with error about $m$, i.e.~$\CKT(\tubes^{T_\rho})\lesssim m = \CKT(\tubes)$. 

\begin{figure}[h!]
\begin{center}
\includegraphics[scale=0.4]{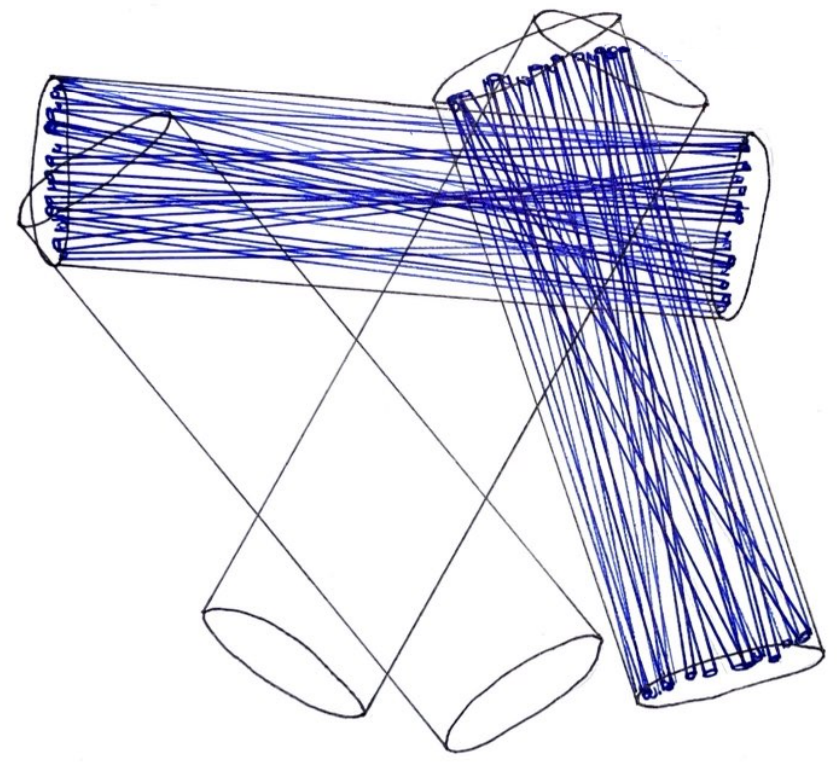}
\caption{ $\tubes_\rho$ (black), and $\tubes$ (blue). For clarity, we have only drawn the tubes from $\tubes$ inside two $\rho$ tubes. Note that the $\rho$ tubes are (comparatively) sparse, while the tubes in $\tubes[T_\rho]$ are densely packed. The situation is similar to that in Figure \ref{stickyVsWellSeparated} (left), except that the set of (rescaled) $\delta$ tubes inside each $\rho$ tube are very dense, and thus $\CKT(\tubes^{T_\rho})$ is large.}
\label{denseInsideRhoTube}
\end{center}
\end{figure}

 Applying the estimate \eqref{defnCEEstimate} from Assertion $\cE(\sigma,\omega'),$ we conclude that
\begin{equation}\label{innerVolEstimate}
\Big|\bigcup_{T^{T_\rho}\in\tubes^{T_\rho}}T^{T_\rho}\Big| \gtrapprox \Big(\frac{\delta}{\rho}\Big)^{\omega'} m^{-1}(\#\tubes[T_\rho])|T^{T_\rho}|\Big(m^{-3/2}(\#\tubes[T_\rho])|T^{T_\rho}|^{1/2}\Big)^{-\sigma}.
\end{equation}

Inequality \eqref{unionRhoTubes} says that about $\rho^{-3+\omega}(\#\tubes_\rho)|T_\rho|\Big((\#\tubes)_\rho|T_\rho|^{1/2}\Big)^{-\sigma}$ distinct $\rho$ balls are needed to cover $\bigcup_{\tubes}T$, and the RHS of \eqref{innerVolEstimate} gives a lower bound for the density of $\bigcup_{\tubes}T$ inside a typical $\rho$ ball from this collection. Combining these estimates and noting that $(\#\tubes_\rho)(\#\tubes[T_\rho]) = \#\tubes$ and $|T_\rho||T^{T_\rho}|=|T|$, we conclude that
\begin{equation}\label{improvementFromRho}
\Big| \bigcup_{T\in\tubes}T\Big| \gtrapprox \rho^{\omega-\omega'}\delta^{\omega'}m^{-1}(\#\tubes)|T|\Big(m^{-3/2}(\#\tubes)|T|^{1/2}\Big)^{-\sigma}.
\end{equation}
If $\rho<\delta^{\zeta}$ for some $\zeta>0$ bounded away from $0$, then \eqref{improvementFromRho} is precisely \eqref{desiredTubesEstimateSketch}, with $\alpha = \zeta(\omega'-\omega)$. 

This concludes the proof of Lemma \ref{weakerPropEquivDE} and hence Proposition \ref{equivDE}, except that in our proof we assumed the existence of a set of $\rho$ tubes that satisfies the following properties:
\begin{itemize}
\item[(a)] $\CKT(\tubes_\rho) = O(1)$.
\item[(b)] Each $\rho$ tube $T_\rho$ contains about $m|T_\rho|/|T|$ tubes from $\tubes$, where $m = \CKT(\tubes)$.
\item[(c)] The sets in $\tubes_\rho$ are \emph{tubes}, i.e.~they have dimensions $\rho\times\rho\times 1$.
\item[(d)] $\rho<\!\!< 1$, in the sense that $\rho = \delta^{\zeta}$ for some $\zeta>0$ bounded away from $0$. 
\end{itemize}

Unfortunately, given a set of $\delta$ tubes $\tubes$, it need not be the case that such a set of $\rho$ tubes satisfying the above properties will always exist. Consider, for example, the case where $\tubes$ is an arrangement of $\delta$ tubes of cardinality $\delta^{-5/2}$, we define $s = \delta^{5/8}$, and each of the roughly $s^{-4}$ essentially distinct $s$ tubes in $B(0,1)\subset\RR^3$ contains one $\delta$ tube from $\tubes$. Examples of this type are called the \emph{well-spaced} case. For such a set $\tubes$, there does not exist a scale $\rho$ satisfying Items (a) -- (d) above. Note, however, that a slightly different statement is true for this arrangement: There are scales $\delta\leq\tau\leq\rho$, and sets of $\tau$ and $\rho$ tubes $\tubes_\tau$ and $\tubes_\rho$ that satisfy the following:
\begin{itemize}
\item[(i)] $\tubes$ has cardinality about $m|T|^{-1}$, where $m = \CKT(\tubes)$.
\item[(ii)] $\CKT(\tubes_\rho) \lesssim (\#\tubes_\rho)|T_\rho|$.
\item[(iii)] Each $\rho$ tube $T_\rho$ satisfies $\CKT(\tubes_\tau^{T_\rho}) = O(1)$, and $\#\tubes_\tau^{T_\rho} \sim |T_\tau^{T_\rho}|^{-1} = (\rho/\tau)^2$.
\item[(iv)] Each $\tau$ tube $T_\tau$ satisfies $\CKT(\tubes^{T_\tau})\lesssim(\#\tubes[T_\tau])|T^{T_\tau}|$.
\item[(v)] $\tau<\!\!< \rho$, in the sense that $\tau = \delta^{\zeta}\rho$ for some $\zeta>0$ bounded away from $0$. 
\end{itemize}
For the well-spaced example described above, we would have  $m=\delta^{-1/2}$, $\tau=\delta$, $\rho = \delta^{1/4}$,  $\tubes_\tau =\tubes$, and  $\tubes_\rho$ is a maximal set of $\rho^{-4}$ essentially distinct $\rho$ tubes. 

The arguments described above can be adapted to this situation: By Item (ii), the $\rho$ tubes satisfy the hypothesis of Assertion $\cE(\sigma,\omega)$, and thus we obtain the volume estimate
\begin{equation}\label{sketchEstimateScaleRho}
\Big| \bigcup_{T_\rho \in\tubes_\rho}T_\rho\Big|  
\gtrapprox \rho^{\omega'} (\#\tubes_\rho)^{\sigma/2}|T_\rho|^{\sigma}.
\end{equation}
Note that the RHS of \eqref{sketchEstimateScaleRho} is precisely the estimate \eqref{defnCEEstimate} from Assertion $\cE(\sigma,\omega')$ (ignoring the multiplicative constant $\kappa$), with $m = (\#\tubes_\rho)|T_\rho|$ and $\ell = O(1)$.

By Item (iii), the $\tau$ tubes inside each $\rho$ tube satisfy the hypotheses of Assertion $\cD(\sigma,\omega)$, and thus for each $\rho$ tube $T_\rho$ we obtain the volume estimate
\begin{equation}\label{sketchEstimateScaleTau}
\Big|\bigcup_{T_\tau^{T_\rho}\in\tubes_\tau^{T_\rho}}T_\tau^{T_\rho}\Big| 
\gtrapprox \Big(\frac{\tau}{\rho}\Big)^\omega |T_\tau^{T_\rho}|^{\sigma/2}.
\end{equation}
Note that the RHS of \eqref{sketchEstimateScaleTau} is precisely the estimate \eqref{defnCDEqn} from Assertion $\cD(\sigma,\omega)$, with $\#\tubes_\tau^{T_\rho} = |T_\tau^{T_\rho}|^{-1}$.

Finally, by Item (iv), the $\delta$ tubes inside each $\tau$ tube satisfy the hypothesis of Assertion $\cE(\sigma,\omega')$, and thus for each $\tau$ tube $T_\tau$  we obtain the volume estimate
\begin{equation}\label{sketchEstimateScaleDelta}
\Big|\bigcup_{T^{T_\tau}\in\tubes^{T_\tau}}T^{T_\tau}\Big| 
\gtrapprox \Big(\frac{\delta}{\tau}\Big)^{\omega'} (\#\tubes[T_\tau])^{\sigma/2}|T^{T_\tau}|^{\sigma}.
\end{equation}
If the $\tau$ tubes are evenly distributed among $\rho$ tubes, and the $\delta$ tubes are evenly distributed among the $\tau$ tubes, then we may suppose that for each $\tau$ tube $T_\tau$ and each $\rho$ tube $T_\rho$, we have $(\#\tubes^{T_\tau})(\#\tubes_\tau^{T_\rho})(\#\tubes_\rho)=\#\tubes$. Thus we can combine \eqref{sketchEstimateScaleRho}, \eqref{sketchEstimateScaleTau}, and \eqref{sketchEstimateScaleDelta} to obtain the following analogue of \eqref{improvementFromRho}:
\begin{equation}
\begin{split}
\Big| \bigcup_{T\in\tubes}T\Big| &
\gtrapprox \Big(\frac{\tau}{\rho}\Big)^{\omega-\omega'}\delta^{\omega'} (\#\tubes)^{\sigma/2}|T|^{\sigma}\\
&=\Big(\frac{\tau}{\rho}\Big)^{\omega-\omega'}\delta^{\omega'}m^{-1}(\#\tubes)|T|\Big(m^{-3/2}(\#\tubes)|T|^{1/2}\Big)^{-\sigma},
\end{split}
\end{equation}
where the second equality used Item (i). By Item (v) we have $\tau/\rho<\delta^{\zeta}$, and thus we obtain \eqref{desiredTubesEstimateSketch} with $\alpha = \zeta(\omega'-\omega)$, as desired.

To prove Lemma \ref{weakerPropEquivDE} (and hence Proposition \ref{equivDE}), we show that for every arrangement of $\delta$ tubes, at least one of the following must hold. 
\begin{itemize}
	\item[(A)] There is a set of $\rho$ tubes satisfying Items (a) - (d) above.
	\item[(B)] There are sets of $\tau$ and $\rho$ tubes satisfying Items (i) - (v) above. 
	\item[(C)] The tubes in $\tubes$ can be efficiently packed inside rectangular prisms of dimensions $s\times t\times 1$, with $s<\!\!< t$.
	\item[(D)] The tubes in $\tubes$ satisfy the \emph{Frostman Convex Wolff Axioms at every scale} (see Definition \ref{convexAtEveryScaleFromAssouadPaper}).
\end{itemize}

To establish the above polychotomy, in Section \ref{arrangementsConvexSetsSec} we develop a general theory for ``factoring'' collections of convex sets in $\RR^n$. Given a set of $\delta$ tubes $\tubes$, this allows us to find a collection of convex sets $\mathcal{W}$ that satisfies the analogues of Items (a) and (b) above with $\mathcal{W}$ in place of $\tubes_\rho$. If these convex sets have dimensions $s\times t\times 1$ with $s<\!\!<t$, then this gives us Item (C). If instead $s\sim t$, then the convex sets in $\mathcal{W}$ are almost tubes. We apply arguments of this type at several carefully chosen scales to show that at least one of Items (A) -- (D) must hold. 

The arguments described thus far establish the desired inequality \eqref{desiredTubesEstimateSketch} in the case where (A) or (B) holds. In Section \ref{factoringTubesSection} we show that Inequality \eqref{desiredTubesEstimateSketch} holds in Case (C); this is done using a careful rescaling argument. Finally, Case (D) is precisely the setting where we can apply the Sticky Kakeya Theorem (as generalized in \cite{WZ23}) to immediately conclude that $\tubes$ satisfies \eqref{desiredTubesEstimateSketch}. 

This concludes the proof sketch of Proposition \ref{equivDE}. We now turn to Proposition \ref{improvingProp}.  


\subsection{A two-scale grains decomposition}\label{twoScaleGrainsDecompIntro}
 In Sections \ref{twoScaleGrainsSec} and \ref{moves123Sec}, we study the structure of arrangements of $\delta$ tubes for which the estimate \eqref{defnCDEqn} from Assertion $\cD(\sigma,\omega)$ is (almost) tight, i.e.~sets of $\delta$ tubes that satisfy the hypotheses of Assertion $\cD(\sigma,\omega)$, and also satisfy an inequality of the form
\[
\Big|\bigcup_{T\in\tubes}Y(T)\Big|\lessapprox \delta^{\omega}(\#\tubes)|T|\big((\#\tubes)|T|^{1/2}\big)^{-\sigma}.
\]
We will assume for now that such a set $\tubes$ exists, and at the end of Section \ref{proofSketchSection} we will arrive at a contradiction. With care, this contradiction will remain when the term $\delta^{\omega}$ is replaced by $\delta^{\omega-\nu}$ for $\nu>0$ a small positive number.

In \cite{Gut14}, Guth proved that under mild ``broadness'' hypotheses, every union of $\delta$ tubes $\bigcup_{\tubes}T$ in $\RR^3$ can be written as a disjoint union of rectangular prisms of dimensions $\delta \times c \times c$, with $c\geq( (\#\tubes)|T|^{1/2} )^{-1}$; see Figure \ref{grainsDecompHeartOfMatter} (left). This lower bound on $c$ is interesting when $\#\tubes$ is substantially smaller than $|T|^{-1}$ (recall that $|T|$ has size roughly $\delta^2$). At the opposite extreme, if $\#\tubes$ has size about $|T|^{-1/2}$ (this is the smallest possible cardinality for $\tubes$ that is allowable, given the broadness hypotheses mentioned above), then grains have dimensions roughly $\delta\times 1\times 1$. We remark that Guth's methods also yield a stronger bound of the form $c \geq \mu (  (\#\tubes)|T|^{1/2} )^{-1}$, where $\mu$ is the number of tubes from $\tubes$ that pass through a typical point, but this stronger bound won't be needed here.

First, we show that there exists a scale $\delta<\!\!<\rho<\!\!< 1$ and a set of $\rho$ tubes $\tubes_\rho$ so that both $\tubes_\rho$ and the rescaled sets $\tubes^{T_\rho}$ (recall \eqref{tubesTRhoDefn}) satisfy the hypotheses of Assertion $\cD(\sigma,\omega)$. In addition, each rescaled set $\tubes^{T_\rho}$ satisfies the broadness hypotheses needed to apply (a variant of) Guth's result. Thus we can write $\bigcup_{\tubes[T_\rho]}T^{T_\rho}$ as a disjoint union of rectangular prisms of dimensions $\delta/\rho \times c \times c$, where $c\geq ((\#\tubes[T_\rho]) |T^{T_{\rho}}|^{1/2})^{-1}$. Note that the grains become larger as $\#\tubes[T_\rho]$ becomes smaller; this numerology will be important later in the argument. Undoing the scaling, we obtain a partition of $\bigcup_{\tubes[T_\rho]}T$ into disjoint $\delta\times \rho c \times c$ rectangular prisms; we will refer to these as grains (see Figure \ref{grainsDecompHeartOfMatter} (right)), and we refer to this set of grains as $\mathcal{G}_{T_\rho}$. Let $\mathcal{G} = \bigcup_{\tubes_\rho}\mathcal{G}_{T_\rho}$. In our discussion below, we will call $\mathcal{G}$ the ``two scale Guth grains decomposition'' of $\tubes$.

Recall that in the proof vignette outlines in Section \ref{vignetteOfProofSection}, we made Simplifying Assumption D. We will now dicuss the technical steps needed to justify this assumption. The main goal of Section \ref{twoScaleGrainsSec} is to define three ``Moves,'' which we will briefly describe below. After these moves have been applied, we obtain a new scale $\rho$ with $\delta<\!\!<\rho<\!\!< 1$; a new set of $\rho$ tubes $\tubes_\rho$ that cover $\tubes$; and a new collection $\mathcal{G}$ of grains that have the following properties:
\begin{itemize}
	\item[(i)] Each grain has dimensions $\delta\times \rho c\times c$, with $ c\geq ((\#\tubes[T_\rho]) |T^{T_{\rho}}|^{1/2})^{-1}$.
	\item[(ii)] Each grain $G\in\mathcal{G}$ is associated to a unique tube $T_\rho\in \tubes_\rho,$ where $G\subset T_\rho,$ and both $G$ and $T_\rho$  point in the same direction (up to uncertainty $\rho$).
	\item[(iii)] Distinct grains from $\mathcal{G}$ associated to the same $\rho$ tube are disjoint.
	\item[(iv)] For each $\rho$ tube $T_\rho$, we have $\bigsqcup G = \bigcup_{T\in\tubes[T_\rho]}T$, where the former union is taken over the set of grains associated to $T_\rho$.
	\item[(v)] Grains associated to different $\rho$ tubes can intersect, but this intersection must be tangential; i.e.~the tangent planes of intersecting grains must agree up to uncertainty $\delta/(\rho c)$.
\end{itemize}
Item (v) means that we can cover $\RR^3$ by boxes of dimensions $\frac{\delta}{\rho}\times c\times c$, so that each grain is contained in $O(1)$ boxes, and two grains intersect only if they are contained in a common box. If we re-scale a box to become the unit cube, then the grains inside this box become $\rho\times \rho\times 1$ rectangular prisms, i.e.~$\rho$ tubes (see Figure \ref{grainsDecompHeartOfMatter}). We introduce the following notation: If $\square$ is a $\frac{\delta}{\rho}\times c\times c$ box, then $\mathcal{G}^\square$ will denote the set of $\rho$-tubes obtained by re-scaling the grains from $\mathcal{G}$ inside $\square$. With this notation, we can state one final property for $\mathcal{G}$:
\begin{itemize}
	\item[(vi)] For each box $\square$, the $\rho$ tubes in $\mathcal{G}^{\square}$ satisfy the hypotheses of $\cE(\sigma,\omega)$, and $\CKT(\mathcal{G}^{\square})\lessapprox 1$. 
\end{itemize}

In a moment, we will describe the Moves needed to find a scale $\rho$; a set of $\rho$ tubes $\tubes_\rho$; and a set of grains $\mathcal{G}$ that satisfy Items (i) -- (vi). We begin by letting $\mathcal{G}$ be the two scale Guth grains decomposition of $\tubes$, as described above. Items (i), (ii), (iii), and (iv) hold for this choice of $\mathcal{G}$, and  properties (ii)-(iv) are preserved throughout the process. 

If Item (v) fails at any point in the process, then we argue by contradiction as follows. Using a $L^2$ argument similar to Cordoba's proof of the Kakeya maximal function conjecture in the plane, we can show that there exists some scale $\tilde\delta>\!\!>\delta$ so that the ``hairbrush'' of a typical grain $G$ (i.e.~the union of the grains $G'\in\mathcal{G}$ with $G'\cap G\neq\emptyset$) fills out (most of) the $\tilde\delta$-neighbourhood of $G$. Let us pretend that the hairbrush fills out all of the $\tilde\delta$ neighbourhood of $G$. Then for each $\delta$ tube $T\in\tubes$, the corresponding $\tilde\delta$ tube $\tilde T = N_{\tilde\delta}(T)$ satisfies $\tilde T \subset \bigcup_{T'\in\tubes}T'$. Thus we can replace our original collection of $\delta$ tubes with a new collection $\tilde\tubes$ of fatter $\tilde\delta$ tubes, and $\bigcup_{\tubes}T =\bigcup_{\tilde\tubes}\tilde T$. The new collection of fatter tubes will satisfy the hypotheses of $\cE(\sigma,\omega)$ (with favorable values of $\CKT(\tilde\tubes)$ and $\FS(\tilde\tubes)$), and hence we can apply the estimate $\cE(\sigma,\omega)$ to $\tilde\tubes$ and obtain a volume estimate for $\big|\bigcup_{\tubes}T\big|$ that is superior to the estimate coming from Assertion $\cD(\sigma,\omega)$. But this contradicts the assumption that the volume estimate from Assertion $\cD(\sigma,\omega)$ was sharp for $\tubes$.

We will now describe the three Moves alluded to above. For ease of exposition, it will be helpful to introduce these Moves in the opposite order that they are defined in Section \ref{moves123Sec}.  

Move \#3  handles the situation when Item (vi) fails (recall the Assertion $\mathcal{D}(\sigma, \omega)$ is sharp for $\tubes$). Using an $L^2$ argument, we show that the hairbrush of each grain $G\in\mathcal{G}$ fills out (most of) a wider grain $\tilde G\supset G$; these wider grains have the same ``length'' $c$, but a substantially larger value of $\rho$. See Figure \ref{breakXIntoUFig} for a visual depiction of this step. 

Unfortunately, after applying Move \#3, it might be the case that $\rho$ has become so large that the inequality $\rho<\!\!<1$ is no longer true. Move \#2 handles this situation. Move \#2 uses a $L^2$ argument to find a new set of grains with a new (substantially larger) length $c$, and a new $\rho$ that satisfies  $\delta<\!\!< \rho <\!\!< 1$. See Figure \ref{longThinSlabFig} for a visual depiction of this step.

Finally, whenever the value of $\rho$ changes, so does the quantity $ ((\#\tubes[T_\rho]) |T^{T_{\rho}}|^{1/2})^{-1}$. Thus after applying Moves \#2 or \#3, it might be the case that $((\#\tubes[T_\rho]) |T^{T_{\rho}}|^{1/2})^{-1}$ has become much larger than $c$, and hence Item (i) fails. Move \#1 handles this case: we throw away our set $\mathcal{G}$ and replace it with the two scale Guth grains decomposition of $\tubes$ that was described above (both Moves \#2 and \#3 maintain the broadness condition needed to invoke the two scale Guth grains decomposition of $\tubes$). This gives us a new grains decomposition with the same value of $\rho$ and a substantially larger value of $c$.

Each of Moves \#1, \#2, and \#3 can be applied to ensure that $\mathcal{G}$ satisfies (some of) the Properties (i) -- (vi) described above. Unfortunately, the application of Move \#1, \#2, or \#3 might destroy other Properties. However, each Move either substantially increases the ``length'' $c$ of the grains, or maintains the length and substantially increases the value of $\rho$. Since $c$ and $\rho$ are bounded above by 1, the process of applying Moves \#1, \#2, and \#3 must halt after a bounded number of steps. The resulting grains decomposition satisfies Properties (i) -- (vi).


\subsection{Refined induction on scales}
In Section \ref{refinedInductionOnScaleSec} we use the two-scale grains decomposition from Section \ref{twoScaleGrainsSec} to apply the estimate from Assertion $\cE(\sigma,\omega)$ at two different scales --- once to the (rescaled) $\delta$ tubes inside each $\rho$ tube, and once to the $\rho$ tubes arising as the re-scaled grains inside each box $\square,$ i.e.~to each arrangement $\mathcal{G}^{\square}$. This is a critical step in the proof of Proposition \ref{improvingProp}, and the entire proof up to this point was carefully structured in order to allow us to apply the estimate $\cE(\sigma,\omega)$ to $\mathcal{G}^{\square}$.

The argument is as follows. Suppose that $\tubes$ is a set of $\delta$ tubes for which the estimate from Assertion $\cD(\sigma,\omega)$ is tight, and let $\tubes_\rho$ and $\mathcal{G}$ be the grains decomposition described in the previous section. Employing a small white lie, we can suppose that there is a number $\mu$ so that each point $x\in \bigcup_{\tubes}T$ is contained in $\sim \mu$ tubes from $\tubes$. We have $\big|\bigcup_{\tubes}T\big| \sim \mu^{-1}(\#\tubes)|T|$, so our goal is to obtain an upper bound for $\mu$. We will suppose there is a number $ \mu_{\operatorname{fine}}$ so that for each $T_\rho$, each point $x\in\bigcup_{\tubes[T_\rho]}T$ is contained in $\sim \mu_{\operatorname{fine}}$ tubes from $\tubes[T_\rho]$.  Finally, we will suppose there is a number $\mu_{\operatorname{coarse}}$ so that each point $x\in \bigcup_{\tubes} T = \bigcup_{G\in\mathcal{G}}G$ is contained in about $\mu_{\operatorname{coarse}}$ grains from $\mathcal{G}$. By Items (ii) and (iv) from Section \ref{twoScaleGrainsDecompIntro}, we have $\mu\lesssim \mu_{\operatorname{fine}}\mu_{\operatorname{coarse}}$, and thus our task is to estimate the latter two quantities. 

Since each rescaled set $\tubes^{T_\rho}$ satisfies the hypotheses of Assertion $\cD(\sigma,\omega)$, we have the estimate 
\begin{equation}\label{introEstimateMuFine}
\mu_{\operatorname{fine}}\lessapprox (\delta/\rho)^{-\omega}\Big(\#\tubes[T_\rho]|T^{T_{\rho}}|^{1/2}\Big)^{\sigma},
\end{equation}
where $\#\tubes[T_\rho]$ has size roughly $(\#\tubes)/(\#\tubes_\rho)$ and $|T^{T_{\rho}}|=|T|/|T_\rho|.$

Our next task is to estimate $\mu_{\operatorname{coarse}}$. We apply $\cE(\sigma,\omega)$ to each set of $\rho$ tubes $\mathcal{G}^{\square}$. (We must use the estimate $\cE(\sigma,\omega)$ rather than $\cD(\sigma,\omega)$, since $\FS(\mathcal{G}^{\square})$ might be large, which entails a separate argument. We will gloss over this issue.) Doing so gives the estimate
\begin{equation}\label{introEstimateMuCoarseGood}
\mu_{\operatorname{coarse}}\lessapprox \rho^{-\omega}\big((\#\mathcal{G}^{\square}) |T_\rho|^{1/2}\big)^\sigma\lessapprox \rho^{-\omega}|T_\rho|^{-\sigma/2}.
\end{equation}
The second inequality in \eqref{introEstimateMuCoarseGood} follows from the fact that $\CKT(\mathcal{G}^{\square})\lessapprox 1 $, and hence $\#\mathcal{G}^{\square} \lessapprox  |T_\rho|^{-1}$. 
Combining \eqref{introEstimateMuFine} and \eqref{introEstimateMuCoarseGood}, we conclude that
\begin{equation}\label{betterMuBd}
\mu\lessapprox  \Big[\delta^{-\omega}\big((\#\tubes)|T|^{1/2}\big)^\sigma\Big]\Big[|T_\rho|(\#\tubes_\rho)\Big]^{-\sigma}.
\end{equation}

The first term in square brackets is the estimate that would follow from applying Assertion $\cD(\sigma,\omega)$ directly to $\tubes$. Thus \eqref{betterMuBd} yields a superior estimate precisely when $\#\tubes_\rho>\!\!>|T_\rho|^{-1}$. Since we assumed that $\tubes$ is a set of tubes for which $\cD(\sigma,\omega)$ is tight, we conclude that $\#\tubes_\rho \lessapprox |T_\rho|^{-1}$. 

The above step was simplified to highlight the main ideas. In reality, we actually need (and prove) a slightly stronger statement: rather than concluding that $\#\tubes_\rho\lessapprox |T_\rho|^{-1}$, we must instead arrive at the estimate $\CKT(\tubes_\rho)\lessapprox 1$. This more difficult estimate is obtained as follows. Suppose to the contrary that $\CKT(\tubes_\rho)>\!\!>1$. Then we can find a convex set $W$ so that $\tubes_\rho[W]$ has cardinality much larger than $|W|/|T_\rho|$ (in fact, we can find many such sets $W$---see Proposition~\ref{factoringConvexSetsProp}). The argument described above is the special case when $W$ is comparable to the unit ball. The general case introduces technical challenges, but in light of the techniques already developed in Sections \ref{arrangementsConvexSetsSec} and \ref{factoringTubesSection} to prove Proposition \ref{equivDE} (see the discussion at the end of Section \ref{proofOfPropEquivDESecIntro}), it does not require any additional new ideas.


\subsection{Multi-scale structure, Nikishin-Stein-Pisier factorization, and Sticky Kakeya}\label{endOfProofSketchSec}
Let us summarize the conclusion of the previous steps: if  $\tubes$ is a set of $\delta$ tubes for which the estimate $\cD(\sigma,\omega)$ is tight, then there is a scale $\delta<\!\!<\rho<\!\!< 1$ and a set of $\rho$ tubes $\tubes_\rho$ with $\CKT(\tubes_\rho)\lessapprox 1$, so that both $\tubes_\rho$ and each (rescaled) set $\tubes[T_\rho]$ satisfy the hypotheses of Assertion $\cD(\sigma,\omega)$, and furthermore, the estimate $\cD(\sigma,\omega)$ is tight for all of these arrangements of tubes. 
 
This last conclusion means that we can iteratively apply the same argument to both $\tubes_\rho$ and each (rescaled) set $\tubes[T_\rho]$. After some pruning, we conclude that there is a sequence of closely spaced scales $\delta = \rho_N < \rho_{N-1} < \ldots < \rho_0 = 1$ and sets $\{\tubes_{\rho_i}\}_{i=1}^N$ covering $\tubes$, with $\CKT(\tubes_{\rho_i})\lessapprox 1$ for each index $i$. 

We would like to apply the Sticky Kakeya Theorem to conclude that $\big|\bigcup_{\tubes}T\big|$ is almost as large as $\sum_{\tubes}|T|$. Indeed, the situation described above almost matches the setup of the Sticky Kakeya Theorem, as generalized in \cite[Theorem 1.8]{WZ23}. Specifically, $\tubes$ would satisfy the hypotheses of \cite[Theorem 1.8]{WZ23} if $\#\tubes\approx \delta^{-2}$. Since $\CKT(\tubes) \lessapprox 1$, we know that $\#\tubes \lessapprox \delta^{-2}$. Unfortunately, however, it could be the case that $\#\tubes$ is much smaller than $\delta^{-2}$.

In Section \ref{stickyKakeyaEveryScaleSec} we use a Nikishin-Stein-Pisier factorization argument to show that if $\#\tubes<\!\!<\delta^{-2}$, then we can construct a new set $\hat\tubes$ consisting of a union of about $\delta^{-2}(\#\tubes)^{-1}$ randomly translated and rotated copies of $\tubes$. This new set $\hat\tubes$ will have cardinality about $\delta^{-2}$. Just like the original set $\tubes$, the new set $\hat\tubes$ will have a sequence of covers $\{\hat\tubes_{\rho_i}\}_{i=1}^N$ with $\CKT(\hat\tubes_{\rho_i})\lessapprox 1$ for each index $i$. Hence we can apply the Sticky Kakeya Theorem to $\hat\tubes$ to conclude that $\big|\bigcup_{T\in\hat\tubes}T\big|\gtrapprox 1$.  Since the volume of $\bigcup_{\tubes}T$ is invariant under translation and rotation (this is a key ingredient for Nikishin-Stein-Pisier factorization), we conclude that 
\begin{equation}\label{goodVolumeBoundTubes}
\big|\bigcup_{T\in\tubes}T\big|\gtrapprox (\#\tubes)|T|.
\end{equation}

But if $\sigma,\omega>0$, then \eqref{goodVolumeBoundTubes} contradicts the assumption that the estimate $\cD(\sigma,\omega)$ is tight for $\tubes$. We conclude that when $\sigma,\omega>0$, there does not exist \emph{any} set $\tubes$ satisfying the hypotheses of Assertion $\cD(\sigma,\omega)$ for which the estimate $\cD(\sigma,\omega)$ is tight. The quantitative version of this statement is Proposition \ref{improvingProp}.


\section{Notation}\label{notationSection}
In the arguments that follow, $\delta>0$ will denote a small positive quantity. Overriding the (informal) notation from Sections \ref{introductionSection} and \ref{proofSketchSection}, we write $A(\delta)\lessapprox_\delta B(\delta)$ if for all $\eps>0$, there exists $K_\eps>0$ so that $A(\delta)\leq K_\eps \delta^{-\eps}B(\delta)$. If the role of $\delta$ is apparent from context, we will often write $A\lessapprox B$. For example if $K$ is a constant independent of $\delta$, then $\log(1/\delta)^K\lessapprox 1$. Similarly, $e^{\sqrt{\log 1/\delta}}\lessapprox 1$. 

In some sections of the paper, it will ease notation to fix certain variables (for example the values of $\sigma$ and $\omega$ from Definition \ref{defnCDE}). In such cases, we will clearly state which variables are fixed, and use bold font throughout that section to denote these fixed variables, and also to denote quantities that depend only on fixed variables. For example we might define $\boldsymbol{\beta} = \boldsymbol{\sigma\omega/100}$.


\subsection{Convex sets and shadings}
In the introduction, we defined a $\delta$-tube to be the $\delta$ neighbourhood of a unit line segment. There are several other types of convex sets that will make frequent appearances in our arguments. A \emph{prism} is a rectangular prism in $\RR^n$ (usually $\RR^3$); we will denote the dimensions by $a\times b\times c\times\ldots,$ with the convention that $a\leq b\leq c\leq\ldots$. Informally, we say a prism in $\RR^3$ is ``flat'' if it has dimensions $a\times b\times c$ with $a<\!\!<b$, and we say it is ``square'' if $b$ and $c$ have comparable size. Finally, we will sometimes refer to the quantities $a$, $b$, and $c$ respectively as the ``thickness,'' ``width,'' and ``length'' of a prism.

Rather than working with rectangular prisms, it will sometimes be convenient to work with ellipsoids, or more general convex sets. This motivates the following definition, which generalizes the definition of $(\tubes,Y)_\delta$ from the introduction. 

\begin{defn}
For $0<a\leq b\leq c$, we write $(\mathcal{P},Y)_{a\times b\times c}$ to denote the following pair: $\mathcal{P}$ is a set of essentially distinct convex subsets of $\RR^3$; for each $P\in\mathcal{P}$, the outer John ellipsoid of $P$ has axes of lengths comparable to $a,b,$ and $c$ respectively. $Y$ is a shading on $\mathcal{P}$, i.e. for each $P\in\mathcal{P}$, we have $Y(P)\subset P$. 

For example, we could write $(\tubes,Y)_{\delta}$ as $(\tubes,Y)_{\delta\times\delta\times 1}$. Finally, we say $(\mathcal{P},Y)_{a\times b\times c}$ is $\lambda$ \emph{dense} if $\sum_{P\in\mathcal{P}}|Y(P)|\geq\lambda\sum_{P\in\mathcal{P}}|P|$. 
\end{defn}

\begin{defn}
If $(\mathcal{P},Y)_{a\times b\times c}$ is a set of prisms and their associated shading and $x\in\RR^3,$ we define
\[
\mathcal{P}_Y(x)=\{P\in\mathcal{P}\colon x\in Y(P)\}.
\]
Similarly, if $\mathcal{P}$ is a set of prisms (or more generally, convex sets) and no shading is present, then we define $\mathcal{P}(x)=\{P\in\mathcal{P}\colon x\in P\}.$
\end{defn}

\begin{defn}
We say a pair $(\mathcal{P}',Y')_{a\times b\times c}$ is a $t$-refinement of $(\mathcal{P},Y)_{a\times b\times c}$ if $\mathcal{P}'\subset\mathcal{P}$; $Y'(P)\subset Y(P)$ for each $P\in\mathcal{P}'$, and $\sum_{P'\in\mathcal{P}'}|Y'(P')|\geq t \sum_{P\in\mathcal P}|Y(P)|$.  In practice, we will often have $t \approx_\delta 1$, in which case we will call it a $\approx_\delta 1$ refinement. 
\end{defn}
\noindent Note that if $(\mathcal{P},Y)_{a\times b\times c}$ is $\lambda$ dense and $(\mathcal{P}',Y')_{a\times b\times c}$ is a $t$-refinement, then $\#\mathcal{P}'\geq \lambda t(\#\mathcal{P})$.

\begin{defn}
If $W\subset\RR^3$ is a convex set whose outer John ellipsoid $E$ has dimensions $a\times b\times c$, we write $\operatorname{dir}(W)\in \operatorname{Gr}(1; \RR^3)$ and $\Pi(W)\in \operatorname{Gr}(2; \RR^3)$ to denote the 1 and 2-dimensional subspaces of $\RR^3$ spanned by the primary and secondary axes of $E$. We have that $\operatorname{dir}(W)$ is meaningfully defined up to accuracy $b/c$, and $\Pi(W)$ is meaningfully defined up to accuracy $a/b$. For example, if $T$ is a $\delta$ tube, then $\dir(T)$ is meaningfully defined up to accuracy $\delta$, while $\Pi(T)$ is only meaningfully defined up to accuracy $1$ (i.e.~$\Pi(T)$ is not a meaningful quantity if $T$ is a $\delta$ tube).
\end{defn}

We will employ the following synecdoche notation: if $\mathcal{P}$ (resp.~$\tubes$, $\mathcal{W}$, etc.) is a collection of convex sets, each of the same volume, then we will use $|P|$ (resp.~$|T|$, $|W|$, etc) to denote the volume of one of these convex sets. In practice, we will abuse notation slightly and continue to employ this notation if the sets in $\mathcal{P}$ have comparable (but not necessarily identical) volume.

\begin{defn}\label{defnPhiW}
Let $W\subset\RR^n$ be a convex set. We define $\phi_W\colon\RR^n\to\RR^n$ to be an affine-linear transformation that maps the outer John ellipsoid of $W$ to the unit ball. For concreteness, if $v_1,\ldots,v_n$ are the axes of the John Ellipsoid, with lengths $\ell_1\leq\ldots\leq \ell_n$, then we select $\phi_W$ so that the $j$-th axis of the John Ellipsoid is mapped to the $x_j$ axis in $\RR^n$. If two more more axes have the same length, then we pick an ordering arbitrarily. 

If $U\subset\RR^n$, we define $U^W = \phi_W(U)$. In particular, if $U$ is a convex subset of $W$ then $U^W$ is a convex subset of the unit ball, and $|U^W|\sim |U|/|W|$. This is compatible with our earlier definition of $T^{T_\rho}$ from \eqref{tubesTRhoDefn}.
\end{defn}

\begin{defn}
Let $\mathcal{U}$ be a collection of convex subsets of $\RR^n$ and let $W$ be a convex subset of $\RR^n$. We define 
\[
\mathcal{U}[W] = \{U\in\mathcal{U}\colon U\subset W\},
\]
and
\[
\mathcal{U}^W = \{U^W\colon U \in \mathcal{U}[W]\}.
\]

If $Y$ is a shading on $\mathcal{U}$, we will use $Y^W$ to denote the corresponding shading on $\mathcal{U}^W$, i.e.~for each $U^W\in\mathcal{U}^W$, we define $Y^W(U^W)=\phi_W(Y(U))$. 
\end{defn}

\begin{rem}
The expression $\mathcal{U}^W $ should not be confused with $\mathcal{U}_W$; the latter notation will be as follows: If $\mathcal{U}$ and $\mathcal{W}$ are sets of convex subsets of $\RR^n$, then $\mathcal{U}_W,\ W\in\mathcal{W}$ will be used to denote a set of subsets of $\mathcal{U}$ that are indexed by the elements of $\mathcal{W}$.
\end{rem}


\subsection{Table of notation}
To aid the reader, we will use certain notation conventions throughout this paper. For example, some symbols (such as $\sigma$ and $\omega$) will be reserved to always have the same meaning. For future reference, we record these notation conventions in the table below

\begin{center}
\begin{longtable}{ | m{2cm} | m{13cm}|  } 
 \hline
{\bf Symbol} & {\bf Meaning}  \\ 
 \hline 
 $\delta,\rho,\tau$ & These variables will denote scales. Typically $\delta\leq\rho\leq\tau$. \\  
  \hline 
 $a,b,c$ & These variables will denote scales; typically the dimensions of a prism.\\
  \hline 
 $\theta$ & $\theta$ will denote an angle\\
\hline
$\eps,\eta,\zeta,\alpha$ & These variables will represent (typically small) exponents, i.e.~they will appear in the form $\delta^\eta$, $\rho^\eps$, etc.\\
\hline
$\kappa,K$ & These variables will represent (positive) multiplicative constants, i.e.~ $|\bigcup T|\geq\kappa\delta^{\eps}$ or $\CKT(\tubes)\leq K\delta^{-\eta}$. Typically $\kappa>0$ is small and $K>\!\!>1$ is large.\\
\hline 
$\sigma,\omega$ & $\sigma$ and $\omega$ and their variants $\sigma',\tilde\sigma,$ etc.~will always be quantities related to the estimates $\cE(\sigma,\omega)$ and $\cD(\sigma,\omega)$. \\
\hline
 $\boldsymbol{\sigma}$, $\boldsymbol{\omega}$   & In Sections \ref{twoScaleGrainsSec} and \ref{moves123Sec}, we will fix values of $\sigma$ and $\omega$ that are kept constant throughout that section. We use bold symbols to denote these fixed numbers, and all subsequent quantities that depend (only) on them. \\
\hline
$T,P,G,S,\square$ & These variables will denote convex sets. Typically $T$ is a tube, $P$ and $G$ are prisms of dimensions $a\times b\times c$, $S$ is a slab, and $\square$ is a ``box'' of dimensions $a\times c\times c$. We use symbols $\tubes,\mathcal{P},\mathcal{G},\mathcal{S}$ to denote sets of such objects.\\
  \hline 
 $\tubes',\tubes_1,\tilde\tubes$ & $\tubes'$ or $\tubes_1$ will denote a subset of $\tubes$. Similarly $\tubes_2$ will denote a subset of $\tubes_1$, etc. $\tilde\tubes$ will denote a new set of tubes that is related to $\tubes$, but not necessarily a subset (for example, $\tilde\tubes$ might consist of the 2-fold dilates of the tubes in $\tubes$).\\
 \hline 
 $(\tubes',Y')_\delta$ & $(\tubes',Y')_\delta$ will denote a refinement of $(\tubes,Y)_\delta$. Similarly for $(\tubes_1,Y_1)_\delta$.\\
\hline
\end{longtable}
\end{center}


\section{Wolff Axioms and Factoring Convex Sets}\label{arrangementsConvexSetsSec}

\subsection{Definitions: Wolff axioms and covers}
\begin{defn}\label{defnOfCover}
Let $\mathcal{U},\mathcal{W}$ be collections of convex sets in $\RR^n$. 
\begin{itemize}
\item[(A)] We say that $\mathcal{W}$ is a \emph{cover} of $\mathcal{U}$ (or $\mathcal{W}$ \emph{covers} $\mathcal{U}$) if $\bigcup_{W\in \mathcal{W}}\mathcal{U}[W]=\mathcal{U}$. We will denote this by $\mathcal{U} \prec \mathcal{W}$.
%
\item[(B)] We say that $\mathcal{W}$ is a $K$-\emph{almost partitioning cover} (resp.~partitioning cover) if it is a cover, and furthermore each $U\in\mathcal{U}$ is contained in at most $K$ sets (resp.~1 set) of the form $\mathcal{U}[W]$.
\item[(D)] We say that $\mathcal{W}$ is a $K$-balanced cover (resp.~balanced cover) if it is a cover, and furthermore the numbers $|W|^{-1}\sum_{U\in\mathcal{U}[W]}|U|$ are comparable for all $W\in\mathcal{W}$, up to a multiplicative factor of $K$ (resp.~2).
\end{itemize}
\end{defn}

\noindent The following is a mild generalization of Definition \ref{KatzTaoAndFrostmanTubesDefn}.
\begin{defnConvexPrime}
Let $\mathcal{U}$ and $\mathcal{W}$ be collections of convex subsets of $\RR^n$.\\
(A) We define the \emph{Katz-Tao Wolff constant of $\mathcal{U}$ with respect to $\mathcal{W}$} to be the infimum of all $C>0$ so that
\begin{equation}\label{defnKTWConst}
\sum_{U\in \mathcal{U}[W]}|U| \leq C |W|\quad\textrm{for all}\ W\in\mathcal{W}.
\end{equation}

\medskip

\noindent (B) We define the \emph{Frostman Wolff constant of $\mathcal{U}$ with respect to $\mathcal{W}$} to be the infimum of all $C>0$ so that 
\begin{equation}\label{FWConst}
\sum_{U\in \mathcal{U}[W]}|U| \leq C |W|\sum_{U\in\mathcal{U}}|U|\quad\textrm{for all}\ W\in\mathcal{W}.
\end{equation}
\end{defnConvexPrime}

\begin{rem}\label{remarksFollowingConvexWolffDefn}$\phantom{1}$\\
(A) To ease notation, we define $\CKT(\mathcal{U})$ (resp.~$\CFC(\mathcal{U})$) to be the Katz-Tao (resp.~Frostman) Wolff constant of $\mathcal{U}$ associated to the set $\mathcal{W}$ of convex subsets of $\RR^n$. We define $\FS(\mathcal{U})$ to be the Frostman Wolff constant of $\mathcal{U}$ associated to the set $\mathcal{W}$ of slabs in $\RR^n$. Note that these definitions are compatible with those from Definition \ref{KatzTaoAndFrostmanTubesDefn}.  

\medskip
\noindent (B) A set $\tubes$ of $\delta$-tubes obeys the Wolff axioms, in the sense of \cite{Wol95} (see Property $(*)$ on p655 and the preceding discussion) if the Katz-Tao Wolff constant of $\tubes$ is small with respect to the set $\mathcal{W}$ consisting of all rectangular prisms of dimensions $10\delta\times \rho\times\ldots\times\rho\times 2$, with $0<\delta\leq\rho\leq 2$. 

\medskip
\noindent (C) For some arguments, it will be useful to consider an analogue of the above definitions where the quantity $|W|$ on the RHS of \eqref{defnKTWConst} is replaced by $|W\cap B(0,1)|/|B(0,1)|$, and similarly for \eqref{FWConst}. This leads to a quantity that transforms naturally under affine maps such as $\phi_W$ from Definition \ref{defnPhiW}.

\medskip

\noindent (D) Note that the above definitions continue to make sense if $\mathcal{U}$ is a multiset. This will be useful in Section \ref{factorizationArgSection}.

\medskip 
\noindent  (E) If the set $\mathcal{U} \neq \emptyset $ consists of convex sets of the same size, then $\CFC(\mathcal{U})\leq C$ implies that $\#\mathcal{U} \geq C^{-1} |U|^{-1}$. To see this, take $W$ to be a convex set in $\mathcal{U}$. Then the LHS of \eqref{FWConst} equals to $|U|$ while the RHS of $\eqref{FWConst}$ equals to $C|U|^2 (\#\mathcal{U})$.  Roughly speaking, if $\CKT(\mathcal{U})$ is small, then $\mathcal{U}$ is ``sparse'', while if $\CFC(\mathcal{U})$ is small, then $\mathcal{U}$ is ``dense.'' 

\end{rem}

\medskip

\begin{rem}\label{FrostmanWolffInheritedUpwardsDownwards}$\phantom{1}$\\
(A) The Frostman Wolff constant is ``inherited upwards'' by covers. More precisely, if $\mathcal{U}$ and $\mathcal{W}$ are collections of convex subsets of $\RR^n$, and if $\mathcal{W}$ is a $K$-balanced cover of $\mathcal{U}$, then
\begin{equation}\label{FrostmanWolffInheritedUpwardsEqn}
\CFC(\mathcal{W})\lesssim K \CFC(\mathcal{U})\quad\textrm{and}\quad \FS(\mathcal{W})\lesssim K \FS(\mathcal{U}).
\end{equation}

\medskip
\noindent (B)
The Katz-Tao Wolff constant is ``inherited downwards'' by covers. More precisely, if $\mathcal{U}$ is a collection of convex subsets of $\RR^n$, and if $W$ is a convex subset of $\RR^n$, then
\begin{equation}
\CKT(\mathcal{U}^W) = \CKT(\mathcal{U}[W]) \leq \CKT(\mathcal{U}).
\end{equation}

\medskip
\noindent (C) The Frostman Slab Wolff Constant is ``sub-multiplicative'' with respect to covers. More precisely, if $\mathcal{U}\prec\mathcal{V}$ are collections of convex subsets of a convex set $W\subset \RR^n$, then in some situations we have that $\FS(\mathcal{U}^W)$ is controlled by $\max_{V\in\mathcal{V}}\FS(\mathcal{U}^V)\FS(\mathcal{V}^W)$. In certain special cases, the same is true for the Katz-Tao Convex Wolff Constant.  See Section \ref{frostmanSlabTransitiveUnderCoversSec} for a precise statement. 
\end{rem}


\subsection{Factoring Convex Sets}\label{factorConvexSetsSec}

As we have observed in Remark \ref{FrostmanWolffInheritedUpwardsDownwards}, Frostman Wolff constants are inherited upwards, while Katz-Tao Wolff constants are inherited downwards. The following definition will help us exploit this observation when performing multi-scale analysis and induction on scale. 

\begin{defn}\label{factoringDefn}
Let $\mathcal{U}$ and $\mathcal{W}$ be collections of convex subsets of $\RR^n$, and let $K>0$. \medskip

\noindent (A) We say that $\mathcal{W}$ \emph{factors} $\mathcal{U}$ \emph{from above with respect to the Katz-Tao (resp.~Frostman) Convex Wolff axioms with error} $K$ if $\mathcal{W}$ covers $\mathcal{U}$, and $\mathcal{W}$ satisfies the Katz-Tao (resp.~Frostman) Convex Wolff axioms with error $K$. 

\medskip
\noindent(B) We say that $\mathcal{W}$ \emph{factors} $\mathcal{U}$ \emph{from below with respect to the Katz-Tao (resp.~Frostman) Convex Wolff axioms with error} $K$ if $\mathcal{W}$ covers $\mathcal{U}$, and for each $W\in\mathcal{W}$ the set  $\mathcal{U}^W$ satisfies the Katz-Tao (resp.~Frostman) Convex Wolff axioms with error $K$. 

\medskip
\noindent(C) We say that $\mathcal{W}$ \emph{factors} $\mathcal{U}$ \emph{from above (resp.~below) with respect to the Katz-Tao (or Frostman) Slab Wolff axioms with error} $K$ if the natural analogue of (A) (resp.~(B)) holds, where the Convex Wolff axioms are replaced by Slab Wolff axioms. 
\end{defn}

\begin{rem}
Definition \ref{factoringDefn} highlights a few special cases of a more general definition: If $\mathcal{U},$ $\mathcal{W},$ and $\mathcal{V}$ are collections of convex subsets of $\RR^n$, we can define what it means for $\mathcal{W}$ to factor $\mathcal{U}$ from above (or below) with respect to the Katz-Tao (or Frostman) Wolff axioms with respect to $\mathcal{V}$. Item (A) and (B) in Definition \ref{factoringDefn} correspond to the special case where $\mathcal{V}$ is the collection of convex sets in $\RR^n$, while Item (C) corresponds to the case where $\mathcal{V}$ is the collection of slabs in $\RR^n$.
\end{rem}

Definition \ref{KatzTaoAndFrostmanTubesDefn}$^\prime$ was carefully formulated to allow the following result, which says that for every collection $\mathcal{U}$ of convex subsets of $\RR^n$, there exists some $\mathcal{W}$ that factors $\mathcal{U}$ from below with respect to the Frostman Convex Wolff axioms, and from above with respect to the Katz-Tao Convex Wolff axioms, both with small error. The precise statement is as follows.


\begin{prop}\label{factoringConvexSetsProp}
Let $\mathcal{U}$ be a finite set of congruent convex subsets of the unit ball in $\RR^n$, each of which contains a ball of radius $\delta$. Let $K = 100^n e^{100 \sqrt{\log(\delta^{-1} \# \mathcal{U})}}$ (the exact shape of $K$ is not important; what matters is that if $\#\mathcal{U}\leq\delta^{-100}$, then $K\lessapprox_\delta 1$). 

Then there exists a set $\mathcal{W}$ of congruent convex subsets of $\RR^n$ and a set $\mathcal{U}'\subset\mathcal{U}$ with the following properties:
\begin{enumerate}[i)]
	\item $\# \mathcal{U}' \geq K^{-1}(\#\mathcal{U})$.
	\item $\mathcal{W}$ is a $K$-balanced, $K$-almost partitioning cover of $\mathcal{U}'$, and 
	\begin{equation}\label{lotsOfUinW}
		\#\mathcal{U}'[W]\geq K^{-1}\CKT(\mathcal{U}')|W||U|^{-1}\quad\textrm{for each}\ W\in\mathcal{W}.
	\end{equation} 
	\item $\mathcal{W}$ factors $\mathcal{U}'$ from above respecting the Katz-Tao Convex Wolff Axioms with error $K$.
	\item $\mathcal{W}$ factors $\mathcal{U}'$ from below respecting the Frostman Convex Wolff Axioms with error $K$.
\end{enumerate}
\end{prop}

\begin{figure}
\includegraphics[width=.42\linewidth]{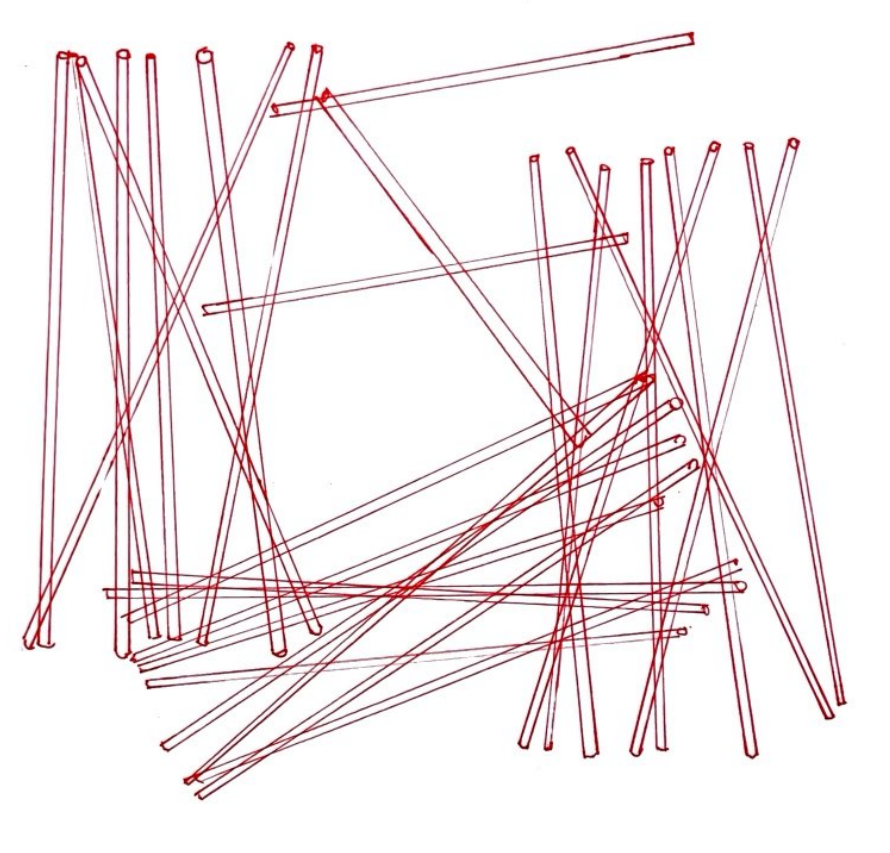}\hfill
\includegraphics[width=.42\linewidth]{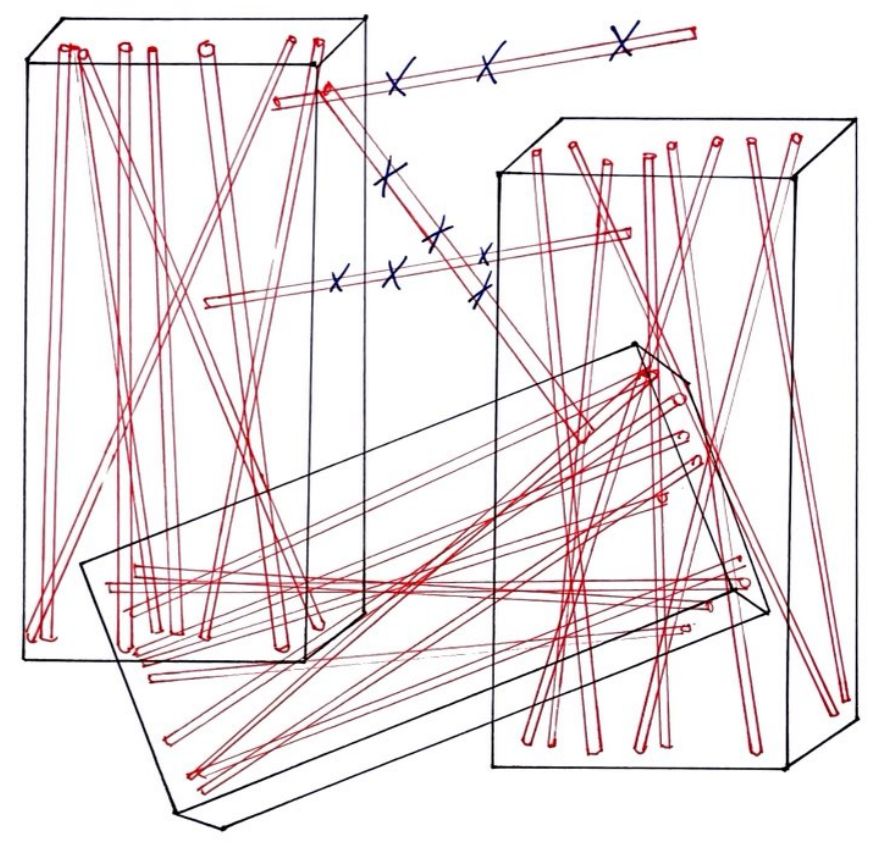}
\caption{Left: $\mathcal{U}$ is a set of tubes (red) that cluster into rectangular prisms. Right: Proposition \ref{factoringConvexSetsProp} locates these prisms (black). The tubes in $\mathcal{U}\backslash \mathcal{U}'$ have been X-ed out. }
\label{factoringInsideBoxesFig}
\end{figure}


Our proof of Proposition \ref{factoringConvexSetsProp} will use the following ``iterated graph pruning'' lemma, which allows us to prune a bipartite graph and find an induced subgraph for which every vertex has many neighbours.

\begin{lem}\label{graphRefinementLemma}
Let $G = (A\sqcup B, E)$ be a bipartite graph. Then there is a sub-graph $G'=(A'\sqcup B', E')$ so that $\#E'\geq \#E/2$; each vertex in $A'$ has degree at least $\frac{\#E}{4\#A}$; and each vertex in $B'$ has degree at least $\frac{\#E}{4\#B}$.
\end{lem}
\noindent Lemma \ref{graphRefinementLemma} is proved via iteratively removing those vertices that have few neighbours. See e.g.~ \cite{DvGo15} for a proof.

\begin{proof}[Proof of Proposition \ref{factoringConvexSetsProp}] $\phantom{1}$\\
{\bf Step 1.} 
Let $\mathcal{U}_0\subset\mathcal{U}$ be a set minimizing the quantity
\begin{equation}\label{eq: minUPrime}
\min_{\substack{\mathcal{U}'\subset \mathcal{U} \\ \mathcal{U}'\neq\emptyset}} \exp\Big[\Big(\log\frac{\# \mathcal{U}}{\# \mathcal{U}'}\Big)^2\Big]\ \CKT(\mathcal{U}').
\end{equation}

Since $\CKT(\mathcal{U}_0)\geq 1$, we have
\[
\exp\Big[\Big(\log\frac{\# \mathcal{U}}{\# \mathcal{U}_0}\Big)^2\Big] \leq \exp\Big[\Big(\log\frac{\# \mathcal{U}}{\# \mathcal{U}_0}\Big)^2\Big]\ \CKT(\mathcal{U}_0)\leq \CKT(\mathcal{U})\leq \#\mathcal{U}.
\]
Re-arranging,
\begin{equation}\label{eq: lowerbdU0}
\# \mathcal{U}_0 \geq e^{-\sqrt{\log(\# \mathcal{U})}}(\# \mathcal{U}).
\end{equation}

Observe that if $\mathcal{U}'\subset \mathcal{U}_0$ with $\# \mathcal{U}'\geq \frac{1}{2}(\#\mathcal{U}_0)$, then 
\[
\exp\Big[\Big(\log\frac{2(\# \mathcal{U})}{\# \mathcal{U}_0}\Big)^2\Big]\ \CKT(\mathcal{U}') \geq \exp\Big[\Big(\log\frac{\# \mathcal{U}}{\# \mathcal{U}'}\Big)^2\Big]\ \CKT(\mathcal{U}') \geq \exp\Big[\Big(\log\frac{\# \mathcal{U}}{\# \mathcal{U}_0}\Big)^2\Big]\ \CKT(\mathcal{U}_0).
\]
Re-arranging and using \eqref{eq: lowerbdU0}, 
\begin{equation}\label{lowerBdOnCKTUPrime}
\CKT(\mathcal{U}') \geq \kappa_0 \CKT(\mathcal{U}_0),\quad \textrm{where}\ \kappa_0=  e^{-2\log 2 \sqrt{\log(\# \mathcal{U})}}.
\end{equation}

\medskip

\noindent{\bf Step 2.}
Select closed convex sets $W_1,W_2,\ldots$ in $\RR^n$ and sets $\mathcal{U}_1\supset\mathcal{U}_2\supset\ldots$ according to the following procedure. Beginning with $j=1$, we select $W_j$ to maximize\footnote{Since $\mathcal{U}$ is a finite set of compact sets, such a maximizer exists; however the proof would work equally well if we merely approximate the maximum within a constant factor.} the quantity $\#\mathcal{U}_{j-1}[W_j]/ |W_j|$. By the definition of $\CKT(\mathcal{U}_{j-1})$, we can select such a $W_j$ so that
\begin{equation}\label{sizeOfWmCalU}
\#\mathcal{U}_{j-1}[W_j] = \CKT(\mathcal{U}_{j-1})\frac{|W_j|}{|U|}.
\end{equation}
(Recall that $|U|$ is the volume of a set from $\mathcal{U}$; all such sets have identical volume). Define $\mathcal{U}_j = \mathcal{U}_{j-1}\backslash \mathcal{U}_{j-1}[W_j]$. Continue this process until $\# \mathcal{U}_j<\frac{1}{2}(\# \mathcal{U}_0)$. 

Let $\mathcal{W}_0 = \{W_1,\ldots,W_{j-1}\}$. Then 
\begin{equation}\label{mostOfU0Kept}
\#\Big(\bigcup_{W\in\mathcal{W}_0}\mathcal{U}_0[W]\Big) = \# (\mathcal{U}_0 \backslash \mathcal{U}_j) >\frac{1}{2}(\# \mathcal{U}_0).
\end{equation}
Furthermore, for each $i=1,\ldots, j$, we have $\#\mathcal{U}_{i-1}\geq\frac{1}{2}(\#\mathcal{U}_{0})$, and hence by \eqref{sizeOfWmCalU} and \eqref{lowerBdOnCKTUPrime},
\begin{equation}\label{lowerBdWiCalUi}
\# \mathcal{U}_{i-1}[W_i] = \CKT(\mathcal{U}_{i-1})\frac{|W_i|}{|U|} \geq \kappa_0 \CKT(\mathcal{U}_0)\frac{|W_i|}{|U|}.
\end{equation}
Hence if $\mathcal{W}'\subset\mathcal{W}_0$, to compare $ \# \Big(\bigcup_{W_i\in\mathcal{W}'} \mathcal{U}_0[W_i] \Big) $ and $\sum_{W_i\in\mathcal{W}'} \# \mathcal{U}_0[W_i] $, 
\begin{equation}\label{numberOfSetsComparable}
\begin{split}
\kappa_0\frac{\CKT(\mathcal{U}_0)}{|U|}\sum_{W_i\in\mathcal{W}'}|W_i|
&\leq \sum_{W_i\in\mathcal{W}'}\# \mathcal{U}_{i-1}[W_i]
= \# \Big(\bigsqcup_{W_i\in\mathcal{W}'} \mathcal{U}_{i-1}[W_i] \Big)\\
&\leq \# \Big(\bigcup_{W_i\in\mathcal{W}'} \mathcal{U}_0[W_i] \Big) 
\leq \sum_{W_i\in\mathcal{W}'} \# \mathcal{U}_0[W_i] 
\leq \frac{\CKT(\mathcal{U}_0)}{|U|}\sum_{W_i\in\mathcal{W}'}|W_i|.
\end{split}
\end{equation}
The  equality in \eqref{numberOfSetsComparable} uses the critical fact that if $i\neq i',$ then $\mathcal{U}_{i-1}[W_i]$ and $\mathcal{U}_{i'-1}[W_{i'}]$ are disjoint.
\medskip

\noindent {\bf Step 3}.
Each $W\in\mathcal{W}_0$ has a John ellipsoid whose axes have lengths $\ell_1,\ldots,\ell_n$. Since each  set $\mathcal{U}_0[W]$ is non-empty and each $U\in \mathcal{U}_0$ contains a ball of radius $\delta$, we have that $\ell_i\geq\delta$ for each $i$. Since the sets in $\mathcal{U}$ are contained in the unit ball, we may suppose that $\ell_i\leq 2$ for each $i$. Thus by dyadic pigeonholing and \eqref{mostOfU0Kept}, there exist $a_1,\ldots, a_n$ and a set $\mathcal{W}_1\subset\mathcal{W}_0$, so that the following two items hold:
\begin{enumerate}
	\item[(i)] Each $W\in\mathcal{W}_1$ has a John ellipsoid whose axes have lengths $\ell_1\leq \ell_2\ldots\leq \ell_n$ with $\ell_i\in [a_i/2, a_i).$
	\item[(ii)] \itemizeEqnVSpacing
\begin{equation}\label{mostOfUPreserved}
\#\Big(\bigcup_{W\in\mathcal{W}_1}\mathcal{U}_0[W]\Big)\geq (100|\log\delta|)^{-n}(\#\mathcal{U}_0).
\end{equation}
\end{enumerate}
Replace each $W$ by a congruent copy of $W_0$---an ellipsoid whose axes have lengths $a_1,\ldots,a_n$, and denote the corresponding set $\mathcal{W}_2$. Observe that \eqref{mostOfUPreserved} remains true with $\mathcal{W}_2$ in place of $\mathcal{W}_1$, and \eqref{numberOfSetsComparable} remains true for all sets $\mathcal{W}'\subset\mathcal{W}_2$, though the first inequality has been weakened by a factor of $2^{n}$ on the RHS. Define $\mathcal{U}_2 = \bigcup_{W\in\mathcal{W}_2}\mathcal{U}_0[W]$ (recall that a sequence sequence $\mathcal{U}_1,\mathcal{U}_2,\ldots$ was defined earlier, and hence $\mathcal{U}_2$ was previously defined, but this is a harmless abuse of notation); we have that the cardinality of $\mathcal{U}_2$ is bounded below by the RHS of \eqref{mostOfUPreserved}.

Since $\mathcal{U}_0[W]=\mathcal{U}_2[W]$ for all $W\in\mathcal{W}_2$, by applying  \eqref{numberOfSetsComparable}  (beginning with the final inequality, and then using the first few inequalities) with $\mathcal{W}'=\mathcal{W}_2$ we conclude that
\begin{equation}\label{pairCount}
\begin{split}
\#\{(U, W)\in \mathcal{U}_2&\times\mathcal{W}_2  \colon U\subset W\}
=\sum_{W\in\mathcal{W}_2}\# \mathcal{U}_0[W]\leq \frac{\CKT(\mathcal{U}_0)}{|U|}\sum_{W\in\mathcal{W}_2}|W|\\
& \leq 2^n \kappa_0^{-1}\#\Big(\bigcup_{W\in\mathcal{W}_2} \mathcal{U}_0[W] \Big) =  2^n\kappa_0^{-1}(\#\mathcal{U}_2).
\end{split}
\end{equation}
In the above estimate, \eqref{numberOfSetsComparable} was used to obtain the first and second inequalities, while the final equality follows from the definition of $\mathcal{U}_2$. 

\medskip

\noindent{\bf Step 4}.
Construct the bipartite incidence graph $(\mathcal{I}, \mathcal{U}_2\times\mathcal{W}_2)$ whose edges consist of those pairs $(U,W)$ with $U\subset W$. This graph has the following properties:
\begin{itemize}
	\item[(i)] $\mathcal{I}$ has at most $2^n \kappa_0^{-1}(\#\mathcal{U}_2)$ edges.
	\item[(ii)] Each $U\in\mathcal{U}_2$ has at least one neighbour.
	\item[(iii)] Each $W\in\mathcal{W}_2$ has between $2^{-n}\kappa_0 \CKT(\mathcal{U}_0)|W_0| |U|^{-1}$ and $\CKT(\mathcal{U}_0)|W_0||U|^{-1}$ neighbours.
\end{itemize}

Note that Items (i) and (iii) imply that
\begin{equation}\label{upperBdOnCalW2}
\#\mathcal{W}_2 \leq 2^{2n} \kappa_0^{-2} \frac{(\#\mathcal{U}_2)|U|}{\CKT(\mathcal{U}_0)|W_0|}.
\end{equation}

We will construct an induced subgraph of $(\mathcal{I}, \mathcal{U}_2\times\mathcal{W}_2)$ as follows. First, remove all $U\in\mathcal{U}_2$ with more than $2^{n+1}\kappa_0^{-1}$ neighbours, and denote the resulting induced subgraph by $(\mathcal{I}_3, \mathcal{U}_3\times\mathcal{W}_2)$; by Items (i) and (ii),  we have $\#\mathcal{U}_3\geq\frac{1}{2}\#\mathcal{U}_2$, and $\#\mathcal{I}_3\geq \#\mathcal{U}_3$. Next, apply Lemma \ref{graphRefinementLemma} (iterated graph pruning) to $(\mathcal{I}_3, \mathcal{U}_3\times\mathcal{W}_2)$. Denote the resulting induced subgraph by  $(\mathcal{I}', \mathcal{U}' \times\mathcal{W})$.

\medskip

\noindent{\bf Step 5}.
We will verify that $\mathcal{U}'$ and $\mathcal{W}$ satisfy Conclusions (i)--(iv) of Proposition \ref{factoringConvexSetsProp}. For Conclusion (i), we have
\begin{align*}
\#\mathcal{U}'& \geq \frac{\kappa_0}{2^{n+1}}(\# \mathcal{I}') \geq \frac{\kappa_0}{2^{n+3}}(\# \mathcal{I}_3) \geq \frac{\kappa_0}{2^{n+3}}(\# \mathcal{U}_3) \geq   \frac{\kappa_0}{2^{n+4}}(\# \mathcal{U}_2) \geq  K^{-1}(\# \mathcal{U}),
\end{align*}
since $\#\mathcal{U}_2$ is bounded below by the RHS of \eqref{mostOfUPreserved};  $\#\mathcal{U}_0$ is bounded below by \eqref{eq: lowerbdU0};   and $K$ was defined in the statement of Proposition~\ref{factoringConvexSetsProp}. 
\medskip

For Conclusion (ii), Since each $U\in\mathcal{U}'$ has at most $K$ neighbours in $(\mathcal{I}', \mathcal{U}'\times\mathcal{W})$, we have that $\mathcal{W}$ is a $K$-almost partitioning cover of $\mathcal{U}'$. It remains to verify \eqref{lotsOfUinW}. Since $\#\mathcal{U}'[W]\leq \CKT(\mathcal{U}') |W| |U|^{-1}$, it will then follow that $\mathcal{W}$ is a $K$-balanced cover of $\mathcal{U}'$. By Lemma \ref{graphRefinementLemma} followed by \eqref{upperBdOnCalW2}, for each $W\in\mathcal{W}$, we have
\begin{equation}\label{goodBdOnCalUPrimeW}
\begin{split}
\#\mathcal{U}'[W] &\geq \frac{1}{4}\Big(\# I_3\Big)\Big(\#\mathcal{W}_2\Big)^{-1} \geq \frac{1}{4}\Big(\frac{1}{2}\#\mathcal{U}_2 \Big)\Big(2^{-2n} \kappa_0^{2}\frac{\CKT(\mathcal{U}_0)|W_0|}{(\#\mathcal{U}_2)|U|}\Big)\\
& \geq \big(2^{-2n-4}\kappa_0^2\big) \CKT(\mathcal{U}_0)|W_0||U|^{-1}.
\end{split}
\end{equation}
Since $\mathcal{U}'\subset\mathcal{U}_0$, we have $\CKT(\mathcal{U}')\leq \CKT(\mathcal{U}_0)$. 

\medskip

For Conclusion (iii), let $V\subset\RR^n$ be a convex set. Since each $U\in \mathcal{U}_3$ has at most $2^{n+1} \kappa_0^{-1}$ neighbours and $(\mathcal{I}', \mathcal{U}'\times \mathcal{W})$ is an induced subgraph of $(\mathcal{I}_3, \mathcal{U}_3\times \mathcal{W}_2)$,  each $U\in\mathcal{U}'$ is contained in at most $2^{n+1}\kappa_0^{-1}$ sets $\mathcal{U}'[W]$, we have
\begin{equation}\label{lowerBdUInV}
\begin{split}
\#\mathcal{U}'[V] & \geq 2^{-n-1}\kappa_0\sum_{W\in\mathcal{W}[V]}\#\mathcal{U}'[W] \geq \big(2^{-3n-5}\kappa_0^3\big) \CKT(\mathcal{U}_0)|W_0||U|^{-1}(\#\mathcal{W}[V]),
\end{split}
\end{equation}
where the final inequality used \eqref{goodBdOnCalUPrimeW}. On the other hand,
\begin{equation}\label{upperBdUInV}
\#\mathcal{U}'[V] \leq \CKT(\mathcal{U}')|V||W_0|^{-1}.
\end{equation}
Comparing \eqref{lowerBdUInV} and \eqref{upperBdUInV}, we see that $\#\mathcal{W}[V] \leq K|V||W_0|^{-1}$, as desired.

\medskip

Finally, for Conclusion (iv), let $W\in\mathcal{W}$ and let $V\subset W$ be a convex set. Then
\begin{equation}\label{UinsideSmallV}
\#(\mathcal{U}'[W])[V] \leq \#\mathcal{U}'[V] \leq \CKT(\mathcal{U}')|V| |U|^{-1}.
\end{equation}
Comparing \eqref{UinsideSmallV} and \eqref{lotsOfUinW} (which we verified  using  \eqref{goodBdOnCalUPrimeW}), we conclude that 
\[
\#(\mathcal{U}'[W])[V] \leq K |V||W|^{-1} (\#\mathcal{U}'[W]).
\]
This is precisely the statement that $\CFC\big((\mathcal{U}')^W\big)\leq K.$
\end{proof}


\subsection{Convex Sets and the Frostman Slab Wolff Axioms}
The goal of this section is to prove that after a refinement, every collection of convex sets can be partitioned into non-interacting pieces, each of which (after an appropriate rescaling) satisfies the Frostman Slab Wolff axioms. The precise statement is as follows.


\begin{prop}\label{slabWolffFactoring}
For all $n\geq 2,\eps>0$, there exists $\eta>0$ and $\kappa,K>0$ so that the following holds for all $\delta>0$. Let $\mathcal{U}$ be a collection of closed convex subsets of the unit ball in $\RR^n$, each of which contains a ball of radius $\delta$. For each $U\in\mathcal{U}$, let $Y(U)\subset U$ be a shading with $|Y(U)|\geq\delta^{\eta}|U|$. 

Then there exists a set $\mathcal{U}'\subset \mathcal{U}$; sets $Y'(U')\subset Y(U'),\ U'\in\mathcal{U}'$; and a set $\mathcal{W}$ of closed convex subsets of $\RR^n$ with the following properties:
\begin{enumerate}[i)]
	\item $\mathcal{W}$ factors $\mathcal{U'}$ from below with respect to the Frostman Slab Wolff axioms with error $\delta^{-\eps}$.

	\item The sets $\mathcal{U}'[W],$ $W\in\mathcal{W}$ do not interact, in the sense that the sets $\big\{ \bigcup_{U'\in\mathcal{U}'[W]}Y'(U'),\ W\in\mathcal{W}\big\}$ are disjoint. 

	\item The subset $\mathcal{U}'$ and the refined shading $Y'$ preserve most of the mass of the original collection $(\mathcal{U},Y)$, in the sense that 
	\begin{equation}\label{deltaEpsFractionMassPreserved}
		\sum_{U'\in\mathcal{U}'}|Y'(U')| \geq \kappa \delta^{\eps}(\log \#\mathcal{U})^{-K} \sum_{U\in\mathcal{U}}|U|.
	\end{equation}
\end{enumerate}
\end{prop}


Proposition \ref{slabWolffFactoring} will rely on the following consequence of Brunn's theorem:
\begin{lem}\label{ellipsoidLem}
Let $U\subset\RR^n$ be a convex set, let $H\subset\RR^n$ be a hyperplane, and let $s>0$, $t\in(0,1]$. Suppose that $|U\cap N_s(H)|= t |U|$. Then $U\subset N_{K_n s/t}(H)$, where $K_n$ depends only on $n$.
\end{lem}
\begin{proof}
Without loss of generality we may suppose that the hyperplane $H$ is given by $\{x_1=0\}$. Let $f(t) = |U\cap \{x_1=t\}|^{\frac{1}{n-1}}$ (here $|\cdot|$ denotes $(n-1)$-dimensional Lebesgue measure), and let $I =\supp(f)$. Our task is to show that $|I|\leq K_n s/t$.

By Brunn's theorem, $f$ is concave on $I$. The result now follows by comparing the estimates $t|U| = |U\cap N_s(H)|\leq (2s)(\sup f)^{n-1}$ and $|U|\geq K_n|I|(\sup f)^{n-1}$ (the latter is a consequence of the concavity of $f$).
\end{proof}

Combining Lemma \ref{ellipsoidLem} with a Cordoba-style $L^2$ argument, we obtain the following.
\begin{lem}\label{findingSlabsLem}
Let $\lambda\in(0,1]$, let $\mathcal{U}$ be a collection of closed convex subsets of the unit ball in $\RR^n$, each of which contains a ball of radius $\delta$. For each $U\in\mathcal{U}$, let $Y(U)\subset U$ be a shading with $|Y(U)|\geq\lambda|U|$. 

Then there exists a set $\mathcal{U}'\subset \mathcal{U}$; sets $Y'(U)\subset Y(U),\ U\in\mathcal{U}'$; and a set $\mathcal{S}$ of infinite slabs (i.e.~the $s$-neighbourhood of a hyperplane in $\RR^n$) with the following properties:
\begin{enumerate}[i)]
	\item $\mathcal{S}$ is a partitioning cover of $\mathcal{U'}$.

	\item If $S\in\mathcal{S}$ has thickness $s$, then the (rescaled) sets $\mathcal{U}'[S]$ have Frostman Slab Wolff constant $O(s^{-1})$. More concretely, if $\tilde S \subset S$ is a $s\times 2\times\ldots\times 2$ slab that contains the convex sets from $\mathcal{U}'[S]$, then $\FS(\mathcal{U}^{'\tilde S})\lesssim s^{-1}$.

	\item The sets $\bigcup_{U\in\mathcal{U}'[S]}Y'(U)$, $S\in\mathcal{S}$ are pairwise disjoint.

	\item $|Y'(U)|\geq \frac{\lambda}{2}|U|$ for each $U\in\mathcal{U'}$, and 
	\begin{equation}\label{enoughMassPreserved}
		\sum_{U\in\mathcal{U}'}|U| \gtrsim \log(\lambda^{-1}\delta^{-1}\#\mathcal{U})^{-1}\lambda \sum_{U\in\mathcal{U}}|U|.
	\end{equation}
\end{enumerate}
\end{lem}


\begin{proof}$\phantom{1}$\\
\noindent {\bf Step 1.}
Define $\mathcal{U}_0=\mathcal{U}$, and define $Y_0(U)=Y(U)$ for each $U\in\mathcal{U}_0$. For $i=1,\ldots$, let $S_i=N_{s_i}(H_i)$ be a slab maximizing the quantity
\begin{equation}\label{slabMaximizingQuantity}
s_i^{-1}\sum_{U\in \mathcal{U}_{i-1}[S_i]}|U|.
\end{equation}
Let $\mathcal{U}^{(i)}=\mathcal{U}_{i-1}[S_i]$; by the maximality of $S_i$, we have that $\mathcal{U}^{(i)}$ satisfies Conclusion (ii) from Lemma \ref{findingSlabsLem}. 

For each $U\in\mathcal{U}_{i-1}$, define $Y_i(U) = Y_{i-1}(U)\backslash S_i$, and define 
\[
\mathcal{U}_{i}=\big\{U\in\mathcal{U}_{i-1}\colon |Y_i(U)|\geq\frac{\lambda}{2}|U|\big\}.
\]
In particular, $\mathcal{U}_{i}\cap \mathcal{U}^{(i)}=\emptyset$.

\medskip
\noindent {\bf Step 2.}
We claim that 
\begin{equation}\label{mostMassInsideSi}
\sum_{U\in\mathcal{U}_{i-1}}|U\cap S_i| \lesssim \log(\lambda^{-1}\delta^{-1}\#\mathcal{U})\sum_{U\in\mathcal{U}^{(i)}}|U|.
\end{equation}
We verify \eqref{mostMassInsideSi} as follows. Since $\sum_{U\in\mathcal{U}^{(i)}}|U|\geq\frac{\lambda}{2} \delta^n$, the contribution from those $U\in\mathcal{U}_{i-1}$ with $|U\cap S_i|\leq \frac{1}{4}\lambda \delta^n(\#\mathcal{U})^{-1}$ is negligible. For each dyadic $t\in [\frac{1}{4}\lambda \delta^n(\#\mathcal{U})^{-1}, 1]$ we have
\begin{equation}\label{tContrib}
\begin{split}
\sum_{\substack{U\in\mathcal{U}_{i-1}\\ |U\cap S|\sim t|U|}}|U\cap S_i| & 
\lesssim t \sum_{\substack{U\in\mathcal{U}_{i-1}\\ |U\cap S|\sim t|U|}}|U|
 \leq t \sum_{\substack{U\in\mathcal{U}_{i-1}[N_{K_n s_i/t}}(H_i)]}|U|\\
& \leq K_n  \sum_{\substack{U\in\mathcal{U}_{i-1}[S_i]}}|U| = K_n\sum_{U\in\mathcal{U}^{(i)}}|U|.
\end{split}
\end{equation}
The second inequality follows from the containment
\[
\big\{U\in\mathcal{U}_{i-1}\colon |U\cap S|\sim t|U|\big\}\subset \big\{ U\in\mathcal{U}_{i-1}[N_{K_ns/t}(H_i)]\big\}
\]
for an appropriately chosen constant $K_n$ depending on $n$; this is Lemma \ref{ellipsoidLem}. The third inequality used the maximality of $S_i$, in the sense of \eqref{slabMaximizingQuantity}. \eqref{mostMassInsideSi} now follows from summing \eqref{tContrib} over dyadic values of $t$.

\medskip
\noindent {\bf Step 3.}
We halt the procedure described above when $\mathcal{U}_N=\emptyset$. Define $\mathcal{U}'=\bigsqcup_{i=1}^N \mathcal{U}^{(i)}$, and for each $U\in\mathcal{U}'$, define $Y'(U)=Y_{i-1}(U)$, where $i$ is the unique index so that $U\in \mathcal{U}^{(i)}$. Conclusions (i) and (iii) follow immediately from the above construction, as does the fact that $|Y'(U')|\geq\frac{\lambda}{2}|U'|$ for each $U'\in\mathcal{U}'$. Conclusion (ii) was already verified in Step 1.

It remains to verify \eqref{enoughMassPreserved}. We claim that each $U\in\mathcal{U} = \mathcal{U}_0$ contributes at least $\frac{\lambda}{2}|U|$ to the sum
\[
\sum_{i=1}^N \sum_{U\in\mathcal{U}_{i-1}}|U\cap S_i|,
\]
in the sense that $\sum_{i: U\in \mathcal{U}_{i-1}} |U\cap S_i|\geq \frac{\lambda}{2}|U|$. 
Indeed, suppose that $j>0$ is the smallest index with $U\not\in \mathcal{U}_j$, and hence $|Y(U)\backslash \bigcup_{i=1}^{j} S_i|<\frac{\lambda}{2}|U|$. But since $|Y(U)|\geq \lambda|U|$, this means that 
\[
\sum_{i=1}^{j} |U\cap S_i| \geq\frac{\lambda}{2}|U|,
\]
as claimed.

We conclude that
\begin{equation}\label{lotsOfContrib}
\sum_{i=1}^N \sum_{U\in\mathcal{U}_{i-1}}|U\cap S_i| \geq \frac{\lambda}{2}\sum_{U\in\mathcal{U}}|U|.
\end{equation}
Comparing \eqref{mostMassInsideSi} and \eqref{lotsOfContrib}, we obtain \eqref{enoughMassPreserved}.
\end{proof}

\medskip
\noindent Proposition \ref{slabWolffFactoring} follows from repeatedly applying Lemma \ref{findingSlabsLem}. We now turn to the details. 
\begin{proof}[Proof of Proposition \ref{slabWolffFactoring}]
Let $\lambda = \frac{1}{4}\delta^{\eta}$. After discarding those $U\in\mathcal{U}$ with $|Y(U)|<\lambda |U|$, we may suppose that $\mathcal{U}$ satisfies the hypotheses of Lemma \ref{findingSlabsLem}. 

Let $\mathcal{U}_0 = \mathcal{U}$ and for each $U\in\mathcal{U}_0$, let $Y_0(U)=Y(U)$. Let $\mathcal{P}_0=\{B(0,1)\}$, and let $\mathcal{W}_0 = \emptyset$. The set $\mathcal{P}_0$ corresponds to convex sets that still need to be ``processed'' by Lemma \ref{findingSlabsLem}, while $\mathcal{W}_0$ will hold the convex sets that satisfy the hypotheses of Proposition \ref{slabWolffFactoring}.

We will iteratively construct sets $\mathcal{U}_i\subset\mathcal{U}_{i-1}$ and $Y_i(U)\subset Y_{i-1}(U)$; a set $\mathcal{W}_i\supset\mathcal{W}_{i-1}$ of convex subsets of $B(0,1)$; and a set $\mathcal{P}_i$ of convex subsets of $B(0,1)$ such that the following properties hold:
\begin{enumerate}
	\item For each $W\in\mathcal{W}_i$, $\FS\big(\mathcal{U}_i^W\big)\lesssim\delta^{-\eps}$ (recall $\mathcal{U}_i^W= \phi_W(\mathcal{U}_i[W])$).
	\item For each $P\in\mathcal{P}_i$, $|P|\leq\delta^{i\eps}|B(0,1)|$.
	\item $\mathcal{W}_i\sqcup\mathcal{P}_i$ is a partitioning cover of $\mathcal{U}_i$.
	\item The sets $\bigcup_{U\in\mathcal{U}_i[V]}Y_i(U)$ are pairwise disjoint, as $V$ ranges over the convex sets in $\mathcal{P}_i\sqcup\mathcal{W}_i$.
	\item $|Y_i(U)|\geq 2^{-i}\lambda |U|$ for each $U\in\mathcal{U}_i$.
	\item \itemizeEqnVSpacing
	\[
		\sum_{V\in\mathcal{P}_i\cup\mathcal{W}_i}\sum_{U\in\mathcal{U}_i[V]}|Y_i(U)| \gtrsim \log(\delta^{-1}\#\mathcal{U})^i\lambda^i \sum_{U\in\mathcal{U}}|U|.
	\]
\end{enumerate}
These six items are trivially satisfied when $i=0$. For the $i$-th step, begin by setting $\mathcal{W}_i = \mathcal{W}_{i-1}$, $\mathcal{P}_i = \emptyset$, and  $\mathcal{U}_i = \bigcup_{W\in \mathcal{W}_{i-1}} \mathcal{U}_{i-1}[W]$, $Y_i(U)=Y_{i-1}(U)$ for each $U\in \mathcal{U}_i$.

 For each $P\in\mathcal{P}_{i-1}$, apply Lemma \ref{findingSlabsLem} (with $2^{-(i-1)}\lambda$ in place of $\lambda$) to each collection $\mathcal{U}_{i-1}^P=\phi_P( \mathcal{U}_{i-1}[P])$ of ellipsoids, and their associated shadings $\phi_P(Y_{i-1}(U))$. We obtain a collection of slabs $\mathcal{S}$, a set $\mathcal{U}_{i-1}'[P]\subset \mathcal{U}_{i-1}[P]$, and a shading, which we denote by $Y_i(U)$,  on the ellipsoids in $\mathcal{U}_{i-1}'[P]$. Add the ellipsoids in $\mathcal{U}_{i-1}'[P]$ and their associated shading $Y_i(U)$ to $\mathcal{U}_i$.   Next, we consider each slab $S\in\mathcal{S}$ in turn.
\begin{itemize}
	\item If a slab $S\in\mathcal{S}$ has thickness $\geq\delta^{\eps}$, then this corresponds to a convex set $W = P\cap \phi_P^{-1}(S)$ for which $\FS\big((\mathcal{U}_{i-1}'[P])^W\big)\lesssim\delta^{-\eps}$. Add this set to $\mathcal{W}_i$.

	\item If a slab $S\in\mathcal{S}$ has thickness $\leq \delta^{\eps}$, then this corresponds to a convex set $P' = P\cap \phi_P^{-1}(S)$ for which 
	\[ 
	|P'|\leq\delta^{\eps}|P| \leq \delta^\eps \big(\delta^{(i-1)\eps}|B(0,1)|\big) = \delta^{i\eps}|B(0,1)|.
	\]
	Add this set to $\mathcal{P}_i$.
\end{itemize}
After this procedure has been performed for each $P\in\mathcal{P}_{i-1}$, Properties 1 and 2 are immediate, while Properties 3-6 follow from their counterparts in Lemma \ref{findingSlabsLem}.

We halt the process when the output $\mathcal{P}_N = \emptyset$. Since each $U\in\mathcal{U}$ has volume at least $\delta^n$, the above procedure must halt after at most $n/\eps$ steps. We let $\mathcal{W}=\mathcal{W}_N$, $\mathcal{U}'=\mathcal{U}_N$, and $Y'(U) = Y_N(U)$. To obtain \eqref{deltaEpsFractionMassPreserved}, $\eta$ must be selected sufficiently small so that $\delta^{N\eta}\leq \delta^{\eps}$, i.e. $\eta \sim \eps^2/n$ and $K\sim n/\eps$ will suffice.
\end{proof}


\subsection{The Frostman Slab Wolff Axioms and Covers}\label{frostmanSlabTransitiveUnderCoversSec}
In this section we will state and prove a precise version of Remark \ref{FrostmanWolffInheritedUpwardsDownwards}(C). We first consider the Frostman Slab Wolff Axioms. 
\begin{lem}\label{frostmanSlabTransitiveUnderCovers}
Let $W\subset \RR^n$ be a convex set and let $\mathcal{U}$ and $\mathcal{V}$ be collections of convex subsets of $W$, with $\mathcal{U}\prec\mathcal{V}$ and $\#\mathcal{U}[V]\leq K(\#\mathcal{U})/(\#\mathcal{V})$ for each $V\in\mathcal{V}$. Suppose that each set in $\mathcal{U}$ has the same volume, and similarly for $\mathcal{V}$. Finally, suppose that each set in $\mathcal{U}^W$ has diameter $\geq 1/100$. 

Then
\begin{equation}\label{boundFSCalUW}
\FS(\mathcal{U}^W)\lesssim \log(2+ |U^W|^{-1}) K\Big(\sup_{V\in\mathcal{V}}\FS(\mathcal{U}^V)\Big)\Big(\FS(\mathcal{V}^W)\Big).
\end{equation}
\end{lem}

\begin{proof}
First, to simplify notation we may suppose wlog that $W = B(0,1)$; indeed, both the hypotheses and conclusion of Lemma \ref{frostmanSlabTransitiveUnderCovers} remain unchanged if we replace $W$ by $W^W$ (the latter is comparable to $B(0,1)$); replace $\mathcal{U}$ by $\mathcal{U}^W$; and replace $\mathcal{V}$ by $\mathcal{V}^W$. In particular, each set in $\mathcal{U}$ now has diameter $\geq 1/100$.

Fix a truncated, thickened hyperplane $S=N_s(H)\cap B(0,1),$ with $\mathcal{U}[S]\neq\emptyset$ (so in particular $s\geq |U|/K_n$, where $K_n$ is a constant depending only on the dimension $n$). We may suppose that $s\leq 2,$ since otherwise we can replace $N_s(H)$ by a hyperplane of the form $N_2(H')$, which has the same intersection with $B(0,1)$.

Since $\mathcal{U}\prec\mathcal{V}$, we have
\begin{equation}\label{boundCalUWS}
\mathcal{U}[S] = \bigcup_{V\in\mathcal{V}}\mathcal{U}[V \cap S]= \bigcup_{t\ \textrm{dyadic}}\bigcup_{\substack{V\in\mathcal{V} \\ |V\cap S|\sim t|V|}}\mathcal{U}[V \cap S],
\end{equation}
where the first union ranges over dyadic values of $t$ between $|U||B(0,1)|^{-1}$ and $1$ (this range is sufficient, since if $|V\cap S|< |U||B(0,1)|^{-1}|V| < |U|$, then $\mathcal{U}[V\cap S]=\emptyset$). Observe that there are $\lesssim \log(2+|U|^{-1})$ dyadic values of $t$ in this range.

Since $\mathcal{U}[S]\neq\emptyset$ and each element of $\mathcal{U}$ has diameter $\geq 1/100$, we have $\diam(S\cap B(0,1))\geq 1/100$, and hence 
\begin{equation}\label{linearGrowthInTheta}
|N_{s/t}(H)\cap B(0,1)|\sim s/t \sim t^{-1}|S|\quad\textrm{for all}\ t\in [s,1]. 
\end{equation}

Next, let $V\in\mathcal{V}$ with $|V\cap S|\sim t|V|$. This means that $|(V\cap S)^V|\sim t$. We have

\begin{equation}\label{numberOfUInsideVThetaIntersection}
\begin{split}
\#\mathcal{U}[V\cap S]
&=\#\{U\in\mathcal{U}\colon U\subset V\cap S\} \\
&=\#\{U\in\mathcal{U}[V]\colon U\subset V\cap S\} \\
&=\#\{U^V\in\mathcal{U}^V\colon U^V \subset (V\cap S)^V\} \\
& \leq \FS(\mathcal{U}^V)|(V\cap S)^V| (\#\mathcal{U}[V])\\
& \lesssim t  \Big(\sup_{V\in\mathcal{V}}\FS(\mathcal{U}^V)\Big)\Big(K \frac{\#\mathcal{U}}{\#\mathcal{V}}\Big).
\end{split}
\end{equation}

On the other hand, by Lemma \ref{ellipsoidLem} we have
\begin{equation}\label{controlNumberVWithThetaIntersection}
\begin{split}
\#\{V\in\mathcal{V}\colon |V\cap S|\sim t|V|\} & \leq \#\{V\in\mathcal{V}\colon V\subset N_{K_n s/t}(H) \}\\ 
&\lesssim \FS(\mathcal{V})|B(0,1)\cap N_{K_n s/t}(H)|(\#\mathcal{V})\\
& \lesssim t^{-1}\FS(\mathcal{V})|S|(\#\mathcal{V}),
\end{split}
\end{equation}
where the final inequality used \eqref{linearGrowthInTheta}.

Using \eqref{numberOfUInsideVThetaIntersection} and \eqref{controlNumberVWithThetaIntersection} to control the cardinality of the union \eqref{boundCalUWS}, we conclude that
\begin{equation*}
\begin{split}
\#\mathcal{U}[S] &\lesssim \sum_{t\ \textrm{dyadic}}\Big( t^{-1}\FS(\mathcal{V})|S|(\#\mathcal{V})\Big)\Big( t  \big(\sup_{V\in\mathcal{V}}\FS(\mathcal{U}^V)\big)\big(K \frac{\#\mathcal{U}}{\#\mathcal{V}}\big)\Big)\\
&\lesssim K_1 |S| (\#\mathcal{U}),\quad K_1 = \log(2+ |U|^{-1}) K\Big(\sup_{V\in\mathcal{V}}\FS(\mathcal{U}^V)\Big)\Big(\FS(\mathcal{V})\Big).
\qedhere
\end{split}
\end{equation*}
\end{proof}

Next we consider Remark \ref{FrostmanWolffInheritedUpwardsDownwards}(C) for the Katz-Tao Convex Wolff Axioms. We will restrict attention to the special case where the convex sets in question are tubes.

\begin{lem}\label{KTWSubMultiplicativeTubes}
Let $0<\delta\leq \rho \leq 1$. Let $\tubes$ be a multiset of $\delta$-tubes and let $\tubes_\rho$ be a cover of $\tubes$. Then
\begin{equation}\label{CKTTubesBd}
\CKT(\tubes)\lesssim  \Big(\sup_{T_\rho\in\tubes_\rho}\CKT(\tubes^{T_\rho})\Big)\Big(\CKT(\tubes_\rho)\Big).
\end{equation}
\end{lem}
\begin{proof}
Let $W\subset\RR^3$ be a convex set with $\tubes[W]\neq\emptyset$. Replacing $W$ by $W\cap B(0,1)$ and then enlarging $W$ by a constant factor, we may assume that $W$ is a prism of dimensions $a\times b\times 2$. Since $\tubes_\rho$ covers $\tubes$, we have 
\begin{equation}\label{decomposeTubesW}
\tubes[W]=\bigcup_{T_\rho\in\tubes_\rho}(\tubes[T_\rho])[W]=\bigcup_{T_\rho\in\tubes_\rho}\tubes[T_\rho\cap W].
\end{equation}
Observe that if $\tubes[T_\rho\cap W]\neq\emptyset$, then $T_\rho\cap W$ must contain a unit line segment, and thus $T_\rho\subset N_{3\rho}(W)$.

Let $\tilde a =\min(a,\rho)$ and $\tilde b = \min(b,\rho)$. Observe that 
\[
|N_{3\rho}(W)|\sim \big(\frac{\rho}{\tilde a}\big)\big(\frac{\rho}{\tilde b}\big) |W|,
\]
and thus
\begin{equation}\label{numberTubesN3rhoW}
\#\tubes_\rho[N_{3\rho}(W)]  \leq  \CKT(\tubes_\rho) \frac{|N_{3\rho}W|}{|T_\rho|} \lesssim \CKT(\tubes_\rho) \frac{|W|}{\tilde a \tilde b}.
\end{equation}
On the other hand, if $\tubes[T_\rho\cap W]$ is non-empty, then $T_\rho\cap W$ is a convex set of dimensions bounded by  $2\tilde a \times2 \tilde b \times 1$, and thus
\begin{equation}\label{numerRescaledTubesInRhoW}
\#\tubes[T_\rho\cap W] \lesssim \CKT(\tubes[T_\rho])\frac{|T_\rho\cap W|}{|T|}\lesssim \CKT(\tubes^{T_\rho})\frac{\tilde a\tilde b}{|T|}.
\end{equation}
\eqref{CKTTubesBd} now follows by combining \eqref{decomposeTubesW}, \eqref{numberTubesN3rhoW}, and \eqref{numerRescaledTubesInRhoW}.
\end{proof}


\section{Factoring tubes into flat prisms}\label{factoringTubesSection}

In this section, we will explore what happens when Proposition \ref{factoringConvexSetsProp} is applied to a set $\tubes$ of $\delta$-tubes. Recall that Proposition \ref{factoringConvexSetsProp} outputs a refinement $\tubes'\subset\tubes$ and a set $\mathcal{W}$ of convex sets. If $\CFC(\tubes)=1$, then $\mathcal{W}=\{B(0,1)\}$. On the other hand, if $\CKT(\tubes)=1$, then $\mathcal{W}=\tubes$. If both $\CKT(\tubes)$ and $\CFC(\tubes)$ are large, then $\mathcal{W}$ will consist of a collection of convex sets, each of which are comparable to a rectangular prism of dimensions $a\times b\times 1,$ for some $\delta\leq a\leq b\leq 1$. The goal of this section is to explore the following theme: if the prisms in $\mathcal{W}$ are flat, in the sense that $a <\!\!<b$, then the union $\bigcup T$ will have larger volume than predicted by the estimate \eqref{defnCEEstimate} from Assertion $\cE(\sigma,\omega)$. The precise statement is as follows.

\begin{prop}\label{aLLbProp}
Let $\omega>0,\ 0<\sigma\leq 2/3$, and suppose $\cE(\sigma,\omega)$ is true. Then for all $\eps>0$, there exists $\kappa,\eta>0$ so that the following holds for all $\delta>0$. Let $(\tubes,Y)_{\delta}$ be $\delta^{\eta}$ dense. Let $\delta\leq a\leq b\leq 1$, and let $\mathcal{W}$ be a $\delta^{-\eta}$ balanced cover of $\tubes$ consisting of congruent copies of an $a\times b\times 2$ prism.

\medskip

\noindent (A) 
Suppose that $\mathcal{W}$ factors $\tubes$ from below with respect to the Frostman Convex Wolff axioms and from above with respect to the Katz-Tao Convex Wolff axioms, both with with error $\delta^{-\eta}$. Suppose as well that $\mathcal{W}$ is a $\delta^{-\eta}$-balanced, $\delta^{-\eta}$-almost partitioning  cover of $\tubes$, and that
$\#\tubes[W] \geq  \delta^{\eta}\CKT(\tubes)\frac{|W|}{|T|}$  for each $W\in \mathcal{W}$ (this condition is satisfied, for example, if $\mathcal{W}$ is the output when Proposition \ref{factoringConvexSetsProp} is applied to $\tubes$). Then
\begin{equation}\label{eq: 5.1}
\Big|\bigcup_{T\in\tubes}Y(T)\Big| \geq \kappa \delta^{\omega+\eps}\Big(\frac{b}{a}\Big)^{\omega} m^{-1}(\#\tubes)|T|\big(m^{-3/2}\ell (\#\tubes)|T|^{1/2}\big)^{-\sigma},
\end{equation}
where $m = \CKT(\tubes)$ and $\ell = \FS(\tubes)$.
\medskip

\noindent (B)
Suppose that $\mathcal{W}$ factors $\tubes$ from above and below with respect to the Frostman Convex Wolff Axioms, both with error $\leq\delta^{-\eta}$. Suppose as well that $\mathcal{W}$ satisfies the Katz-Tao Convex Wolff axioms at scale $b$ in the following sense: for all $W\in\mathcal{W}$ we have $\CKT(\mathcal{W}[N_b(W)])\leq\delta^{-\eta}$. Then
\begin{equation}\label{PropALLbPropConclusionBEstimate}
\Big|\bigcup_{T\in\tubes}Y(T)\Big| \geq \kappa \delta^{\omega+\eps}\Big(\frac{b}{a}\Big)^{\omega}\big((\#\tubes)^{1/2}|T|\big)^{\sigma}.
\end{equation}
 \end{prop}

\noindent Note that \eqref{PropALLbPropConclusionBEstimate} agrees with \eqref{eq: 5.1} when $\CKT(\tubes)= (\#\tubes) |T|$, in which case both $\CFC(\tubes)$ and $\FS(\tubes)$ have size $\sim 1$. 

In Section \ref{refinedInductionOnScaleSec} we will need the following mild generalization of Proposition \ref{aLLbProp}(A).

\begin{prop}\label{PropALLbPropGen}
Let $\omega>0,\ 0<\sigma\leq 2/3$, and suppose $\cE(\sigma,\omega)$ is true. Then for all $\eps>0$, there exists $\kappa,\eta>0$ so that the following holds for all $0<\delta\leq\rho\leq a\leq b\leq 1$. Let $(\tubes,Y)_{\delta}$ be $\delta^{\eta}$ dense and let $\tubes_\rho$ be a $\delta^{-\eta}$ balanced cover of $\tubes$. Let $\mathcal{W}$ be a $\delta^{-\eta}$ balanced cover of $\tubes_\rho$ consisting of congruent copies of an $a\times b\times 2$ prism.

Suppose that $\tubes_\rho$ factors $\tubes$ from below with respect to the Frostman Slab Wolff Axioms with error $\ell'$. Suppose that $\mathcal{W}$ factors $\tubes_\rho$ from above with respect to the Katz-Tao Convex Wolff axioms and from below with respect to the Frostman Convex Wolff axioms, both with error $\delta^{-\eta}$. Then
\begin{equation}
\Big|\bigcup_{T\in\tubes}Y(T)\Big| \geq \kappa \delta^{\omega+\eps}\Big(\frac{b}{a}\Big)^{\omega} m^{-1}(\#\tubes)|T|\big(m^{-3/2}\ell \ell' (\#\tubes)|T|^{1/2}\big)^{-\sigma},
\end{equation}
where $m = \CKT(\tubes)$ and $\ell = \FS(\tubes)$.
\end{prop}
Proposition \ref{aLLbProp}(A) is the special case of Proposition \ref{PropALLbPropGen} where $\rho=\delta$ and $\tubes_\rho=\tubes$.

The main goal of Section \ref{factoringTubesSection} is to prove Propositions \ref{aLLbProp} and \ref{PropALLbPropGen}. A second goal is to introduce two cousins of the estimate $\cE(\sigma,\omega),$ and to show that these three estimates are equivalent. The first estimate is (formally) weaker: it is the special case of the estimate $\cE(\sigma,\omega)$ when $\ell$ has size about 1. 

\begin{defn}\label{TCEDefn}
We say that \emph{Assertion $\TcE(\sigma,\omega)$ is true} if the following holds:\\
For all $\eps>0$, there exists $\kappa,\eta>0$ such that the following holds for all $\delta>0$. Let $(\tubes,Y)_\delta$ be $\delta^{\eta}$ dense, and suppose $\FS(\tubes)\leq\delta^{-\eta}$. Then
\begin{equation}\label{TCEEstimate}
\Big|\bigcup_{T\in\tubes}Y(T)\Big|\geq \kappa \delta^{\omega+\eps}m^{-1}(\#\tubes)|T|\big(m^{-3/2} (\#\tubes)|T|^{1/2}\big)^{-\sigma},
\end{equation}
where  $m = \CKT(\tubes)$.
\end{defn}

The second estimate is (formally) stronger: it is a generalization of the estimate $\cE(\sigma,\omega)$ where $\delta$-tubes are replaced by congruent convex sets of diameter 1.


\begin{defn}\label{defnF}
We say that \emph{Assertion $\cF(\sigma,\omega)$ is true} if the following holds:\\
For all $\eps>0$, there exists $\kappa,\eta>0$ such that the following holds for all $0<a\leq b\leq 1$. Let $(\mathcal{P},Y)_{a\times b\times 1}$ be $a^{\eta}$ dense. Then
\begin{equation}
\Big|\bigcup_{P\in\mathcal{P}}Y(P)\Big|\geq \kappa a^\eps b^{\omega}m^{-1}(\#\mathcal P)|P|(m^{-3/2}\ell (\#\mathcal{P})|P|^{1/2})^{-\sigma}D^{-\sigma},
\end{equation}
where  
\begin{equation}\label{defnOfD}
m = \CKT(\mathcal{P}),\quad \ell = \FS(\mathcal{P}),\quad\textrm{and}\quad  D = \max_{P\in\mathcal{P}}\sup_{\rho\in[a,b]}\frac{|P|}{|N_\rho(P)|}\big(\#\mathcal{P}[N_{\rho}(P)]\big)^{1/2}.
\end{equation}
\end{defn}

\begin{rem} \label{rmk: RHSF}
	When $\sigma \in (0, 2/3]$, the term 
	\[
	m^{-1} (\#\mathcal P)|P|(m^{-3/2}\ell (\#\mathcal{P})|P|^{1/2})^{-\sigma}  D^{-\sigma}
	\] 
	is always at most $1$. To see this,  since  $\#\mathcal{P} \leq m |P|^{-1}$  and $\sigma \in (0,  2/3]$, it remains to show that $ ((\#\mathcal{P}) |P| )^{1/2} l^{-1} |P|^{1/2} \leq D $.  This is true because  $  \#\mathcal{P}\leq  \#\mathcal{P}[N_b(P)] \cdot b^{-4} \leq D^2 b^{-2} a^{-2},$ and $\ell \geq 1$. 
\end{rem}

\begin{rem}\label{discussionOfDRemark}
If $\mathcal{P}$ is non-empty, then by selecting $\rho=a$ we see that the quantity $D = D(\mathcal{P})$ from \eqref{defnOfD} is always $\gtrsim 1$. In general, $D$ can be as large as $(b/a)^{1/2}$: when $\rho=b$, it is possible for about $(b/a)^3$ essentially distinct $a\times b\times 1$ prisms to fit inside the $b$ tube $N_b(P)$. If this happens, then the RHS of \eqref{defnOfD} becomes $\frac{ab}{b^2}(b/a)^{3/2}=(b/a)^{1/2}$. However, there are several important situations where we can guarantee that $D$ has size roughly 1. We describe three of these below.
\end{rem}

\noindent {\bf Situation 1.}
If $a=b$, then since the prisms in $\mathcal{P}$ are essentially distinct, we have $\#\mathcal{P}(N_a(P))\sim 1$ and hence $D\sim 1$. In particular, this means that $\cF(\sigma,\omega)\implies \cE(\sigma,\omega)$.

\noindent {\bf Situation 2.}
Suppose that for each $P\in\mathcal{P}$, we have $\CKT(\mathcal{P}[N_b(P)])\leq K,$ for some $K\geq 1$. This means that the Katz-Tao Convex Wolff constant of $\mathcal{P}$ might be large, but if we restrict attention to those prisms $\mathcal{P}'$ contained inside a tube of diameter $b$, then the Katz-Tao Convex Wolff constant of $\mathcal{P}'$ is small (this is the setup for Item (B) of Proposition \ref{aLLbProp}). Then for each $\rho\in[a,b]$ and $P\in\mathcal{P}$ we have 
$
\#\mathcal{P}[N_\rho(P)]\lesssim K \frac{\rho\cdot b\cdot 1}{a\cdot b\cdot 1} =K\frac{\rho}{a},
$
and hence $D\lesssim K^{1/2}.$

\medskip

\noindent {\bf Situation 3.}
Let $\tubes$ be a set of essentially distinct $\delta$-tubes contained in a $s\times t\times 2$ prism $W$, with $0<s\leq t\leq 2$. Let $\mathcal{P} = \phi_W(\tubes)$. Then the sets in $\mathcal{P}$ are comparable to rectangular prisms of dimensions $a\times b\times 2$ with $a=\delta/t$ and $b = \delta/s$. Let us estimate the quantity $D=D(\mathcal{P})$ for this arrangement. For each $\rho\in[a,b]$ and each $P\in\mathcal{P}$, $\#\mathcal{P}[N_\rho(P)]$ counts the number of $\delta$-tubes from $\tubes$ contained inside a rectangular box of dimensions roughly $\delta\times (\delta \rho/a)\times 2$. Since the tubes in $\tubes$ are essentially distinct, at most $O\big(\big(\frac{(\delta \rho/a)}{\delta}\big)^2\big)=O\big(\rho^2/a^2\big)$ tubes from $\tubes$ can be contained in such a box, i.e. $D \lesssim \sup_{\rho\in[a,b]} \big(\frac{ab}{\rho b}\big)\big(\frac{\rho}{a}\big) = 1.$


\subsection{A few frequently used Cordoba-type $L^2$ arguments}\label{frequentlyUsedCordoba}
In this section, we will explore several variants of the following argument: To show that a union $\bigcup_{P\in\mathcal{P}}P$ is large, it suffices to show that the quantity $\Vert \sum_{P\in\mathcal{P}}\chi_P\Vert_2^2 = \sum_{P,P'\in\tubes}|P\cap P'|$ is small, and then use Cauchy-Schwartz to conclude that
\[
\Big|\bigcup_{P\in\mathcal{P}}P\Big| \geq \Big(\sum_{P\in\mathcal P}|P|\Big)^2\ \Big/\ \Big\Vert \sum_{P\in\mathcal{P}}\chi_P\Big\Vert_2.
\]
This argument was used by Cordoba \cite{Cor77} to prove the Kakeya maximal function conjecture in $\RR^2$, so we will call this style of argument a ``Cordoba-type $L^2$ argument.''

\subsubsection{A volume estimate for slabs}
In this section we will use a Cordoba-type $L^2$ argument to estimate the volume of a union of slabs. The precise statement is as follows.
\begin{lem}\label{cordobaLem}
Let $\delta,\lambda>0$. Let $\mathcal{S}$ be a collection of $\delta\times 1\times\ldots\times 1$ slabs (\emph{n.b.~these slabs need not be essentially distinct}), and let $Y$ be a $\lambda$-dense shading on $\mathcal{S}$. Let $m = \CKT(\mathcal{S})$. Then
\begin{equation}\label{cordobaEstimate}
\Big|\bigcup_{S\in\mathcal{S}}Y(S)\Big| \gtrsim |\log\delta|^{-1}m^{-1}\lambda^2(\#\mathcal{S})|S|.
\end{equation}
\end{lem}
\begin{proof}
Fix $S\in\mathcal{S}$. By Lemma \ref{ellipsoidLem} (applied to the outer John ellipsoid of each element of $\mathcal{S}$), we have that for each $t \in[\delta,1]$, we have 
\[
\#\{S'\in\mathcal{S}\colon |S\cap S'| \sim t|S|\} \leq \# \{S'\in\mathcal{S}\colon S'\subset N_{C\delta/t}(S)\} \lesssim m \frac{\delta}{t}|S|^{-1}\sim\frac{m}{t},
\]
where $C = C(n)$ depends only on $n$. Thus
\begin{equation}\label{L2Estimate}
\begin{split}
\Big\Vert \sum_{S\in\mathcal{S}}\chi_{Y(S)}\Big\Vert_2^2 &
\leq \Big\Vert \sum_{S\in\mathcal{S}}\chi_{S}\Big\Vert_2^2 \lesssim \sum_{S\in\mathcal{S}} \sum_{\substack{\delta\leq t\leq 1\\ t\phantom{.}\operatorname{dyadic}}}\sum_{\substack{S'\in\mathcal{S}\\
|S\cap S'|\sim t|S|}} t|S|\\
& \lesssim \sum_{S\in\mathcal{S}} \sum_{\substack{\delta\leq t\leq 1\\ t\phantom{.}\operatorname{dyadic}}}\big(\frac{m}{t}\big) t|S| \lesssim |\log\delta| m(\#\mathcal{S})|S|.
\end{split}
\end{equation}
Let $E = \bigcup_{S\in\mathcal{S}}Y(S)$. Using Cauchy-Schwartz, we have
\begin{equation}\label{applyCS}
\Big(\lambda |S|(\#\mathcal{S})\Big)^2\leq \Big(\int \chi_{E} \sum_{S\in\mathcal{S}}\chi_{Y(S)}\Big)^2 \leq |E| \Big\Vert \sum_{S\in\mathcal{S}}\chi_{Y(S)}\Big\Vert_2^2.
\end{equation}
The result now follows by comparing \eqref{L2Estimate} and \eqref{applyCS}.
\end{proof}

\noindent We will highlight a few instances where Lemma \ref{cordobaLem} will be helpful.
\begin{itemize}
	\item $(\mathcal{S},Y)$ is $\delta^{\eta}$ dense, and $\CFC(\mathcal{S})\lessapprox_\delta \delta^{-\eta}$. Then the RHS of \eqref{cordobaEstimate} is $\gtrapprox_\delta \delta^{3\eta}$.

	\item $\mathcal{R}$ is a set of $\delta\times 1$ rectangles inside a $\rho\times 2$ rectangle $W$: we will apply Lemma \ref{cordobaLem} to $\mathcal{R}^W$.

	\item $\tubes$ is a set of $\delta$-tubes inside a $100\delta\times b\times 1$ prism $W$: we will apply Lemma \ref{cordobaLem} to $\tubes^W$.
\end{itemize}

\subsubsection{Tangential vs transverse prism intersection}

The next result says that if a collection of $a\times b\times 1$ prisms intersect transversely, in the sense that their tangent planes make large angle at a typical point of intersection, then the union of these prisms fills out a large fraction of a thickened neighbourhood of these prisms. The precise statement is Lemma \ref{OnePrismWithTangentPlanes} and Corollary \ref{prismsWithTangentPlanes} below. Before stating that result, we need a few definitions.

\begin{defn}\label{regularShading}
Let $W\subset\RR^n$, let $Y(W)\subset W$ be a shading, and let $\delta>0$. We say that the shading $Y(W)$ is \emph{regular} at scales $\geq\delta$ if for each $x\in Y(W)$ and each $r\in[\delta, 1]$, we have
\begin{equation}\label{regularShadingIneq}
|Y(W)\cap B(x,r)| \geq (100\log(1/\delta))^{-1} |Y(W)|\Big(\frac{|B(x,r)\cap W|}{|W|}\Big).
\end{equation}
If the quantity $\delta$ is apparent from context, then we will omit it and say that $Y(W)$ is regular. 
\end{defn}

The next lemma says after a harmless refinement, every shading has a regular subshading. This is Lemma 2.7 from \cite{KWZ}; see also \cite[Lemma 2.3]{WZ23}. 
\begin{lem}\label{everySetHasARegularShadingLem}
Let $W\subset\RR^n$ and let $Y(W)$ be a shading. Then there is a regular shading $Y'(W)\subset Y(W)$ with $|Y'(W)|\geq\frac{1}{2}|Y(W)|$. 
\end{lem}

The next result says that if a prism $P_0$ is incident to many prisms that intersect $P_0$ non-tangentially, then the union of these prisms fills out a thickened neighbourhood of $P_0$.

\begin{lem}\label{OnePrismWithTangentPlanes}
Let $0<a\leq b\leq c$ and let $\lambda>0$. Let $P_0$ be a $a\times b\times c$ prism with shading $Y_0(P_0)$. Let $(\mathcal{P},Y)_{a\times b\times c}$ be a set of prisms and their associated shading. Suppose that $|Y_0(P_0)|\geq\lambda|P_0|$; $|Y(P)|\geq\lambda|P|$ for each $P\in\mathcal{P}$; and each shading $Y(P)$ is regular, in the sense of Definition \ref{regularShading}.

Let $\theta\in[\frac{a}{b}, 1]$, and suppose that
\[
\theta \leq \theta_{\operatorname{min}}:= \frac{a}{b} + \inf_{x\in Y_0(P_0)} \sup_{P\in \mathcal{P}_Y(x)} \angle\big(\Pi(P_0),\Pi(P)\big).
\]
Then 
\begin{equation}\label{nbhdMostlyFullOnePrism}
\Big|N_{b \theta}(P_0) \cap \bigcup_{P\in\mathcal{P}}Y(P)\Big| \gtrapprox_a \lambda^4|N_{b \theta}(P_0)|.
\end{equation}
\end{lem}

\begin{figure}[h!]
\centering
\begin{overpic}[ scale=0.45]{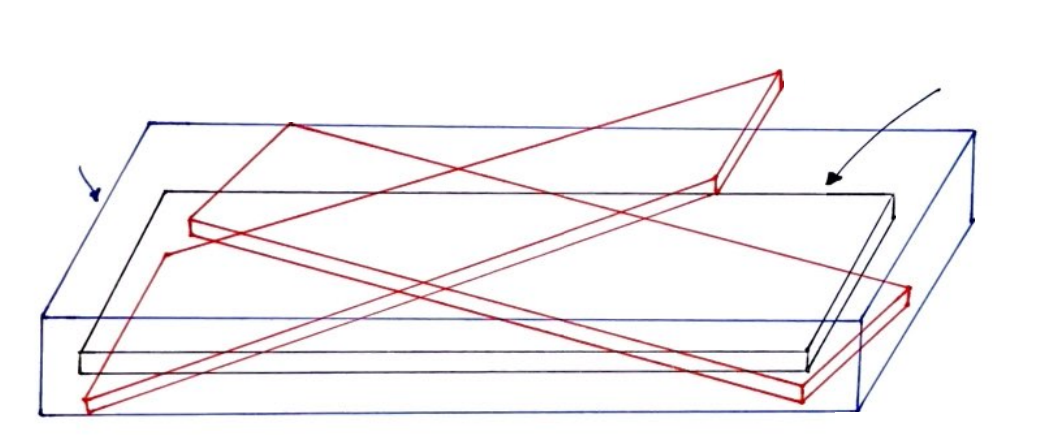}
 \put (90,32) {$P_0$}
 \put (-5,27) {$N_{b\theta}(P_0)$}
\end{overpic}
\caption{The geometry of Lemma \ref{OnePrismWithTangentPlanes}. 
For each $x\in Y_0(P_0)$, there is a prism (red) from $\mathcal{P}$ that intersects $P_0$ at $x$ at angle at least $\theta_{\operatorname{min}}$. For clarify, only two such prisms have been drawn. The union of these (red) prisms fills out a large fraction of every $b$-ball centered at a point of $P_0$.}
\label{prismsFillOutTubeFig}
\end{figure}

\begin{rem}
The exponent $\lambda^4$ in \eqref{nbhdMostlyFull} is not important --- the exponent $\lambda^{100}$ would work equally well for our applications of Lemma \ref{OnePrismWithTangentPlanes}.
\end{rem}
\begin{proof}
The argument is similar in spirit to Wolff's hairbrush argument from \cite{Wol95}. After dyadic pigeonholing, we can select a number $\theta_1\in[\theta, 1]$ and a set $Y'_0(P_0)\subset Y_0(P_0)$ with $|Y'_0(P_0)|\gtrapprox_a |Y_0(P_0)|$ so that 
\[
\sup_{P\in \mathcal{P}_Y(y)} \angle\big(\Pi(P_0),\Pi(P)\big)\sim\theta_1\quad\textrm{for each}\ y\in Y_0'(P_0).
\]
Applying Lemma \ref{everySetHasARegularShadingLem} (and further refining $Y_0'(P_0)$ by a factor of 2), we may suppose that $Y_0'(P_0)$ is regular.

Let $b'=b\frac{\theta}{\theta_1}$; note that $a\leq b'\leq b$. Divide $P_0$ into sub-prisms of dimensions $a\times b'\times b'$; we have that $\gtrapprox \lambda \frac{|P|}{a (b')^2}$ of these sub-prisms intersect $Y(P)$. Thus to obtain \eqref{nbhdMostlyFull}, it suffices to show that for each $x\in Y'(P_0)$, we have
\begin{equation}\label{localizedNBThetaNbhdBd}
\Big|N_{b \theta}\big(P_0 \cap B(x,b')\big) \cap \bigcup_{P\in\mathcal{P}}Y(P)\Big| \gtrapprox_a \lambda^3 (b \theta)(b')^2.
\end{equation}

Fix a choice of $x\in Y'(P_0)$. Our goal is to show that \eqref{localizedNBThetaNbhdBd} holds for this choice of $x$. Let $E = Y'(P_0)\cap B(x,b')$. Since the shading $Y'(P_0)$ is regular, we have $|E|\gtrapprox_a \lambda a(b')^2$. For each $y\in E$, let $P_y\in\mathcal{P}_Y(y)$ be a prism with $\angle(\Pi(P_0),\Pi(P))\sim\theta_1$.

$B(x,b')\cap \bigcup_{y\in E}P_y$ is contained in a prism of dimensions comparable to $b\theta \times b'\times b'$ that is concentric with the $a\times b'\times b'$ prism $B(x,b')\cap P$; denote this $b\theta\times b'\times b'$ prism by $Q$. Then $\phi_Q(B(x,b')\cap P)$ is comparable to a prism of dimensions $\frac{a}{b\theta}\times 1\times 1$. For notational convenience, we will select coordinates so that this prism is given by $\tilde P = [0, \frac{a}{b\theta}]\times[0,1]\times[0,1]$. 

For each $y\in E$, we have that $\phi_Q(B(x,b')\cap P_y)$ is comparable to a $\frac{a}{b\theta}\times 1\times 1$ prism, and each such prism intersects the $(z_2,z_3)$-plane with angle $\sim 1$, i.e.~the projection of the normal vector $v_y$ of this prism to the $(z_2,z_3)$ plane has magnitude $\gtrsim 1$. 
Since the shading of each $P_y\in\mathcal{P}$ is regular, we have that $\phi_Q \big(Y(P_y)\cap B(x,b')\big)$ is a subset of the prism $\phi_Q \big(P_y\cap B(x,b')\big)$ that has density $\gtrapprox_a\lambda$. 

After pigeonholing, we can select a set $E'\subset E$ with $|E'|\gtrsim|E|$, so that for each $y\in E'$, $\phi_Q\big(P_y\cap B(x,b')\big)$ is comparable to a $\frac{a}{b\theta}\times 1\times 1$ prism whose normal vector $v_y$ makes angle $\leq 1/100$ with some fixed unit vector $v$, and the projection of $v$ to the $(z_2,z_3)$ plane has magnitude $\gtrsim 1$. Let $v'$ be the projection of $v$ to the $(z_2,z_3)$ plane. 

By Fubini, we can find a line segment $L\subset [0,\frac{a}{b\theta}]\times[0,1]\times[0,1]$ pointing in direction $v'$, with $|L\cap\phi_Q(E')|\geq|\phi_Q(E')|$; here the left $|\cdot|$ denotes one-dimensional Lebesgue measure, while the right $|\cdot|$ denotes three-dimensional Lebesgue measure. Let $y_1,\ldots,y_N$, $N\gtrsim \frac{b\theta}{a}|L\cap\phi_Q(E')| \gtrapprox_a \lambda \frac{b\theta}{a}$ be a $\frac{a}{b\theta}$-separated subset of $L\cap\phi_Q(E')$, and let 
\[
\mathcal{S} = \big\{\phi_Q\big(P_{y_i}\cap B(x,b')\big)\colon i=1,\ldots,N\big\}.
\] 
$\mathcal{S}$ is a set of convex subsets of $\RR^3$, each of which is comparable to a $\frac{a}{b\theta}\times 1\times 1$ slab. Each $S\in\mathcal{S}$ has a $\gtrapprox_a \lambda$-dense shading (the shading is the set $\phi_Q\big(Y(P_{y_i})\cap B(x,b')\big)$), so in particular the pair $(\mathcal{S},Y)_{\frac{a}{b\theta}\times 1\times 1}$ is $\gtrapprox_a\lambda$ dense.  

We claim that $\CKT(\mathcal{S})\lesssim 1$. To verify this, let $W\subset\RR^3$ be a convex set, and suppose $\#\{S\in\mathcal{S}\colon S\subset W\}=M$. We need to show that
\begin{equation}\label{boundOnM}
M\lesssim |W||S|^{-1}.
\end{equation}
 If $M=0$ there is nothing to prove. If $M=1$, then clearly $|W|\geq|S|$. Suppose instead that $M\geq 2$. Then $|W\cap L|\geq (M-1)\frac{a}{b\theta}\geq\frac{M}{2b\theta}\gtrsim M|S|$, since $W\cap L$ contains at least $M$ points on $L$ that are $\frac{a}{b\theta}$-separated. On the other hand, since $\mathcal{S}[W]$ is non-empty, we know that $W$ contains a $\frac{a}{b\theta}\times 1\times 1$ prism $S\in \mathcal{S}[W]$ whose normal direction makes angle $\leq\frac{1}{50}$ with the direction of $v$ (whose projection to the $(z_2, z_3)$-plane is $v'$, i.e.~the direction of $L$). Since $W$ is convex, $W$ contains the convex hull of $S\cup (W\cap L)$, which has three-dimensional volume $\gtrsim |W\cap L|\gtrsim M|S|$. Thus $|W|\gtrsim M|S|,$ which gives \eqref{boundOnM}.

Applying Lemma \ref{cordobaLem} (a Cordoba-type $L^2$ argument for slabs), we conclude that 
\[
\Big|\bigcup_{S\in\mathcal{S}}Y(S)\Big|\gtrapprox_a\lambda^2(\#\mathcal{S})|S| \gtrapprox_a \lambda^2\big(\lambda \frac{b\theta}{a}\big)\big(\frac{a}{b\theta}\big) = \lambda^3.
\]
To conclude the proof, we verify that 
\begin{equation}
\begin{split}
\Big|N_{b \theta}\big(P_0 \cap B(x,b')\big) \cap \bigcup_{P\in\mathcal{P}}Y(P)\Big| 
&\geq \Big|N_{b \theta}\big(P_0 \cap B(x,b')\big)\cap\bigcup_{y\in E'}Y(P_y)\Big|\\
&\geq (b\theta)(b')^2 \Big|\phi_Q\Big(B(x,b')\cap\bigcup_{y\in E'}Y(P_y)\Big)\Big|\\
&\geq (b\theta)(b')^2 \Big|\phi_Q\Big(B(x,b')\cap\bigcup_{i=1}^N Y(P_{y_i})\Big)\Big|\\
&\geq (b\theta)(b')^2 \Big|\bigcup_{S\in\mathcal{S}}Y(S)\Big| \gtrapprox_a \lambda^3 (b\theta)(b')^2.\qedhere
\end{split}
\end{equation}
This establishes \eqref{localizedNBThetaNbhdBd}, as desired.
\end{proof}


In practice, we will often use Lemma \ref{OnePrismWithTangentPlanes} in situations where each prism in $P_0\in \mathcal{P}$ satisfies the hypotheses of the lemma. The following definitions help make that precise.

\begin{defn}\label{thetaMinDefn}
Let $(\mathcal{P},Y)_{a\times b\times c}$ be a collection of prisms and their associated shading, and let $x\in\bigcup_{P\in\mathcal{P}}Y(P)$. We define
\[
\theta(x) = \frac{a}{b}+ \sup_{P,P'\in\mathcal{P}_Y(x)}\angle(\Pi(P),\Pi(P')),
\]
where $\mathcal{P}_Y(x) = \{P\in\mathcal{P}\colon x\in Y(P)\}$. We define 
\[
\theta_{\operatorname{min}}=\inf \theta(x),
\]
where the infimum is taken over all $x\in\bigcup_{P\in\mathcal{P}}Y(P)$. Note that $\theta_{\operatorname{min}}$ depends on the pair $(\mathcal{P},Y)_{a\times b\times c}$. The choice of $\mathcal{P}$ and $Y$ will be apparent from context.
\end{defn}

Be considering each prism $P_0\in\mathcal{P}$ in turn, Lemma \ref{OnePrismWithTangentPlanes} now has the following corollary
\begin{cor}\label{prismsWithTangentPlanes}
Let $0<a\leq b\leq c$, $\lambda>0$. Let $(\mathcal{P},Y)_{a\times b\times c}$ be a set of prisms and their associated shading. Suppose that each shading $Y(P)$ is regular, in the sense of Definition \ref{regularShading}, and satisfies $|Y(P)|\geq\lambda|P|$.
Then for each $P_0\in\mathcal{P}$, we have
\begin{equation}\label{nbhdMostlyFull}
\Big|N_{b \theta_{\operatorname{min}}}(P_0) \cap \bigcup_{P\in\mathcal{P}}Y(P)\Big| \gtrapprox_a \lambda^4|N_{b \theta_{\operatorname{min}}}(P)|.
\end{equation}
\end{cor}


\subsection{Assertions $\cF$, $\cE$, and $\TcE$ are equivalent}\label{tubesVsPrismsSubSec}
As noted in Remark \ref{discussionOfDRemark}, Situation 1, we have $\cF(\sigma,\omega)\implies \cE(\sigma,\omega)\implies \TcE(\sigma,\omega)$. In this section we will prove that the reverse implications hold. More precisely, we have the following. 


\begin{prop}\label{EiffF}
Let $0 <\sigma\leq 2/3,\ \omega> 0$. Then $\cF(\sigma,\omega) \Longleftrightarrow \cE(\sigma,\omega) \Longleftrightarrow \TcE(\sigma,\omega)$.
\end{prop}

The goal of Section \ref{tubesVsPrismsSubSec} is to prove Proposition \ref{EiffF}. We begin with several lemmas that describe the structure of arrangements of rectangular prisms. Recall the quantity $\theta_{\operatorname{min}}$ from Definition \ref{thetaMinDefn}. When $\theta_{\operatorname{min}}\sim 1$, then a typical pair of intersecting prisms intersect transversely, and by Corollary \ref{prismsWithTangentPlanes} this means that each prism can be replaced by a thickened neighbourhood that is comparable to an $a$-tube. In the next lemma, we will analyze what happens when $\theta_{\operatorname{min}}$ is small. I.e.~we consider a collection of $a\times b\times 1$ prisms $\mathcal{P}$ that intersect tangentially, in the sense that their tangent planes make small angle at a typical point of intersection. 

Informally, the argument is as follows. We will begin by describing three Moves, which we will then apply repeatedly. Note that these are \emph{not} the Moves described in Section \ref{twoScaleGrainsDecompIntro} --- those Moves will be described in Sections \ref{twoScaleGrainsSec} and \ref{moves123Sec}.

Since the tangent plane of a prism is constant along the entire prism, this means that the collection of prisms can be partitioned into smaller sub-collections $\mathcal{P}_1,\ldots,\mathcal{P}_N$, where for each index $i$, the normal vector of the tangent plane of each prism in $\mathcal{P}_i$ makes small angle with a fixed unit vector $v_i$. If $\mathcal{P}_i$ and $\mathcal{P}_j$ are two such sub-collections, then the corresponding sets $\bigcup_{P\in\mathcal{P}_i} P$ and $\bigcup_{P\in\mathcal{P}_j} P$ are mostly disjoint. We can cover the prisms in each set $\mathcal{P}_i$ by a set of slabs whose tangent planes have normal vector $v_i$. We then rescale and continue this process inside each slab. This is Move $\#1$.

Next, we can apply Proposition \ref{slabWolffFactoring} (factoring with respect to the Frostman Slab Wolff Axioms) to further break our collection of prisms into sub-collections, each of which satisfy the Frostman Slab Wolff Axioms with small error. This is Move $\#2$.

Finally we apply Lemma \ref{everySetHasARegularShadingLem} (every shading has a regular sub-shading). This is Move $\#3$.

We iteratively apply Moves 1, 2, and 3 until our set of prisms $\mathcal{P}$ has been covered by a set $\mathcal{W}$ of convex subsets of $\RR^3$, so that each set $\mathcal{P}^W$ satisfies the hypotheses of Corollary \ref{prismsWithTangentPlanes} with $\theta_{\operatorname{min}}$ about 1, and $\mathcal{P}^W$ satisfies the Frostman Slab Wolff Axioms with error about 1. The precise statement is as follows.


\begin{lem}\label{isolatePrismsBroad}
For all $\eps>0$, there exists $\eta,c>0$ so that the following holds for all $0<a\leq b\leq 1$. Let $(\mathcal{P},Y)_{a\times b\times 1}$ be $a^\eta$ dense. Then there exists a refinement $(\mathcal{P}',Y')_{a\times b\times 1}$ with $\sum_{P'\in\mathcal{P}'}|Y'(P')|\gtrapprox_a a^\eps\sum_{P\in\mathcal{P}}|P|$ and cover $\mathcal{W}$ of $\mathcal{P}'$ by congruent convex sets, such that the following holds.
\begin{itemize}
	\item[(A)] The shading $Y'$ is regular, in the sense of Definition \ref{regularShading}.
	\item[(B)] $\mathcal{W}$ is a $\approx_a 1$ balanced cover of $\mathcal{P}'$, and factors $\mathcal{P}'$ from below with respect to the Frostman Slab Wolff axioms, with error $a^{-\eps}$. 
	\item[(C)] All of the convex sets in $\bigcup_{W\in\mathcal{W}}\mathcal{P}^{\prime W}$ have the same dimensions up to a multiplicative factor of 2, i.e.~each one is comparable (after a suitable rigid transformation) to a common convex set $\tilde P$ (see Remark \ref{failureToBeCongrudentAfterRescaling} below).
	\item[(D)] For each $W\in\mathcal{W}$, the sets $\mathcal{P}^{\prime W}$ and their associated shading satisfy $\theta_{\operatorname{min}}\geq a^\eps$, in the sense of Definition \ref{thetaMinDefn}.
	\item[(E)] The sets $\bigcup_{P'\in\mathcal{P}'[W]}Y'(P)$, $W\in\mathcal{W}$ are pairwise disjoint.
\end{itemize}
\end{lem}
\begin{rem}\label{failureToBeCongrudentAfterRescaling}
Note that even though the prisms in $\mathcal{P}$ (resp.~$\mathcal{W}$) all have the same dimensions, it could happen that the prisms in $\mathcal{P}^W$ have differing dimensions --- see Figure \ref{nonCongruentImagesFig} for an example of this phenomena. Conclusion (C) (which is achieved by pigeonholing) ensures that this does not happen.
\end{rem}

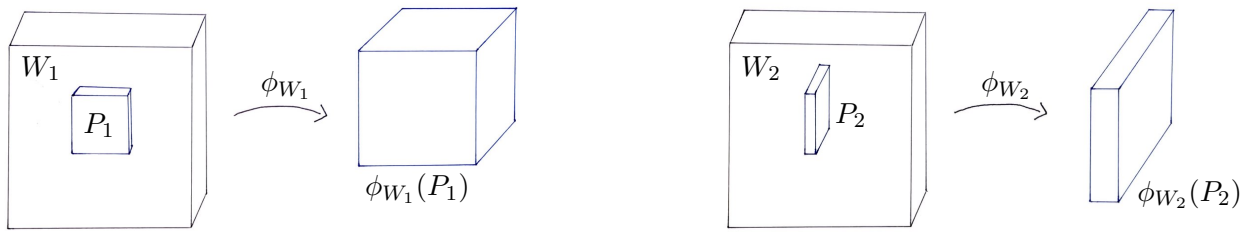
\begin{figure}[h!]
\begin{overpic}[width=.42\linewidth]{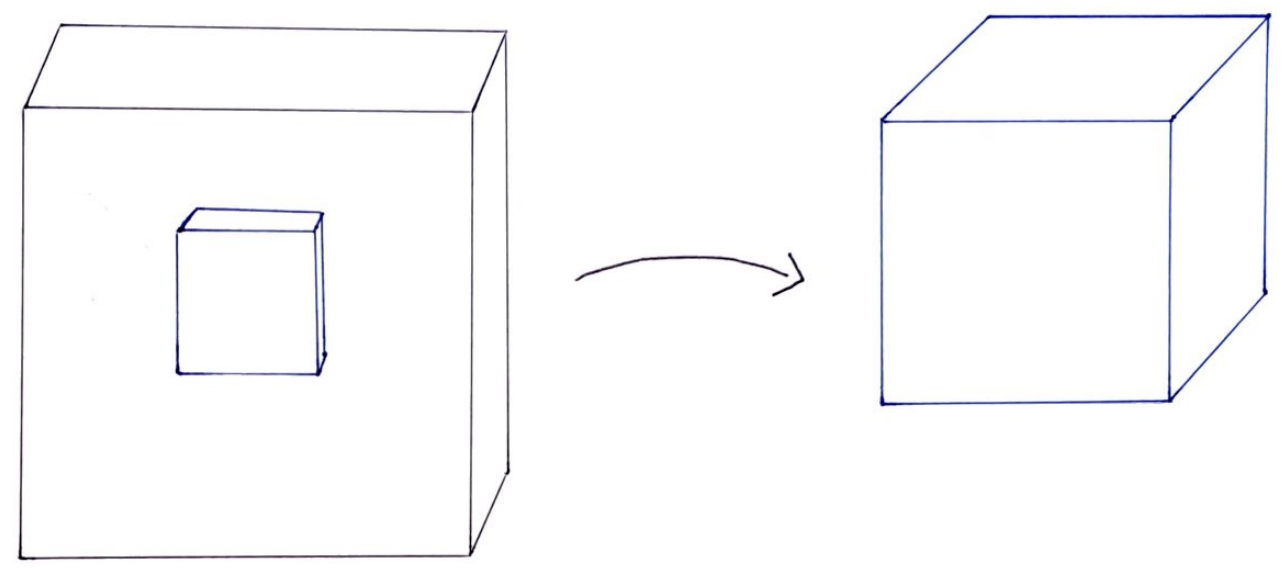}
	\put (4,31) {$W_1$}
	\put (16,20) {$P_1$}
	\put (50,28) {$\phi_{W_1}$}
	\put (70,8) {$\phi_{W_1}(P_1)$}
\end{overpic}
\hfill
\begin{overpic}[width=.42\linewidth]{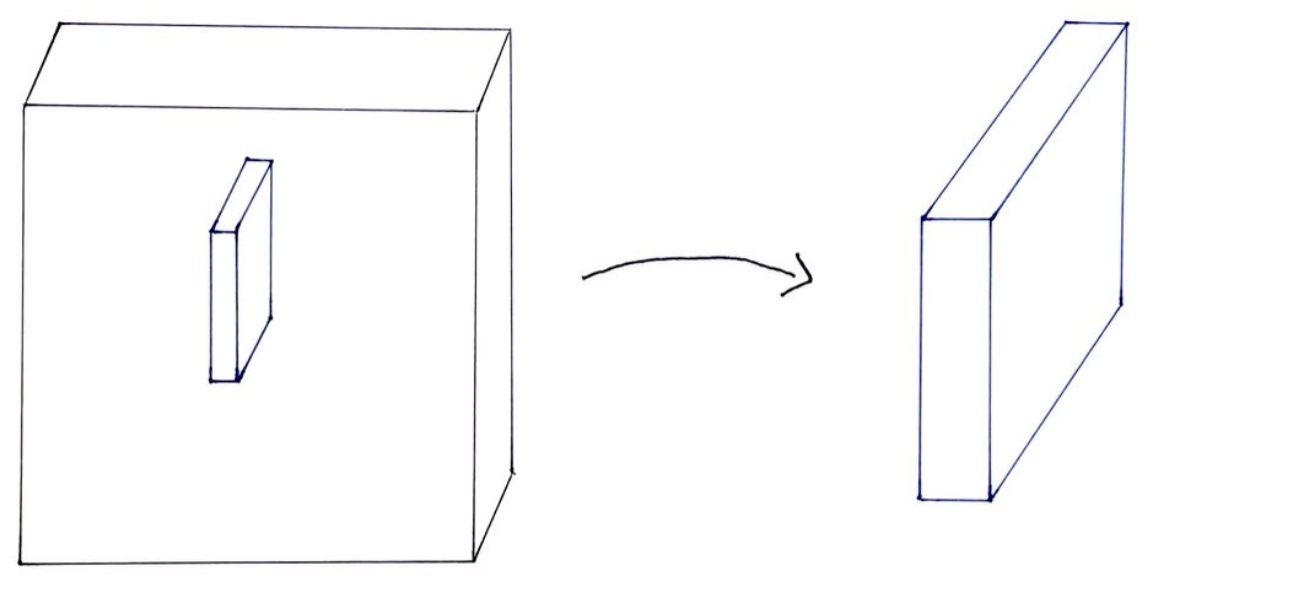}
\put (4,31) {$W_2$}
\put (22,22) {$P_2$}
\put (50,28) {$\phi_{W_2}$}
\put (80,7) {$\phi_{W_2}(P_2)$}
\end{overpic}
\caption{ An example where the convex sets $P_1$ and $P_2$ are congruent, and $U_1$ and $U_2$ are congruent, but $\phi_{U_1}(P_1)$ and $\phi_{U_2}(P_2)$ are not congruent.}
\label{nonCongruentImagesFig}
\end{figure}


\begin{proof}
The main ideas of the proof were already outlined above; in brief, we apply Moves 1, 2, and 3 described above, in that order. We iteratively repeat this process until Conclusions (A), (B), (D), and (E) of Lemma \ref{isolatePrismsBroad} are satisfied. Each iteration increases the volume of each surviving convex set by at least $a^{-\eps}$. On the other hand, each convex set initially had volume $ab\geq a^2$, and at each step all convex sets are contained inside the unit ball, and thus have volume at most $O(1)$. Hence the iterative process detailed above halts after at most $2/\eps$ steps. If $\eta>0$ is chosen sufficiently small (depending on $\eps$), then the resulting refinement will satisfy $\sum_{P'\in\mathcal{P}'}|Y'(P')|\gtrapprox_a a^\eps\sum_{P\in\mathcal{P}}|P|$. Finally, we dyadically pigeonhole the set $\mathcal{P}'$ to select a refinement that satisfies Conclusion (C).
\end{proof}


We are now ready to prove Proposition \ref{EiffF}. The idea is as follows. Given a set $(\mathcal{P},Y)_{a\times b\times 1}$ of prisms, we apply Lemma \ref{isolatePrismsBroad} to cover a refinement of  $\mathcal{P}$ by a collection of convex sets $\mathcal{W}$, and then apply Corollary \ref{prismsWithTangentPlanes} to each collection $\mathcal{P}^W$ --- this gives us a collection of $\tilde b$ tubes associated to $\mathcal{P}^W$, for some $\tilde b\geq b$. The collection of $\tilde b$ tubes satisfies the Frostman Slab Wolff axioms with error $\lessapprox 1$, and satisfies the Katz-Tao Convex Wolff axioms with some error $\tilde m$ that we will analyze later. We now apply the estimate $\TcE(\sigma,\omega)$ to this collection of $\tilde b$ tubes, and then undo the transformation $\phi_W$. Summing the contributions from each $W\in\mathcal{W}$, we obtain an estimate for $\Big|\bigcup Y(P)\Big|$ that becomes better as $\tilde m$ becomes larger. Thus we are faced with the task of estimating the size of $\tilde m$ --- this quantity is closely related to the quantity $D$ from \eqref{defnOfD}. We now turn to the details.

\begin{proof}[Proof of Proposition \ref{EiffF}]
It suffices to prove that $\TcE(\sigma,\omega)\implies \cF(\sigma,\omega)$. Fix $0 <\sigma\leq 2/3,\ \omega> 0$, and suppose $\TcE(\sigma,\omega)$ is true. Fix $\eps>0$. Then there exists $\eta_0>0$ so that the volume estimate \eqref{TCEEstimate} holds for all pairs $(\tubes,Y)_{\delta}$ that are $\delta^{\eta_0}$ dense and satisfy $\FS(\tubes)\leq\delta^{-\eta_0}$. 

Let $\kappa,\eta>0$ be quantities to be determined below. Let $0<a\leq b\leq 1$ and let $(\mathcal{P},Y)_{a\times b\times 1}$ be $a^{\eta}$ dense. Our goal is to prove that
\begin{equation}\label{desiredIneqForPropEiffF}
\Big|\bigcup_{P\in\mathcal{P}}Y(P)\Big|\geq \kappa a^\eps b^{\omega}m^{-1}(\#\mathcal P)|P|(m^{-3/2}\ell (\#\mathcal{P})|P|^{1/2})^{-\sigma}D^{-\sigma},
\end{equation}
with $m,\ell$, and $D$ as defined in \eqref{defnOfD}.

\medskip

\noindent {\bf Step 1.}
Let $\eta_1$ be a small quantity to be determined later (we will select $\eta_1$ very small compared to $\eta_0$, and $\eta$ very small compared to $\eta_1$). We will choose $\eta>0$ sufficiently small so that we can apply Lemma \ref{isolatePrismsBroad} with $\eta_1/4$ in place of $\eps$ to $(\mathcal{P},Y)_{a\times b \times 1}$. Denote the output of that lemma by $(\mathcal{P}_1,Y_1)_{a\times b\times 1}$ and $\mathcal{W}$. By Conclusion (B) of Lemma \ref{isolatePrismsBroad} we have $\FS(\mathcal{P}_1^{W})\leq a^{-\eta_1/4}$ for each $W\in\mathcal{W}_1$. By Conclusion (E) of Lemma \ref{isolatePrismsBroad}, we have
\begin{equation}\label{disjointnessOfP1}
\Big|\bigcup_{P\in\mathcal{P}}Y(P)\Big|\geq \sum_{W\in\mathcal{W}}\Big|\bigcup_{P\in\mathcal{P}_1[W]}Y_1(P)\Big|.
\end{equation}

By Conclusion (C) of Lemma \ref{isolatePrismsBroad}, there are numbers $\tilde a \leq\tilde b\leq 1$ with $\tilde b\geq b$ so that for each $W\in\mathcal{W}$ and each $P^W\in\mathcal{P}_1^W$, $P^W$ is comparable to a $\tilde a\times\tilde b\times 1$ prism. Conclusions (A) and (D) of Lemma \ref{isolatePrismsBroad} say that for each $W\in\mathcal{W}$, the pair $(\mathcal{P}_1^W,Y_1^W)_{\tilde a\times \tilde b\times 1}$ satisfies the hypotheses of Corollary \ref{prismsWithTangentPlanes}, and hence we can apply Corollary \ref{prismsWithTangentPlanes} to conclude that for each $P^W\in\mathcal{P}_1^W$ we have

\begin{equation}\label{densityOfTildeBNbhd}
\Big|N_{\tilde b}(P^W) \cap \bigcup_{P^{\prime W}\in\mathcal{P}_1^W}Y_1^W(P^{\prime W})\Big| \gtrsim  a^{\eta_1}|N_{\tilde b}(P^W)|.
\end{equation}
In words, \eqref{densityOfTildeBNbhd} says the following: for each $P^W\in \mathcal{P}_1^W$, the set $N_{\tilde b}(P^W)$ is comparable to a $\tilde b$-tube. This $\tilde b$-tube has almost full intersection with the union of shadings $\bigcup_{P^{\prime W}\in\mathcal{P}_1^W}Y_1^W(P^{\prime W})$. See Figure \ref{prismsFillOutTubeFig} for an illustration of this situation.

\medskip
\noindent{\bf Step 2.} As noted above, each set $N_{\tilde b}(P^W)$ is comparable to a $\tilde b$ tube. Let $\tubes_W$ be a maximal, essentially distinct subset of $\{N_{\tilde b}(P^W)\colon P^W\in\mathcal{P}_1^W\}$. For each $T\in\tubes_W$, define the shading
\[
Y(T) = T\cap \bigcup_{P^W\in\mathcal{P}_1^W}Y_1^W(P^W).
\]
Then \eqref{densityOfTildeBNbhd} says that $(\tubes_W,Y)_{\tilde b}$ is $\gtrsim  a^{\eta_1}$ dense. After pigeonholing and refining  $\mathcal{W}$ and $\tubes_W$ (which induces a refinement on $\mathcal{P}_1^W$), we can ensure that $\#\mathcal{P}_1^W[T]$ is roughly the same (up to a factor of 2) for each $T\in\tubes_W, W\in \mathcal{W}$. Abusing notation, we continue to refer to these refined sets as $\mathcal{W}, \tubes_W$ and $\mathcal{P}_1^W$. 

At this point, $\tubes_W$ is a balanced cover of $\mathcal{P}_1^W$, and we still have $\FS(\mathcal{P}_1^W)\lessapprox_a a^{-2\eta_1}.$ Since $\tubes_W$ is a balanced cover of $\mathcal{P}_1^W$, by Remark  \ref{FrostmanWolffInheritedUpwardsDownwards}(A) (i.e.~Frostman Wolff constants are inherited upwards), we have that $\tilde\ell := \max_{W\in\mathcal{W}}\FS(\tubes_W)\lessapprox_a a^{-2\eta_1}$. Finally, define $\tilde m := \max_{W\in\mathcal{W}} \CKT(\tubes_W)$.

\medskip
\noindent{\bf Step 3.} At this point, we have covered our set of prisms $\mathcal{P}_1$ by a collection of convex sets $\mathcal{W}$. For each rescaled set $\mathcal{P}_1^W$, we have located a collection of tubes $\tubes_W$, whose shadings are almost full. The relation between the volumes of these objects is as follows:
\begin{equation}\label{relationVolumePWT}
\Big|\bigcup_{P\in\mathcal{P}_1[W]}Y_1(P)\Big| \sim |W|\Big|\bigcup_{P^W\in\mathcal{P}_1^W}Y_1^W(P^W)\Big| \geq |W|\Big|\bigcup_{T\in\tubes_W}Y(T)\Big|.
\end{equation}

Our next task is to estimate the volume of $\bigcup_{\tubes_W}Y(T)$. First we consider the case where $\tilde b \leq a^{\eps/5}$. In this case, $(\tubes_W,Y)_{\tilde b}$ is $\gtrsim {\tilde b}^{5\eta_1/\eps}$ dense, and $\FS(\tubes_W)\lessapprox_a {\tilde b}^{-10\eta_1/\eps}$. If $\eta_1$ is selected sufficiently small depending on $\eta_0$ and $\eps$ (for example, $\eta_1 = \eps\eta_0/20$ will suffice), then we can apply the estimate $\TcE(\sigma,\omega)$ with $\eps/2$ in place of $\eps$ to conclude that
\begin{equation}\label{volumeBdUnionTildeP}
\begin{split}
\Big|\bigcup_{T\in\tubes_W}Y(T)\Big| & \gtrsim  \tilde b^{\omega+\eps/2}\tilde m^{-1}(\#\tubes_W)|T|\Big(\tilde m^{-3/2}\tilde \ell (\#\tubes_W)|T|^{1/2} \Big)^{-\sigma}\\
& \gtrapprox_a a^{\eps/2+2\eta_1\sigma} b^{\omega}\tilde m^{-1}(\#\tubes_W)|T|\Big(\tilde m^{-3/2} (\#\tubes_W)|T|^{1/2} \Big)^{-\sigma}.
\end{split}
\end{equation}
On the other hand, if $\tilde b  \geq a^{\eps/5}$, then the estimate \eqref{volumeBdUnionTildeP} follows from the fact that for every $T\in\tubes_W$ we have
\begin{equation}
\begin{split}
\Big|\bigcup_{T\in\tubes_W}Y(T)\Big| & \geq |Y(T)| \gtrsim a^{\eta_1}{\tilde b}^2\geq a^{\eps/2},
\end{split}
\end{equation}
which is stronger than \eqref{volumeBdUnionTildeP}.

\medskip

\noindent{\bf Step 4.}
We have estimated the volume of $\bigcup_{\tubes_W}Y(T)$ for each $W\in\mathcal{W}$. Our next task is to combine these estimates in order to estimate the volume of $\bigcup_{P\in\mathcal{P}}Y(P)$. 

We know that each prism $W\in\mathcal{W}$ has the same dimensions up to a factor of 2; call these dimensions $s\times t\times 1$. Then for each $W\in\mathcal{W}$ and each $T\in \tubes_W$, we have that $\phi_W^{-1}(T)$ is comparable to a rectangular prism of dimensions $s\tilde b\times t\tilde b\times 1$. 

Let $\hat{\mathcal{P}}_W=\{ \phi_W^{-1}(T)\colon T\in\tubes_W\}$, i.e.~$\hat{\mathcal{P}}_W$ is a set of $s\tilde b\times t\tilde b\times 1$ prisms contained in $W$; $|T|\sim|\hat P|/|W|$; and $\#\hat{\mathcal{P}}_W=\#\tubes_W$. For each $W\in\mathcal{W}$ and each $\hat P =  \phi_W^{-1}(T) \in \hat{\mathcal{P}}_W$, define the natural shading $\hat Y(\hat P) = \phi_W^{-1}(Y(T))$.

\eqref{disjointnessOfP1} allows us to combine the (rescaled) volume estimates \eqref{relationVolumePWT} and \eqref{volumeBdUnionTildeP} from each $W\in\mathcal{W}$. Defining $\hat{\mathcal{P}}=\bigsqcup \hat{\mathcal{P}}_W$, we have
\begin{equation}\label{volumeEstimateY2T}
\begin{split}
\Big|\bigcup_{P\in\mathcal{P}}Y(P)\Big| & \gtrapprox_a |W| \sum_{W\in\mathcal{W}}\Big[a^{\eps/2+\eta_1\sigma}b^{\omega}\tilde m^{-1}(\#\hat{\mathcal{P}}_W)\frac{|\hat P|}{|W|}\Big(\tilde m^{-3/2} (\#\hat{\mathcal{P}}_W)\Big(\frac{|\hat P|}{|W|}\Big)^{1/2} \Big)^{-\sigma}\Big].\\
& \approx_a a^{\eps/2+\eta_1\sigma} b^{\omega}\tilde m^{-1}(\#\hat{\mathcal{P}})|\hat P|\Big(\tilde m^{-3/2} \frac{\#\hat{\mathcal{P}}}{\#\mathcal{W}}\Big(\frac{|\hat P|}{|W|}\Big)^{1/2} \Big)^{-\sigma}.
\end{split}
\end{equation} 
 
\medskip

\noindent{\bf Step 5.} To understand the RHS of \eqref{volumeEstimateY2T} we must estimate $\tilde m$. Recall that in Step 2, we have pigeonholed to ensure that $\#\mathcal{P}_1^W[T]$ is roughly the same for each $T\in\tubes_W,\ W\in\mathcal{W}$. Thus each $\hat P\in\hat{\mathcal{P}}$ satisfies $\#\mathcal{P}_1[\hat P]\approx_a \frac{\#\mathcal{P}_1}{\#\hat{\mathcal{P}}}$. Thus we have
\begin{equation}\label{estimateForTildeM}
\tilde m \lessapprox_a m \frac{|\hat P|}{|P|}\frac{\#\hat{\mathcal{P}}}{\#\mathcal{P}_1}.
\end{equation}
Thus we can estimate the RHS of \eqref{volumeEstimateY2T} as follows.
\begin{equation}\label{secondEstimateBigcupCalP}
\begin{split}
\Big|\bigcup_{P\in\mathcal{P}}Y(P)\Big| & \gtrapprox_a a^{\eps/2+\eta_1\sigma} b^{\omega}m^{-1}(\#\mathcal{P}_1)|P|
\Big( m^{-3/2}\Big(\frac{|P|}{|\hat P|}\frac{\#\mathcal{P}_1}{\#\hat{\mathcal{P}}}\Big)^{3/2} \frac{\#\hat{\mathcal{P}}}{\#\mathcal{W}}\Big(\frac{|\hat P|}{|W|}\Big)^{1/2} \Big)^{-\sigma}\\
& 
\gtrapprox_a a^{\eps/2+3\eta_1} b^{\omega}m^{-1}(\#\mathcal{P})|P|
\Big( m^{-3/2} (\#\mathcal{P})|P|^{1/2}\Big)^{-\sigma}\\
&\qquad\qquad\qquad\qquad\qquad\cdot \Big[(\#\mathcal{W})|W|^{1/2}\Big]^{\sigma}\Big[\frac{|P|}{|\hat P|}\Big(\frac{\#\mathcal{P}}{\#\hat{\mathcal{P}}}\Big)^{1/2}\Big]^{-\sigma}.
\end{split}
\end{equation}
In the above computation, we used the fact that $\#\mathcal{P}_1\gtrapprox_a\delta^{-\eta_1}(\#\mathcal{P}).$

\medskip

\noindent{\bf Step 6.}
Compare the RHS of \eqref{secondEstimateBigcupCalP} with \eqref{desiredIneqForPropEiffF}. It remains to analyze the final two terms on the RHS of \eqref{secondEstimateBigcupCalP}. We begin with the penultimate term. Since the sets in $\mathcal{W}$ are convex and have diameter $\sim 1$, each $W\in\mathcal{W}$ is contained in a slab of volume  $O(|W|^{1/2})$. Thus
\[
\#\mathcal{P}_1[W] \lesssim \FS(\mathcal{P})|W|^{1/2}(\#\mathcal{P}).
\]
Since $\#\mathcal{P}_1[W]\approx_a \frac{\#\mathcal{P}_1}{\#\mathcal{W}}\gtrapprox_a a^{\eta_1} \frac{\#\mathcal{P}}{\#\mathcal{W}}$, we conclude that $(\#\mathcal{W})|W|^{1/2}\gtrapprox_a a^{\eta_1} \FS(\mathcal{P})^{-1}=a^{\eta_1}\ell^{-1}$. 

\medskip

\noindent{\bf Step 7.}
We now turn to the final term on the RHS of \eqref{secondEstimateBigcupCalP}. Recall that $\mathcal{P}$ is a collection of prisms of dimensions $a\times b\times 1$, while $\hat{\mathcal{P}}$ is a collection of prisms that all have the same, but unknown, dimensions  $s\tilde b\times t\tilde b\times 1$ --- call these dimensions $a'\times b'\times 1$, with $a\leq a'\leq b'$ and $b'\geq b$. Our desired estimate \eqref{desiredIneqForPropEiffF} will follow from \eqref{secondEstimateBigcupCalP} and the estimate
\begin{equation}\label{needToControlPInsideHatP}
\frac{|P|}{|\hat P|}\Big(\frac{\#\mathcal{P}}{\#\hat{\mathcal{P}}}\Big)^{1/2} \lesssim D = \max_{P\in\mathcal{P}}\sup_{\rho\in[a,b]}\frac{|P|}{|N_\rho(P)|}\big(\#\mathcal{P}[N_{\rho}(P)]\big)^{1/2}.
\end{equation}
\medskip

Fix $\hat P\in\hat{\mathcal P}$. Let $\mathcal{P}^\dag$ be a maximal set of essentially distinct $\min(a',b)\times b\times 1$ prisms contained in $\hat P$, so that each $P\in\mathcal{P}[\hat P]$ is contained in at least one $P^\dag\in\mathcal{P}^\dag$. We claim that
\begin{equation}\label{cardPDag}
\#\mathcal{P}^\dag \sim (|\hat P|/|P^\dag|)^2.
\end{equation}
Indeed, when $a'\leq b$, the RHS of \eqref{cardPDag} is $(b'/b)^2$ and the numerology comes from the fact that a $b'\times 1$ rectangle can be filled with about $(b'/b)^2$ essentially distinct $b\times 1$ rectangles. When $a'\geq b$ the RHS of \eqref{cardPDag} is $(a'b'/b^2)^2$, and the numerology comes from the fact that a $a'\times b'\times 1$ prism can be filled with about $(b'/a')^2$ essentially distinct $a'\times a'\times 1$ tubes, and each of these tubes can be filled with about $(a'/b)^4$ essentially distinct $b\times b\times 1$ tubes.

By the definition of $D$, we have
$
D\geq  \frac{|P|}{|P^{\dag}|}\big(\#\mathcal{P}[P^{\dag}]\big)^{1/2},
$
i.e. $\#\mathcal{P}[P^{\dag}] \leq \big(D \frac{|P^{\dag}|}{|P|}\big)^2.$ 
Thus by \eqref{cardPDag}, 
\[
\#\mathcal{P}[\hat P]\lesssim\Big(\frac{|\hat P|}{|P^\dag|}\Big)^2 \#\mathcal{P}[P^{\dag}] \lesssim  \Big(\frac{|\hat P|}{|P^\dag|}\Big)^2 \Big(D \frac{|P^{\dag}|}{|P|}\Big)^2,
\]
which is \eqref{needToControlPInsideHatP}.
\end{proof}


\subsection{Proof of Proposition \ref{aLLbProp}: Tubes that factor through flat boxes}
With Proposition \ref{EiffF} in hand, we are now ready to prove Proposition \ref{aLLbProp}. Fix $\omega>0,\ 0<\sigma\leq 2/3$, and suppose $\cE(\sigma,\omega)$ is true (and thus by Proposition \ref{EiffF}, $\cF(\sigma,\omega)$ is true). Let $\kappa,\eta>0$ be small quantities to be specified below. Let $(\tubes,Y)_{\delta}$ be $\delta^{\eta}$ dense, let $\delta\leq a\leq b\leq 1$, and let $\mathcal{W}$ be a set of congruent copies of an $a\times b\times 2$ prism $W_0$, as described in the statement of Proposition \ref{aLLbProp}.

\medskip

\noindent {\bf Step 1.}
 After dyadic pigeonholing, we can find a refinement $(\tubes_1,Y_1)_\delta$ of $(\tubes,Y)_\delta$ with $\sum_{T\in\tubes_1}|Y_1(T)|\gtrapprox_\delta\sum_{T\in\tubes}|Y(T)|$, and a subset $\mathcal{W}_1\subset\mathcal{W}$ so that for each $W\in\mathcal{W}_1$ and each $x \in \bigcup_{T\in\tubes_1[W]}Y_1(T)$, we have
	\begin{equation}\label{densityRho}
		\Big| B(x, a) \cap \bigcup_{T\in\tubes}Y(T)\Big| \sim \lambda |B(x,a)|,\quad \textrm{with}\quad \lambda \geq |W|^{-1}\Big|\bigcup_{T\in\tubes_1[W]}Y_1(T)\Big|,
	\end{equation}
where the ``density'' $\lambda=\lambda(W)$ is the same (up to a factor of 2) for all $W\in\mathcal{W}$. In words, \eqref{densityRho} says that if we blur the shading $\bigcup_{\tubes_1[W]}Y_1(T)$ at scale $a$ (for example by convolving with the characteristic function of $B(0,a)$), then each point in the shading has density $\sim \lambda$.

After further pigeonholing, we can ensure that each set $(\tubes_1[W], Y_1)_{\delta},\ W\in\mathcal{W}_1$ is $\approx_\delta \delta^{O(\eta)}$ dense; $\mathcal{W}_1$ is a $\approx_\delta \delta^{-O(\eta)}$ balanced cover of $\tubes_1$; and each set $\tubes_1^W$ obeys the Frostman Convex Wolff axioms with error $\lessapprox_\delta \delta^{-O(\eta)}$. Abusing notation, we will continue to refer to the output of this pigeonholing by $(\tubes_1,Y_1)_\delta$ and $\mathcal{W}_1$.

\medskip

\noindent {\bf Step 2.} Fix $W\in\mathcal{W}_1$. Then $(\tubes_1^W, Y_1^W)_{\delta/b\times \delta/a\times 1}$ is a collection of convex sets, each comparable to a $\delta/b\times \delta/a\times 1$ prism, with a $\delta^{O(\eta)}$-dense shading. We first consider the case where $b\geq\delta^{1-\eps/100}$, so $\delta/b \leq \delta^{\eps/100}$. In this case, the shading on $(\tubes_1^W, Y_1^W)_{\delta/b\times \delta/a\times 1}$ is $(\delta/b)^{O(\eta)/\eps}$ dense. If $\eta$ is selected sufficiently small, then we can apply the estimate $\cF(\sigma,\omega)$ with $\eps/2$ in place of $\eps$ (recall that $\cF(\sigma,\omega)$ is true in light of Proposition \ref{EiffF}) to conclude that
\begin{equation}\label{firstEstimateRescaledTubes}
\Big| \bigcup_{T^W \in \tubes_1^W}Y_1^W(T^W)\Big| \geq \kappa_\eps \big(\frac{\delta}{b}\big)^{\eps/2} \big(\frac{\delta}{a}\big)^\omega m_W^{-1}(\#\tubes^W)|T^W|
\Big(m_W^{-3/2}\ell_W(\#\tubes^W)|T^W|^{1/2}\Big)^{-\sigma},
\end{equation}
where $m_W = \CKT(\tubes^W)$ and $\ell_W = \FS(\tubes^W)\leq \CFC(\tubes^W)\lessapprox_\delta\delta^{-O(\eta)}$. Recall as well that $|T^W| \sim |T|/|W|$. Note that the estimate $\cF(\sigma,\omega)$ involves the additional term ``$D$'' defined in \eqref{defnOfD}, but as discussed in Remark \ref{discussionOfDRemark}, Situation 3, we have $D\lesssim 1$.

Since $\CFC(\tubes^W)\lessapprox_\delta \delta^{-O(\eta)}$, we have that $m_W\lessapprox_\delta \delta^{-O(\eta)}|T^W|(\#\tubes^W)$, and thus \eqref{firstEstimateRescaledTubes} allows us to estimate the density $\lambda$ (recall that $\lambda$ was defined in \eqref{densityRho}):
\begin{equation}\label{secondEstimateRescaledTubes}
\lambda \geq \Big| \bigcup_{T^W \in \tubes^W}Y^W(T^W)\Big| \gtrapprox_\delta \kappa_\eps \big( \frac{ \delta}{b}\big)^{\eps/2} \delta^{O(\eta)} \big(\frac{\delta}{a}\big)^\omega 
\Big((\#\tubes^W)^{1/2}|T^W|\Big)^{\sigma}.
\end{equation}
Note that Inequality \eqref{secondEstimateRescaledTubes} is currently only valid when $b\geq\delta^{1-\eps/100}$. Next we consider the case where $b<\delta^{1-\eps/100}$, so each $T^W\in\tubes_1^W$ has dimensions comparable to $d_1 \times d_2 \times 1$ for some $1\geq d_2\geq d_1 \geq\delta^{\eps/100}$. This is true because \[
\lambda \geq 
|Y^W(T^W)| \geq \delta^{O(\eta)} |T^W|\geq \delta^{\eps/50 + O(\eta)}\geq \delta^{\eps/10}. 
\]


\medskip

\noindent {\bf Step 3.} For each $W\in\mathcal{W}_1$, define the shading 
\[
\tilde Y_1(W) = W \cap N_a\Big(\bigcup_{T\in\tubes_1[W]}Y_1(T)\Big).
\]
By \eqref{densityRho}, we have
\begin{equation}\label{twoScaleEstimate}
\Big|\bigcup_{T\in\tubes}Y(T)\Big| \gtrsim \lambda \Big|\bigcup_{W\in\mathcal{W}_1}\tilde Y_1(W)\Big|.
\end{equation}

Our next task is to show that $(\mathcal{W}_1,\tilde Y_1)_{a\times b\times 1}$ is $\delta^{O(\eta)}$ dense. The argument is shown in Figure \ref{thinTubesFatTubesBoxFig}. Fix $W\in\mathcal{W}$. Since $(\tubes_1^W,Y_1^W)$ is $\gtrapprox \delta^{O(\eta)}$ dense, we can select a refinement so that each $T\in\tubes_1[W]$ satisfies $|Y_1(T)|\gtrapprox\delta^{O(\eta)}|T|$. After pigeonholing, we can select a set $\tubes_a$ of essentially distinct $a$-tubes that form a $\approx_\delta 1$ balanced cover of $\tubes_1[W]$. Since $(\tubes_1^W,Y_1^W)$ satisfies the Frostman Convex Wolff axioms with error $\lessapprox_\delta \delta^{-O(\eta)}$, we have that $\tubes_a$  satisfies the Frostman Convex Wolff axioms with error $\lessapprox_\delta \delta^{-O(\eta)}$. The shading $Y_a(T_a)=T_a\cap N_a\big(\bigcup_{T\in\tubes[W]}Y_1(T)\big)$ is $\approx_\delta \delta^{O(\eta)}$ dense. Applying Lemma \ref{cordobaLem} to $(\tubes_a^W, Y_a^W)_{a/b\times 1\times 1}$ and then undoing the scaling $\phi_W$, we conclude that
\[
\Big| W \cap N_a\Big(\bigcup_{T\in\tubes_1[W]}Y_1(T)\Big)\Big| \geq \Big|\bigcup_{T_a\in\tubes_a}Y_a(T_a)\Big|\gtrapprox_\delta\delta^{O(\eta)}|W|.
\]
\begin{figure}[h!]
\begin{centering}
\includegraphics[scale=0.4]{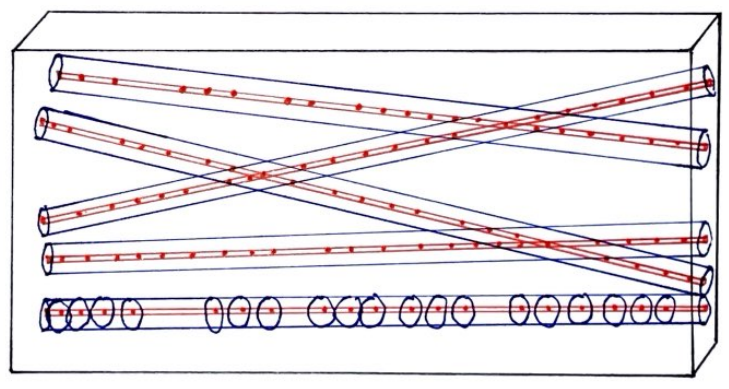}
\caption{Each $a$-tube in $\tubes_a$ (blue tubes) contains at least one $\delta$-tube from $\tubes_1[W]$ (red tube); this $\delta$-tube has a dense shading (red dots), which in turn gives us a dense shading on the $a$-tube that contains it (blue balls; for clarify we have only drawn these for one of the tubes in $\tubes_a$). Finally, since $W$ (black box) has thickness $a$, and the $a$-tubes inside $W$ satisfy a (rescaled) Frostman  Convex Wolff axioms, an $L^2$ argument says that the union of $a$-tubes has almost full volume.}
\label{thinTubesFatTubesBoxFig}
\end{centering}
\end{figure}

\medskip

\noindent {\bf Step 4.} 
It remains to estimate the RHS of \eqref{twoScaleEstimate}.
 We first consider the case where $a\leq\delta^{\eps/100}$. 
In this case, $(\mathcal{W}_1,\tilde Y_1)_{a\times b\times 1}$ is $a^{O(\eta/\eps)}$ dense. If $\eta$ is selected sufficiently small depending on $\eps$, then we can apply the estimate $\cF(\sigma,\omega)$ to conclude that
\begin{equation}\label{firstEstimatePrismsW}
\Big|\bigcup_{W\in\mathcal{W}_1}\tilde Y_1(W)\Big|  \geq \kappa_\eps a^{\eps/2} b^\omega \tilde m^{-1}(\#\mathcal{W}_1)|W|\Big(\tilde m^{-3/2}\tilde\ell(\#\mathcal{W}_1)|W|^{1/2}\Big)^{-\sigma},
\end{equation}
where $\tilde m = \CKT(\mathcal{W}_1)$ and $\tilde\ell = \FS(\mathcal{W}_1).$  Note that the estimate $\cF(\sigma,\omega)$ involves the additional term ``$D$'' defined in \eqref{defnOfD}. However, we claim that both both Parts (A) and (B) of Proposition \ref{aLLbProp}, we have 
\begin{equation}\label{DNotTooBigClaim}
D\lesssim\delta^{-\eta}.
\end{equation}

We verify this claim as follows. In Part (A) of Proposition \ref{aLLbProp}, we have the hypothesis that $\CKT(\mathcal{W})\leq\delta^{-\eta}$. As discussed in Remark \ref{discussionOfDRemark}, Situation 1, this ensures that $D\lesssim \delta^{-\eta}$. In Part (B) of Proposition \ref{aLLbProp}, we have the hypothesis that 
\[
\CKT(\mathcal{W}[N_b(W)])\leq\delta^{-\eta}\quad\textrm{for all}\ W\in\mathcal{W}.
\]
As discussed in Remark \ref{discussionOfDRemark}, Situation 2, this ensures that $D\lesssim \delta^{-\eta}$. This establishes \eqref{DNotTooBigClaim}.

In summary, we have established \eqref{firstEstimatePrismsW} when $a\leq\delta^{\eps/100}$. If instead $a>\delta^{\eps/100}$, then \eqref{firstEstimatePrismsW} follows from the estimate 
\begin{equation}\label{eq: large a}
\Big|\bigcup_{W\in\mathcal{W}_1}\tilde Y_1(W)\Big|\gtrsim \delta^{\eps/20} \gtrsim \delta^{\eps/20}  \tilde m^{-1}(\#\mathcal{W}_1)|W|.
\end{equation}

Combining \eqref{secondEstimateRescaledTubes}, \eqref{twoScaleEstimate}, and \eqref{firstEstimatePrismsW} (when $a \leq \delta^{\eps/100}$) or \eqref{eq: large a} (when $a\geq \delta^{\eps/100}$), we conclude that
\begin{equation}\label{volumeBdTrappedInPrisms}
\begin{split}
\Big|\bigcup_{T\in\tubes}Y(T)\Big| 
&\gtrapprox_\delta \delta^{\omega+\eps+O(\eta)- \eps^2/200}\Big(\frac{b}{a}\Big)^{\omega }\tilde m^{-1}(\#\mathcal{W}_1)|W|\Big(\tilde m^{-3/2}\tilde\ell(\#\mathcal{W}_1)|W|^{1/2}\Big)^{-\sigma}\Big(\Big(\frac{\#\tubes}{\#\mathcal{W}_1}\Big)^{1/2}\frac{|T|}{|W|}\Big)^{\sigma}\\
& = \delta^{\omega+\eps-\eps^2/300}\Big(\frac{b}{a}\Big)^{\omega}\tilde m^{-1}(\#\mathcal{W}_1)|W|\Big(\tilde m^{-3/2}\tilde\ell(\#\mathcal{W}_1)^{3/2}|W|^{3/2}(\#\tubes)^{-1/2}|T|^{-1}\Big)^{-\sigma}.
\end{split}
\end{equation}

\medskip

\noindent {\bf Step 5.} 
It remains to analyze the RHS of \eqref{volumeBdTrappedInPrisms}. Our analysis will differ for Parts (A) and (B) of Proposition \ref{aLLbProp}. We begin with Part (A). We have $\tilde m \lesssim \delta^{-\eta}$, and since $\mathcal{W}_1$ is a $\approx_\delta\delta^{-O(\eta)}$ balanced cover of $\tubes_1$, by Remark \ref{FrostmanWolffInheritedUpwardsDownwards}(A) (i.e.~Frostman Wolff constants are inherited upwards) we have $ \delta^{\eta} \tilde \ell \lessapprox_\delta \FS(\tubes)=\ell$. Since  $\mathcal{W}$ is a $\delta^{-\eta}$-balanced, $\delta^{-\eta}$-almost partitioning cover of $\tubes$, 
$\#\tubes[W] \geq  \delta^{\eta}\CKT(\tubes)\frac{|W|}{|T|}$  for each $W\in \mathcal{W}$,  and $\mathcal{W}_1$ is a refinement of $\mathcal{W}$, we have 
$(\#\mathcal{W}_1)|W|\gtrapprox_\delta \delta^{\eta}m^{-1}(\# \tubes)|T|$. The RHS of \eqref{volumeBdTrappedInPrisms} becomes
\[
\delta^{\omega+\eps}\Big(\frac{b}{a}\Big)^{\omega} m^{-1}(\#\tubes)|T|\Big( m^{-3/2} \ell(\#\tubes)|T|^{1/2}\Big)^{-\sigma},
\]
as claimed.

Next we consider part (B). We have $\tilde m \lessapprox_\delta \delta^{-\eta}|W|(\#\mathcal{W}) \lessapprox_{\delta} \delta^{-\eta}|W|(\#\mathcal{W}_1)$, and $\tilde\ell\lessapprox_\delta \delta^{-\eta}$. Thus the RHS of \eqref{volumeBdTrappedInPrisms} becomes
\[
\delta^{\omega+\eps}\Big(\frac{b}{a}\Big)^{\omega}\Big((\#\tubes)^{1/2}|T|\Big)^{\sigma}.
\]
This concludes the proof of Proposition \ref{aLLbProp}.


\subsection{Proof of Proposition \ref{PropALLbPropGen}: Factoring at two scales}
The proof of Proposition \ref{PropALLbPropGen} is almost identical to the proof of  Proposition \ref{aLLbProp}. Rather than present a more complicated unified proof of the two results, for clarity of exposition we have opted to instead briefly sketch the proof of Proposition \ref{PropALLbPropGen} and highlight where the two proofs differ. 

We begin by refining the shading $(\tubes,Y)_\delta$ to find a subset $(\tubes_1,Y_1)_\delta$ that has at least average density on balls of radius $a$; this is the analogue of \eqref{densityRho}. By Lemma \ref{frostmanSlabTransitiveUnderCovers}, we have that for each $W\in\mathcal{W}$, $\FS(\tubes^W)\lessapprox_\delta \lessapprox_{\delta} \FS(\tubes^{T_{\rho}})  \cdot \FS(\tubes_\rho^W) \lessapprox_{\delta} \delta^{-\eta}\ell'$. Thus the analogue of \eqref{firstEstimateRescaledTubes} is 
\begin{equation}
\Big| \bigcup_{T^W \in \tubes_1^W}Y_1^W(T^W)\Big| \gtrsim \big(\frac{\delta}{b}\big)^{\eps/2} \big(\frac{\delta}{a}\big)^\omega m^{-1}(\#\tubes^W)|T^W|
\Big(m^{-3/2}(\delta^{-\eta}\ell')(\#\tubes^W)|T^W|^{1/2}\Big)^{-\sigma},
\end{equation}
where $m = \CKT(\tubes)$, and this gives us the following analogue of \eqref{secondEstimateRescaledTubes}:
\begin{equation}\label{secondEstimateRescaledTubesAnalogue}
\lambda \gtrapprox_\delta \big( \frac{\delta}{b}\big)^{\eps/2} \delta^{O(\eta)} \big(\frac{\delta}{a}\big)^\omega m^{-1}(\#\tubes[W])|T|
\Big(m^{-3/2}\ell'(\#\tubes^W)|T^W|^{1/2}\Big)^{-\sigma}.
\end{equation}

The next step is to define a dense shading on $\mathcal{W}$; it is here that we use the fact that $\mathcal{W}$ factors $\tubes_\rho$ from below with respect to the Frostman Convex Wolff axioms with error $\leq \delta^{-\eta}$ --- this allows us to use the same argument as in Step 3 from the proof of Proposition \ref{aLLbProp} to show that the shading $\tilde{Y}_1(W)$ is $\delta^{O(\eta)}$ dense. 

Finally, we have the following analogue of \eqref{firstEstimatePrismsW}: 
\begin{equation}\label{firstEstimatePrismsWAnalogue}
\Big|\bigcup_{W\in\mathcal{W}_1}\tilde Y_1(W)\Big|  \gtrsim a^{\eps/2} b^\omega \tilde m^{-1}(\#\mathcal{W}_1)|W|\Big(\tilde m^{-3/2}\tilde\ell(\#\mathcal{W})|W|^{1/2}\Big)^{-\sigma},
\end{equation}
where $\tilde m = \CKT(\mathcal{W}_1)\leq\delta^{-\eta}$ (by hypothesis) and $\tilde\ell = \FS(\mathcal{W}_1)\lessapprox\delta^{-\eta}\ell$ (by Remark \ref{FrostmanWolffInheritedUpwardsDownwards}(A) ).

Combining \eqref{secondEstimateRescaledTubesAnalogue} and \eqref{firstEstimatePrismsWAnalogue} (using the same argument that was used to obtain \eqref{volumeBdTrappedInPrisms}), we conclude that 
\[
\Big|\bigcup_{T\in\tubes}Y(T)\Big| \gtrsim  \delta^{\eps+\omega } \big(\frac{b}{a}\big)^\omega m^{-1}(\#\tubes)|T| \Big(m^{-3/2}\ell \ell'(\#\tubes)|T|^{1/2}\Big)^{-\sigma}.
\]


\subsection{Tubes organized into  to slabs}
We conclude this section by using the tools developed thus far to prove the following. Let $(\tubes,Y)_\delta$ be a set of tubes and their associated shading. Suppose that for each $T\in\tubes$, there is a $\delta \times b\times 1$ slab $S\supset T$ that has large intersection with $\bigcup_{\tubes}Y(T)$. Then provided $\CKT(\tubes)$ is small, $\bigcup_{\tubes}Y(T)$ has larger volume than one would expect from the estimate $\cE(\sigma,\omega)$; the estimate becomes better as $\CKT(\tubes)$ becomes smaller and $b$ becomes larger. The precise statement is as follows. 

\begin{lem}\label{inflateTubesToSlabsLem}
Let $\omega>0,\sigma\in[0,2/3]$ and suppose $\cE(\sigma,\omega)$ is true. For all $\eps>0$, there exist $\kappa,\eta>0$ so that the following holds for all $\delta>0$. Let $(\tubes,Y)_\delta$ be a set of tubes and their associated shading. Let $b\geq\delta$ and suppose that for each $T\in\tubes$ there exists a $\delta\times b\times 1$ prism $S\supset T$ with $\big|S\cap\bigcup_{T\in\tubes}Y(T)\big| \geq \delta^\eta|S|$. Then
\begin{equation}\label{betterBoundPlanks}
\Big|\bigcup_{T\in\tubes}Y(T)\Big| \geq \kappa \delta^\eps b^{\omega}m^{-1}(\#\tubes)|T|\big(m^{-1}\ell(\#\tubes)|T|^{1/2}\big)^{-\sigma} \Big(\frac{|S|}{|T|}\Big)^{\sigma/2},
\end{equation}
where $m = \CKT(\tubes)$ and $\ell = \FS(\tubes)$.
\end{lem}
\begin{rem}
Note that the exponent of $m$ in \eqref{betterBoundPlanks} is only $m^{\sigma-1}$, rather than the usual estimate $m^{(3/2)\sigma-1}$ from $\cE(\sigma,\omega)$. While this is likely not optimal, in practice we will only apply Lemma \ref{inflateTubesToSlabsLem} with $m$ of size about 1, so the distinction will not be important.
\end{rem}

\begin{proof}
After pigeonholing, we can select a $\approx_\delta 1$ refinement $(\tubes_1,Y_1)_\delta$ of $(\tubes,Y)_\delta$ and a set $\mathcal{S}$ of essentially distinct $\delta\times b\times 1$ slabs with the following properties:
\begin{itemize}
	\item For each $T\in\tubes_1$ with corresponding slab $S(T)$, there is a slab $S\in\mathcal{S}$ comparable to $S(T)$. We denote this by $T\sim S$
	\item There is an integer $N$ so that each slab in $\mathcal{S}$, there are between $N$ and $2N$ tubes $T\in\tubes_1$ with $T\sim S$.
\end{itemize}
Abusing notation, we will replace each slab in $\mathcal{S}$ with its 10-fold dilate. Then
\begin{itemize}
\item[(i)] $\mathcal{S}$ covers $\tubes_1$,
\item[(ii)] $\#\tubes[S]\geq N$ for each $S\in\mathcal{S}$.
\item[(iii)] $\#\mathcal{S}\sim N^{-1}(\#\tubes_1)$.
\item[(iv)]The shading $Y(S)=S\cap\bigcup_{T\in\tubes}Y(T)$ is $\gtrsim \delta^{\eta}$ dense. 
\end{itemize}
From Items (ii) and (iii) we conclude that 
\[
\FS(\mathcal{S})\lesssim \FS(\tubes_1)\lessapprox_\delta\delta^{-\eta} \ell,
\]
and
\begin{equation}\label{boundCKTS}
\CKT(\mathcal{S}) \lesssim \CKT(\tubes_1)\frac{\#\mathcal{S}}{\#\tubes_1}\frac{|S|}{|T|}\lessapprox_\delta\delta^{-\eta}m\frac{\#\mathcal{S}}{\#\tubes}\frac{|S|}{|T|}.
\end{equation}
We would like to use Proposition \ref{EiffF} and apply the estimate $\cF(\sigma,\omega)$ to obtain a lower bound for the volume of $\bigcup Y(S)$. Before doing so, we should estimate the quantity $D$ from \eqref{defnOfD}. By  Remark~\ref{discussionOfDRemark}, Situation 2, and \eqref{boundCKTS}, we have

\[
D\lesssim \Big( \sup_{S\in \mathcal{S}} \CKT(\mathcal{S}[N_b(S)])\Big)^{1/2}
\lessapprox_{\delta} \Big( \delta^{-\eta} m \frac{\#\mathcal{S}}{\#\tubes} \frac{|S|}{|T|} \Big)^{1/2}.
\]
If $\eta,\kappa>0$ are chosen sufficiently small depending on $\omega,\sigma,$ and $\eps$, then we can apply the estimate $\cF(\sigma,\omega)$ to conclude that
\begin{equation*}
\begin{split}
\Big|\bigcup_{T\in\tubes}Y(T)\Big| &\geq \Big|\bigcup_{S\in\mathcal{S}}Y(S)\Big|\\
& \geq \kappa \delta^{\eps/2} b^{\omega}\CKT(\mathcal{S})^{-1}(\#\mathcal{S})|S|\Big(\CKT(\mathcal{S})^{-3/2}\FS(\mathcal{S})(\#\mathcal{S})|S|^{1/2}\Big)^{-\sigma} D^{-\sigma}\\
& \gtrapprox_\delta \kappa \delta^{\eps/2+2\eta} b^{\omega} m^{-1}(\#\tubes)|T|\Big(m^{-1}\ell(\#\tubes)|T|^{1/2}\Big)^{-\sigma}\Big(m^{1/2} \frac{|S|}{|T|}\frac{(\#\mathcal{S})^{1/2}}{(\#\tubes)^{1/2}} D^{-1}\Big)^{\sigma}\\
& \gtrapprox_\delta \kappa \delta^{\eps/2+2\eta} b^{\omega} m^{-1}(\#\tubes)|T|\Big(m^{-1}\ell(\#\tubes)|T|^{1/2}\Big)^{-\sigma} \Big(\frac{|S|}{|T|}\Big)^{\sigma/2}.\qedhere
\end{split}
\end{equation*}

\end{proof}

Often, we will use the following weaker version of Lemma \ref{inflateTubesToSlabsLem}.
\begin{cor}\label{corOfInflateTubesToSlabsLem}
Let $\omega>0,\sigma\in (0,2/3]$ and suppose $\cE(\sigma,\omega)$ is true. For all $\eps>0$, there exist $\eta,c>0$ so that the following holds for all $\delta>0$. Let $(\tubes,Y)_\delta$ be a set of tubes and their associated shading. Let $\rho\geq\delta$ and suppose that for each $T\in\tubes$ there exists a $\rho$ tube $T_\rho\supset T$ with $\big|T_\rho \cap\bigcup_{T\in\tubes}Y(T)\big| \geq \delta^\eta|T_\rho|$. Then
\begin{equation}\label{betterBoundPlankscor}
\Big|\bigcup_{T\in\tubes}Y(T)\Big| \geq \kappa \delta^\eps \rho^{\omega}m^{-1}(\#\tubes)|T|\big(m^{-1}\ell(\#\tubes)|T|^{1/2}\big)^{-\sigma}   \Big( \frac{\rho}{\delta} \Big)^{\sigma/2},
\end{equation}
where $m = \CKT(\tubes)$ and $\ell = \FS(\tubes)$.
\end{cor}


\section{Assertions $\cD$ and $\cE$ are equivalent}\label{cEIffcDSec}
Our goal in this section is to prove Proposition \ref{equivDE}. To do so, we will need a result from \cite{WZ23}, which (informally) says that if a set of $\delta$-tubes satisfies the Frostman Convex Wolff Axioms at many different scales, then the union of these tubes must have large volume. To state the result precisely, we recall Definition 2.12 from \cite{WZ23}.

\begin{defn}\label{convexAtEveryScaleFromAssouadPaper}
Let $K\geq 1,\delta>0$. We say a set $\tubes$ of $\delta$-tubes in $\RR^3$ satisfies the \emph{Frostman Convex Wolff Axioms at every scale with error $K$} if the tubes in $\tubes$ are essentially distinct, and for every $\rho_0\in [\delta,1]$, there exists $\rho\in [\rho_0, K\rho_0)$ and a set of $\rho$-tubes $\tubes_\rho$ that satisfies the following properties.
\begin{itemize}
	\item[(i)] $\tubes_\rho$ is a $K$-balanced partitioning cover of $\tubes$.
	\item[(ii)] For each $T_\rho\in\tubes_\rho$, $\tubes^{T_{\rho}}$ satisfies the Frostman Convex Wolff Axioms with error $K$.
\end{itemize}
\end{defn}

Next we recall Theorem 5.2 from \cite{WZ23}.
\begin{thm}\label{WZThm52}
For all $\eps>0$, there exists $\eta,\kappa>0$ so that the following holds for all $\delta>0$. Let $\tubes$ be a set of $\delta$-tubes that satisfy the Frostman Convex Wolff Axioms at every scale with error $\delta^{-\eta}$, and let $Y(T)$ be a $\delta^{\eta}$ dense shading. Then
\begin{equation}\label{conclusionStickySelfSimThm}
\Big|\bigcup_{T\in\tubes}Y(T)\Big|\geq \kappa\delta^\eps.
\end{equation}
\end{thm}

When $\eps<\omega$, the conclusion of Theorem \ref{WZThm52} gives a volume estimate that is superior to the volume estimate \eqref{defnCEEstimate} from Assertion $\cE(\sigma,\omega)$. However, the hypotheses of Assertion $\cE(\sigma,\omega)$ are weaker. The next result says that for every collection $\tubes$ of tubes that satisfies the Frostman Convex Wolff axioms, at least one of the following must occur:
\begin{itemize}
	\item[(A)] $\tubes$ satisfies the hypothesis of Theorem \ref{WZThm52}. This is good, provided we select $\eps<\omega$.
	\item[(B)] At a suitable scale, $\tubes$ satisfies the Katz-Tao Convex Wolff axioms with error roughly 1. This is good, if we know that $\cD(\sigma,\omega_1)$ is true for some $\omega_1<\omega.$
	\item[(C)] At a suitable scale, $\tubes$ is factored by flat rectangular prisms, and thus satisfies the hypotheses of Proposition \ref{aLLbProp}. This is good, since it gives a stronger volume estimate than $\cE(\sigma,\omega)$.
\end{itemize}
The precise statement is as follows. 


\begin{prop}\label{tubeTricotProp}
Let $\zeta_1\geq\zeta_2\geq\zeta_3>0$. Then there exists $\eta>0$ such that the following holds for all $\delta>0$. Let $\tubes$ be a set of $\delta$-tubes satisfying the Frostman Convex Wolff Axioms with error $\leq\delta^{-\eta}$. Then after replacing $\tubes$ by a $\approx_\delta 1$ refinement, at least one of the following is true. 

\begin{enumerate}
	\item[(A)] $\tubes$ satisfies the Frostman Convex Wolff Axioms at every scale with error $\delta^{-\zeta_1},$ in the sense of Definition \ref{convexAtEveryScaleFromAssouadPaper}.
	
	\item[(B)] There exists $\delta\leq\tau<\rho\leq 1$, with $\tau\leq\delta^{\zeta_1/5}\rho;$ a balanced partitioning cover $\tubes_{\tau}$ of $\tubes$; and a balanced partitioning cover $\tubes_{\rho}$ of $\tubes_{\tau}$ such that the following is true:
	\begin{enumerate}
		\item[(i)] $\CFC(\tubes^{T_\tau})\lesssim \delta^{-\zeta_2}$ for each $T_\tau\in\tubes_\tau$.
		\item[(ii)] $\CKT(\tubes_\tau^{T_\rho})\lesssim \delta^{-\zeta_2}$ and $\#\tubes_\tau^{T_\rho}\geq\delta^{\zeta_2}(\rho/\tau)^2$ for each $T_\rho\in\tubes_\rho$.
		\item[(iii)] $\CFC(\tubes_\rho) \lessapprox_{\delta} \delta^{-\eta}. $
	\end{enumerate}

	\item[(C)] There exists $\delta\leq a<b\leq 1$ with $a\leq\delta^{\zeta_2/100}b$, and a set $\mathcal{W}$ of $a\times b\times 1$ prisms that satisfies the hypotheses of Proposition \ref{aLLbProp}(B): $\mathcal{W}$ factors $\tubes$ above and below with respect to the Frostman Convex Wolff Axioms, with error $O(\delta^{-\zeta_3})$. And for each $W\in\mathcal{W}$, we have $\CKT(\mathcal{W}[N_b(W)])\lesssim\delta^{-\zeta_3}$. 
\end{enumerate}
\end{prop}
\noindent We will defer the proof of Proposition \ref{tubeTricotProp} to Section \ref{tubeTricotProp}.

Using Proposition \ref{tubeTricotProp}, we will prove the following weaker form of Proposition \ref{equivDE}; recall that $\TcE$ is defined in Definition \ref{TCEDefn}.

\begin{lem}\label{weakerPropEquivDE}
Let $0<\sigma\leq 2/3$. For all $\omega,t>0$, there exists $\alpha>0$ so that the following holds for all $\omega'\geq \omega+t$.  Suppose $\cD(\sigma,\omega)$ and $\cE(\sigma,\omega')$ are true. Then $\TcE(\sigma, \omega'-\alpha)$ is true.
\end{lem}

\begin{proof}
Let $\alpha = \alpha(\sigma, \omega,t)>0$ be a small number to be specified below. Let $\eta,\kappa>0$ be small numbers that depend on $\omega,\omega',$ and $\sigma$. Our goal is to prove that if $(\tubes,Y)_{\delta}$ is $\delta^{\eta}$ dense with $\FS(\tubes)\leq\delta^{-\eta}$, then 
\begin{equation}\label{desiredOmegaOmegaPrimeVolumeBd}
\Big|\bigcup_{T\in\tubes}\Big|\geq \kappa\delta^{\omega'-\alpha}m^{-1}(\#\tubes)|T|\Big( m^{-3/2} (\#\tubes)|T|^{1/2} \Big)^{-\sigma},
\end{equation}
with $m = \CKT(\tubes)$. Note that the estimate \eqref{desiredOmegaOmegaPrimeVolumeBd} is slightly stronger than the desired estimate $\TcE(\sigma, \omega'-\alpha)$, since there is no additional $\delta^{\eps}$ loss; this stronger estimate is possible since $\omega'-\omega>0$, which gives us a bit of ``wiggle room.''

\medskip

\noindent{\bf Step 1.}
Without loss of generality we may suppose that $|Y(T)|\geq \delta^{\eta}|T|$ for each $T\in\tubes$. Let $\zeta_i=\zeta_i(\sigma, \omega,t)$, $i=1,\ldots,5$ be small quantities to be chosen below; we will have $\zeta_{i+1}$ very small compared to $\zeta_i$, and $\eta$ very small compared to $\zeta_5$. 

We first consider the case where there exists a subset $\tubes'\subset\tubes$ with $\#\tubes'\geq \delta^{\eta}(\#\tubes)$, and $\CFC(\tubes')\leq\delta^{-\zeta_4}$ (this assumption will remain until Step 4, where we will consider the case where no such subset exists). Abusing notation slightly, we will continue to use $\tubes$ to refer to this subset. In particular, we have that 
\begin{equation}\label{lotsOfTubes}
\#\tubes \gtrapprox_\delta \delta^{\eta+\zeta_4}m|T|^{-1},\quad\textrm{and}\quad \CFC(\tubes)\leq\delta^{-\zeta_4}.
\end{equation}

If $\zeta_4$ and $\eta$ are chosen sufficiently small compared to $\zeta_1,\zeta_2,\zeta_3$, then we can apply Proposition \ref{tubeTricotProp} to $(\tubes,Y)_{\delta}$, with $\zeta_1,\zeta_2,\zeta_3$ as specified above. We will select $\zeta_1=\zeta_1(\omega)$ and $\eta$ sufficiently small so that if Conclusion (A) of Proposition \ref{tubeTricotProp} holds, then we can use Theorem \ref{WZThm52} to conclude that
\[
\Big|\bigcup_{T\in\tubes}Y(T)\Big| \gtrsim \delta^{\omega},
\]
and hence \eqref{desiredOmegaOmegaPrimeVolumeBd} holds provided we choose $\alpha(\omega,t)\leq t$. Henceforth we shall assume that Conclusion (A) of Proposition \ref{tubeTricotProp} does \emph{not} hold.

\medskip

\noindent{\bf Step 2.} Suppose that Conclusion (B) of Proposition \ref{tubeTricotProp} holds. We will define shadings $Y_\tau$ and $Y_\rho$ on the sets of tubes $\tubes_\tau$ and $\tubes_\rho$ as follows. For each $T_\tau\in\tubes_\tau$, we refine the shading on $\tubes[T_{\tau}]$ to have average density on balls of radius $\tau$ (see \eqref{densityRho} and the surrounding discussion). We define the shading $Y_\tau(T_\tau)$ to be the union of those $\tau$-balls that intersect $\bigcup_{\tubes[T_\tau]}Y(T)$. We perform the analogous procedure to define $Y_{\rho}$ (this induces a refinement on the shadings $Y_\tau$ and $Y$).

After these steps have been performed, $(\tubes_\rho, Y_\rho)_\rho$ is $\delta^{O(\eta)}$ dense; each pair $(\tubes_\tau^{T_\rho},Y_\tau^{T_\rho})_{\tau/\rho}$ is $\delta^{O(\eta)}$ dense; and each pair $(\tubes^{T_\tau}, Y^{T_{\tau}})_{\delta/\tau}$ is $\delta^{O(\eta)}$ dense. Furthermore, we have
\begin{equation}\label{boundYTByProduct}
\Big|\bigcup_{T\in\tubes}Y(T)\Big|\geq 
\bigg(\Big|\bigcup_{T_\rho\in\tubes_\rho}Y_\rho(T_\rho)\Big|\bigg)
\bigg( \inf_{T_\rho\in\tubes_\rho}\Big|\bigcup_{T_\tau^{T_\rho} \in\tubes_\tau^{T_\rho}}Y_{\tau}^{T_\rho}(T_\tau^{T_\rho})\Big|\bigg)
\bigg( \inf_{T_\tau\in\tubes_\tau}\Big|\bigcup_{T^{T_\tau} \in\tubes^{T_\tau}}Y^{T_\tau}(T^{T_\tau})\Big|\bigg).
\end{equation}
Our next task is to estimate the three terms on the RHS of \eqref{boundYTByProduct} as follows.
\begin{align}
\label{estimateScaleRho}
\Big|\bigcup_{T_\rho\in\tubes_\rho}Y_\rho(T_\rho)\Big| & \geq \delta^{\eps_1}\rho^{\omega'}\big((\#\tubes_\rho)^{1/2}|T_\rho|\big)^\sigma,\\
\label{estimateScaleRhoTau}
\inf_{T_\rho\in\tubes_\rho}\Big|\bigcup_{T_\tau^{T_\rho} \in\tubes^{T_\rho}}Y_{\tau}^{T_\rho}(T_\tau^{T_\rho})\Big| 
& \geq \delta^{\eps_1} \Big(\frac{\tau}{\rho}\Big)^{\omega}\Big(\Big(\frac{\#\tubes_\tau}{\#\tubes_\rho}\Big)^{1/2} \frac{|T_\tau|}{|T_\rho|}\Big)^\sigma,\\
\label{estimateScaleTauDelta}
\inf_{T_\tau\in\tubes_\tau}\Big|\bigcup_{T^{T_\tau} \in\tubes^{T_\tau}}Y^{T_\tau}(T^{T_\tau})\Big| 
& \geq \delta^{\eps_1} \Big(\frac{\delta}{\tau}\Big)^{\omega'}\Big(\Big(\frac{\#\tubes}{\#\tubes_\tau}\Big)^{1/2} \frac{|T|}{|T_\tau|}\Big)^\sigma,
\end{align}
where $\eps_1 = \frac{\zeta_1}{24}(\omega-\omega')$.

First, observe that if we choose $\eta$ sufficiently small, then \eqref{estimateScaleRho} (resp.~\eqref{estimateScaleRhoTau} or \eqref{estimateScaleTauDelta}) immediately holds if $\rho >\delta^{\eps_1/2}$ (resp.~$\delta/\tau > \delta^{\eps_1}/2$). This is because the volume of the union \eqref{estimateScaleRho} is bounded by the volume of a single tube. In particular, we have the following
\begin{itemize}
\item Either \eqref{estimateScaleRho} automatically holds, or $(\tubes_\rho,Y_\rho)_\rho$ is $\rho^{O(\eta/\eps_1)}$ dense and satisfies $\CFC(\tubes_\rho)\lesssim \rho^{-2\zeta_2/\eps_1}$.
\item Since $\eps_1\leq \zeta_1/24$ and $\tau \leq \delta^{\zeta_1/5}\rho$, each pair $(\tubes_\tau^{T_\rho},Y_\rho^{T_\rho})_{\tau/\rho}$ is $(\tau/\rho)^{O(\eta/\eps_1)}$ dense and satisfies $\CKT(\tubes_\tau^{T_\rho})\lesssim (\tau/\rho)^{-2\zeta_2/\eps_1}$ and $\#\tubes_\tau^{T_\rho} \geq (\tau/\rho)^{2\zeta_2/\eps_1}(\rho/\tau)^2$. 
\item Either \eqref{estimateScaleTauDelta} automatically holds, or each pair $(\tubes^{T_\tau}, Y^{T_{\tau}})_{\delta/\tau}$ is $(\delta/\tau)^{O(\eta/\eps_1)}$ dense and satisfies $\CKT(\tubes^{T_\tau})\lesssim (\delta/\tau)^{-2\zeta_2/\eps_1}$.
\end{itemize}

If we select $\eta$ and $\zeta_2$ sufficiently small, depending on $\omega$ and $t$ (recall that $t\leq \omega-\omega'$) ($\eta$ and $\zeta_2$ also depend on $\eps_1$, but $\eps_1$ only depends on $\zeta_1=\zeta_1(\omega)$, $\omega$ and $t$), then in light of the three bullet points above, the estimates \eqref{estimateScaleRho}, \eqref{estimateScaleRhoTau}, and \eqref{estimateScaleTauDelta} follow from $\cE(\sigma,\omega')$, $\cD(\sigma,\omega)$, and $\cE(\sigma,\omega')$, respectively. 

Combining \eqref{boundYTByProduct}, \eqref{estimateScaleRho}, \eqref{estimateScaleRhoTau}, and \eqref{estimateScaleTauDelta}, we conclude that
\begin{equation}\label{improvedFrostmanBd}
\Big|\bigcup_{T\in\tubes}Y(T)\Big|  \geq \delta^{3\eps_1}\Big(\frac{\rho}{\tau}\Big)^{\omega-\omega'}\delta^{\omega'}\Big((\#\tubes)^{1/2}|T|\Big)^\sigma
\geq \delta^{\frac{\zeta_1(\omega-\omega')}{8}}\delta^{\omega'}\Big((\#\tubes)^{1/2}|T|\Big)^\sigma.
\end{equation}
Combining \eqref{lotsOfTubes} and \eqref{improvedFrostmanBd}, we verify that \eqref{desiredOmegaOmegaPrimeVolumeBd} holds, provided $\alpha(\omega,t)\leq \frac{\zeta_1 t}{16}$; $\zeta_4<\frac{\zeta_1 t}{17}$; and $\eta>0$ is chosen sufficiently small. Henceforth we shall assume that Conclusion (B) of Proposition \ref{tubeTricotProp} does \emph{not} hold.

\medskip

\noindent{\bf Step 3.} Suppose that Conclusion (C) of Proposition \ref{tubeTricotProp} holds, i.e. there is a set $\mathcal{W}$ of $a\times b\times 1$ prisms, with $a\leq\delta^{\zeta_2/100}b$, so that $\mathcal{W}$ satisfies the hypotheses of Proposition \ref{aLLbProp}(B): $\mathcal{W}$ factors $\tubes$ above and below with respect to the Frostman Convex Wolff Axioms with error $O(\delta^{-\zeta_3})$. And for each $W\in\mathcal{W}$ we have $\CKT(\mathcal{W}[N_b(W)])\lesssim\delta^{-\zeta_3}$. 

Let $\eps_2 = \zeta_2\omega/200$. If $\zeta_3$ is selected sufficiently small depending on $\zeta_2$ and $\eps_2$ (both of these numbers in turn ultimately only depend on $\omega$ and $t$) and if $\eta>0$ is selected sufficiently small, then by Proposition \ref{aLLbProp}(B) (with $\eps_2$ in place of $\eps$) we have
\begin{equation}\label{case3Estimate}
\begin{split}
\Big|\bigcup_{T\in\tubes}Y(T)\Big| &\gtrsim \delta^{\omega'+\eps_2}\Big(\frac{b}{a}\Big)^{\omega'}\big((\#\tubes)^{1/2}|T|\big)^{\sigma}\\
& \geq \delta^{\omega'+\eps_2}\delta^{-\frac{\zeta_2\omega}{100}}\big((\#\tubes)^{1/2}|T|\big)^{\sigma}\\
& \geq \delta^{\omega'-\frac{\zeta_2\omega}{200}}\big((\#\tubes)^{1/2}|T|\big)^{\sigma}.
\end{split}
\end{equation}
Combining \eqref{lotsOfTubes} and \eqref{case3Estimate},  we verify that \eqref{desiredOmegaOmegaPrimeVolumeBd} holds, provided $\alpha(\omega,t)\leq \frac{\zeta_2\omega}{400}$; $\zeta_4<\frac{\zeta_2\omega}{500}$; and $\eta>0$ is chosen sufficiently small.

\medskip

\noindent{\bf Step 4.}
It remains to consider the case where every subset $\tubes'\subset\tubes$ with $\#\tubes'\geq \delta^{\eta}(\#\tubes)$ satisfies $\CFC(\tubes')>\delta^{-\zeta_4}$. Apply Proposition \ref{factoringConvexSetsProp} (factoring a collection of convex sets) to $\tubes$, and denote the output by $\tubes'$ and $\mathcal{W}$. Then Item ii) of Proposition~\ref{factoringConvexSetsProp} implies $\#\tubes[W]\approx_{\delta} \CKT(\tubes') \frac{|W|}{|T|}$. Since $\#\tubes'\geq\delta^{\eta}(\#\tubes)$, we have $\CFC(\tubes')>\delta^{-\zeta_4}$. We also have $\CFC(\tubes^{\prime W})\lessapprox_\delta 1$ for each $W\in\mathcal{W}$, from which it follows that $|W|\lessapprox_\delta \delta^{\zeta_4}$ (recall that the sets $W\in\mathcal{W}$ are congruent, and thus they all have identical volume).

If the prisms in $\mathcal{W}$ are flat, in the sense that they are comparable to $a\times b\times 1$ prisms with $a\leq\delta^{\zeta_5}b$, then we can apply Proposition \ref{aLLbProp}(A) with $\eps_3 = \zeta_5\omega'/2$ to conclude that
\begin{equation}\label{conclusionStep4Flat}
\begin{split}
\Big|\bigcup_{T\in\tubes}Y(T)\Big| & \gtrsim \delta^{\omega'+\eps_3}\Big(\frac{b}{a}\Big)^{\omega'} (m')^{-1}(\#\tubes')|T|\big((m')^{-3/2}(\ell') (\#\tubes)|T|^{1/2}\big)^{-\sigma}\\
&\gtrapprox \delta^{\omega'-\frac{\zeta_5\omega'}{2}+2\eta}m^{-1}(\#\tubes)|T|\big(m^{-3/2} (\#\tubes)|T|^{1/2}\big)^{-\sigma}.
\end{split}
\end{equation}
In the second inequality, we used the fact that $\#\tubes'\geq \delta^{\eta}\#\tubes$; $\ell':=\FS(\tubes')\lessapprox \delta^{-\eta}$; $m':=\CKT(\tubes')\leq m$; and $\sigma\leq 2/3$. We conclude that \eqref{desiredOmegaOmegaPrimeVolumeBd} holds, provided $\alpha(\omega,t)\leq \frac{\zeta_5\omega}{4} \leq\frac{\zeta_5\omega'}{4}  $ and $\eta>0$ is chosen sufficiently small.

Finally, we consider the case where the prisms in $\mathcal{W}$ are not flat, in the sense that $a\geq\delta^{\zeta_5}b$. In this case we can replace each prism $W\in\mathcal{W}$ by its $b$-neighbourhood, and then refine the corresponding set of $b$-tubes $\tubes_b$ by a factor of $(b/a)^3 \leq \delta^{-3\zeta_5}$ (this is the number of essentially distinct $a\times b\times 1$ prisms that can fit inside a $b$ tube) so that the tubes in $\tubes_b$ are essentially distinct. Since $|W|\lessapprox_\delta \delta^{\zeta_4}$ and $a\geq\delta^{\zeta_5}b$, we have 
\begin{equation}\label{lowerBdOnB}
b\lessapprox_\delta \delta^{\zeta_4/2 - \zeta_5}\leq \delta^{\zeta_4/3}.
\end{equation} 
To recap, the set $\tubes_b$ has the following properties.
\begin{itemize}
	\item $\CKT(\tubes_b)\lesssim \frac{b}{a}\CKT(\mathcal{W})\lessapprox_\delta \delta^{-\zeta_5}\leq b^{-3\zeta_5/\zeta_4}.$
	\item $\FS(\tubes_b)\lesssim\delta^{-3\zeta_5}\FS(\mathcal{W})\lessapprox_\delta \delta^{-3\zeta_5}\FS(\tubes) \lessapprox_\delta \delta^{-3\zeta_5-\eta}$.
	\item For each $T_b\in\tubes_b$, we have $\CFC(\tubes^{T_b})\lesssim \frac{b}{a}\CFC(\tubes^{W})\lessapprox_\delta \delta^{-\zeta_5}$, where $\mathcal{W}\ni W\subset T_b$ is the prism containing all of the tubes in $\tubes[T_b]$.

	\item For each $T_b\in\tubes_b$, we have $\#\tubes^{T_b} \lessapprox_{\delta}  \CKT(\mathcal{W}) \frac{b}{a}  \#\tubes^W \lessapprox_{\delta}  \frac{b}{a} m \frac{|W|}{|T|} \lessapprox_{\delta} \delta^{-2\zeta_5} m \frac{|T_b|}{|T|}.$ 
\end{itemize}
(For the second item above, we make crucial use of the fact that $\FS(\tubes)\leq\delta^{-\eta}$, which is why the conclusion of Lemma \ref{weakerPropEquivDE} only says that $\TcE(\sigma,\omega'-\alpha)$ is true, rather than the superficially stronger statement $\cE(\sigma,\omega'-\alpha)$).

Mirroring the argument in Step 2, refine the shadings $Y(T)$ on each set of tubes $\tubes[T_b]$ to have average density on balls of radius $b$. We define the shading $Y_b(T_b)$ to be the union of those $b$-balls that intersect $\bigcup_{\tubes[T_b]}Y(T)$. Then $(\tubes_b, Y_b)_b$ is $\delta^{\eta}\geq b^{3\eta/\zeta_4}$ dense. We thus have the following analogue of \eqref{boundYTByProduct}:
\begin{equation}\label{boundYTByProductTwoTerms}
\Big|\bigcup_{T\in\tubes}Y(T)\Big|\geq 
\bigg(\Big|\bigcup_{T_b\in\tubes_b}Y_\rho(T_b)\Big|\bigg)
\bigg( \inf_{T_b\in\tubes_b}\Big|\bigcup_{T^{T_b} \in\tubes^{T_b}}Y^{T_b}(T^{T_b})\Big|\bigg).
\end{equation}

Let $\eps_4 = \frac{\zeta_4(\omega-\omega')}{12}$. If $\zeta_5$ and $\eta$ are selected sufficiently small depending on $\zeta_4$ and $\eps_4$ (which in turn depend on $\omega$ and $t$), then we can use the estimate \eqref{defnCDEqn} from Assertion $\cD(\sigma,\omega)$ to conclude that
\begin{equation}\label{volumeBdOnTb}
\Big|\bigcup_{T_b\in\tubes_b}Y_\rho(T_b)\Big| \gtrsim b^{\omega+\eps_4}(\#\tubes_b)|T_b|\big((\#\tubes_b)|T_b|^{1/2}\big)^{-\sigma}.
\end{equation}

Finally, we would like to obtain the estimate
\begin{equation}\label{volumeBdOnTInsideTb}
\Big|\bigcup_{T^{T_b} \in\tubes^{T_b}}Y^{T_b}(T^{T_b})\Big|\gtrsim \delta^{\eps_4}\big(\frac{\delta}{b}\big)^{\omega'}\Big((\#\tubes^{T_b})^{1/2}\frac{|T|}{|T_b|}\Big)^{\sigma}.
\end{equation}
When $\delta/b>\delta^{\eps_4/2}$, we have that \eqref{volumeBdOnTInsideTb} follows from the elementary fact that the shading on each $Y(T)$ is $\gtrapprox \delta^{\eta}$ dense, and the union on the LHS of \eqref{volumeBdOnTInsideTb} is bounded by the volume of a single tube. On the other hand, when $\delta/b\leq\delta^{\eps_4/2}$, we have that \eqref{volumeBdOnTInsideTb} follows from the estimate \eqref{defnCEEstimate} from Assertion $\cE(\sigma,\omega')$, provided $\zeta_5$ is chosen sufficiently small depending on $\zeta_4$ and $\eps_4$, since $\CFC(\tubes^{T_b})\lessapprox \delta^{-\zeta_5}$. 

Since $\#\tubes^{T_b}\lessapprox m \delta^{-2\zeta_5}\frac{|T|}{|T_b|}$ for each $T_b\in\tubes_b$ (where $m= \CKT(\tubes)$), \eqref{volumeBdOnTInsideTb} becomes
\begin{equation}\label{volumeBdOnTInsideTbUseful}
\Big|\bigcup_{T^{T_b} \in\tubes^{T_b}}Y^{T_b}(T^{T_b})\Big|
\gtrsim \delta^{\eps_4+2\zeta_5}\big(\frac{\delta}{b}\big)^{\omega'} m^{-1}(\#\tubes^{T_b})\frac{|T|}{|T_b|}\Big(m^{-3/2}(\#\tubes^{T_b})\Big(\frac{|T|}{|T_b|}\Big)^{1/2}\Big)^{-\sigma}.
\end{equation}
Combining \eqref{boundYTByProductTwoTerms}, \eqref{volumeBdOnTb}, \eqref{volumeBdOnTInsideTbUseful}, and \eqref{lowerBdOnB}, we conclude that if we select $\zeta_5\leq\zeta_4(\omega-\omega')/12$, then
\begin{equation}
\begin{split}
\Big|\bigcup_{T\in\tubes}Y(T)\Big|
& \geq  \delta^{\omega'+\eps_4+2\zeta_5}b^{\omega-\omega' + \eps_4} m^{-1}(\#\tubes)|T|\Big(m^{-3/2}(\#\tubes)|T|^{1/2}\Big)^{-\sigma}\\
& \geq  \delta^{\omega' - \frac{\zeta_4(\omega-\omega')}{12}} m^{-1}(\#\tubes)|T|\Big(m^{-3/2}(\#\tubes)|T|^{1/2}\Big)^{-\sigma}.
\end{split}
\end{equation}
We conclude that \eqref{desiredOmegaOmegaPrimeVolumeBd} holds, provided $\alpha(\omega,t)\leq \frac{\zeta_4t}{12},$ and $\eta>0$ is chosen sufficiently small.
\end{proof}


\noindent We  now use Lemma \ref{weakerPropEquivDE} to prove Proposition \ref{equivDE}.
\begin{proof}[Proof of Proposition \ref{equivDE}]
Let $0\leq\sigma\leq 2/3,\ \omega\geq 0,$ and suppose that Assertion $\cD(\sigma,\omega)$ is true. Fix $t>0$ and let $\alpha=\alpha(\sigma, \omega,t)>0$ be the output of
 Lemma \ref{weakerPropEquivDE}. Since a $\delta$-tube has volume $\sim\delta^2$, we always have that $\cE(\sigma, 2)$ is true. Now suppose that $\cE(\sigma, \omega')$ is true for some $\omega'\in [\omega+t,2]$. Applying Lemma \ref{weakerPropEquivDE} followed by Proposition \ref{EiffF}, we have
 \[
\cE(\sigma, \omega')\implies \TcE(\sigma, \omega'-\alpha)\implies \cE(\sigma, \omega'-\alpha).
 \]
 Iterating the above argument, we conclude that $\cE(\sigma, \omega'')$ is true for some $\omega''\leq\omega+t$, so in particular $\cE(\sigma, \omega+t)$ is true. 

 However, $t>0$ was arbitrary, and by the definition of  $\cE$, it is clear that the set 
 \[
 \{\omega'\in [\omega,2]\colon \cE(\sigma, \omega')\ \textrm{is true}\}
 \]
 is a closed interval. We conclude that $\cE(\sigma,\omega)$ is true. 
\end{proof}
This concludes the proof of Proposition \ref{equivDE}, except that we must still prove Proposition \ref{tubeTricotProp}. We do this below. 


\subsection{Proof of Proposition \ref{tubeTricotProp}: A factoring trichotomy}$\phantom{1}$\\
{\bf Step 1.}
We begin by regularizing the set $\tubes$. Let $\eta>0$ be a small quantity to be determined below, with $N=1/\eta$ an integer. Define $\delta_i = \delta^{i/N},\ i=1,\ldots, N$. By iterated pigeonholing and replacing $\tubes$ by a $|\log\delta|^{-N}$-refinement, we may suppose that for each $i=1,\ldots,N$ there exists a set $\tubes_{\delta_i}$ of $\delta_i$-tubes that is a balanced partitioning cover of $\tubes_{\delta_{i+1}}$. We will call numbers of the form $\delta_i$ ``admissible scales.'' In particular, for each admissible scale $\delta_i$, we have
\begin{equation}\label{cardTubesTDeltaI}
\#\tubes[T_{\delta_i}]\approx_\delta \frac{\#\tubes\phantom{_{\delta_i}}}{\#\tubes_{\delta_i}}\quad\textrm{for every}\ T_{\delta_i}\in\tubes_{\delta_i}.
\end{equation}

Next we apply  Proposition \ref{factoringConvexSetsProp} to each set $\tubes^{T_{\delta_i}}$; the output of Proposition \ref{factoringConvexSetsProp} is a $\approx 1$ refinement of $\tubes^{T_{\delta_i}}$ (this induces a $\approx 1$ refinement of $\tubes$ and all sets $\tubes_{\delta_j}$ for $j>i$; observe that \eqref{cardTubesTDeltaI} remains true after this refinement; to simplify notation, we still use $\tubes_{\delta_i}$ and $\tubes$ to denote these refinement) and a set $\mathcal{W}$ of convex subsets of $\RR^3$.  If $|W|\geq\delta^{\zeta_1/2}$, then $\CFC(\tubes^{T_{\delta_i}})\leq\delta^{-\zeta_1/2}$. In this case, we say $\tubes^{T_{\delta_i}}$ is of \emph{Type 1}. If $|W|<\delta^{\zeta_1/2}$, then we say we say $\tubes^{T_{\delta_i}}$ is of \emph{Type 2}.

We now proceed as follows. For each $i=N-1,\ldots,1$, if at least half the sets $\tubes^{T_{\delta_i}},\ T_{\delta_i}\in\tubes_{\delta_i}$ are of Type 1, then we refine $\tubes_{\delta_i}$ to only consist of those $T_{\delta_i}$ for which $\tubes^{T_{\delta_i}}$ are of Type 1; observe that \eqref{cardTubesTDeltaI} remains true after this refinement. We say that $\tubes$ has passed stage $i$. On the other hand, if at least half the sets $\tubes^{T_{\delta_i}},\ T_{\delta_i}\in\tubes_{\delta_i}$ are of Type 2, then we say that $\tubes$ has failed stage $i$.

Suppose that $\tubes$ passes every stage $i = N-1,\ldots, 1$. Then by \eqref{cardTubesTDeltaI} we have that for each $i=1,\ldots,N$, the set $\tubes_{\delta_i}$ (this consists of those $\delta_i$-tubes that survived the refinements described above) is a $\approx_\delta 1$-balanced partitioning cover of $\tubes$. Furthermore, since $\tubes$ passed stage $i$, we have that for each $T_{\delta_i}\in\tubes_{\delta_i}$ we have that $\CFC(\tubes^{T_{\delta_i}})\lessapprox_\delta \delta^{-\zeta_1/2}.$ We conclude that $\tubes$ satisfies the Frostman Convex Wolff Axioms at every scale with error $O(\delta^{-\zeta_1})$, and hence Conclusion (A) of Proposition \ref{tubeTricotProp} holds. 

\medskip

\noindent{\bf Step 2.}
Suppose that $\tubes$ fails some stage $i\geq 1$. After pigeonholing and replacing $\tubes_{\delta_i}$ and $\tubes$ by a $\approx_\delta 1$ refinement, we may suppose that there exists $\delta\leq a\leq b\leq 1$ so that for each $T_{\delta_i}\in\tubes_{\delta_i}$, the output of Proposition \ref{factoringConvexSetsProp} applied to $\tubes^{T_{\delta_i}}$ consists of a set $\mathcal{W}_{T_{\delta_i}}$ of $\frac{a}{\delta_i}\times \frac{b}{\delta_i}\times 1$ prisms that forms a $\approx_\delta 1$ balanced cover of $\tubes^{T_{\delta_i}}$, and factors $\tubes^{T_{\delta_i}}$ from above (resp.~below) with respect to the Katz-Tao Convex Wolff Axioms (resp.~Frostman Convex Wolff Axioms) with error $\lessapprox_\delta 1$.

Since the tubes in $\tubes_{\delta_i}$ are essentially distinct, we can further refine $\tubes_{\delta_i}$ by a $\sim 1$ factor so that every pair of distinct tubes  $T_{\delta_i},T_{\delta_i}'\in \tubes_{\delta_i}$ that intersect must satisfy $\angle\big(\dir(T_{\delta_i}),\ \dir(T_{\delta_i}')\big)\geq 100\delta_i$, so in particular we have 
\begin{equation}\label{smallDiameterIntersectingRhoiTubes}
\operatorname{diam}(2T_{\delta_i}\cap 2T_{\delta_i}')\leq \frac{1}{2}.
\end{equation}

Define 
\[
\mathcal{W}=\bigsqcup_{T_{\delta_i}\in\tubes_{\delta_i}}\phi_{T_{\delta_i}}^{-1}\big(\mathcal{W}_{T_{\delta_i}}\big).
\]
Then $\mathcal{W}$ is a collection of convex sets, each of which is comparable to a $a\times b\times 1$ prism. Recall that since $\tubes$ failed stage $i$, we have that the prisms in $\mathcal{W}$ are substantially smaller than the tubes in $\tubes_{\delta_i}$; specifically, we have 
\begin{equation}\label{WMuchSmallerThanTauDeltaI}
|W|\lesssim \delta^{\zeta_1/2}|T_{\delta_i}|.
\end{equation}

We claim that the convex sets in $\mathcal{W}$ are essentially distinct. To verify this claim, we argue as follows. Every pair of convex sets $W,W'$ from the same set $\mathcal{W}_{T_{\delta_i}}$ are essentially distinct. On the other hand, if $W\in \mathcal{W}_{T_{\delta_i}}$ and $W'\in \mathcal{W}_{T_{\delta_i}'}$ for distinct tubes $T_{\delta_i}$ and $T_{\delta_i}'$, then $\diam(W\cap W') \leq \diam(T_{\delta_i}\cap T_{\delta_i}')\leq \frac{1}{2}$ by \eqref{smallDiameterIntersectingRhoiTubes}, from which it follows that $W$ and $W'$ are essentially distinct. 

Since $\mathcal{W}$ is a $\approx_\delta 1$ balanced cover of $\tubes$, and $\CFC(\tubes)\lessapprox_\delta\delta^{-\eta}$, by Remark \ref{FrostmanWolffInheritedUpwardsDownwards}(A) (Frostman Wolff Axioms are inherited upwards), we have $\CFC(\mathcal{W})\lessapprox_\delta\delta^{-\eta}$; we will select $\eta>0$ sufficiently small so that $\CFC(\mathcal{W})\lesssim\delta^{-\zeta_3}$.

\medskip

\noindent{\bf Step 3.} Suppose that the prisms in $\mathcal{W}$ are flat, in the sense that $a\leq\delta^{\zeta_2/100}b$. Our task is to show that $\mathcal{W}$ and our refined set $\tubes$ satisfies the conditions of Conclusion (C) from Proposition \ref{tubeTricotProp}.

Each $W\in\mathcal{W}$ came from some set $\mathcal{W}_{T_{\delta_i}}$. We claim that if $W'\in\mathcal{W}$ satisfies $W'\subset N_{b}(W)$, then we must have that $W'$ came from the same set $\mathcal{W}_{T_{\delta_i}}$, i.e.
\begin{equation}\label{allWPrimeComeFromSameTube}
\mathcal{W}[N_{b}(W)] = (\mathcal{W}[T_{\delta_i}] )[N_b(W)]=\phi_{T_{\delta_i}}^{-1}(\mathcal{W}_{T_{\delta_i}})[N_{b}(W)].
\end{equation}
To verify this claim, we argue by contradiction: suppose instead that $W'$ came from a distinct tube $T_{\delta_i}'$, then we would have  $W'\subset N_{b}(W) \cap T_{\delta_i}' \subset 2T_{\delta_i}\cap T_{\delta_i}',$ and in particular the latter set would have diameter $\geq 1$. But this is forbidden by \eqref{smallDiameterIntersectingRhoiTubes}. 

\eqref{allWPrimeComeFromSameTube} implies that 
\[
\CKT\big(\mathcal{W}[N_{b}(W)]\big) = \CKT\big( (\mathcal{W}[T_{\delta_i}])[N_{b}(W)] \big) \leq \CKT(\mathcal{W}_{T_{\rho_i}})\lessapprox_\delta 1.
\]
Thus $\mathcal{W}$ and our refined set $\tubes$ satisfies the conditions of Conclusion (C) from Proposition \ref{tubeTricotProp}.

\medskip

\noindent{\bf Step 4.}
Now we consider the case where the prisms in $\mathcal{W}$ are not flat, in the sense that $a> \delta^{\zeta_2/100}b$. Define $\tau$ to be the smallest admissible scale greater than or equal to $b$. Since at most $O( (b/a)^3)$ essentially distinct $a\times b\times 1$ prisms can fit inside a $b$-tube, and at most $O(\delta^{-4/N})=O(\delta^{-4\eta})$ essentially distinct $b$-tubes can fit inside a $\tau$-tube, we have that $\mathcal{W}[T_\tau]\lesssim \delta^{-4\eta}(b/a)^3\lesssim \delta^{-4\zeta_2/100}$ for every $\tau$-tube $T_{\tau}$ (for the last inequality we select $\eta\leq \zeta_2/400$). After pigeonholing $\tubes,\mathcal{W}$, and $\tubes_\tau$, we may suppose that $\tubes_\tau$ is a balanced partitioning cover of $\mathcal{W}$. We have 
\begin{equation}\label{almostEqualCardBd}
1\leq \#\mathcal{W}[T_{\tau}]\lesssim\delta^{-4\zeta_2/100}\quad\textrm{for each}\ T_{\tau}\in\tubes_\tau.
\end{equation}
Since $\CFC(\tubes[W])\lessapprox_\delta 1$ for each $W\in\mathcal{W}$, we conclude from \eqref{almostEqualCardBd} that $\CFC(\tubes[T_\tau])\lessapprox_\delta \delta^{-4\zeta_2/100}$ for each $T_\tau\in\tubes_\tau$; this gives us Conclusion (B.i).  

At this point, we have correctly identified the scale $\tau$ from Conclusion (B) of Proposition \ref{tubeTricotProp}. What about the scale $\rho$? One candidate is $\delta_i$; by \eqref{WMuchSmallerThanTauDeltaI} we have $\tau\lesssim \delta^{\zeta_1/4-2\zeta_2/100}\delta_i\leq \delta^{\zeta_1/5}\delta_i$, as specified in Conclusion (B). 

The scale $\delta_i$ satisfies some of the required properties of Conclusion (B). Recall that for each $T_{\delta_i}\in\tubes_{\delta_i}$, we have $\CKT(\mathcal{W}[T_{\delta_i}])\lessapprox_\delta 1$. Since $|T_\tau|=(b/a)|W|\leq \delta^{-\zeta_2/100}|W|$, by \eqref{almostEqualCardBd} we have 
\begin{equation}\label{goodKatzTaoBoundTauInsideRhoI}
\CKT\big(\tubes_\tau[T_{\delta_i}]\big)\leq \delta^{-5\zeta_2/100}\quad\textrm{for each}\ T_{\delta_i}\in\tubes_{\delta_i}.
\end{equation}
This is half of Conclusion (B.ii). If $\frac{\#\tubes_\tau}{\#\tubes_{\delta_i}}\geq \delta^{100\eta}(\delta_i/\tau)^2$ (in fact a weaker estimate $\frac{\#\tubes_\tau}{\#\tubes_{\delta_i}}\geq \delta^{\zeta_2} (\delta_i/\tau)^2$ suffices), then after a refinement, $\tubes_{\delta_i}$ satisfies Conclusion (B.ii). Conclusion (B.iii) then follows from Remark \ref{FrostmanWolffInheritedUpwardsDownwards}(A) (Frostman Wolff Axioms are inherited upwards), and we are done.

Suppose instead that $\frac{\#\tubes_\tau}{\#\tubes_{\delta_i}}< \delta^{100\eta}(\delta_i/\tau)^2$. Let $\rho$ be the minimum of all admissible scales in $[\delta_i,1]$ for which
\begin{equation}\label{minimalCoverOfTubesByRho}
\frac{\#\tubes_\tau}{\#\tubes_\rho}\geq \delta^{100\eta}(\rho/\tau)^2.
\end{equation}
Such a choice of $\rho\in[\rho_i,1]$ must exist, since $\CFC(\tubes_{\tau})\lessapprox_\delta \delta^{-\eta},$ and hence $\#\tubes_{\tau}\gtrapprox_\delta\delta^\eta\tau^{-2},$ from which it follows that $\rho=1$ satisfies \eqref{minimalCoverOfTubesByRho} and hence $\rho=1$ is a valid candidate.

In particular, for this choice of $\rho$ we have
\begin{equation}\label{minimalCoverOfTubesByRhoUpperLower}
\delta^{100\eta}(\rho/\tau)^2 \leq \frac{\#\tubes_\tau}{\#\tubes_\rho} \leq \delta^{50\eta}(\rho/\tau)^2,
\end{equation}
since if the RHS of \eqref{minimalCoverOfTubesByRhoUpperLower} failed, then \eqref{minimalCoverOfTubesByRho} would hold for a smaller admissible scale, which would contradict the minimality of $\rho$. 

Suppose for the moment that there exists a $\gtrapprox_{\delta} 1$-refinement of $\tubes_{\tau}$ (abusing notation, we will continue to call this set $\tubes_\tau$) such that 
\begin{equation}\label{goodFrostmanBoundTauInsideRho}
\CFC\big(\tubes_\tau^{T_{\rho}}\big)\leq \delta^{-\zeta_2/2}\quad\textrm{for at least half of the tubes}\ T_{\rho}\in\tubes_{\rho}.
\end{equation}
Then \eqref{minimalCoverOfTubesByRhoUpperLower} plus \eqref{goodFrostmanBoundTauInsideRho} implies that Conclusion (B.ii) holds, and Conclusion (B.iii) follows from Remark \ref{FrostmanWolffInheritedUpwardsDownwards}(A) (Frostman Wolff axioms are inherited upwards). Thus if \eqref{goodFrostmanBoundTauInsideRho} is true, then Conclusion (B) of Proposition \ref{tubeTricotProp} holds, and we are done.

\medskip

\noindent{\bf Step 5.} We claim that either \eqref{goodFrostmanBoundTauInsideRho} holds (and we are done, as discussed in the previous step), or else Conclusion (C) of Proposition \ref{tubeTricotProp} holds.

We will verify this claim as follows. Suppose that \eqref{goodFrostmanBoundTauInsideRho} fails.  
Applying Proposition \ref{factoringConvexSetsProp} to each set $\tubes^{T_\rho}$ and then undoing the scaling $\phi_{T_\rho}$, we obtain a refinement of $\tubes[T_\rho]$ (abusing notation, we will continue to call this set $\tubes[T_\rho]$), and a set $\mathcal{U}_{T_\rho}$ of $s\times t\times 1$ prisms contained in $T_\rho$ (so in particular $t\leq\rho$) that factors $\tubes[T_\rho]$ from below with respect to the Frostman Convex Wolff axioms with error $\lessapprox_\delta 1$ and each $T\in \tubes[T_{\rho}]$ is contained in $\lessapprox_{\delta} 1$ sets $U\in \mathcal{U}_{T_{\rho}}$.

Pigeonhole and refine $\tubes_{\rho}$ to consist of those tubes $T_{\rho}$ for which  $\CFC\big(\tubes_\tau^{T_{\rho}}\big)>\delta^{-\zeta_2/2}$  (such a refinement exists because \eqref{goodFrostmanBoundTauInsideRho} fails) and the corresponding sets $\mathcal{U}_{T_{\rho}}$  are comparable to $s\times t\times 1$ prisms for a common pair of numbers $(s,t)$.  In particular, this implies 
\begin{equation}\label{ULessapproxBd}
|U|\lessapprox_\delta \delta^{\zeta_2/2}|T_\rho|.
\end{equation}
(C.f.~\eqref{WMuchSmallerThanTauDeltaI} when $s\geq \tau$. When $s\leq \tau$, this is true because $\tau \lesssim \delta^{\zeta_1/5}\delta_i\leq \delta^{\zeta_1/5}\rho$).

Suppose for the moment that $s\geq \tau$.  Then $\CFC(\tubes_{\tau}^U)\lessapprox_{\delta} 1$ by Remark \ref{FrostmanWolffInheritedUpwardsDownwards}(A) (Frostman Wolff axioms are inherited upwards).   By \eqref{minimalCoverOfTubesByRhoUpperLower} and our hypothesis that $\CFC(\tubes_{\tau}^{T_{\rho}}) > \delta^{-\zeta_2/2}$, we have 
\begin{equation}\label{eq: CKTtubestauTrho}
	 \CKT(\tubes_{\tau}[T_{\rho}]) \geq \delta^{-\zeta_2/2+100\eta}.
\end{equation} We conclude that  by Item ii) of Proposition~\ref{factoringConvexSetsProp}, 
\[
\#(\tubes[T_\rho])[U]\gtrapprox_{\delta} \CKT(\tubes[T_\rho]) \Big(\frac{|U|}{|T|}\Big) \quad\textrm{for each}\ U\in \mathcal{U}_{T_\rho}.
\]
Since $s\geq \tau$ and $\tau$ is an admissible scale, we have 
\[
\CKT(\tubes[T_\rho]) \gtrapprox_{\delta} \frac{\#\tubes}{\#\tubes_\tau} \frac{|T|}{|T_\tau|} \CKT(\tubes_\tau[T_\rho]),
\]  
and so 
\begin{equation}\label{numberOfTTauInWPrime}
\#(\tubes_\tau[T_\rho])[U] \gtrapprox_\delta
\CKT\big(\tubes_\tau[T_{\rho}]\big)\Big(\frac{|U|}{|T_{\tau}|}\Big)
\geq \delta^{-\zeta_2/2+100\eta}\Big(\frac{|U|}{|T_{\tau}|}\Big)\quad\textrm{for each}\ U\in \mathcal{U}_{T_\rho}.
\end{equation}

\noindent We will show that 
\begin{equation}\label{sSmall}
s<\delta^{\zeta_2/100}t.
\end{equation} 
To verify \eqref{sSmall}, suppose to the contrary that $s\geq\delta^{\zeta_2/100}t$; we will obtain a contradiction. If $s\geq\delta^{\zeta_2/100}t$ then  $t\leq \delta^{\zeta_2/4}\rho$.  By \eqref{numberOfTTauInWPrime} and the fact that each tube in $\tubes_{\tau}[T_{\rho}]$ is contained in $\lessapprox_{\delta} 1$ sets $U\in \mathcal{U}_{T_{\rho}}$, we have
\[
\# \mathcal{U}_{T_\rho}  \lessapprox_{\delta} \frac{ \#\tubes_{\tau}[T_{\rho}]}{\inf_{U\in \mathcal{U}_{T_\rho}}\#(\tubes_{\tau}[T_{\rho}])[U]}\lessapprox_\delta \delta^{\zeta_2/2-100\eta}  (\#\tubes_\tau[T_\rho])\frac{|T_{\tau}|}{|U|}\leq \delta^{\frac{48}{100}\zeta_2}(\#\tubes_\tau[T_\rho])\big(\frac{\tau}{t}\big)^2,
\]
and thus if $t'$ is the smallest admissible scale greater than or equal to $t$, then
\begin{equation}\label{lotsOfTauInsideTPrime}
\#\tubes_{t'} \leq \# \Big( \bigcup_{T_\rho\in\tubes_\rho}\mathcal{U}_{T_\rho}\Big) \lessapprox_\delta  \delta^{\frac{48}{100}\zeta_2}(\#\tubes_\tau)\big(\frac{\tau}{t}\big)^2\leq \delta^{\frac{47}{100}\zeta_2}(\#\tubes_\tau)\big(\frac{\tau}{t'}\big)^2.
\end{equation}
From \eqref{lotsOfTauInsideTPrime} we see that $t'\geq\delta_i$ --- if not, then there would exist a $t'$-tube $T_{t'}$ with $\#\tubes_\tau[T_{t'}]\gtrapprox_\delta \delta^{-\frac{47}{100}\zeta_2}\frac{|T_{t'}|}{|T_{\tau}|}$, and this $t'$-tube would be contained in some $\delta_i$-tube; but this violates \eqref{goodKatzTaoBoundTauInsideRhoI}.

Comparing \eqref{lotsOfTauInsideTPrime} and \eqref{minimalCoverOfTubesByRho}, we see that $t' \in [t, \delta^{\zeta_2/5}\rho] \subset [\delta_i, \delta^{\zeta_2/5}\rho]$ is an admissible scale that satisfies \eqref{minimalCoverOfTubesByRhoUpperLower}. But this contradicts the assumption that $\rho$ was the minimal such scale. We conclude that \eqref{sSmall} must hold.

When \eqref{sSmall} holds,  we are  in precisely the same situation as the beginning of Step 3:
The set of flat (recall \eqref{sSmall}) prisms $\mathcal{U} = \bigcup_{T_\rho\in\tubes_\rho}\mathcal{U}_{T_\rho}$ plays the role of $\mathcal{W}$, and the set of essentially distinct $\rho$-tubes $\tubes_\rho$ plays the role of $\tubes_{\delta_i}$. An identical argument with the same numerology (up to harmless $\delta^{\eta}$ factors) shows that Conclusion (C) of Proposition \ref{tubeTricotProp} holds.

Now suppose \eqref{sSmall} does not hold, so $\delta^{\zeta_2/100}t\leq s<\tau$.  If $s \geq \tau \delta^{\zeta_2/20}$, then together with $s\geq \delta^{\zeta_2/100}t$, 
 \[
 \CKT(\tubes_\tau[T_\rho]) \leq \delta^{-\zeta_2/4} \CKT(\mathcal{U}_{T_\rho}) \lessapprox_{\delta} \delta^{-\zeta_2/4},
 \]
 which is a contradiction to \eqref{eq: CKTtubestauTrho}.  So we may assume $ s\leq \tau\delta^{\zeta_2/20}$ and $s\geq \delta^{\zeta_2/100}t$, and  we are in the same situation as the beginning of Step 4 with $(s,t)$ in place of $(a,b)$, $\mathcal{U}$ in place of $\mathcal{W}$, and $\tubes_\rho$ in place of $\tubes_{\delta_i}$. However, we have the additional condition that  $s/\rho \leq \delta^{\zeta_2/20} \tau/\delta_i$; this means that the prisms in $\mathcal{U}$ are substantially flatter than the prisms in $\mathcal{W}$. We return to the beginning of Step 4 and repeat the argument; we iterate this process until either Conclusion (B) or Conclusion (C) holds; this must occur after at most $20/\zeta_2$ iterations. Note that each iteration of this process induces a $\approx_\delta 1$ refinement of $\tubes$, etc.~but since this process repeats at most $20/\zeta_2$ times, this refinement is harmless.


\section{A two-scale grains decomposition for tubes in $\RR^3$}\label{twoScaleGrainsSec}

In \cite{KLT00}, Katz, \L{}aba and Tao proved that every union of $\delta$ tubes in $\RR^3$ coming from the discretization of a (hypothetical) Kakeya set with upper Minkowski dimension $5/2$ can be written as a union of ``grains,'' (i.e.~rectangular prisms) of dimensions roughly $\delta\times\delta^{1/2}\times\delta^{1/2}$. Guth \cite{Gut14} generalized this result and proved that every union of $\delta$ tubes in $\RR^3$ satisfying a certain broadness hypothesis can be written as a union of grains of dimensions roughly $\delta\times t\times t$, where the diameter $t$ is related to the number of tubes in the arrangement and the volume of their union.

The purpose of this section is to prove a structural statement for unions of $\delta$ tubes in $\RR^3$, in the spirit of the Katz-\L{}aba-Tao and Guth results described above. This is Proposition \ref{grainsDecomposition} below. As discussed in the introduction, Proposition \ref{grainsDecomposition} is a key step in the proof of Proposition \ref{improvingProp}---Proposition \ref{grainsDecomposition} helps us find the correct scales and arrangements of convex sets to which we can apply Assertion $\cE(\sigma,\omega)$.

In brief, Proposition \ref{grainsDecomposition} explores what happens when we cover an arrangement of $\delta$-tubes by $\rho$-tubes, apply (a variant of) Guth's grains theorem inside each re-scaled $\rho$-tube, and then analyze how the resulting grains coming from the $\delta$ tubes inside different $\rho$-tubes interact. The specific hypotheses and conclusions of Proposition \ref{grainsDecomposition} are somewhat technical; they were adapted to match the needs of the arguments in Section \ref{refinedInductionOnScaleSec}. In order to state Proposition \ref{grainsDecomposition} we will require a few definitions. 

\begin{defn}\label{twoScaleGrainsDecomp}
Let $\lambda>0$, $0<\delta\leq\rho\leq 1$ and $\delta\leq a\leq b\leq c$ with $\rho = b/c$. Let $(\tubes,Y)_\delta$ be a set of $\delta$ tubes and their associated shading, and let $\tubes_\rho$ be a balanced partitioning cover of $\tubes$. We say $(\mathcal{P},Y)_{a\times b\times c}$ is a \emph{robustly $\lambda$-dense two-scale grains decomposition of $(\tubes,Y)_\delta$ with regard to (wrt) $\tubes_\rho$} if the following is true:
\begin{itemize}

\item[(i)] For each $P\in\mathcal{P}$, there is a unique $\tubes_\rho\in\tubes_{\rho}$ satisfying $P\subset T_\rho$ and $\angle(\dir(P),\dir(T_{\rho}))\leq 2\rho$. This induces a partition $\mathcal{P}=\bigsqcup_{\tubes_{\rho}} \mathcal{P}_{T_{\rho}}$.

\item[(ii)] For each $T_{\rho}\in\tubes_\rho$, the sets $\{Y(P)\colon P\in P\in\mathcal{P}_{T_\rho}\}$ are disjoint, and we have
\begin{equation}\label{equalityOfTubesGrains}
\bigcup_{T\in\tubes[T_\rho]}Y(T) = \bigsqcup_{P\in\mathcal{P}_{T_\rho}}Y(P).
\end{equation}

\item[(iii)] The pair $(\mathcal{P},Y)_{a\times b\times c}$ is $\lambda$-dense, and furthermore there exists a number $\mu$ so that for each $T_\rho\in\tubes_\rho$ and each $x\in \eqref{equalityOfTubesGrains}$, we have $\#\tubes[T_\rho]_Y(x)\sim\mu$.

\item[(iv)] For each $T_{\rho}\in\tubes_\rho$ and each pair $T\in\tubes[T_\rho]$ and $P\in\mathcal{P}_{T_\rho}$ with $Y(T)\cap Y(P)\neq\emptyset$, we have that $T$ exits $P$ through the ``long end'' (See Figure \ref{exitLongEndsFig}), and $Y(T)\cap P \subset Y(P)$.
\end{itemize}
\end{defn}

\begin{figure}
 \includegraphics[width=.3\linewidth]{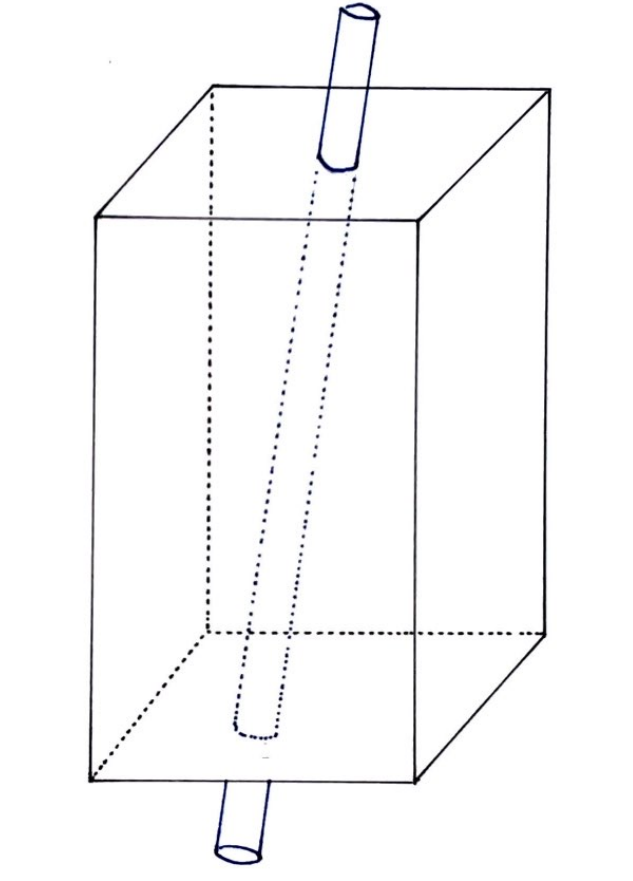}\hfill
  \includegraphics[width=.3\linewidth]{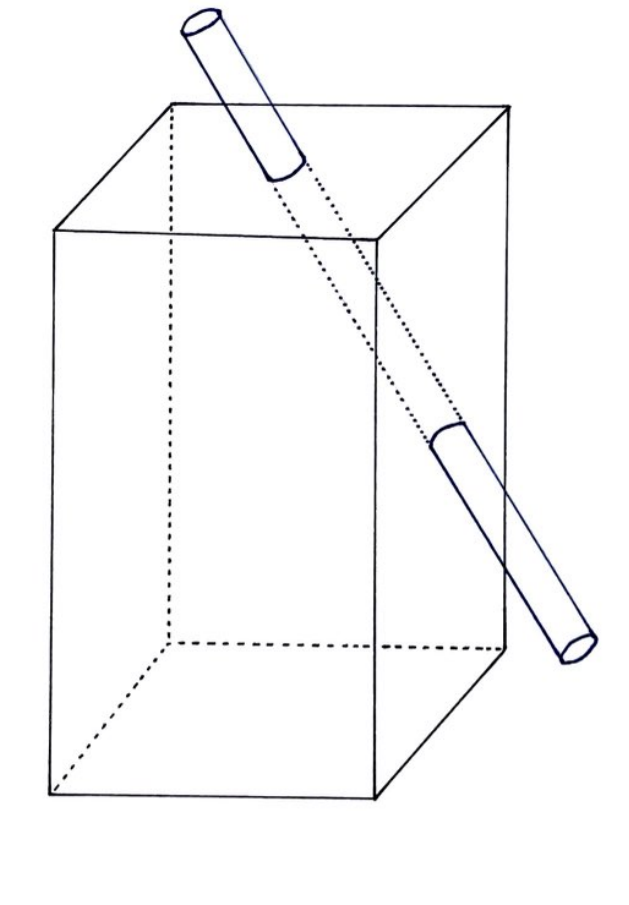}\hfill
   \includegraphics[width=.3\linewidth]{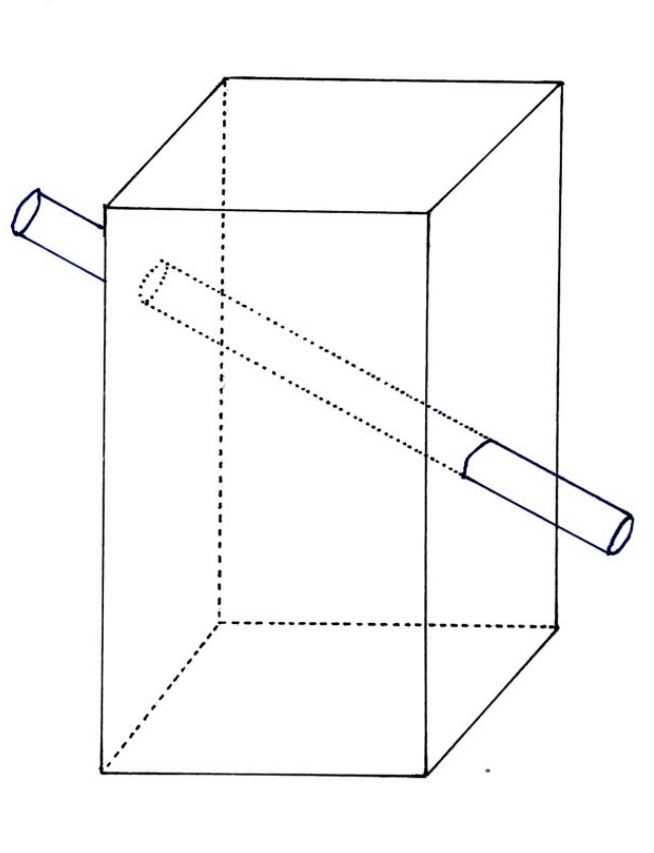}\hfill
\caption{A tube exiting a grain through the long ends (left) vs failing to do so (middle and right).}
\label{exitLongEndsFig}
\end{figure}

\begin{rem}
Conclusion (iv) implies that $a\geq 2\delta$. Usually we will be interested in the case where $a\sim\delta$, though it will sometimes be useful to consider larger values of $a$.
\end{rem}

\begin{defn}\label{defnPSquare}
Let $P$ be a $a\times b\times c$ prism. Define $\square(P)$ to be the $\frac{ac}{b}\times c\times c$ prism containing $P$ with the same center and normal direction as $P$ (the latter condition means that $\Pi(P) = \Pi(\square(P))$.  Observe that both $\Pi(P)$ and $\Pi(\square(P))$ are defined up to accuracy $a/b$. 

Let $\mathcal{P}$ be a set of $a\times b\times c$ prisms, and let $W$ be a convex set. Define
\[
\mathcal{P}\langle W\rangle = \{P\in\mathcal{P}\colon \square(P)\subset W\}.
\]
Note that since $P\subset\square(P)$, we have $\mathcal{P}\langle W\rangle\subset\mathcal{P}[W]$.
\end{defn}

Observe that if $P,P'$ are intersecting $a\times b\times c$ prisms and $\angle(\Pi(P),\Pi(P'))\leq K\frac{a}{b}$ for some $K\geq 1$, then the $K$-fold thickenings (i.e. the prisms of dimensions $\frac{ac}{b}\times c\times c$ with the same center, direction, and tangent plane) of $\square(P)$ and $\square(P')$ are comparable up to factors of $K$. In particular, if $(\mathcal{P},Y)_{a\times b\times c}$ is a pair of prisms and their associated shadings, with $\angle(\Pi(P),\Pi(P'))\leq K\frac{a}{b}$ for all pairs $P,P'\in\mathcal{P}$ for which $Y(P)\cap Y(P')\neq\emptyset$, then we can find a set $\mathcal{W}$ of prisms of dimensions comparable to $K\frac{ac}{b}\times c\times c$ so that each $P\in\mathcal{P}$ is contained in at least 1, and at most $O(1)$ sets of the form $\mathcal{P}\langle W\rangle,\ W\in\mathcal{W}.$ Furthermore, the sets $\big\{\bigcup_{P\in\mathcal{P}\langle W\rangle}Y(P),\ W\in\mathcal{W}\big\}$ are $O(1)$ overlapping.

In our arguments below, we will exploit the above observation when $K\leq\delta^{-\eps}$ for a small $\eps>0$. Since we will analyze each set $\mathcal{P}\langle W\rangle$ individually, we will be interested in the quantity $\CKT(\mathcal{P}\langle W\rangle)$, rather than the (potentially much larger) quantity $\CKT(\mathcal{P})$. Each set $W$ contains at most $100K^{100}$ essentially distinct prisms of the form $\square(P)$, and thus if $K$ is not too large then $\CKT(\mathcal{P}\langle W\rangle)$ is controlled by $\sup_P \CKT(\mathcal{P}\langle 2\square(P)\rangle)$, where $2\square(P)$ denotes the 2-fold dilate of $\square(P)$. This motivates the following definition. 
\begin{defn}
Let $\mathcal{P}$ be a set of $a\times b\times c$ prisms. Define
\[
\CKT^{\operatorname{loc}}(\mathcal{P})=\max_{P\in\mathcal{P}}\CKT\big(\mathcal{P}\langle2\square(P)\rangle\big).
\]
\end{defn}

With these definitions, we can now state the main result of Section \ref{twoScaleGrainsSec}.

\begin{prop}\label{grainsDecomposition}
Let $\omega>0,\ \sigma\in(0,2/3]$, and $\zeta\in(0,\omega/1000)$. Suppose that $\cE(\sigma,\omega)$ is true. Then there exists $\alpha,\eta,\kappa>0$ so that the following holds for all $\delta>0$. Let $(\tubes,Y)_\delta$ be $\delta^\eta$ dense, with $\CKT(\tubes)\leq\delta^{-\eta}$ and $\FS(\tubes)\leq\delta^{-\eta}$. Then at least one of the following must hold.
\begin{itemize}
	\item[(A)] \itemizeEqnVSpacing
		\[
			\Big|\bigcup_{T\in\tubes}Y(T)\Big| \geq \kappa \delta^{\omega-\alpha}(\#\tubes)|T|\big((\#\tubes)|T|^{1/2}\big)^{-\sigma}.
		\]

	\item[(B)] 
There exist the following:
\begin{itemize} 
	\item Numbers $\rho$ and $\delta\leq a\leq b\leq c\leq 1$, with $\rho=b / c$.
	\item A $\delta^{\zeta}$ refinement $(\tubes',Y')_\delta$ of $(\tubes,Y)_\delta$.
	\item A set $\tubes_{\rho}$ of $\rho$ tubes.
	\item A pair $(\mathcal{G}, Y)_{a \times b\times c}$.
	\end{itemize}
These objects have the following properties:
	\begin{itemize}
	\item[(i)] $(\tubes',Y')_\delta$  is $\delta^{\zeta}$ dense, $\CKT(\tubes')\leq\delta^{-\zeta}$, and $\FS(\tubes')\leq\delta^{-\zeta}$. 
	\item[(ii)] $\tubes_{\rho}$ is a balanced partitioning cover of $\tubes'$ that factors $\tubes'$ above and below with respect to the Frostman Slab Wolff Axioms with error $\delta^{-\zeta}$. 
	\item[(iii)] $(\mathcal{G}, Y)_{ a \times  b\times  c}$ is a robustly $\delta^\zeta$-dense two-scale grains decomposition of $(\tubes',Y')_\delta$ wrt $\tubes_\rho$.
	\item[(iv)] \itemizeEqnVSpacing
		\begin{equation}\label{cSuffBig}
				 c \geq \delta^{\zeta} \frac{\rho}{\delta}\frac{\#\tubes_\rho}{\#\tubes_\delta}.
		\end{equation} 
	\item[(v)] $\delta^{1-\omega/100}\leq\rho\leq \delta^{\omega/100}$.
	\item[(vi)] $\CKT^{\operatorname{loc}}(\mathcal{G})\leq \delta^{-\zeta}$.
	\end{itemize}
\end{itemize}
\end{prop}

\noindent An overview for the proof of Proposition \ref{grainsDecomposition} is outlined in Section \ref{twoScaleGrainsDecompIntro}.

\medskip
\noindent{\bf Fixing $\omega$ and $\sigma$}. In our proof of Proposition \ref{grainsDecomposition}, the values of $\omega>0$, $\sigma\in(0,2/3]$ and $\zeta>0$ will never change. Thus to simplify our exposition below, we will fix values of $\boldsymbol{\omega}$, $\boldsymbol{\sigma},$ and $\boldsymbol{\zeta}$, which will remain unchanged throughout Sections \ref{twoScaleGrainsSec} and \ref{moves123Sec}. In particular, some of our definitions (such as Definition \ref{broadnessDefn}, which defines broadness) will depend on these quantities. By fixing them in advance, we can suppress this dependence.

As previewed in Section \ref{twoScaleGrainsDecompIntro}, we will make crucial use Guth's methods to find a grains decomposition inside each re-scaled set $\tubes^{T_\rho}$. In order to apply Guth's techniques, the tubes need to satisfy a ``broadness'' condition. We describe this below.


\subsection{Broadness}\label{broadnessSection}
For the definition that follows, we will fix $\boldsymbol{\beta}=\boldsymbol{\omega}\boldsymbol{\zeta}/100.$

\begin{defn}\label{broadnessDefn}
Let $\delta>0$ and $K\geq 1$. We say a multi-set $\mathcal{V}\subset S^{n-1}$ of unit vectors in $\RR^n$ is \emph{broad with error $K$ at scales $\geq\delta$} if for all unit vectors $v_0\in\RR^n$ and all $r\in[\delta,1]$ we have
\[
\# \{v\in\mathcal{V}\colon \angle(v,v_0)\leq r\} \leq K r^{\boldsymbol{\beta}} (\#\mathcal{V}).
\]

More generally, let $\mathcal{V}$ be contained in a disk $D\subset S^{n-1}$ of radius $\rho$. We say $\mathcal{V}$ is broad with error $K$ at scales $\geq\delta$ inside $D$ if for all unit vectors $v_0\in\RR^n$ and all $r\in[\delta,\rho]$ we have
\[
\# \{v\in\mathcal{V}\colon \angle(v,v_0)\leq r\} \leq K (r/\rho)^{\boldsymbol{\beta}} (\#\mathcal{V}).
\]
\end{defn}

\begin{defn}\label{broadnessDefnShading}
Let $(\tubes,Y)_\delta$ be a collection of tubes and their associated shading, and let $\tubes_\rho$ be a cover of $\tubes$. 

\begin{itemize}
	\item[(A)] We say that $(\tubes,Y)_\delta$ is \emph{broad with error $K$} if for each $x\in\bigcup_{T\in\tubes}Y(T)$, the set of unit vectors $\{\dir(T)\colon x\in Y(T)\}$ is broad with error $K$ at scales $\geq\delta$. 
	\item[(B)] We say that $(\tubes,Y)_\delta$ is \emph{broad with error $K$} relative to the cover $\tubes_\rho$ if for each $T_\rho\in\tubes_\rho$, the set $(\tubes^{T_\rho},Y^{T_\rho})_{\delta/\rho}$ is broad with error $K$.
\end{itemize}
\end{defn}

The next result is a variant of the ``two ends'' reduction. In general, a set $\mathcal{V}$ of unit vectors need not be broad (with small error). However, the next result says that every set of unit vectors is broad when localized inside $\rho$ disks, for some value of $\rho$. The precise statement is as follows. 

\begin{lem}\label{findingGoodBroadness}
Let $\delta>0$ and let $\mathcal{V}$ be a set of vectors in $\RR^n$ pointing in $\delta$-separated directions. Then there exists a scale $\rho\in[\delta,1]$; a set $\mathcal{B}$ of disjoint balls $B\subset S^{n-1}$ of radius $\rho$; and sets $\mathcal{V}_B\subset\mathcal{V}\cap B$ so that the following holds.
\begin{itemize}
\item[(i)] Each set $\mathcal{V}_B$ has cardinality $\gtrsim \rho^{\boldsymbol{\beta}}(\#\mathcal{V})$.
\item[(ii)] $\mathcal{V}_B$ is broad with error $100$ at scales $\geq\delta$ inside $B$.
\item[(iii)] $\bigcup_{\mathcal{B}} \mathcal{V}_B \gtrsim (\log 1/\delta)^{-1}(\#\mathcal{V})$.
\end{itemize}
\end{lem}
\begin{proof}
We will greedily construct a sequence of sets $\mathcal{V}=\mathcal{V}_0\supset\mathcal{V}_1\supset\ldots$ as follows. For each index $i\geq 1$, let $v_i$ be a unit vector and $r_i\in[\delta,1]$ a radius that maximizes the quantity
\begin{equation}\label{maximalityOfVR}
 r_i^{-\boldsymbol{\beta}}(\#\mathcal{V}_{i-1}\cap B(v_i,r_i)).
\end{equation}
Let $\mathcal{W}_i = \mathcal{V}_{i-1}\cap B(v_i,r_i)$ and let $\mathcal{V}_i = \mathcal{V}_{i-1}\backslash B(v_i,\ 100r_i)$. Note that $\#(\mathcal{V}_{i-1}\cap B(v_i,3r_i))\leq 100^{\boldsymbol{\beta}}(\#\mathcal{W}_i)$. Continue this process until $\#\mathcal{V}_i \leq \frac{1}{2}\#\mathcal{V}$. 

We claim that for each index $i$, each unit vector $v$, and each $r\geq \delta$, we have
\begin{equation}\label{nonConcentratedInsideWi}
\#(\mathcal{W}_i \cap B(v,r)) \leq 100 (r/r_i)^{\boldsymbol{\beta}}(\#\mathcal{W}_i).
\end{equation}
This will give Conclusion (ii). To verify \eqref{nonConcentratedInsideWi}, suppose to the contrary that \eqref{nonConcentratedInsideWi} failed for some pair $(v,r)$, then this would contradict the maximality of $(v_i,r_i)$ in \eqref{maximalityOfVR}. 

Note as well that since $r_i = 1$ is a valid choice of $r$, we have $\#\mathcal{W}_i\geq r_i^{\boldsymbol{\beta}}(\#\mathcal{V}_{i-1})\geq\frac{1}{2}r_i^{\boldsymbol{\beta}}(\#\mathcal{V})$. Finally, note that if $r_i\leq r_j \leq 2r_i$, then the balls $2B(v_i, r_i)$ and $2B(v_j, r_j)$ must be disjoint. Indeed, we may suppose that both $\mathcal{W}_i$ and $\mathcal{W}_j$ are non-empty. If $i<j$ then we must have $B(v_j,r_j)\cap B(v_i, 100r_i)^c\neq\emptyset$, while if $j<i$ then we must have $B(v_i,r_i)\cap B(v_i, 100r_j)^c\neq\emptyset$. In either case, since $r_i \leq r_j \leq 2r_i$, this means that $2B(v_i,r_i)\cap 2B(v_j,r_j)=\emptyset$. 

To conclude the proof, use dyadic pigeonholing to select a scale $\rho\in[\delta,1]$ so that $\# \bigcup_{i\colon \rho/2\leq r_i\leq \rho}\mathcal{W}_i\gtrsim(\log 1/\delta)^{-1}(\#\mathcal{V})$. This gives us the collection $\mathcal{B}$ of disjoint balls (it is harmless for us to replace each ball $B(v_i,r_i)$ with $B(v_i,\rho)$; as noted above, the balls remain disjoint).
\end{proof}

\begin{rem}
Lemma \ref{findingGoodBroadness} is similar to the standard ``two-ends'' broadness reduction. However, the standard two-ends broadness reduction typically replaces $\mathcal{V}$ by $\mathcal{V}\cap B$ for a \emph{single} $\rho$ ball $B$. The resulting set has cardinality $\gtrsim\rho^{\boldsymbol{\beta}}(\#\mathcal{V})$.  For our applications, this would have introduced an unacceptably large reduction in the cardinality of $\mathcal{V}$.
\end{rem}

\begin{cor}\label{coveringTubesBroadCor}
Let $(\tubes,Y)_\delta$ be a set of $\delta$ tubes and their corresponding shading. Then there exists a $\gtrapprox_\delta 1$ refinement $(\tubes',Y')_\delta$ of $(\tubes,Y)_\delta$; a scale $\rho\in[\delta,1]$; and a balanced partitioning cover $\tubes_\rho$ of $(\tubes',Y')$, so that the following holds:
\begin{itemize}
	\item[(i)] Each point $x\in\RR^3$ is contained $\lessapprox_\delta \rho^{-\boldsymbol{\beta}}$ sets of the form $\bigcup_{\tubes'[T_\rho]}Y'(T)$, $T_\rho\in\tubes_\rho$.
	\item[(ii)] $\tubes'$ is broad with error $\approx_\delta 1$ relative to the cover $\tubes_\rho$.
\end{itemize}
\end{cor}
\begin{proof}
Using dyadic pigeonholing, we can select a number $\mu\geq 1$ and a $(\log 1/\delta)^{-1}$ refinement $(\tubes,Y_1)_\delta$ of $(\tubes,Y)$ with the property that $\#\{T\in\tubes\colon x\in Y_1(T)\}\sim\mu$ for each $x\in\bigcup_{T\in\tubes}Y_1(T)$. Apply Lemma \ref{findingGoodBroadness} to each point $x\in \bigcup Y_1(T)$. We obtain a scale $\rho_x\in[\delta,1]$, and a set of disjoint $\rho_x$ balls $\mathcal{B}_x$ (each ball in $\mathcal{B}_x$ is a subset of $S^2$). After further dyadic pigeonholing we can select the following:
\begin{itemize}
	\item A common scale $\rho$;
	\item Multiplicities $\nu \geq 1$ and $N\lesssim\rho^{-\boldsymbol{\beta}}$;
	\item A set $\mathcal{B}$ of $\rho$ balls (each ball in $\mathcal{B}$ is a subset of $S^2$), whose $100\rho$ neighbourhoods are disjoint;
	\item A $\gtrsim (\log 1/\delta)^{-1}$ refinement $(\tubes,Y_2)_\delta$ of $(\tubes,Y_1)_\delta$.
\end{itemize}
We can select the above numbers and sets so that the following property holds: for each $x\in \bigcup_{\tubes} Y_2(T)$, the set of unit vectors $\{\dir(T)\colon x\in Y_2(T)\}\subset S^2$ can be covered by a union of $N$ balls from $\mathcal{B}$, and for each such ball, we have that the set $B\cap \{\dir(T)\colon x\in Y_2(T)\}$ has cardinality $\nu$ and is broad with error $O(1)$ inside $B$.

After refining $\tubes$ by a factor of $\approx_{\delta} 1$, we obtain a new pair $(\tubes_2,Y_2)_\delta$ that is a $\approx_{\delta} 1$ refinement of $(\tubes,Y_2)_\delta$, and a set $\tubes_\rho$ of $\rho$ tubes with the property that $\tubes_\rho$ is a balanced partitioning cover of $\tubes_2$, and furthermore the family of convex sets $\{3T_\rho\colon T_\rho\in\tubes_\rho\}$ is a partitioning cover of $\tubes_2$. This means that for each $T_\rho\in\tubes_\rho$, $\tubes_2[T_\rho]=\tubes_2[3T_\rho]$. 

For each $T\in\tubes_2$ with $T\subset T_\rho\in\tubes_\rho$, define
\[
Y_3(T) = \{x\in Y_2(T)\colon \#\{ T'\in\tubes_2[T_\rho]\colon x\in Y_2(T')\} \geq \kappa_0 \nu\}.
\]
Since $(\tubes_2,Y_2)_\delta$ is a $\approx_{\delta} 1$ refinement of $(\tubes,Y_2)_\delta$, if $\kappa_0>0$ is chosen sufficiently small (depending on the implicit constant mentioned previously), then $(\tubes_2,Y_3)_\delta$ is a $\approx_{\delta} 1$ refinement of $(\tubes_2,Y_2)_\delta$. Furthermore, for each point $x\in\bigcup_{T\in\tubes[3T_\rho]}Y_3(T) = \bigcup_{T\in\tubes[T_\rho]}Y_3(T)$, we have that the set of unit vectors $\{\dir(T)\colon T\in\tubes_2[3T_\rho]=\tubes_2[T_\rho],\ x\in Y_3(T)\}$ is broad inside $B(\dir(T_\rho), 2\rho)$ with error $O(1)$, in the sense of Definition \ref{broadnessDefn}.  
We conclude that the pair $(\tubes_2,Y_3)_\delta$ and $\{3T_\rho, T_\rho\in\tubes_\rho\}$ satisfy the conclusions of Corollary \ref{coveringTubesBroadCor} (with $3\rho$ in place of $\rho$).
\end{proof}

The next result says that every pair $(\tubes,Y)_\delta$ of $\delta$ tubes and their associated shading is either broad at \emph{some} scale $\delta<\!\!<\rho<\!\!< 1$, or else the tubes in $(\tubes,Y)_\delta$ are almost disjoint.

\begin{lem}\label{findingBroadScale}
Let $\delta>0$ and let $(\tubes,Y)_\delta$ be a pair of $\delta$ tubes and their associated shading. Then at least one of the following must occur:
\begin{itemize}
	\item[(A)] \itemizeEqnVSpacing
		\begin{equation}\label{goodConclusionR1R2Scale}
			\Big|\bigcup_{T\in\tubes}Y(T)\Big| \gtrapprox_\delta \delta^{\boldsymbol{\omega}/2} \sum_{T\in\tubes}|Y(T)|.
		\end{equation}

	\item[(B)] There is a scale $\rho\in[\delta^{1-\boldsymbol{\omega}/100},\delta^{\boldsymbol{\omega}/100}],$ a $\approx_\delta 1$ refinement $(\tubes',Y')_\delta$ of $(\tubes,Y)$, and a balanced partitioning cover $\tubes_\rho$ of $\tubes$, so that  $(\tubes',Y')_\delta$ is broad with error $O(1)$ relative to the cover $\tubes_\rho$.
\end{itemize}
\end{lem}

\begin{proof}
Let $r=\delta^{\boldsymbol{\omega}/100}$. After replacing $(\tubes,Y)_\delta$ by a $\sim 1$ refinement, we can find a partitioning cover $\tubes_{r}$ of $\tubes$. Apply Corollary \ref{coveringTubesBroadCor} to each set $(\tubes^{T_{r}},Y^{T_{r}})_{\delta/r}$. After dyadic pigeonholing $\tubes_{r}$ and $\tubes$, we can suppose that the resulting scale, which we will call $\tilde\rho$, is the same for each $T_{r}\in\tubes_{r}$. Define $\rho = \tilde\rho r$, and let $(\tubes',Y')$ be the $\approx_\delta 1$ refinement of $(\tubes,Y)_\delta$ consisting of the tubes and their associated shadings coming from the conclusion of Corollary \ref{coveringTubesBroadCor} for each $T_{r}\in\tubes_{r}$.

Recall that the tubes $\tubes_{r}$ form a partitioning cover of $\tubes'$, and for each $T_{r}\in\tubes_{r}$, the (re-scaled) $\rho$ tubes coming from Corollary \ref{coveringTubesBroadCor} form a partitioning cover of $\tubes'[T_{r}]$. We conclude that if we define $\tubes_\rho$ to be the union of these $\rho$ tubes, then $\tubes_\rho$ is a partitioning cover of $\tubes'$, and furthermore $(\tubes',Y')_\delta$ is broad with error $O(1)$ relative to the cover $\tubes_\rho$. 

At this point, we have constructed the pair $(\tubes',Y')_\delta$ and $\tubes_\rho$, which satisfies all of the requirements for Conclusion (B) of Lemma \ref{findingBroadScale} with one exception --- we know that $\rho\in[\delta,\delta^{\boldsymbol{\omega}/100}]$, but we do not know that $\rho\in[\delta^{1-\boldsymbol{\omega}/100},\delta^{\boldsymbol{\omega}/100}].$  Our next task is to show that if   $\rho<\delta^{1-\boldsymbol{\omega}/100}$, then Conclusion (A) holds. 

Since the tubes in $\tubes_{r}$ are essentially distinct, $O(r^{-2})$ tubes can pass through a common point. By Conclusion (i) from Corollary \ref{coveringTubesBroadCor}, we have that for each tube $T_{r}$ and each $x\in T_{r}$ we have
\[
\# \{ T_\rho\in\tubes_\rho[T_{r}]\colon x \in \bigcup_{T\in\tubes'[T_\rho]}Y'(T)\}\lessapprox_\delta \rho^{-\boldsymbol{\beta}}\leq\delta^{-\boldsymbol{\omega}/100}.
\]
Recall $\boldsymbol{\beta} =\boldsymbol{\omega}\boldsymbol{\zeta}/100$ was fixed at the beginning of Section 7.1. 

Finally, for each $T_{\rho}\in\tubes_\rho$, at most $(\rho/\delta)^2$ tubes from $\tubes'[T_\rho]$ can pass through a common point and at most $r^{-2}$ distinct $T_r\in \mathbb{T}_r$ can pass through a common point. We conclude that 
\[
\#\tubes'_{Y'}(x)\lesssim \delta^{-\boldsymbol{\omega}/100}\Big(\frac{\rho}{\delta r}\Big)^2\quad\textrm{for each}\ x\in\RR^3.
\]

In summary, if  $\rho<\delta^{1-\boldsymbol{\omega}/100}$, then Conclusion (A) holds. Otherwise, Conclusion (B) holds. 
\end{proof}


We conclude this section with two results on the union of broad sets of vectors. The first is a straightforward result saying that a union of broad sets is broad. We omit the proof.
\begin{lem}\label{broadnessUnion}
Let $\delta>0$, let $B\subset S^{n-1}$ be a disk, and let $\mathcal{V}_i$, $i=1,\ldots,N$ be sets of unit vectors in $B$. Suppose that each set $\mathcal{V}_i$ is broad with error $K$ at scales $\geq\delta$ inside $B$. Then the multi-set $\bigsqcup_{i=1}^N \mathcal{V}_i$ is broad with error $K$ at scales $\geq\delta$ inside $B$. 
\end{lem}

The second result described how broadness combines across scales.
\begin{lem}\label{broadnessAcrossScales}
Let $0<\delta<\rho\leq 1$. Let $\mathcal{U}$ be a set of unit vectors in $\RR^n$. For each $u\in\mathcal{U}$, let $\mathcal{V}_u\subset B(u,\rho)\subset S^{n-1}$ be contained in the disk of radius $\rho$ centered at $u$. 
Suppose that $\mathcal{U}$ is broad with error $K_1$ at scales $\geq\rho$, and that each set $\mathcal{V}_u$ is broad with error $K_2$ at scales $\geq\delta$ inside $B(u,\rho)$.
Suppose furthermore that each set $\mathcal{V}_u$ has comparable cardinality (up to a factor of 2).
Then the multi-set $\bigsqcup_{u\in\mathcal{U}}\mathcal{V}_u$ is broad with error $O(K_1K_2)$ at scales $\geq\delta$.
\end{lem}
\begin{proof}
Let $\mathcal{V} = \bigsqcup_{u\in\mathcal{U}}\mathcal{V}_u$. By hypothesis, there is a number $M_1$ so that $M_1\leq \#\mathcal{V}_u\leq 2M_1$ for each set $\mathcal{V}_u$. Let $X=\sup_{u_0}\#\{ u\in\mathcal{U}\colon \angle(u,u_0)\leq\rho\}$, where the supremum is taken over all unit vectors $u_0\in S^{n-1}$. Since $\mathcal{U}$ is broad with error $K_1$ at scales $\geq \rho$, if we choose a unit vector $u_0$ achieving the above supremum, then 
\[
X \leq \#\{u\in\mathcal{U}\colon \angle(u,u_0)\leq\rho\}\leq K_1 \rho^{\boldsymbol{\beta}}(\#\mathcal{U}),
\]
and thus $\#\mathcal{U}\geq K_1^{-1}\rho^{-\boldsymbol{\beta}}X$.

First we consider the case where $r\in[\delta,\rho]$. For each unit vector $v_0$, there are at most $O(X)$ vectors $u\in\mathcal{U}$ for which $\{v\in\mathcal{V}_u\colon \angle(v,v_0)\leq r\}$ is non-empty. For each such $u$, we have
\begin{equation*}
\#\{v\in\mathcal{V}_u\colon \angle(v,v_0)\leq r\} 
\lesssim K_2(r/\rho)^{\boldsymbol{\beta}}(2M_1) 
\lesssim K_1K_2r^{\boldsymbol{\beta}}(\#\mathcal{U})X^{-1}M_1.
\end{equation*}
Thus the total contribution from all such $u$ is $\lesssim K_1K_2r^{\boldsymbol{\beta}}(\#\mathcal{V})$, as desired. 

Next we consider the case where $r\in[\rho,1]$. Then 
\begin{equation*}
\#\{v\in\mathcal{V}\colon \angle(v,v_0)\leq r\}
\lesssim  2M_1(\#\{u\in\mathcal{U}\colon \angle(u,v_0)\leq 2r\})
\leq M_1 K_1 r^{\boldsymbol{\beta}}(\#\mathcal{U}) = K_1 r^{\boldsymbol{\beta}}(\#\mathcal{V}).\qedhere
\end{equation*}
\end{proof}


\subsection{Broadness and the Frostman Slab Wolff axioms}
Given a pair $(\tubes,Y)_\delta$, Lemma \ref{findingBroadScale} allows us to find a set $\tubes_\rho$ for which the pairs $(\tubes^{T_\rho},Y^{T_\rho})_{\delta/\rho}$ are broad. The next result says that under suitable hypotheses, $\tubes_\rho$ will factor $\tubes$ with respect to the Frostman Slab Wolff axioms. The precise statement is as follows.

\begin{lem}\label{bigVolOrFactorFrostmanSlabLem}
Suppose that $\cE(\boldsymbol{\sigma},\boldsymbol{\omega})$ is true, and let $\eps>0$. Then there exists $\alpha,\eta,\kappa>0$ so that the following holds for all $0<\delta\leq\rho\leq 1$. Let $(\tubes,Y)_\delta$ be $\delta^{\eta}$ dense, with $\CKT(\tubes)\leq\delta^{-\eta}$ and $\FS(\tubes)\leq\delta^{-\eta}$. Let $\tubes_\rho$ be a balanced cover of $\tubes$, and suppose $(\tubes,Y)_\delta$ is broad with error $\delta^{-\eta}$ relative to the cover $\tubes_\rho$. Then at least one of the following must hold.

\begin{itemize}
\item[(A)]\itemizeEqnVSpacing
\begin{equation}\label{conclusionBigVolumeInFactorFrostmanSlab}
\Big| \bigcup_{T\in\tubes}Y(T) \Big| \geq \kappa \delta^{\boldsymbol{\omega}-\alpha}(\#\tubes)|T|\big((\#\tubes)|T|^{1/2}\big)^{-\boldsymbol{\sigma}}.
\end{equation} 

\item[(B)] There exists a $\delta^{\eps}$ refinement $(\tubes',Y')_\delta$ and a set $\tubes_\rho'\subset\tubes_\rho$, so that $\tubes_\rho'$ factors $\tubes'$ above and below with respect to the Frostman Slab Wolff Axioms with error $\delta^{-\eps}$.
\end{itemize}
\end{lem}

In brief, the proof is as follows. We apply Proposition \ref{slabWolffFactoring} inside each set $\tubes^{T_\rho}$ to cover the tubes in $\tubes[T_\rho]$ by prisms $\mathcal{W}$, with the property that each set $\tubes^W$ satisfies the Frostman Slab Wolff axioms with small error, and the tubes in $\tubes[T_\rho]$ from distinct prisms do not interact (i.e.~their shadings are almost disjoint). Each prism $W$ is contained inside its corresponding $\rho$ tube $T_\rho$. If $W$ is almost as large as $T_\rho$, then $\FS(\tubes^{T_\rho})$ must be almost as small as $\FS(\tubes^{W}),$ which in turn has size about 1. If this happens, then Conclusion (B) holds.

Suppose instead that the prisms are much smaller than $T_\rho$; we will refer to the dimensions of these prisms as $s\times t\times 1$, with $\delta\leq s\leq t$. Since each prism is contained inside a $\rho$ tube, we also have $t\leq\rho$. Thus if the prisms are much smaller than the $\rho$ tubs, then in particular we must have $s<\!\!<\rho$. 

Next, we will make use of the assumption that the tubes are broad relative to the cover $\tubes_\rho$ in order to show that $t\sim\rho$. Since $s<\!\!<\rho$, this means that the prisms are flat. Specifically, broadness ensures that a typical pair of tubes passing through a common point $x\in\bigcup_{T\in\tubes[T_\rho]}Y(T)$ make angle roughly $\rho$. Such a pair of tubes is contained in a common prism $W$, from which it follows that $W$ must have ``width'' $\rho$, i.e.~ each prism $W$ has dimensions roughly $s\times \rho\times 1$. 

To summarize, at this point in the argument each $\delta$ tube is contained inside a flat prism $W$ of dimensions roughly $s\times \rho \times 1$ with $s<\!\!<\rho$. Furthermore, at a typical point $x\in Y(T)$, there is at least one other tube from $\tubes$ contained inside the same flat prism $W$, and this second tube intersects $T$ at angle roughly $\rho$. Thus the hairbrush of $T$ (i.e. the union of set of tubes intersecting $T$), when restricted to the $\frac{\rho}{s}\delta$ neighbourhood of $T$, fills out a rectangular slab of dimensions roughly $\delta\times \frac{\rho}{s}\delta\times 1$. But this is precisely the setting where Lemma \ref{inflateTubesToSlabsLem} asserts that $\bigcup_{\tubes}Y(T)$ is larger than we would expect from the estimate $\cE(\sigma,\omega)$, and thus Conclusion (A) holds. We now turn to the details.

\begin{proof}[Proof of Lemma \ref{bigVolOrFactorFrostmanSlabLem}]$\phantom{1}$\\
\noindent{\bf Step 1.} 
We may suppose that $\rho\leq\delta^{\eps/10}$, or else Conclusion (B) follows by selecting a single $\rho$-tube that contains at least $\rho^4(\#\tubes)$ tubes from $\tubes$. We may also suppose that $\rho \geq \delta^{\eps/10}$, or else Conclusion (B) holds trivially by taking $\tubes'=\tubes$ and $\tubes_\rho'=\tubes_\rho$. 

By pigeonholing and replacing $\tubes_\rho$ by a subset $\tubes_{\rho,1}$, we can find a $\approx_\delta 1$ refinement $(\tubes_1,Y_1)_\delta$ of $(\tubes,Y)_\delta$ so that $\tubes_{\rho,1}$ is a balanced cover of $\tubes_1$, and furthermore each set $(\tubes_1^{T_\rho},Y^\rho)_{\delta/\rho}$ is $\gtrapprox_\delta\delta^{\eta}\gtrapprox_\rho\rho^{10\eta/\eps}$ dense.

Let $\eps_1=\eps{\boldsymbol{\beta}}/1600$. Apply Proposition \ref{slabWolffFactoring} with $\eps_1$ in place of $\eps$ to each set $\tubes_1^{T_\rho}$; this gives us a collection of convex sets $\mathcal{W}_{T_\rho}$ that factors  a $(\delta/\rho)^{\eps_1}$-fraction of $\tubes_1^{T_\rho}$  (abusing notation, we continue to use $\tubes_1^{T_\rho}$ to denote this $(\delta/\rho)^{\eps_1}$-fraction) from below with respect to the Frostman Slab Wolff axioms with error $(\delta/\rho)^{-\eps_1}\leq \delta^{-\eps_1}$.  We may do this, provided $\eta>0$ is selected sufficiently small depending on $\eps$ and $\eps_1$. 

After further pigeonholing (which induces a further $\approx_\delta 1$ refinement $(\tubes_2,Y_2)_\delta$ of $(\tubes_1,Y_1)$ and replaces $\tubes_{\rho,1}$ by a subset $\tubes_{\rho,2}$), we may suppose that the convex sets in $\phi_{T_\rho}^{-1}(\mathcal{W}_{T_\rho})$ all have common dimensions (up to a factor of 2) for each $T_\rho\in\tubes_{\rho,2}$; call these dimensions $s\times t\times 1$. We may also suppose that the value of $\FS(\tubes_2^{T_\rho})$ is the same (up to a factor of 2) for all $T_\rho\in\tubes_{\rho,2}$, and that the mass $\sum_{T\in\tubes[T_\rho]}|Y_2(T)| $ is the same (up to a factor of 2) for each $T_\rho\in\tubes_{\rho,2}$.

Let $\mathcal{W} = \bigcup_{T_\rho\in\tubes_\rho}\phi_{T_\rho}^{-1}(\mathcal{W}_{T_\rho})$. To summarize, the situation is as follows:
\begin{itemize}
\item[(i)] We have $\sum_{T\in\tubes_2[T_\rho]}|Y_2(T)|\gtrapprox_\delta \delta^{\eps_1}\sum_{T\in\tubes[T_\rho]}|Y(T)|$ for each $T_\rho\in\tubes_{\rho,2}$.

\item[(ii)] We have a collection $\mathcal{W}$ of convex sets of dimensions $s\times t\times 1$, with $\tubes_2 \prec \mathcal{W} \prec \tubes_\rho$. 

\item[(iii)] For each $T_\rho\in\tubes_{\rho,2}$, the sets $\bigcup_{\tubes_2[W]}Y_2(T),\ W\in\mathcal{W}[T_\rho]$ are disjoint. 

\item[(iv)] For each $W\in\mathcal{W}$, we have $\FS(\tubes_2^W)\leq\delta^{-\eps_1}.$
\end{itemize}

By Items (ii) and (iv), for each $T_\rho\in\tubes_\rho$ we have
\begin{equation}\label{crudeEstimateFSTubes2TRho}
\FS(\tubes_2^{T_\rho})\leq \Big(\sup_{W\in\mathcal{W}[T_\rho]}\FS(\tubes_2^W)\Big)\Big(\frac{|T_\rho|}{|W|}\Big)^{100} \leq \delta^{-\eps/2}\Big(\frac{|T_\rho|}{|W|}\Big)^{100}.
\end{equation}
The last term in the above inequality is an (intentionally) crude estimate for the number of essentially distinct $s\times t\times 1$ prisms that can fit inside a $\rho$-tube. 

If $\Big(\frac{|T_\rho|}{|W|}\Big)^{100}\leq \delta^{-\eps/2}$, then Conclusion (B) holds with $(\tubes',Y')_{\delta} = (\tubes_2,Y_2)$ and $\tubes_\rho = \tubes_{\rho,2}$ and we are done. 

\medskip

\noindent{\bf Step 2.} 
We shall suppose henceforth that
\begin{equation}\label{WSignificantlySmallerTRho}
|W|\leq \delta^{\eps/200}|T_\rho|.
\end{equation}
Our goal is to prove the \eqref{conclusionBigVolumeInFactorFrostmanSlab} holds for an appropriate choice of $\alpha$.

Each prism in $\mathcal{W}$ has dimensions $s\times t\times 1$, with $t\leq\rho$. We claim the reverse inequality is almost true. Specifically, we have
\begin{equation}\label{tAlmostAsBigRho}
t\gtrapprox_\delta \delta^{\frac{2\eps_1}{{\boldsymbol{\beta}}}}\rho.
\end{equation}
Recall from Definition \ref{broadnessDefnShading} that for each $T_\rho\in\tubes_{\rho,2}$ and each $x\in\bigcup_{T\in\tubes[T_\rho]}Y(T)$, we have 
\begin{equation}\label{angularNonConcentration}
\#\{ T\in\tubes[T_\rho] \colon x\in Y(T),\ \angle(v,\dir(T))\leq r\}\leq \delta^{-\eta} (\frac{r}{\rho})^{\boldsymbol{\beta}} \#\{ T\in\tubes[T_\rho] \colon x\in Y(T)\},\quad v\in S^2,\ r\geq\delta.
\end{equation}
By Item (i) above, there exists at least one point $x\in\RR^3$ for which 
\begin{equation}
\label{pointNotMuchLost}
\#\{T\in\tubes_2[T_\rho]\colon x\in Y_2(T)\}\gtrapprox_\delta \delta^{\eps_1} \#\{T\in\tubes[T_\rho]\colon x\in Y(T)\}>0,
\end{equation}
and hence for this choice of $x$ we have
\begin{equation}\label{angularNonConcentrationInsideT2}
\#\{ T\in\tubes_2[T_\rho] \colon x\in Y_2(T),\ \angle(v,\dir(T))\leq r\}\leq \delta^{-\eta-\eps_1} (\frac{r}{\rho})^{\boldsymbol{\beta}} \#\{ T\in\tubes_2[T_\rho] \colon x\in Y_2(T)\},\quad v\in S^2,\ r\geq\delta.
\end{equation}

On the other hand, by Item (iii) above, the tubes $\{T\in\tubes_2[T_\rho]\colon x\in Y_2(T)\}$ are all contained in a common prism $W\in \mathcal{W}$, and thus must all make angle $\leq 2t$ with the direction $v$ of this prism. Selecting $r=2t$ in \eqref{angularNonConcentrationInsideT2} and comparing with  (iii),  
 we obtain \eqref{tAlmostAsBigRho} (provided we select $\eta\leq\eps_1$).  

Comparing \eqref{WSignificantlySmallerTRho} and \eqref{tAlmostAsBigRho}, we see that the prisms in $\mathcal{W}$ must be flat, i.e.
\begin{equation}\label{boundSvsT}
s \sim t^{-1}|W| \lesssim t^{-1} \delta^{\eps/200}\rho^2 \lessapprox_\delta \delta^{\frac{\eps}{200}-\frac{4\eps_1}{\boldsymbol{\beta}}}t \leq \delta^{\frac{\eps}{400}}t.
\end{equation}

\medskip

\noindent{\bf Step 3.} 
Apply Lemma \ref{everySetHasARegularShadingLem} to replace each shading $Y_2(T),\ T\in\tubes_2$ with a regular sub-shading $Y_3(T)\subset Y_2(T)$. Define $\tubes_3=\tubes_2$. We say a point $x\in\bigcup_{T\in\tubes}Y(T)$ has \emph{survived} if 
\[
\#\{T\in\tubes_3\colon x\in Y_3(T)\}\geq \kappa_0\delta^{2\eps_1} \#\{T\in\tubes\colon x\in Y(T)\}.
\] 
Let $Y_4(T)\subset Y_3(T)$ consist of surviving points and let $\tubes_4=\tubes_3$; we will choose the constant $\kappa_0$ sufficiently small so that $(\tubes_4,Y_4)_{\delta}$ is a $\gtrapprox_\delta\delta^{\eps_1}$ refinement of $(\tubes_3,Y_3)$. 

Observe that if the point $x$ has survived, then there is a unique prism $W\in\mathcal{W}$ with $x\in \bigcup_{T\in\tubes_3[W]}Y(T)$, and at least two tubes $T,T'\in \tubes_3[W]$ with $x\in Y_3(T),\ x\in Y_3(T')$ with $\angle(\dir(T),\dir(T'))\gtrapprox_\delta \delta^{\frac{2\eps_1}{\boldsymbol{\beta}}}\rho$. 

For each $T\in\tubes_4$, let $S(T)\supset T$ be the $\delta\times \frac{t}{s}\delta \times 1$ prism with coaxial line $T$, and plane parallel to $\Pi(W)$, where $W\in\mathcal{W}$ is the unique prism covering $T$. Recall that by \eqref{boundSvsT}, we have $\frac{t}{s}\delta\geq \delta^{1-\frac{\eps}{400}}.$ 
A Cordoba style $L^2$ argument (see i.e.~Lemma \ref{cordobaLem} and its proof for an example of a similar argument) shows that for each $T\in\tubes_4$, 
\[
\Big|S(T)\cap\bigcup_{T\in\tubes_3}Y_3(T)\Big|\gtrapprox_\delta \delta^{4\eps_1}\frac{|Y_4(T)|}{|T|}|S|.
\]
Let $b = \frac{t}{s}\delta\geq \delta^{1-\eps/400}$. Applying Lemma \ref{inflateTubesToSlabsLem}, we obtain Conclusion (A) for $\alpha = \eps {\boldsymbol{\omega}}/500$, provided we select $\eta>0$ sufficiently small.
\end{proof}


\subsection{The iteration base case: Guth's grains decomposition}
In our proof sketch from Section \ref{twoScaleGrainsDecompIntro}, we described a single-scale grains decomposition due to Guth. In this section we will state the result precisely. 

\begin{prop}\label{GuthGrainsProp}
Let $\eps>0$. Then there exists $\eta,\kappa>0$ so that the following holds for all $\delta>0$. Let $(\tubes,Y)_\delta$ be $\delta^\eta$ dense and be broad with error $\delta^{-\eta}$. Suppose that the tubes in $\tubes$ are contained in a common 1 tube $T_1$.

Then there is a $\delta^\eps$ refinement $(\tubes',Y')$ that is $\delta^\eps$ dense and is broad with error $\leq\kappa^{-1}\delta^{-\eps}$, and a number $\mu\geq 1$ so that $\mu\sim \#\tubes'_{Y'}(x)$ for each $x\in \bigcup_{\tubes'} Y'(T)$. In addition, there is a number $c\geq  \kappa \mu  \delta^\eps (\delta\#\tubes)^{-1}$;  and a pair $(\mathcal{G},Y)_{\delta\times c\times c}$, so that $(\mathcal{G},Y)_{\delta\times c\times c}$ is a robustly $\delta^{\eps}$-dense two-scale grains decomposition of $(\tubes',Y')_\delta$ wrt $\{T_1\}$ (the latter is a set consisting of a single 1 tube).
\end{prop}

\begin{rem}\label{remarkLongEndsExitSquareGrain}
Proposition \ref{GuthGrainsProp} says that $(\mathcal{G},Y)_{\delta\times c\times c}$  is a two-scale grains decomposition of $(\tubes',Y')_\delta$ wrt $\{T_1\}$. A two-scale grains decomposition is defined in Definition \ref{twoScaleGrainsDecomp}, and Item (iv) from that definition specifies that if $Y'(T)\cap Y(G)\neq\emptyset$, then $T$ exits $G$ through its long ends, in the sense of Figure \ref{exitLongEndsFig}. Since the grains $(\mathcal{G},Y)_{\delta\times c\times c}$ from Proposition \ref{GuthGrainsProp} are square, the definition is somewhat ambiguous in this setting. However, Proposition \ref{GuthGrainsProp} is only used to prove Corollary \ref{grainsDecompFromBroadCor}. Thus the 1 tube $T_1$ should be thought of as the anisotropic rescaling of a $\rho$ tube $T_\rho$. What is needed is the following: the images of the tubes in $\tubes'$ under the anisotropic scaling sending $T_1$ to $T_\rho$ must exit the images of the grains in $\mathcal{G}$ through their long ends.
\end{rem}

Proposition \ref{GuthGrainsProp} is a variant of Guth's grains decomposition from \cite{Gut14}. Since this precise statement does not appear in \cite{Gut14} (the hypotheses in \cite{Gut14} are stated slightly differently), we will provide a proof in Appendix \ref{guthGrainsDecompAppendix}.

If $(\tubes,Y)_\delta$ is broad relative to a set of $\rho$ tubes $\tubes_\rho$, then we can apply Proposition \ref{GuthGrainsProp} to each re-scaled set $(\tubes^{T_\rho},Y^{T_\rho})_\delta$.



\begin{cor}\label{grainsDecompFromBroadCor}
Let $\eps>0$. Then there exists $\kappa,\eta>0$ so that the following holds for all $\delta>0$.
Let $(\tubes,Y)_\delta$ be $\delta^\eta$ dense and let $\tubes_\rho$ be a balanced partitioning cover of $\tubes$. Suppose that $(\tubes,Y)_\delta$ is broad with error $\delta^{-\eta}$ relative to the cover $\tubes_\rho$, and that $\big|\bigcup_{T\in\tubes}Y(T)|\leq\delta^{\eps}(\#\tubes)|T|$.

Then there is a $\delta^\eps$ refinement $(\tubes',Y')_\delta$ of $(\tubes,Y)_\delta$ that is $\delta^\eps$ dense; a subset $\tubes_\rho'\subset \tubes_\rho$; a number $c\gtrsim \frac{\rho}{\delta}\frac{\#\tubes_\rho}{\#\tubes}$; and a pair $(\mathcal{P},Y)_{\delta \times b\times c}$ with $\rho=b/c$ that is a robustly $\delta^\eps$-dense two-scale grains decomposition of $(\tubes',Y')_\delta$ wrt $\tubes_\rho'$. Finally, $(\tubes',Y')_\delta$ is broad with error $\leq \kappa^{-1}\delta^{-\eps}$ relative to the cover $\tubes_\rho'$.
\end{cor}

Note that the estimate on the size of $c$ in Corollary \ref{grainsDecompFromBroadCor} omits the term $\mu$ (though this term is used to compensate for the $\delta^{\eps}$ loss in Proposition \ref{GuthGrainsProp}); this is because the weaker estimate  $c\gtrsim \frac{\rho}{\delta}\frac{\#\tubes_\rho}{\#\tubes}$ will be sufficient in the arguments that follow.

Corollary \ref{grainsDecompFromBroadCor} has two important consequences. First, when combined with Lemma \ref{findingBroadScale}, it says that if $(\tubes,Y)_\delta$ is an arrangement for which $\cE(\boldsymbol{\sigma},\boldsymbol{\omega})$ is tight, then $(\tubes,Y)_\delta$ admits a two-scale grains decomposition. In Section \ref{twoScaleGrainsDecompIntro}, we called this the ``Guth grains decomposition'' of $\tubes$.  The precise statement is as follows

\begin{lem}\label{existsTwoScaleGrainsLem}
Suppose that $\cE(\boldsymbol{\sigma},\boldsymbol{\omega})$ is true and let $\eps>0$. Then there exists $\alpha,\eta,\kappa>0$ so that the following holds for all $\delta>0$. Let $(\tubes,Y)_\delta$ be $\delta^{\eta}$ dense, with $\CKT(\tubes)\leq\delta^{-\eta}$ and $\FS(\tubes)\leq\delta^{-\eta}$. Then at least one of the following must hold.

\begin{itemize}
\item[(A)]\itemizeEqnVSpacing
\begin{equation}
\Big| \bigcup_{T\in\tubes}Y(T) \Big| \geq \kappa \delta^{\boldsymbol{\omega}-\alpha}(\#\tubes)|T|\big((\#\tubes)|T|^{1/2}\big)^{-\boldsymbol{\sigma}}.
\end{equation} 

\item[(C)] There exist the following:
\begin{itemize}
	\item A scale $\rho\in [\delta^{1-\boldsymbol{\omega}/100},\delta^{\boldsymbol{\omega}/100}]$.
	\item A $\delta^\eps$ refinement $(\tubes',Y')_\delta$ of $(\tubes,Y)_\delta$.
	\item A balanced partitioning cover $\tubes_\rho$ of $\tubes'$.
	\item Numbers $\delta\leq b\leq c$ with $b/c=\rho$ and $c\gtrsim \frac{\rho}{\delta}\frac{\#\tubes_\rho}{\#\tubes}$.
	\item A pair $(\mathcal{P},Y)_{\delta \times b\times c}$.
\end{itemize}
So that the following holds:
\begin{itemize}
	\item[(i)] $(\tubes',Y')_\delta$ is broad with error $\leq\delta^{-\eps}$ relative to the cover $\tubes_\rho$.
	\item[(ii)]$(\mathcal{P},Y)_{\delta \times b\times c}$ is a robustly $\delta^\eps$-dense two-scale grains decomposition of $(\tubes',Y')_\delta$ wrt $\tubes_\rho$.
\end{itemize}
\end{itemize}
\end{lem}
\begin{rem}
Note that the Conclusions of Lemma \ref{existsTwoScaleGrainsLem} are labelled (A) and (C), rather than (A) and (B). We chose this convention in order to have parallelism with Conclusions (A), (B), and (C) of Moves \#1, \#2, and \#3 below.
\end{rem}

Lemma \ref{existsTwoScaleGrainsLem} will serve as the starting point for the iterative process described in Section \ref{twoScaleGrainsDecompIntro}. In the following subsections, we will describe the three Moves in this iterative process.


\subsection{Moves \#1, \#2, \#3: Parallel structure}\label{movesParallelStructureSec}

In the following sections, we will describe three Moves, which we will iteratively apply to the two-scale grains decomposition that we obtained from Lemma \ref{existsTwoScaleGrainsLem}. Each of these Moves are expressed as a lemma, and these three lemmas have similar structure. In particular, the three lemmas have the same hypotheses, and have similar conclusions. 

\medskip

\noindent\fbox{%
    \parbox{\textwidth}{%
    {\bf Common setup for Moves \#1, \#2, \#3: Hypotheses}\\
        Suppose that $\cE(\boldsymbol{\sigma},\boldsymbol{\omega})$ is true and let $\eps\in (0, \boldsymbol{\zeta}/2]$. Then there exists $\alpha,\eta,\kappa>0$ so that the following holds for all $0<\delta\leq 1$, $\rho\geq \delta^{1-\boldsymbol{\omega}/100}$, and all $\delta\leq a\leq b\leq c$ with $b/c=\rho$. 

		Let $(\tubes,Y)_\delta$ be $\delta^\eta$ dense, with $\CKT(\tubes)\leq\delta^{-\eta}$ and $\FS(\tubes)\leq\delta^{-\eta}$. Let $\tubes_\rho$ be a balanced partitioning cover of $\tubes$, and suppose that $(\tubes,Y)_\delta$ is broad with error $\delta^{-\eta}$ relative to $\tubes_\rho$. Let $(\mathcal{P},Y)_{a \times b\times c}$ be a robustly $\delta^\eta$-dense two-scale grains decomposition of $(\tubes,Y)_\delta$ wrt $\tubes_\rho$.
    }%
}

Each of Moves \#1, \#2, and \#3 will have three possible conclusions, which we label (A), (B), and (C). Conclusion (A) is the same for all three moves.\\

\medskip  
\noindent\fbox{%
    \parbox{\textwidth}{%
    \noindent {\bf Common setup for Moves \#1, \#2, \#3: Conclusion (A).}
		\begin{equation}\label{volumeLowerBoundConclusionA}
		\Big| \bigcup_{T\in\tubes}Y(T) \Big| \geq \kappa \delta^{\boldsymbol{\omega}-\alpha}(\#\tubes)|T|\big((\#\tubes)|T|^{1/2}\big)^{-\boldsymbol{\sigma}}.
		\end{equation} 
    }%
}
In Inequality \eqref{volumeLowerBoundConclusionA}, $\kappa,\alpha>0$ are the quantities from the {\bf Common setup for Moves \#1, \#2, \#3: Hypotheses} described above.

\medskip
\noindent Conclusion (B) is not identical for the three Moves, but shares many common elements. We describe these below

\medskip  

\noindent\fbox{%
    \parbox{\textwidth}{%

    \noindent {\bf Common setup for Moves \#1, \#2, \#3: Conclusion (B).}

There are $\delta^{\eps}$ refinements $(\tubes',Y')_\delta$ and $(\mathcal{P}',Y')_{a \times b\times c}$  of $(\tubes,Y)_\delta$ and $(\mathcal{P},Y)_{a \times b\times c}$, respectively, and a set $\tubes'_\rho\subset\tubes_\rho$, so that the following holds.
\begin{itemize}
	\item[(i)] $(\tubes',Y')_\delta$ is $\delta^\eps$ dense, $\CKT(\tubes')\leq\delta^{-\eps}$, and $\FS(\tubes')\leq\delta^{-\eps}$. 
	\item[(ii)] $\tubes'_\rho$ is a balanced partitioning cover of $\tubes'$, and $(\tubes',Y')_\delta$ is broad with error $\delta^{-\eps}$ relative to $\tubes'_{\rho}$.
	\item[(iii)] $(\mathcal{P}',Y')_{a \times b\times c}$ is a robustly $\delta^{\eps}$-dense two-scale grains decomposition of $(\tubes',Y')_\delta$ wrt $\tubes'_\rho$.
	\item[(iv)] Moves \#1, \#2, \#3 will have additional conclusions specific to that Move.
\end{itemize}
    }%
}

\medskip

\noindent Conclusion (C) is not identical for the three Moves, but shares many common elements. We describe these below

\medskip  

\noindent\fbox{%
    \parbox{\textwidth}{%

\noindent   {\bf Common setup for Moves \#1, \#2, \#3: Conclusion (C)}.\\
There exist the following:
\begin{itemize} 
	\item Numbers $\tilde\rho$ and $\tilde a\leq\tilde b\leq\tilde c\leq 1$, with $\tilde \rho=\tilde b / \tilde c$.
	\item A $\delta^{\eps}$ refinement $(\tubes',Y')_\delta$ of $(\tubes,Y)_\delta$.
	\item A set $\tubes_{\tilde \rho}$ of $\tilde\rho$ tubes.
	\item A pair $(\tilde{\mathcal{P}},\tilde Y)_{\tilde a \times \tilde b\times \tilde c}$.
	\end{itemize}
These objects have the following properties:
	\begin{itemize}
	\item[(i)] $(\tubes',Y')_\delta$  is $\delta^{\eps}$ dense, $\CKT(\tubes')\leq\delta^{-\eps}$, and $\FS(\tubes')\leq\delta^{-\eps}$. 
	\item[(ii)] $\tubes_{\tilde \rho}$ is a balanced partitioning cover of $\tubes'$, and $(\tubes',Y')_\delta$ is broad with error $\delta^{-\eps}$ relative to $\tubes_{\tilde \rho}$.
	\item[(iii)] $(\tilde{\mathcal{P}}, \tilde Y)_{\tilde a \times \tilde b\times \tilde c}$ is a robustly $\delta^\eps$-dense two-scale grains decomposition of $(\tubes',Y')_\delta$ wrt $\tubes_{\tilde \rho}$.
	\item[(iv)] $\delta^{1-\boldsymbol{\omega}/100}\leq \tilde\rho\leq 1$.
	\item[(v)] Moves \#1, \#2, \#3 will have additional conclusions specific to that Move.
	\end{itemize}
    }%
}

\begin{rem}\label{outputMatchesInput}
Observe that Items (i), (ii), and (iii) of Conclusion (B) (resp.~Items (i) -- (iv) of Conclusion (C)) say that the output $(\tubes',Y')_\delta$; $\tubes_{\rho}'$; and $(\mathcal{P}', Y')_{\delta \times b\times c}$ (resp.~$\tubes_{\tilde\rho}$ and $ (\tilde{\mathcal{P}},\tilde{Y})_{\tilde a\times \tilde  b\times  \tilde c}$) of Moves \#1, \#2, and \#3 match the ``input'' hypotheses (i.e. the {\bf Common setup for Moves \#1, \#2, \#3: Hypotheses}), except that exponent $\eta$ has been weakened to $\eps$. This will allow us to iteratively apply these Moves many times.
\end{rem}


\subsection{Using Moves \#1, \#2, \#3 to prove Proposition \ref{grainsDecomposition}}
We will state and prove Moves \#1, \#2, and \#3 in Section \ref{moves123Sec}. However, by using the parallel structure described above we can already state the hypotheses and conclusions of these three moves. We do so in the table below.

\begin{center}
\begin{tabular}{ | m{0.9cm} | m{1.9cm} | m{6cm}| m{6cm}| } 
 \hline
{\bf Move} & {\bf Lemma} & {\bf Conclusion (B), Item (iv)} & {\bf Conclusion (C), Item (v)} \\ 
\hline
\#1 & \ref{replaceGrainsWLongerEnsureLargeC} & $c \geq \delta^{\boldsymbol{\zeta}} (\rho/\delta)\ (\#\tubes_\rho / \#\tubes_\delta)$ & $\tilde c\geq \delta^{-\boldsymbol{\zeta}}c$, $\tilde a=\delta,$ and $\tilde\rho=\rho$ \\
\hline
\#2 & \ref{squareGrainsGetLonger}  & $\rho\leq \delta^{\boldsymbol{\omega}/100}$ & $\tilde c\geq \delta^{-\boldsymbol{\omega}/100}c $ \\
\hline
\#3 & \ref{WiderGrainsSmallCKT}  & $\CKT^{\operatorname{loc}}(\mathcal{P}')\leq \delta^{-\boldsymbol{\zeta}}$ & $\tilde\rho\geq \delta^{-\boldsymbol{\zeta}/1000}\rho $ and $\tilde c\geq c$ \\
\hline
\end{tabular}
\end{center}

We will now match the Conclusions of Proposition \ref{grainsDecomposition} to their counterparts from Moves \#1, \#2, \#3.

\begin{itemize} 
	\item Conclusion (A) for Moves \#1, \#2, \#3 matches Conclusion (A) of Proposition \ref{grainsDecomposition}.
	\item Conclusion (B), Items (i) and (iii) for Moves \#1, \#2, \#3 matches Conclusion (B), Items (i) and (iii), respectively, of  Proposition \ref{grainsDecomposition}.

	\item Conclusion (B), Item (ii) for Moves \#1, \#2, \#3, plus Lemma \ref{bigVolOrFactorFrostmanSlabLem} either yields Conclusion (A), or Conclusion (B), Item (ii) of  Proposition \ref{grainsDecomposition}.

	\item Conclusion (B), Item (iv) for Move \#1 matches Conclusion (B), Item (iv) of Proposition \ref{grainsDecomposition}.
	\item Conclusion (B), Item (iv) for Move \#2, plus the hypothesis $\rho\geq\delta^{1-\boldsymbol{\omega}/100}$, matches Conclusion (B), Item (v) of  Proposition \ref{grainsDecomposition}.
	\item Conclusion (B), Item (v) for Move \#3 matches Conclusion (B), Item (vi) of  Proposition \ref{grainsDecomposition}.
\end{itemize}

\begin{proof}[Proof of Proposition~\ref{grainsDecomposition}]
We proceed as follows. Let $\eta_0<\eta_1<\ldots<\eta_N$ be a sequence of numbers that we will determine below. Let $(\tubes,Y)_\delta$ be $\delta^{\eta_0}$ dense, with $\CKT(\tubes)\leq\delta^{-\eta_0}$ and $\FS(\tubes)\leq\delta^{-\eta_0}$. If $\eta_0$ is sufficiently small depending on $\eta_1$, then we can apply the iteration base case, Lemma \ref{existsTwoScaleGrainsLem}, with $\eta_1$ in place of $\eps$. If Conclusion (A) of Lemma \ref{existsTwoScaleGrainsLem} holds, then Conclusion (A) of Proposition \ref{grainsDecomposition} holds, and we are done.

Suppose instead that Conclusion (B) of Lemma \ref{existsTwoScaleGrainsLem} holds. The output of Conclusion (B) is precisely the set of objects required for the {\bf Common setup for Moves \#1, \#2, \#3: Hypotheses}, with $\eta_1$ in place of $\eta$. 

We now repeatedly apply Moves \#1, \#2, \#3, with $\eta_{j+1}$ in place of $\eps$ at stage $j$. We may do so, provided $\eta_j$ is sufficiently small compared to $\eta_{j+1}$. 

\begin{itemize}
	\item If Conclusion (A) occurs at any point, then Conclusion (A) of Proposition \ref{grainsDecomposition} holds, and we halt. 
	\item If Conclusion (B) occurs for Move \# i, for some $i=1,2,3$,  then we switch to a different move. 
	\item If Conclusion (B) occurs for all three moves in succession, then Items (i), (iii), and (iv) of Conclusion (B) of Proposition \ref{grainsDecomposition} hold. We halt and apply Lemma~\ref{bigVolOrFactorFrostmanSlabLem} to show either Conclusion (B), Item (ii) holds, or else Conclusion (A) holds. We conclude that at least one of Conclusion (A) or Conclusion (B) of Proposition \ref{grainsDecomposition} holds.
	\item If Conclusion (C) occurs for Move \# i, then at least one of the following must occur:
		\begin{itemize}
			\item $c$ does not decrease, and $\rho$ becomes larger by $\delta^{-\boldsymbol{\zeta}/1000}$; this can occur at most $1000/\boldsymbol{\zeta}$ times in a row.
			\item $c$ becomes larger by $\min\{\delta^{-\boldsymbol{\zeta}}, \delta^{-\boldsymbol{\omega}/100}\}$. This can occur at most $1/\boldsymbol{\zeta}+100/ \boldsymbol{\omega}$ times total. 
		\end{itemize}
\end{itemize}
In particular, the iterative process described above must halt after at most $N = \frac{1000}{\boldsymbol{\zeta}}\big(\frac{2}{\boldsymbol{\zeta}}+\frac{100}{\boldsymbol{\omega}}\big)$ steps. We will choose  $\eta_{N+1}$ below, and then select each of $\eta_N,\eta_{N-1},\ldots,\eta_0$ in turn. Finally, we define $\eta$ (the quantity from Proposition \ref{grainsDecomposition} ) by $\eta  = \eta_0$. 

The refinement $(\tubes',Y')_\delta$, the set $\tubes_{\rho}$, and the pair $(\mathcal{G},Y)_{a\times b\times c}$ satisfy all of the conclusions of Proposition \ref{grainsDecomposition}, Conclusion (B), except that for Item (ii), we have not yet shown that $\tubes_\rho$ factors $\tubes'$ from below with respect to the Frostman Slab Wolff Axioms with error $\delta^{-\zeta}$. However, since $(\tubes',Y')_\delta$ is broad with respect to $\tubes_\rho$, we can apply Lemma \ref{bigVolOrFactorFrostmanSlabLem} and conclude that either either this is indeed the case (after replacing $(\tubes',Y')_\delta$, $\tubes_\rho$, and $(\mathcal{G},Y)_{a\times b\times c}$  by a suitable refinement), or else Conclusion (A) of Proposition \ref{grainsDecomposition} holds.

This completes the proof of Proposition \ref{grainsDecomposition}, except that we have not proved Moves \#1, \#2, \#3. We shall do so in the next section. 
\end{proof}


\section{Moves \#1, \#2, and \#3}\label{moves123Sec}

\subsection{Move \#1: Replacing grains with longer grains to ensure $c\geq \delta^{\boldsymbol{\zeta}}\frac{\rho}{\delta}(\#\tubes_\rho)/(\#\tubes)$}
Our goal in this section is to state and prove Move \#1, as described in Section \ref{twoScaleGrainsDecompIntro} and Section \ref{movesParallelStructureSec}.

\begin{lem}\label{replaceGrainsWLongerEnsureLargeC}
We assume the {\bf Common setup for Moves \#1, \#2, \#3: Hypotheses} from Section \ref{movesParallelStructureSec}. Then at least one of the following must hold.

\begin{itemize}
	\item[(A)] Conclusion (A) of the common setup for Moves \#1, \#2, \#3.
	\item[(B)] Conclusion (B) of the common setup for Moves \#1, \#2, \#3. In addition, 
		\begin{itemize}
			\item[(iv)]
				\itemizeEqnVSpacing
				\begin{equation}
				 c \geq \delta^{\boldsymbol{\zeta}} \frac{\rho}{\delta}\frac{\#\tubes_\rho}{\#\tubes_\delta}.
				\end{equation}
		\end{itemize}

	\item[(C)] Conclusion (C) of the common setup for Moves \#1, \#2, \#3. In addition, 
	\begin{itemize}
		\item[(v)] $\tilde a=\delta,$ $\tilde\rho=\rho$, and $\tilde c\geq \delta^{-\boldsymbol{\zeta}}c$.
	\end{itemize}
\end{itemize}
\end{lem}

\begin{proof}
If Conclusion (B) fails, then we discard the cover $(\mathcal{P},Y)_{a\times b\times c}$ and replace it with the cover coming from Corollary \ref{grainsDecompFromBroadCor} applied to $(\tubes,Y)_\delta$ and $\tubes_\rho$. Conclusion (C) of Lemma \ref{replaceGrainsWLongerEnsureLargeC} was chosen to match the output of Corollary \ref{grainsDecompFromBroadCor}. Note that the prisms coming from Corollary \ref{grainsDecompFromBroadCor} have length $\tilde c\geq \frac{\rho}{\delta}\frac{\#\tubes_\rho}{\#\tubes}$. If Conclusion (B) fails, then this quantity is $\geq \delta^{-\boldsymbol{\zeta}}c$, as claimed. 
\end{proof}


\subsection{Move \#2: Replacing square grains with longer grains}
In this section we will use geometric arguments in the spirit of Cordoba's proof of the Kakeya maximal function conjecture in $\RR^2$ to show the following: if $(\tubes,Y)_\delta$ has a two-scale grains decomposition consisting of square grains, i.e.~grains of dimensions $a\times b\times c$ with $\rho = b/c> \delta^{\boldsymbol{\omega}/100}$, then either $\bigcup Y(T)$ is large, or else we can construct a new two-scale grains decomposition of $(\tubes,Y)$ with significantly longer grains. The precise statement is as follows.

\begin{lem}\label{squareGrainsGetLonger}
We assume the {\bf Common setup for Moves \#1, \#2, \#3: Hypotheses} from Section \ref{movesParallelStructureSec}. Then at least one of the following must hold.

\begin{itemize}
\item[(A)] Conclusion (A) of the common setup for Moves \#1, \#2, \#3.
\item[(B)] Conclusion (B) of the common setup for Moves \#1, \#2, \#3. In addition, 
		\begin{itemize}
			\item[(iv)] $\rho\leq \delta^{\boldsymbol{\omega}/100}$. 
		\end{itemize}

\item[(C)] Conclusion (C) of the common setup for Moves \#1, \#2, \#3. In addition, 
	\begin{itemize}
		\item[(v)] $\tilde c\geq \delta^{-\boldsymbol{\omega}/100}c$.
	\end{itemize}

\end{itemize}
\end{lem}

\begin{proof}$\phantom{1}$\\
\noindent{\bf Step 1.} 
Let $0<\eps_1<\eps_2$ be small quantities to be chosen below. We will choose $\eps_1$ very small compared to $\eps_{2}$; we will choose $\eps_2$ very small compared to $\eps$; we will choose $\alpha,\eta$ very small compared to $\eps_1$. 

First, we claim that either
\begin{equation}\label{aAlmostDelta}
a\leq \delta^{1-\eps_1},
\end{equation}
or else Conclusion (A) immediately holds, provided we choose $\alpha$ and $\eta$ sufficiently small depending on $\eps_1$. 

We verify this claim as follows. Suppose that \eqref{aAlmostDelta} fails. Then after replacing $(\tubes,Y)_\delta$ by an $\approx_\delta 1$ refinement, we have that for each $x\in\bigcup_{T\in\tubes}Y(T)$,
\[
\Big|B(x,\delta^{1-\eps_1})\cap\bigcup_{T\in\tubes}Y(T)\Big|\gtrsim \delta^{\eta}|B(x,\delta^{1-\eps_1})|.
\]
But from this it follows (see Corollary \ref{corOfInflateTubesToSlabsLem}) that Conclusion (A) holds, provided we select $\alpha\leq \boldsymbol{\omega}\eps_1/2$ and $\eta>0$ sufficiently small. Henceforth we shall suppose that \eqref{aAlmostDelta} holds.

\medskip

\noindent{\bf Step 2.} 
We will regularize the set $\tubes$. By dyadic pigeonholing and replacing $(\tubes,Y)_\delta$ by a $\gtrsim(\log 1/\delta)^{-1/\eps_1}$ refinement $(\tubes_1,Y_1)_\delta$, we can suppose that 
\begin{itemize}
	\item[(i)] For each scale of the form $\tau_i=\delta^{\eps_1 i}$, $i=1,\ldots,\eps_1^{-1}$, there exists a balanced partitioning cover $\tubes_{\tau_i}$ of $\tubes$, and a number $\mu_i$ so that $\#(\tubes_1[T_{\tau_i}])_{Y_1}(x)\sim\mu_i$ for each $T_{\tau_i}\in\tubes_{\tau_i}$ and each $x\in\bigcup_{T\in\tubes_1[T_{\tau_i}]}Y_1(T)$. 
	
	\item[(ii)] There exists a number $\mu_{\operatorname{fine}}$ so that for each $T_\rho\in\tubes_\rho$ and each $x\in\bigcup_{T\in\tubes_1[T_\rho]}Y_1(T)$, we have $\#(\tubes_1[T_\rho]_{Y_1})(x)\sim\mu_{\operatorname{fine}}$.

	\item[(iii)] For each $T_\rho\in\tubes_\rho$ and each $x\in\bigcup_{T\in\tubes_1[T_\rho]}Y_1(T)$, we have 
	\[
		(\#\tubes_1[T_\rho])_{Y_1}(x)\gtrsim (\log 1/\delta)^{-1/\eps_1}\big(\#\tubes[T_\rho]_Y(x)\big).
	\]
\end{itemize}
\noindent Item (iii) implies that $(\tubes_1,Y_1)_\delta$ is broad with error $\lessapprox_\delta \delta^{-\eta}$ relative to $\tubes_\rho$.

Since the tubes in $\tubes_\rho$ are essentially distinct, at most $O(\rho^{-2})$ tubes from $\tubes_\rho$ can pass through a common point. We conclude that either Conclusion (B) holds, or
\[
\Big|\bigcup_{T\in\tubes}Y(T)\Big| \gtrsim \mu_{\operatorname{fine}}^{-1}\rho^2 \sum_{T\in\tubes}|Y_1(T)| \geq \delta^{\frac{\boldsymbol{\omega}}{50}+\eta}\mu_{\operatorname{fine}}^{-1}(\#\tubes)|T|.
\]
If $\mu_{\operatorname{fine}}$ is small, then Conclusion (A) holds. More precisely, if Conclusion (A) fails then 
\begin{equation}\label{muFineIsBig}
\mu_{\operatorname{fine}} \geq \delta^{-\boldsymbol{\omega}+\frac{\boldsymbol{\omega}}{50}+2\eta+\alpha}\Big((\#\tubes)|T|^{1/2}\Big)^{\boldsymbol{\sigma}}
\geq \delta^{-\frac{24}{25} \boldsymbol{\omega}},
\end{equation}
where for the final inequality we used the fact that $\FS(\tubes)\leq\delta^{-\eta}$ to conclude that $(\#\tubes)|T|^{1/2}\geq\delta^{\eta}$. Henceforth we shall suppose that \eqref{muFineIsBig} is true.

\medskip

\noindent{\bf Step 3.}
Let $\tilde c= \delta^{-\boldsymbol{\omega}/100}c$; a bit later in the argument we will abuse notation and replace $\tilde c$ by a number of the form $K\delta^{-\boldsymbol{\omega}/100}c$ for $1\leq K\lesssim 1$. We will describe a procedure that finds a prism $\tilde P$ of dimensions $a\times \tilde \rho \tilde c\times\tilde c$, for some $\tilde \rho = \tilde \rho(\tilde P) \in [\delta^{1-\boldsymbol{\omega}/100},  \delta^{\boldsymbol{\omega}/100}]$, that satisfies some of the properties from Conclusion (C). This procedure is illustrated in Figure \ref{longThinSlabFig}. In Step 4, we will iterate this procedure multiple times.


\begin{figure}[h!]
\begin{overpic}[width=.42\linewidth]{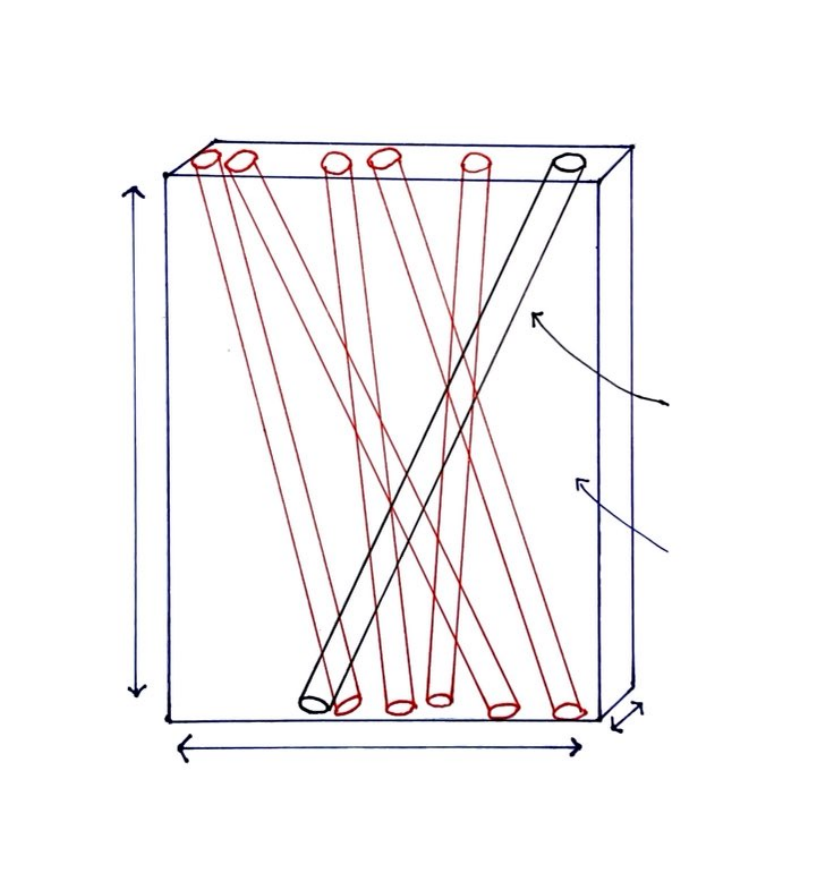}
	\put (11, 50) {$c$}
	\put (41, 11) {$b$}
	\put (71, 15) {$a$}
	\put (76, 34) {$P$}
	\put (77, 52) {$T\cap P$}

\end{overpic}
\hfill
\begin{overpic}[width=.42\linewidth]{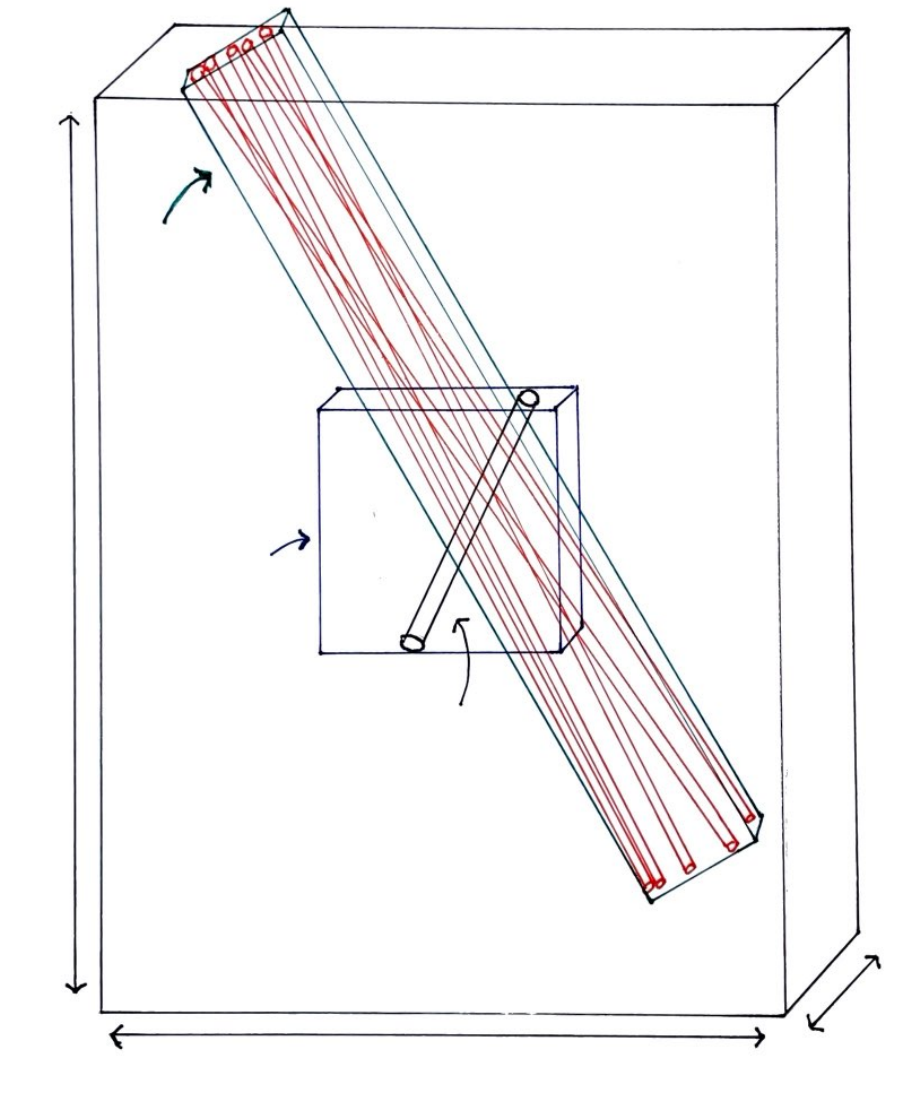}
    \put (-8,52) {$\delta^{-\frac{\omega}{100}}c$}
    \put (31,1) {$\delta^{-\frac{\omega}{100}}b$}
    \put (76,7) {$a$}
    \put (21,46) {$P$}
    \put (35,32) {$T\cap P$}
    \put (11, 75) {$\tilde P$}
    \put (57, 80) {$P^\dag$}
\end{overpic}
\caption{Finding a prism $\tilde P$}
\label{longThinSlabFig}
\end{figure}
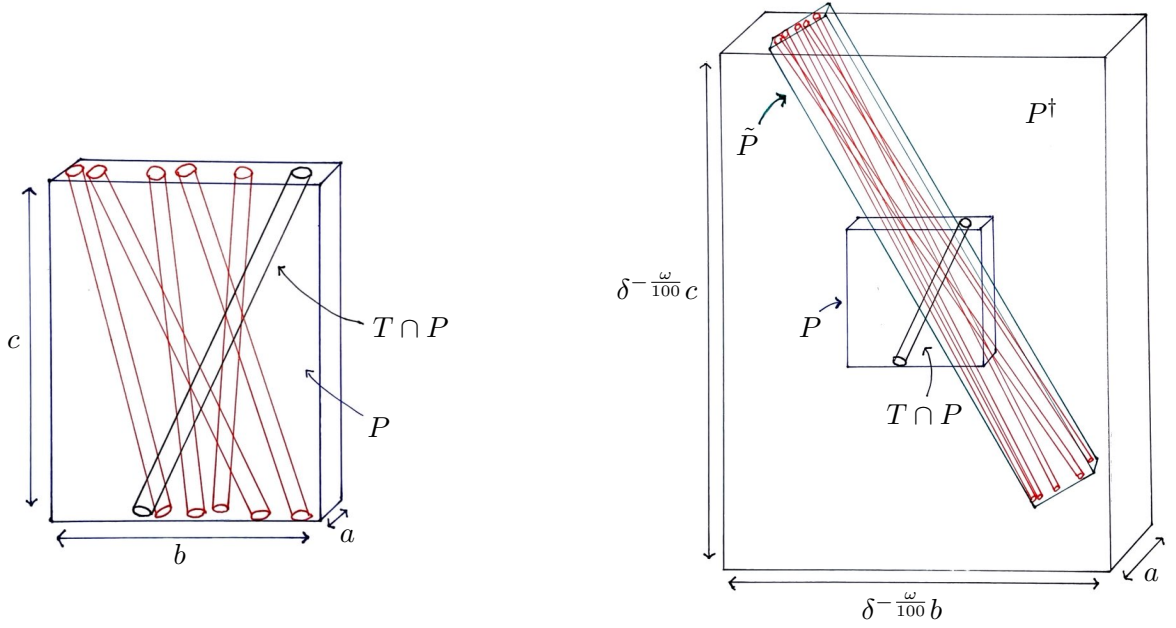

Recall that $(\tubes_1,Y_1)_\delta$ is $\gtrapprox_\delta\delta^{\eta}$ dense. This means that for a typical tube and a typical point $x\in Y_1(T)$, we expect $|B(x,\tilde c)\cap Y_1(T)|$ to have size $\gtrapprox_\delta \delta^{\eta} \tilde c \delta^2$. This should also hold for any (reasonably dense) sub-shading $Y'(T)\subset Y_1(T)$. The next definition makes this heuristic precise: Given a shading  $Y'(T)\subset T$, we define $Y'_{\operatorname{reg}}(T)\subset Y'(T)$ to consist of those points $x\in Y(T)$ for which 
\begin{equation}\label{bigOnTildeCBalls}
|B(x, \tilde c)\cap Y'(T)|\geq\frac{1}{100} \delta^{2\eta} \tilde c\delta^2,
\end{equation}
i.e. $Y'_{\operatorname{reg}}$ consists of those points where $Y'(T)$ has at density at least $\delta^{2\eta}/100$ at scale $\tilde c$. Observe that if $(\tubes_1,Y')_\delta$ is $\delta^{2\eta}$ dense, then $(\tubes,Y'_{\operatorname{reg}})$ is a $\sim 1$ refinement of $(\tubes_1,Y')$.

With the above definition, we proceed as follows. Let $(\tubes_1,Y')_\delta$ be any refinement of $(\tubes_1,Y_1)_\delta$ that satisfies $\sum |Y'(T)|\geq\frac{1}{2}\sum |Y_1(T)|$ (so in particular, $(\tubes_1,Y')_\delta$ is $\gtrapprox_\delta \delta^\eta$ dense).

We claim that we can select a $\rho$ tube $T_\rho$; a $\delta$ tube $T_{\operatorname{stem}}\in\tubes_1[T_\rho]$; a prism $P\in\mathcal{P}_{T_\rho}$ (recall Definition \ref{twoScaleGrainsDecomp}, Item (i)) and a number $\kappa\sim 1$ so that the set
\begin{equation}\label{defnE}
E = \{x \in Y'(T_{\operatorname{stem}})\cap Y(P)\colon \#\tubes_1[T_\rho]_{Y_{\operatorname{reg}}'}(x)\geq \kappa \mu_{\operatorname{fine}}\}
\end{equation}
satisfies 
\begin{equation}\label{EIsLarge}
|E|\gtrsim \delta^{2\eta}c\delta^2.
\end{equation}

To verify this claim, let us temporarily define the refinement $(\tubes_1,Y'')_\delta$ given by 
\[
Y''(T) = Y'(T)\cap\{x\colon \#\tubes_1[T_\rho]_{Y_{\operatorname{reg}}'}(x)\geq \kappa \mu_{\operatorname{fine}}\},
\]
where $T_\rho$ is the unique $\rho$ tube from $\tubes_\rho$ containing $T_{\operatorname{stem}}$. If $\kappa\sim 1$ is chosen sufficiently small, then $(\tubes_1,Y'')_\delta$ is a $1/2$ refinement of $(\tubes_1,Y')_\delta$. By pigeonholing we can select a tube $T_{\operatorname{stem}}$ and a point $x\in T_{\operatorname{stem}}$ so that the two adjacent tube segments $T_{\operatorname{seg}}^{(1)},T_{\operatorname{seg}}^{(2)}$ of length $c/10$ whose intersection contains $x$ (see Figure \ref{tubeSegmentsAndPrism}, Left) satisfy $|Y''(T_{\operatorname{stem}})\cap T_{\operatorname{seg}}^{(i)}|\gtrsim \delta^{2\eta}c\delta^2$, $i=1,2$. Next, select a prism $P\in\mathcal{P}_{T_\rho}$ with $x\in Y(P)$  (by Definition~\ref{twoScaleGrainsDecomp}, Item (ii), such $P$ is unique); by Definition \ref{twoScaleGrainsDecomp}, Item (iv), at least one of the segments $T_{\operatorname{seg}}^{(i)}$  must be almost contained in $P$, in the sense that $T_{\operatorname{seg}}^{(i)}\backslash N_\delta(x)\subset P$  (see Figure \ref{tubeSegmentsAndPrism}, Right), and thus the set $E$ from \eqref{defnE} contains at least one of the sets $\big(Y''(T_{\operatorname{stem}})\cap T_{\operatorname{seg}}^{(i)}\big)\backslash N_\delta(x)$. This yields the volume bound \eqref{EIsLarge}.

\begin{figure}[h!]
\centering
\begin{overpic}[ scale=0.35]{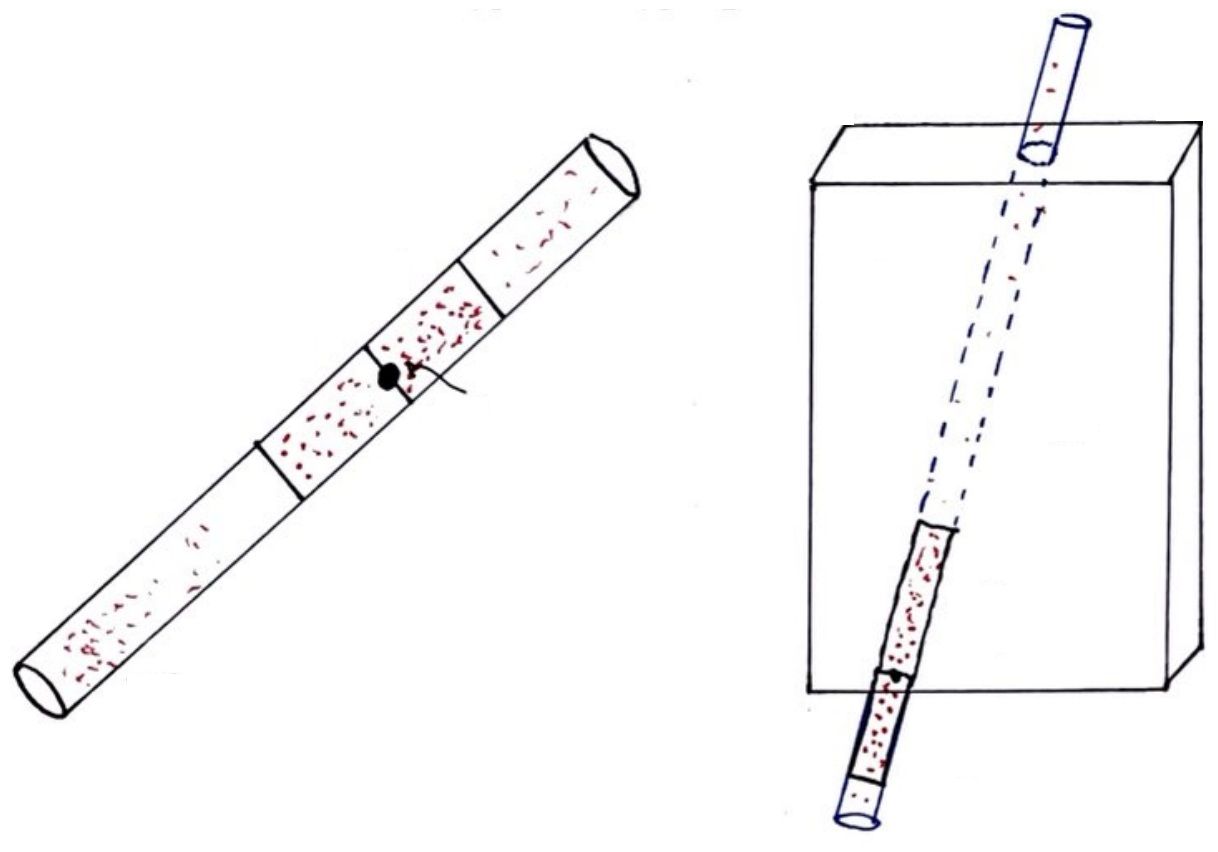}
 \put (11,10) {$T_{\operatorname{stem}}$}
 \put (16,38) {$T_{\operatorname{seg}}^{(1)}$}
 \put (25,47) {$T_{\operatorname{seg}}^{(2)}$}
 \put (39,35) {$x$}
 \put (71,46) {$P$}
 \put (83,36) {$T_{\operatorname{stem}}$}
 \put (75,5) {$T_{\operatorname{seg}}^{(1)}$}
 \put (78,17) {$T_{\operatorname{seg}}^{(2)}$}
\end{overpic}
\caption{Left: The two tube segments on either side of the point $x$. Both tube segments have rich shadings.\\ Right: Since $x\in P$ and $T_{\operatorname{stem}}$ exits $P$ through its ``long ends,'' at least one of the tube segments must be almost contained in $P$, in the sense that $T_{\operatorname{seg}}^{(i)}\backslash N_\delta(x)\subset P$.}
\label{tubeSegmentsAndPrism}
\end{figure}

Since $(\tubes_1,Y_1)_\delta$ is broad with error $\lessapprox_\delta \delta^{-\eta}$ relative to $\tubes_\rho$, we have that for each $x\in E$, there are $\gtrapprox_\delta \mu_{\operatorname{fine}}$ tubes $T\in \tubes_1[T_\rho]_{Y_{\operatorname{reg}}'}(x)$ with $\angle(T,T_{\operatorname{stem}})\gtrapprox_\delta \delta^{\eta/\boldsymbol{\beta}}\rho$. Since each such tube intersects $T_{\operatorname{stem}}$ in a set of dimensions at most $\delta\times\delta\times \frac{\delta}{\delta^{\eta/\boldsymbol{\beta}}\rho}$, we conclude that there is a set $\mathcal{H} = \mathcal{H}(T_{\operatorname{stem}})\subset \tubes_1[T_\rho]$ with the following properties:
\begin{itemize}
\item $\#\mathcal{H}\gtrapprox_{\delta} (\mu_{\operatorname{fine}} \delta^{-1} c)(\delta^{\eta/\boldsymbol{\beta}}\rho) \delta^{2\eta}$.

\item Each $T\in \mathcal{H}$ intersects $P$, and exits $P$ through the ``long ends.''

\item Each $T\in \mathcal{H}$ satisfies $|B\cap Y'(T)| \gtrsim \delta^{2\eta}|B\cap T| \sim \delta^{2\eta}\tilde c \delta^2$, where $B$ is the ball of radius $\tilde c$ with the same center as $P$. 
\end{itemize}
The second item follows from Definition \ref{twoScaleGrainsDecomp} Item (iv), plus the fact that $T\in \mathcal{H}$ implies that $T$ and $P$ are both associated to the same $\rho$ tube, and $Y(T)\cap  Y(P)\neq\emptyset$. The third item follows from the fact that the set $E$ from \eqref{defnE} was defined with respect to the shading $Y_{\operatorname{reg}}'$ (recall \eqref{bigOnTildeCBalls} for the definition of $Y_{\operatorname{reg}}'$).

As a consequence, using $a\in [\delta, \delta^{1-\eps_1}]$, $\tilde{c}=c\delta^{-\boldsymbol{\omega}/100}$, and \eqref{muFineIsBig},  we have
\[
\sum_{T\in \mathcal{H}}|B \cap Y'(T)| \gtrsim \Big(\mu_{\operatorname{fine}}\delta^{-1+2\eta +\eta/\boldsymbol{\beta}} c\rho \Big)\Big(\delta^{2\eta}\tilde c \delta^2\Big)\gtrsim \mu_{\operatorname{fine}}\delta^{\eta/\boldsymbol{\beta}+4\eta+\boldsymbol{\omega}/100 +\eps_1}(a\cdot \rho\tilde c\cdot \tilde c)\gtrsim \delta^{-\frac{1}{2}\boldsymbol{\omega}}(a\cdot \rho\tilde c\cdot\tilde c),
\]
where $B$ is the ball of radius $\tilde c$ with the same center as $P$. To ensure that the final inequality holds, we select $\eta\leq \boldsymbol{\beta}\boldsymbol{\omega}/100$ and $\eps_1\leq \boldsymbol{\omega}/100$. 

We claim that by pigeonholing, we can select a prism $P^\dag\supset P$ (the larger prism on the right side of Figure \ref{longThinSlabFig}) of dimensions $2a\times 2\rho\tilde c\times\tilde c$ so that
\begin{equation}\label{lotsOfMassInsidePDag}
\sum_{\substack{T\in \mathcal{H}\\ T\ \textrm{long end}\ P^\dag}}|P^\dag  \cap Y'(T)| \gtrsim \delta^{2\eta} \sum_{\substack{T\in \mathcal{H}\\ T\ \textrm{long end}\ P^\dag}} |P^{\dag}\cap T|  \gtrsim \delta^{-\frac{1}{4}\boldsymbol{\omega}}|P^\dag|,
\end{equation}
where ``$T\ \textrm{long end}\ P^\dag$'' means that $T$ exits $P^\dag$ through its long ends. 

To verify this claim, note that for each tube $T\in \mathcal{H},$ there exists at least one $a\times \rho\tilde c\times\tilde c$ prism $P^\dag$ so that $T$ exits $P^\dag$ through its long end. On the other hand, there are only $\lesssim (\tilde c/c)^2 = \delta^{-\frac{1}{50}\boldsymbol{\omega}}$ essentially distinct prisms of dimensions $a\times \rho\tilde c\times\tilde c$ that contain $P$. The result now follows from pigeonholing. (Note that if $T$ exits a prism $P_1^\dag$ through its long ends, and if $P_2^\dag$ is comparable to $P_1^\dag$, then $T'$ exits the 2-fold thickening of $P_2^\dag$ through its long ends, where the 2-fold thickening is the prism obtained by increasing the two smaller dimensions of $P_2^\dag$ by a factor of 2. This is why the dimension of our prisms have increased to  $2a\times 2\rho\tilde c\times\tilde c$ at this step).

Apply Corollary \ref{coveringTubesBroadCor} (finding a broad scale) to the set $\mathcal{H}$ with the shading $P^\dag\cap Y'(T)$, $T\in\mathcal{H}$. We obtain a scale $\tilde\rho\in[\delta,1]$; a set $\mathcal{H}'\subset\mathcal{H}$ (each $T\in\mathcal{H}'$ exists $P^\dag$ through its long ends); a sub-shading of the shadings $\{P^\dag\cap Y'(T), T\in\mathcal{H}'\}$, which we will denote by $Y_{P^{\dag}}(T)$; and a balanced partitioning cover $\tubes^{\mathcal{H}}_{\tilde\rho}$ of $\mathcal{H}'$.

Note that $Y_{P^{\dag}}(T)\subset P^\dag\cap Y'(T),$ 
and the latter set is contained in a tube segment of dimensions comparable to $\delta\times\delta\times \tilde c$. In particular, $Y_{P^{\dag}}(T)$ is \emph{not} a $\delta^{O(\eta)}$ dense shading of $T$. However, Corollary \ref{coveringTubesBroadCor} guarantees that the shadings are ``relatively'' dense inside $T\cap P^\dag$, thus
\begin{equation}\label{mostMassPreservedInTildeHPrime}
\sum_{T\in\mathcal{H}'}|Y_{P^{\dag}}(T)|\gtrapprox_\delta \delta^{2\eta}\tilde c\delta^2(\#\mathcal{H}').
\end{equation}

 Note that \eqref{lotsOfMassInsidePDag} remains true if the shading $P^\dag\cap Y'(T)$ on the LHS of \eqref{lotsOfMassInsidePDag}  is replaced by $Y_{P^{\dag}}(T)$, provided the RHS is weakened by an additional $\approx_\delta 1$ factor, i.e. we have
\begin{equation}\label{YPDagHasEnoughMass}
\sum_{T\in\mathcal{H}'}|Y_{P^{\dag}}(T)| \gtrapprox_\delta \delta^{-\frac{1}{4}\boldsymbol{\omega}}|P^\dag|.
\end{equation}

We claim that
\begin{equation}\label{rhoDagBig}
\tilde\rho\geq\delta^{1-\boldsymbol{\omega}/100}.
\end{equation}
We verify this claim as follows. Each point $x\in P^\dag$ is contained in at most $(\tilde\rho/\delta)^2 (\tilde\rho)^{-2\boldsymbol{\beta}}$ of the shadings $\{Y_{P^\dag}(T),\ T\in\mathcal{H}\}$. Thus if \eqref{rhoDagBig} failed, then by \eqref{YPDagHasEnoughMass} we would have
\[
|P^\dag|\geq \Big|\bigcup_{T\in\mathcal{H}}P^\dag\cap Y_{P^\dag}(T)\Big|\gtrapprox_\delta \Big(\big(\frac{\tilde\rho}{\delta}\big)^2 (\tilde\rho)^{-2\boldsymbol{\beta}}\Big)^{-1}\Big(\delta^{-\frac{1}{4}\boldsymbol{\omega}}|P^\dag|\Big)\gtrsim \delta^{-\frac{1}{8}\boldsymbol{\omega}}|P^\dag|,
\]
which is impossible. For the final inequality, we used the assumption that $\boldsymbol{\beta}\leq\boldsymbol{\omega}/100$.

By \eqref{mostMassPreservedInTildeHPrime} and pigeonholing, we can select $T_1\in\mathcal{H}'$ with $|Y_{P^\dag}(T_1)|\gtrapprox_\delta  \delta^{2\eta}\tilde c\delta^2$. Let $ T_{\tilde\rho}\in\tubes_{\tilde\rho}^{\mathcal{H}}$ be the (unique) $\tilde\rho$ tube with $T_1\in\mathcal{H}'[T_{\tilde\rho}]$. For each $x\in Y_{P^\dag}(T_1)$, the directions of the tubes in $\mathcal{H}'[T_{\tilde\rho}]_{Y_{P^\dag}}(x)$ are broad with error $\lessapprox_\delta 1$ inside the $\tilde\rho$ cap centered at $\dir(T_{\tilde\rho})$. In particular, the intersection of each of these tubes with $P^\dag$ is contained in $P^\dag\cap N_{\tilde\rho\tilde c}(T_1)$; the latter set is contained in a $2a\times \tilde\rho\tilde c\times\tilde c$ prism; call this prism $\tilde P$---this is the green prism in Figure \ref{longThinSlabFig}. Note that since each $T\in\mathcal{H}'$ exits $P^\dag$ through its long ends, each of the tubes $T\in \mathcal{H}'[T_{\tilde\rho}]_{Y_{P^\dag}}(x)$ described above exit $\tilde P$ through its long ends.

Applying a standard Cordoba-type $L^2$ argument\footnote{In brief, we select a set of tubes $T'\in \mathcal{H}'$ that make angle $\sim \tilde\rho$ with $T_1$ and intersect $T_1$ at $\delta/\tilde\rho$ separated points; this latter collection of tubes, restricted to the $2a\times\tilde\rho\tilde c\times \tilde c$ prism described above, satisfies the Katz-Tao Convex Wolff Axioms with error $\lesssim 1$. Each of these tubes has a shading $Y'(T)\cap\tilde P$ that satisfies $|Y'(T)\cap\tilde P|\gtrsim \delta^{2\eta}\tilde c\delta^2$, and hence the union of these shadings is almost disjoint.}, we conclude that if we define 
\begin{equation*}
\tubes_{\tilde P} = \{T\in \mathcal{H}'[T_{\tilde \rho}]\colon T\cap\tilde P\neq\emptyset,\ T\ \textrm{exits}\ \tilde P\ \textrm{through its long ends}\},
\end{equation*}
and define the shading $Y_{\tilde P}(T)=Y_{P^\dag}(T)\cap\tilde P$ (note if $T\in \tubes_{\tilde{P}}$, $T$ exists $\tilde{P}$ through its long ends, so $Y_{\tilde{P}}(T)= Y_{P^{\dag}}(T)$, we rename it to be $Y_{\tilde{P}}(T)$ just for notational convenience),
then the set
\[
Y(\tilde P) = \bigcup_{T\in \tubes_{\tilde P}} Y_{\tilde P}(T)
\]
satisfies $|Y(\tilde P)|\gtrsim \delta^{4\eta+4\eps_1}|\tilde P|$ (here the $\delta^{4\eps_1}$ loss comes from the fact that the prism $\tilde P$ is not a $\delta\times \tilde\rho\tilde c\times\tilde c$ prism, but rather a $2a\times \tilde\rho\tilde c\times\tilde c$ prism with $a\in [\delta, \delta^{1-\eps_1}]$). Furthermore, for each $x\in Y(\tilde P)$, the set of unit vectors $\{\dir(T)\colon T\in (\tubes_{\tilde P})_{Y_{\tilde P}}(x)\}$ is broad with error $\lessapprox_\delta 1$ inside the $\tilde\rho$-cap $B(\dir(\tilde P),\tilde\rho)$.

\medskip

\noindent {\bf Step 4.}
We summarize the conclusion from Step 3. Given a refinement $(\tubes_1,Y')_\delta$ of $(\tubes_1,Y_1)_\delta$ that satisfies $\sum |Y'(T)|\geq\frac{1}{2}\sum |Y_1(T)|$, we have located the following objects:
\begin{itemize}
	\item A scale $\tilde\rho$.
	\item A $2a\times\tilde\rho\tilde c\times\tilde c$ prism $\tilde P$ and a shading $Y(\tilde P)$ on this prism.
	\item A set of tubes $\tubes_{\tilde P}$ and a shading $Y_{\tilde P}(T)\subset Y'(T)\cap\tilde P$ on these tubes. 
\end{itemize}
 These objects have the following properties:
\begin{itemize}
	\item Each $T\in\tubes_{\tilde P}$ exits $\tilde P$ through its long ends, in the sense of Figure \ref{exitLongEndsFig}, Left.
	\item $Y(\tilde P) = \bigcup_{T\in\tubes_{\tilde P}}Y_{\tilde P}(T)$, and $|Y(\tilde P)|\gtrapprox_\delta \delta^{4\eta+4\eps_1}|\tilde P|$.
	\item For each $x\in Y(P)$, the tubes in $(\tubes_{\tilde P})_{Y_{\tilde P}}(x)$ point in directions that are broad with error $\lessapprox_\delta 1$ inside the $\tilde\rho$ cap $B(\dir(\tilde P),\tilde\rho)$. 
\end{itemize}

We will now iteratively apply the argument from Step 3. We begin by setting $(\tubes,Y')=(\tubes,Y_1)$ and $\tilde{\mathcal{P}}_0=\emptyset$. As long as $\sum |Y'(T)|\geq\frac{1}{2}\sum |Y_1(T)|$, we proceed as follows:
\begin{itemize}
\item  Apply the argument from Step 3.
\item Place the prism $\tilde P$ located in Step 3 into the multiset $\tilde{\mathcal{P}}_0$ (i.e. if the prism is already present, then we add another copy).
\item For each $T\in\tubes_{\tilde P}$, replace the shading $Y'(T)$ with $Y'(T)\backslash Y_{\tilde P}(T)$.
\end{itemize} 
We repeat the above steps until $\sum |Y'(T)|<\frac{1}{2}\sum |Y_1(T)|$, at which point we halt.

Let us examine the output from the above procedure. First, we have
\begin{equation}\label{massBdTildeP}
\sum_{\tilde P\in\tilde{\mathcal{P}}_0}\sum_{T\in\tubes_{\tilde P}} |Y_{\tilde P}(T)| \geq\frac{1}{2}\sum_{T\in\tubes}|Y_1(T)|\gtrapprox_\delta \sum_{T\in\tubes}|Y(T)|.
\end{equation}

After dyadic pigeonholing, we can select a multiset $\tilde{\mathcal{P}}_1\subset \tilde{\mathcal{P}}_0$ so that each $\tilde P\in \tilde{\mathcal{P}}_1$ has a common value of $\tilde\rho$ (up to a factor of 2). Abusing notation slightly, we will denote this value by $\tilde\rho$. We will choose $\tilde{\mathcal{P}}_1$ so that the bound \eqref{massBdTildeP} (the first and final terms) remains true with $\tilde{\mathcal{P}}_1$ in place of $\tilde{\mathcal{P}}_0$.

For each $T\in\tubes_1$, define $Y_2(T) = \bigcup_{\tilde P} Y_{\tilde P}(T)$, where the union is taken over those $\tilde P\in\tilde{\mathcal P}_1$ with $T\in\tubes_{\tilde P}$. For notational consistency, define $\tubes_2=\tubes_1$. Then $(\tubes_2,Y_2)_\delta$ is an $\approx_\delta 1$ refinement of $(\tubes_1,Y_1)_\delta$. 

\medskip

\noindent {\bf Step 5.}
We will summarize the conclusion from Step 4. We have located the following objects:
\begin{itemize}
	\item A $\approx_\delta 1$ refinement $(\tubes_2,Y_2)_\delta$ of $(\tubes_1,Y_1)_\delta$, which in turn is an $\approx_\delta 1$ refinement of $(\tubes,Y)_\delta$.
	\item A scale $\tilde\rho \geq\delta^{1-\boldsymbol{\omega}/100}$.
	\item A multiset $\tilde{\mathcal P}_1$ of $2a\times\tilde\rho\tilde c\times\tilde c$ prisms, and a shading $Y_1(\tilde P)$ on these prisms (In Step 4 this shading was called $Y(\tilde P)$). Note that the prisms in $\tilde{\mathcal P}_1$ might not be essentially distinct.
	\item For each prism $\tilde P\in\tilde{\mathcal{P}}_1$, a set of tubes $\tubes_{\tilde P}\subset \tubes_2$.
\end{itemize}
 These objects have the following properties:
\begin{itemize}
	\item[(a)] For each $\tilde P\in\tilde{\mathcal{P}}_1$, each $T\in\tubes_{\tilde P}$ exits $\tilde P$ through the long ends.
	\item[(b)] For each $\tilde P\in\tilde{\mathcal{P}}_1$, we have $Y_1(\tilde P) = \tilde P\cap \bigcup_{T\in\tubes_{\tilde P}}Y_2(T)$, and $|Y_1(\tilde P)|\gtrapprox_\delta \delta^{4\eta+4\eps_1}|\tilde P|$.
	\item[(c)] For each $\tilde P\in\tilde{\mathcal{P}}_1$ and for each $x\in Y_1(\tilde P)$, the tubes in $(\tubes_{\tilde P})_{Y_2}(x)$ point in directions that are broad with error $\lessapprox_\delta 1$ inside the $\tilde\rho$ cap $B(\dir(\tilde P),\tilde \rho)$.
\end{itemize}
Let us compare the above items to Conclusion (C) of Lemma \ref{squareGrainsGetLonger}. We have that Items (i) and (iv) of Conclusion (C) are currently satisfied.  We will work towards satisfying the other Items.

First, the prisms in $\tilde{\mathcal P}_1$ might not be essentially distinct. This is not a minor failure fixable by a $\sim 1$ refinement, but instead a dramatic failure --- it is possible that a very large number of prisms from $\tilde{\mathcal P}_1$ are all pairwise comparable, or even identical. We can fix this problem by merging comparable $2 a \times\tilde\rho\tilde c\times\tilde c$ prisms in $\tilde{\mathcal P}_1$ into a single $4 a\times \tilde\rho(2\tilde c)\times(2\tilde c)$ prism. We will refer to this new, post-merger set of prisms as $\tilde{\mathcal{P}}_2$. Our shading $Y_2(\tilde P)$ on our newly constructed prisms is given by the union of the shadings of the corresponding prisms from $\tilde{\mathcal P}_1$, and the set $\tubes_{\tilde P}\subset\tubes_2$ is the union of the sets $\{\tubes_{\tilde P_1}\}$ from the corresponding prisms from $\tilde{\mathcal P}_1$. 

Item (a) from the start of Step 5 might no longer hold for our newly constructed prisms $\tilde{\mathcal{P}}_2$, but this is a minor failure --- we can restore it by replacing each $4 a \times2\tilde\rho(2\tilde c)\times(2\tilde c)$ prism by the prism of dimensions $100  a\times 100\tilde\rho\tilde c\times 2\tilde c$ with the same center and axes (see Figure \ref{comparablePrismsBadExitFig}). Annoyingly, this might destroy the property that the prisms are essentially distinct, but this time, this is only a minor failure --- essential distinctness can be restored by a $\sim 1$ refinement of the prisms (this refinement induces a $\sim 1$ refinement of the shading $Y_2$ on $\tubes_2$). Denote the new set of prisms created through this process by $\tilde{\mathcal{P}}_3$. Abusing notation slightly, we will redefine the quantities $\tilde c,$ and $\tilde\rho$ and let $\tilde a=100a$ (increasing each by a $\sim 1$ multiplicative factor) so that the prisms in $\tilde{\mathcal{P}}_3$ still have dimensions $\tilde a\times\tilde \rho\tilde c\times\tilde c$.

\begin{figure}[h!]
\centering
\begin{overpic}[ scale=0.25]{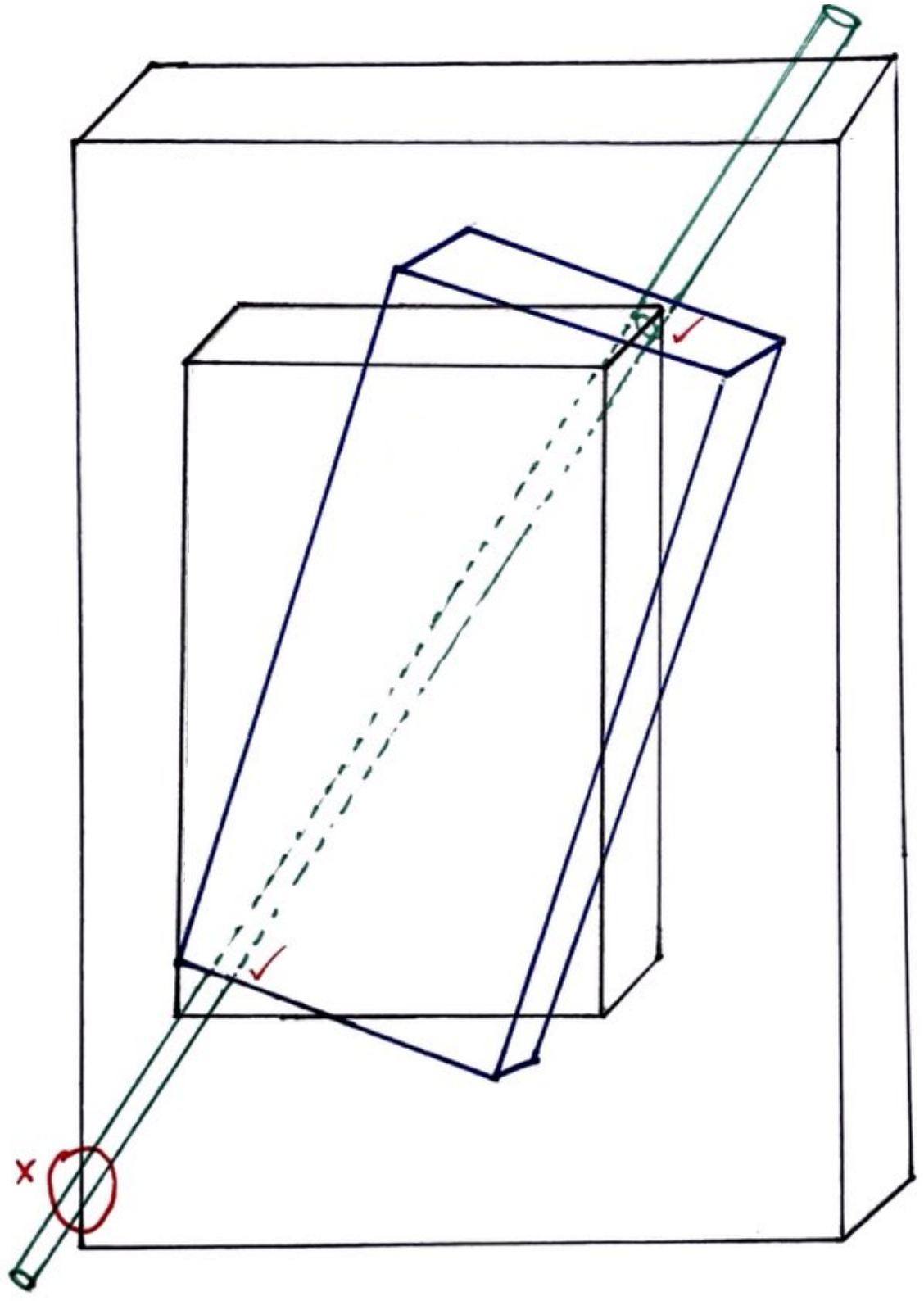}
 \put (17,60) {$\tilde P$}
 \put (29,62) {$\tilde P'$}
 \put (12,81) {$2\tilde P$}
\end{overpic}
\caption{The prisms $\tilde P$ and $\tilde P'$ are comparable, and thus both are replaced by a common $4 a\times \tilde\rho(2\tilde c)\times(2\tilde c)$ prism (which happens to be $2\tilde P$, i.e.~the 2-fold dilate of $\tilde P$). This creates a problem (circled in red): A tube (green) that exits the prism $\tilde P'$ through its long ends might fail to exit $2\tilde P$ through its long ends. However, this problem can be fixed by replacing the $4 a\times \tilde\rho(2\tilde c)\times(2\tilde c)$ prism by a slightly wider and thicker (but not taller) prism.
}
\label{comparablePrismsBadExitFig}
\end{figure}

Note that Item (b) above continues to hold for our newly constructed prisms and their associated shading. Crucially, Item (c) also continues to hold; this follows from Lemma \ref{broadnessUnion} (a union of sets of broad vectors is broad).  More specifically, the sets of broad vectors through each point are disjoint because in Step 4, for each $T\in \tubes_{\tilde{P}}$, we have replaced the shading $Y'(T)$ with $Y'(T)\setminus Y_{\tilde{P}}(T)$, so through each point $x$, there is  at most one  $\tilde{P}$ such that $x\in Y_{\tilde{P}}(T)$.  Since the sets are disjoint, their union is a set instead of multi-set, so Lemma~\ref{broadnessUnion} implies that the union (as a set)  of sets of broad vectors is broad.

\medskip

\noindent {\bf Step 6.}
In Step 5 we constructed a set $\tilde{\mathcal P}_3$ of essentially distinct $\tilde a\times\tilde\rho\tilde c\times\tilde c$ prisms, and a shading $Y_3$ on these prisms. We shall refer to this pair as $(\tilde{\mathcal P}_3,Y_3)_{\tilde a\times \tilde\rho \tilde c\times \tilde c}$. For each $\tilde P\in\tilde{\mathcal P}_3$, we have a set $\tubes_{\tilde P}\subset\tubes_2$; each of these tubes exits $\tilde P$ through its long ends.

Our current task is to further refine the pair $(\tilde{\mathcal P}_3,Y_3)_{\tilde a\times \tilde\rho \tilde c\times \tilde c}$ and $(\tubes_2,Y_2)_\delta$ to more closely match Conclusion (C) of Lemma \ref{squareGrainsGetLonger}.
Item (ii) of Conclusion (C) refers to a partitioning cover $\tubes_{\tilde \rho}$ of $\tubes'$. We will construct this as follows.  To begin, let $\{T_{\tilde\rho}\}$ be a set of $\tilde\rho$ tubes with the following properties:
\begin{itemize}
\item Every $\delta$ tube is contained in at least one tube from $\{T_{\tilde\rho}\}$.
\item For every $\tilde a\times\tilde\rho\tilde c\times\tilde c$ prism $\tilde P$, there is at least one $\tilde\rho$ tube $T_{\tilde\rho}\in  \{T_{\tilde\rho}\}$ with $\tilde P\subset T_{\tilde\rho}$ and $\angle(\dir(\tilde P),\dir(T_{\tilde \rho}))\leq 2\tilde\rho$.
\item The tubes in $\{ T_{\tilde\rho}\}$ are \emph{weakly} essentially distinct, in the following sense: for each fixed $T_{\tilde\rho}\in\{T_{\tilde\rho}\}$, there are $O(1)$ other tubes from $\{T_{\tilde\rho}\}$ that are comparable to $T_{\tilde\rho}$.
\end{itemize}

Next, by pigeonholing the set $\{T_{\tilde\rho}\}$ by a $O(1)$ factor, we can select a set $\tubes_{\tilde\rho}\subset\{T_{\tilde\rho}\}$ that is \emph{strongly} essentially distinct, in the following sense: for each pair of distinct tubes $T_{\tilde\rho},T_{\tilde\rho}'$ from $\tubes_{\tilde\rho}$, we have that $N_{100\tilde\rho}(T_{\tilde\rho})\cap N_{100\tilde\rho}(T_{\tilde\rho}')$ has diameter at most $1/2$, and in particular no $\delta$ tube can be contained in both $N_{100\tilde\rho}(T_{\tilde\rho})$ and  $N_{100\tilde\rho}(T_{\tilde\rho}')$. We will select the set $\tubes_{\tilde\rho}$ so that the following properties hold:

\begin{itemize}
	\item[(i)] If we define $\tubes_4$ to be the set of tubes $T\in\tubes_2$ contained in some $T_{\tilde\rho}\in\tubes_{\tilde\rho}$, i.e. $\tubes_4=\bigcup_{T_{\tilde\rho}\in\tubes_{\tilde\rho}}\tubes_2[T_{\tilde\rho}]$, and define $Y_4$ to be the restriction of $Y_2$ to $\tubes_4$, then $(\tubes_4,Y_4)_\delta$ is a $\sim 1$ refinement of $(\tubes_2,Y_2)_\delta$. 

	\item[(ii)] Similarly to the previous item, if we define $\tilde{\mathcal{P}}_4$ to be the set of prisms $\tilde P\in \tilde{\mathcal{P}}_3$ with the property that there exists $T_{\tilde\rho}\in\tubes_{\tilde\rho}$ with $\tilde P\subset T_{\tilde\rho}$ and $\angle(\dir(\tilde P),\dir(T_{\tilde \rho}))\leq 2\tilde\rho$, and define 

		\begin{equation}\label{defnY4Shading}
		Y_4(\tilde P) = Y_3(\tilde P)\cap \bigcup_{T\in\tubes_{\tilde P}\cap\tubes_4}Y_4(T),
		\end{equation}
		then $(\tilde{\mathcal{P}}_4,Y_4)_{\tilde a\times\tilde\rho\tilde c\times\tilde c}$ is a $\gtrsim 1$ refinement  of $(\tilde{\mathcal{P}}_3,Y_3)_{\tilde a\times\tilde\rho\tilde c\times\tilde c}$. 
\end{itemize}

\begin{itemize}
	\item[(a)] A consequence of Item (i) is that $\tubes_{\tilde\rho}$ is a partitioning cover of $\tubes_4$, and in fact more is true:  the sets $\{N_{100\tilde\rho}(T_{\tilde \rho})\colon T_{\tilde \rho}\in\tubes_{\tilde\rho}\}$ form a partitioning cover of of $\tubes_4$.

	\item[(b)] A consequence of Item (ii) is that for each $\tilde P\in \tilde{\mathcal{P}}_4$, there is a unique $T_{\tilde\rho}\in\tubes_{\tilde\rho}$ that satisfies the two properties $\tilde P\subset T_{\tilde \rho}$ and $\angle(\dir(\tilde P),\dir(T_{\tilde \rho}))\leq 2\tilde\rho$. This induces a partition 
	\begin{equation}\label{mathcalP4Partition}
		\tilde{\mathcal{P}}_4=\bigsqcup_{\tubes_{\tilde \rho}} (\tilde{\mathcal{P}}_4)_{T_{\tilde\rho}}.
	\end{equation}
	(C.f. Definition \ref{twoScaleGrainsDecomp}, Item (i).)

	\item[(c)] A consequence of Items (i) and (ii) is that if $T_{\tilde\rho}\in\tubes_{\tilde\rho}$, $\tilde P\in (\tilde{\mathcal{P}}_4)_{T_{\tilde \rho}},$ and $T\in\tubes_{\tilde P}\cap\tubes_4,$ then $T\in\tubes_4[T_{\tilde\rho}]$. This is because $T$ exits $\tilde P$ through its long ends, and hence  $\angle(\dir(T),\dir(\tilde P))\leq 10\tilde\rho$.

\end{itemize}

Our next task is to estimate the quantity $\sum_{\tilde P\in\tilde{\mathcal{P}}_4}|\tilde P|$. Let $\tau_i$ be the scale from Step 2 satisfying $\delta^{\eps_1}\rho\leq\tau_i<\rho$. Since $(\tilde{\mathcal{P}}_4,Y_4)_{\tilde a\times\tilde\rho\tilde c\times\tilde c}$ is $\gtrapprox_\delta \delta^{4\eta+4\eps_1}$ dense, we have
\begin{equation}\label{volumeLowerBoundSumOfGrains}
\begin{split}
\sum_{\tilde P\in\tilde{\mathcal{P}}_4}|\tilde P| & \lessapprox_\delta \delta^{-4\eta-4\eps_1} \sum_{\tilde P\in\tilde{\mathcal{P}}_4}|Y_4(\tilde P)|\\
&=\delta^{-4\eta-4\eps_1} \sum_{T_{\tilde\rho}\in\tubes_{\tilde\rho}}\Big|\bigsqcup_{\tilde P\in(\tilde{\mathcal{P}}_4)_{T_{\tilde\rho}}}Y_4(\tilde P) \Big|\\
&\leq \delta^{-4\eta-4\eps_1} \sum_{T_{\tilde\rho}\in\tubes_{\tilde\rho}} \Big|\bigcup_{T\in\tubes_4[T_{\tilde\rho}]}Y_4(T)\Big|\\
&\lesssim \delta^{-4\eta-4\eps_1} \sum_{T_{\tilde\rho}\in\tubes_{\tilde\rho}}\Big( \mu_i^{-1} \sum_{T\in\tubes_1[T_{\tilde\rho}]}|Y_1(T)|\Big)\\
&= \delta^{-4\eta-4\eps_1}\mu_i^{-1}\sum_{T\in\tubes_1}|Y_1(T)|\\
& \lesssim   \delta^{-4\eta-4\eps_1}\mu_i^{-1}\sum_{T\in\tubes_2}|Y_2(T)|.
\end{split}
\end{equation}
For the third line we used \eqref{defnY4Shading}.  For the fourth line, we used the fact that each $\tilde\rho$ tube  contains some $\tau_i$-tube and Item (i) in Step 2. For the last line, we used \eqref{massBdTildeP} and the definition of $Y_2(T) = \bigcup_{\tilde P} Y_{\tilde P}(T)$.

\medskip

\noindent {\bf Step 7.}

We would like to show that after a suitable refinement, $(\tilde{\mathcal{P}}_4,Y_4)_{\tilde a\times\tilde\rho\tilde c\times\tilde c}$ is a robustly $\delta^{\eps}$-dense  two-scale grains decomposition of $(\tubes_4,Y_4)_\delta$ wrt $\tubes_{\tilde \rho}$, in the sense of Definition \ref{twoScaleGrainsDecomp}. Currently, the biggest obstacle is Item (ii) from Definition \ref{twoScaleGrainsDecomp}. In particular, it need not be the case that the sets $\{Y_4(\tilde P)\colon \tilde P \in (\tilde{\mathcal{P}}_4)_{T_{\tilde\rho}}\}$ are pairwise disjoint.


We will fix this problem as follows. We claim that either Conclusion (A) of Lemma \ref{squareGrainsGetLonger} is true (and thus we are done), or there exists a refinement $(\tubes_5,Y_5)_\delta$ of $(\tubes_4,Y_4)$, and a refinement $(\tilde{\mathcal{P}}_5,Y_5)_{\tilde a\times\tilde\rho\tilde c\times\tilde c}$ of $(\tilde{\mathcal{P}}_4,Y_4)_{\tilde a\times\tilde\rho\tilde c\times\tilde c}$ with the following properties:
\begin{itemize}
	\item For each $T_\rho\in\tubes_\rho$, the sets $\{Y_5(\tilde P)\colon \tilde P\in(\tilde{\mathcal{P}}_5)_{T_{\tilde\rho}}\}$ are pairwise disjoint (here $(\tilde{\mathcal{P}}_5)_{T_{\tilde\rho}} = \tilde{\mathcal{P}}_5 \cap (\tilde{\mathcal{P}}_4)_{T_{\tilde\rho}}$; recall \eqref{mathcalP4Partition}).
	\item $(\tilde{\mathcal{P}}_5,Y_5)_{\tilde a\times\tilde\rho\tilde c\times\tilde c}$ is a $\gtrapprox_\delta \delta^{\eps_2}$ refinement of $(\tilde{\mathcal{P}}_4,Y_4)_{\tilde a\times\tilde\rho\tilde c\times\tilde c}$.
	\item The pair $(\tubes_5,Y_5)_\delta$ is a $\gtrapprox_\delta \delta^{\eps_2}$ refinement of $(\tubes_4,Y_4)_\delta$, where $\tubes_5 = \tubes_4$, and the shading $Y_5$ is given by

	\begin{equation}\label{defnOfY5Shading}
	Y_5(T)=Y_4(T)\cap\bigcup_{\substack{\tilde P\in (\tilde{\mathcal P}_5)_{T_{\tilde \rho}} \\ T\in\tubes_{\tilde P} \cap \tubes_4  }}Y_5(\tilde P),
	\end{equation}
	where $T_{\tilde\rho}$ is the unique $\tilde  \rho$ tube containing $T$.

	\item For each $T_{\tilde\rho}\in\tubes_{\tilde\rho}$, we have
	\begin{equation}\label{unionShadingY4EqualsY5}
		\bigcup_{T\in\tubes_5[T_{\tilde\rho}]}Y_5(T) = \bigsqcup_{\tilde P\in(\tilde{\mathcal{P}}_5)_{T_{\tilde\rho}}}Y_5(\tilde P).
	\end{equation}
\end{itemize}

\noindent We will prove this claim in Step 8 below. Let us accept this claim for the moment. 

The pair $(\tilde{\mathcal{P}}_5,Y_5)_{\tilde a\times\tilde\rho\tilde c\times\tilde c}$ and $(\tubes_5,Y_5)_\delta$ now satisfy Items (i), (ii), and (iv) from Definition \ref{twoScaleGrainsDecomp}. Items (i) and (ii) are immediate. We can verify Item (iv) as follows: If $T$ and $\tilde P$ are associated to a common $\rho$ tube, and if $Y_5(T)\cap Y_5(\tilde P)\neq\emptyset$, then we must have $T\in\tubes_4\cap\tubes_{\tilde P}$, and hence we have that $T$ exits $\tilde P$ through its long end, and also $Y_5(T)\cap\tilde P\subset Y_5(\tilde P)$ (this follows from the definition of the shading $Y_5$ from \eqref{defnOfY5Shading}), as desired. 

It remains to obtain Item (iii) from Definition \ref{twoScaleGrainsDecomp}. By dyadic pigeonholing, there is a number $\mu$; a set $\tubes_{\tilde\rho}'$; and an $\approx_\delta1$ refinement $(\tubes_6,Y_6)_\delta$ of $(\tubes_5,Y_5)_\delta$ so that the following holds:
\begin{itemize}
	\item $\tubes_{\tilde\rho}'$ is a balanced partitioning cover of $\tubes_6$.
	\item For each $T_{\tilde\rho}\in\tubes_{\tilde\rho}'$ and each $x\in\bigcup_{T\in\tubes_6[T_{\tilde\rho}]}Y_6(T)$, we have $\#\big((\tubes_6[T_{\tilde\rho}])_{Y_6}(x)\big)\sim\mu$.
\end{itemize}

Let $\tilde{\mathcal{P}}_6 = \bigcup_{T_{\tilde\rho}\in\tubes_{\tilde\rho}'}(\tilde{\mathcal{P}}_5)_{T_{\tilde\rho}}$ and let $Y_6(\tilde P)\subset Y_5(\tilde P)$ be the shading so that \eqref{unionShadingY4EqualsY5} continues to hold with $(\tubes_6,Y_6)_\delta$ in place of  $(\tubes_5,Y_5)_\delta$, and $(\tilde{\mathcal{P}}_6,Y_6)_{\tilde a\times\tilde\rho\tilde c\times\tilde c}$ in place of $(\tilde{\mathcal{P}}_5,Y_5)_{\tilde a\times\tilde\rho\tilde c\times\tilde c}$.

The triple $(\tubes_6,Y_6)_\delta$, $(\tilde{\mathcal{P}}_6,Y_6)_{\tilde a\times\tilde\rho\tilde c\times\tilde c}$, and $\tubes_{\tilde\rho}$ continue to satisfy Items (i), (ii), and (iv) from Definition \ref{twoScaleGrainsDecomp}. To verify Item (iii), we need to estimate the density of the shading on $(\tilde{\mathcal{P}}_6,Y_6)_{\tilde a\times\tilde\rho\tilde c\times\tilde c}$. From \eqref{volumeLowerBoundSumOfGrains}, we have
\begin{equation}\label{upperBdSumVolumePrisms}
\sum_{\tilde P \in\tilde{\mathcal{P}}_6}|\tilde P|\lessapprox_\delta \delta^{-4\eta-4\eps_1}\mu_i^{-1}\sum_{T\in\tubes_2}|Y_2(T)|.
\end{equation}

Note that $\mu\lesssim\mu_i \delta^{-2\eps_1}$ (recall that $\mu_i$ is the multiplicity associated to scale $\tau_i$, which was chosen in Step 6). Thus we can compute
\begin{equation}\label{lowerBdOnMassOfShading}
\begin{split}
\sum_{\tilde P \in\tilde{\mathcal{P}}_6}|Y_6(\tilde P)| & = \sum_{T_{\tilde\rho}\in\tubes_{\tilde\rho}'}\Big| \bigsqcup_{\tilde P \in(\tilde{\mathcal{P}}_6)_{T_{\tilde\rho}}}Y_6(\tilde P)\Big|\\
& = \sum_{T_{\tilde\rho}\in\tubes_{\tilde\rho}'}\Big| \bigcup_{T\in\tubes_6[T_{\tilde\rho}]}Y_6(T)\Big|\\
& \gtrsim \sum_{T_{\tilde\rho}\in\tubes_{\tilde\rho}'}\Big(\mu^{-1}\sum_{T\in\tubes_6[T_{\tilde\rho}]}|Y_6(T)|\Big)\\
& \gtrsim \mu_i^{-1} \delta^{2\eps_1} \sum_{T\in\tubes_6}|Y_6(T)|\\
& \gtrapprox_\delta \mu_i^{-1}\delta^{\eps_2+2\eps_1} \sum_{T\in\tubes_6}|Y_2(T)|.
\end{split}
\end{equation}

Comparing \eqref{upperBdSumVolumePrisms} and \eqref{lowerBdOnMassOfShading}, we conclude that $(\tilde{\mathcal{P}}_6,Y_6)_{\tilde a\times\tilde\rho\tilde c\times\tilde c}$  is $\gtrapprox_\delta\delta^{4\eta+6\eps_1+\eps_2}\geq\delta^{2\eps_2}$ dense. We now select $\eps_2$ sufficiently small (depending on $\eps$) so that this quantity is $\geq\delta^{\eps}$. We conclude that Conclusion (C) of Lemma \ref{squareGrainsGetLonger} holds.

This concludes the proof of Lemma \ref{squareGrainsGetLonger}, except that we must still prove the Claim stated at the beginning of Step 7. We will do this below.

\medskip

\noindent {\bf Step 8.} 
Our final task is to prove the Claim from Step 7. For notational convenience, we will abuse notation and rename the set $(\tilde{\mathcal P}_4,Y_4)_{\tilde a \times  \tilde\rho\tilde c\times\tilde c}$ as $(\tilde{\mathcal P},Y)_{\tilde a \times  \tilde\rho\tilde c\times\tilde c}$. Informally, the idea is as follows: If the shadings $\{Y(\tilde P)\colon \tilde P\in\tilde{\mathcal{P}}_{T_{\tilde \rho}}\}$ have small overlap, then we can refine these these shadings to be disjoint. On the other hand, if the shadings have large overlap, then since the prisms in $\tilde{\mathcal{P}}_{T_{\tilde \rho}}$ are essentially distinct and all satisfy $\angle(\operatorname{dir}(\tilde P), \dir(T_{\tilde \rho}))\lesssim \tilde \rho$, we have that the prisms in $\tilde{\mathcal{P}}_Y(x)$  (i.e. the prisms passing through a typical point) must have differing tangent planes (i.e. there must exist prisms $\tilde P,\tilde P'\in \tilde{\mathcal{P}}_Y(x)$ for which $\angle(\Pi(\tilde P), \Pi(\tilde P'))$ is large). We then apply Lemma \ref{OnePrismWithTangentPlanes} to show that the thickened neighbourhood of a typical prism in $\tilde{\mathcal P}$ has large intersection with $\bigcup_{\tilde P\in\tilde{\mathcal{P}}}Y(\tilde P)$, and this in turn means that the thickened neighbourhood of a typical tube in $\tubes_4$ has large intersection with $\bigcup_{\tubes}Y(T)$. By Corollary \ref{corOfInflateTubesToSlabsLem}, this yields Conclusion (A) of Lemma \ref{squareGrainsGetLonger}. We now turn to the details.

Using Lemma \ref{everySetHasARegularShadingLem} (every shading has a regular sub-shading), we may select a $\gtrsim 1$ refinement $(\tilde{\mathcal P}',Y')_{\tilde a \times  \tilde\rho\tilde c\times\tilde c}$ with the property that each shading $Y'( \tilde P),\ \tilde P\in \tilde{\mathcal{P}}'$ is regular (recall Definition \ref{regularShading}) and satisfies $|Y'(\tilde P)|\gtrapprox_\delta  \delta^{4\eta+4\eps_1}|\tilde P|$. 

After dyadic pigeonholing, we may suppose there exists a number $\nu$ and a $\gtrapprox_\delta$ refinement $(\tilde{\mathcal P}'',Y'')_{\tilde a \times  \tilde\rho\tilde c\times\tilde c}$ of $(\tilde{\mathcal P}',Y')_{\tilde a \times \tilde\rho\tilde c\times\tilde c}$ so that for each $T_{\tilde \rho}\in\tubes_{\tilde \rho}$ and each point $x\in \bigcup_{\tilde P\in(\tilde{\mathcal{P}}'')_{T_{\tilde \rho}}}Y''(\tilde P)$, we have 
$\#((\tilde{\mathcal{P}}''_{T_{\tilde \rho}})_{Y''}(x))\sim\nu$. 

First, we will consider the case where
\begin{equation}\label{nuGeqDeltaEps2}
\nu \geq \delta^{-\eps_2}.
\end{equation}
We will show that Conclusion (A) of Lemma \ref{squareGrainsGetLonger} is true for a suitably chosen value of $\alpha$.  Observe that the prisms in $((\tilde{\mathcal{P}}')_{T_{\tilde \rho}})_{Y'}(x)$ are essentially distinct, and they all satisfy $\angle(\dir(T_{\tilde \rho}),\dir(\tilde P))\leq 2\tilde \rho$. Furthermore, they all (by definition) pass through the common point $x$. Thus for each point $x\in \bigcup_{\tilde P\in(\tilde{\mathcal{P}}'')_{T_ {\tilde \rho}}}Y''(\tilde P)$, there must exist a pair of prisms $\tilde P,\tilde P'$ from this set with 
\[
\angle\big(\Pi(\tilde P^{T_{\tilde \rho}}),\ \Pi\big((\tilde P')^{T_{\tilde \rho}}\big)\big)\gtrsim \nu^{1/2}   (\tilde a/(\tilde \rho \tilde c)).
\]
(For comparison, $\Pi(\tilde P^{T_{\tilde \rho}})$ and $\Pi\big((\tilde P')^{T_{\tilde \rho}}\big)$ are defined up to uncertainty $\tilde a/ ( \tilde \rho \tilde c )$ ).

From the above discussion, we see that for each $T_{\tilde \rho}\in\tubes_{\tilde \rho}$, each $\tilde P_0\in\tilde{\mathcal{P}}''_{T_{\tilde \rho}}$ and each $x\in Y''( \tilde P_0)$, we have
\[ 
 \sup_{\tilde P\in ((\tilde{\mathcal{P}}'_{T_{\tilde \rho}})_{Y'}(x)) }\angle\big(\Pi(\tilde P_0^{T_{\tilde \rho}}),\Pi(\tilde P^{T_{\tilde \rho}})\big) \gtrsim  \nu^{1/2} \tilde a/(\tilde \rho \tilde c),
\]
and thus
\[ 
\inf_{x\in Y''( \tilde P_0)} \sup_{\tilde P\in ((\tilde{\mathcal{P}}'_{T_{\tilde \rho}})_{Y'}(x)) }\angle\big(\Pi(\tilde P_0^{T_{\tilde \rho}}),\Pi( \tilde P^{T_{\tilde \rho}})\big) \gtrsim  \nu^{1/2} \tilde a/(\tilde \rho \tilde c).
\]
But this is precisely the condition we need to apply Lemma \ref{OnePrismWithTangentPlanes} with $\lambda\approx_\delta \delta^{4\eta+4\eps_1}$. Let $\tilde{\mathcal{P}}'''$ be the set of those prisms $\tilde P_0\in\tilde{\mathcal{P}}''$ satisfying $|Y''(\tilde P_0)|\gtrapprox_\delta \delta^{4\eta+4\eps_1}$. Undoing the scaling, we conclude that for each $\tilde P_0\in\mathcal{P}'''$ we have
\begin{equation}\label{nbhdMostlyFullOnePrismApplication}
\Big|N_{\nu^{1/2} \tilde a}(\tilde P_0) \cap \bigcup_{\tilde P\in\tilde{\mathcal{P}}}Y(\tilde P)\Big| 
\gtrapprox_\delta \delta^{16\eta+16\eps_1}|N_{\nu^{1/2} \tilde a}(\tilde P_0)|.
\end{equation}

But this means that after refining $(\tubes_4,Y_4)_\delta$ by an $\approx_\delta 1$ factor, there is a pair $(\tubes_4',Y_4')_\delta$ so that for each $x\in\bigcup_{T\in\tubes_4'}Y_4'(T)$, we have 
\begin{equation}\label{tubesIntersectWithBallLowerBd}
\Big|B(x, \nu^{1/2}\tilde a)\cap\bigcup_{T\in\tubes}Y(T)\Big|\gtrapprox_\delta \delta^{O(\eta+\eps_1)}|B(x, \nu^{1/2}\tilde a )|. 
\end{equation}
By Corollary \ref{corOfInflateTubesToSlabsLem} (and using \eqref{nuGeqDeltaEps2}), we conclude that Conclusion (A) holds, provided $\alpha\leq \boldsymbol{\omega}\eps_2/2$, and provided $\eps_1$ and $\eta$ are selected sufficiently small (depending on $\boldsymbol{\omega},\eps_2,$ and the implicit constant on the RHS of \eqref{tubesIntersectWithBallLowerBd}).

Finally, we will consider the case where \eqref{nuGeqDeltaEps2} fails, i.e.
\begin{equation}\label{nuLeqDeltaEps2}
\nu \leq \delta^{-\eps_2}.
\end{equation}
This means that for each $T_{\tilde\rho}\in\tubes_{\tilde\rho}$, the sets $\{Y''(\tilde P)\colon \tilde P \in (\tilde{\mathcal{P}}'')_{T_{\tilde \rho}}\}$ are $\leq \delta^{-\eps_2}$ overlapping. By pigeonholing, we can select a refinement $(\tilde{\mathcal{P}}_5,Y_5)_{\tilde a\times\tilde\rho\tilde c\times\tilde c}$ of $(\tilde{\mathcal{P}}_4,Y_4)_{\tilde a\times\tilde\rho\tilde c\times\tilde c}$ satisfying the four Items listed in Step 7. 
\end{proof}


\subsection{Move \#3: Replacing grains with wider grains with small $\CKT^{\operatorname{loc}}$}

\begin{lem}\label{WiderGrainsSmallCKT}
We assume the {\bf Common setup for Moves \#1, \#2, \#3: Hypotheses} from Section \ref{movesParallelStructureSec}. Then at least one of the following must hold.

Suppose that $\cE(\boldsymbol{\sigma},\boldsymbol{\omega})$ is true and let $\eps>0$. Then there exists $\alpha,\eta,c>0$ so that the following holds for all $0<\delta\leq \rho\leq 1$, and all $\delta\leq a\leq b\leq c$ with $b/c=\rho$. 

Let $(\tubes,Y)_\delta$ be $\delta^\eta$ dense, with $\CKT(\tubes)\leq\delta^{-\eta}$ and $\FS(\tubes)\leq\delta^{-\eta}$. Let $\tubes_\rho$ be a balanced partitioning cover of $\tubes$, and suppose that $(\tubes,Y)_\delta$ is broad with error $\delta^{-\eta}$ relative to $\tubes_\rho$. Let $(\mathcal{P},Y)_{a \times b\times c}$ be a robustly $\delta^\eta$-dense two-scale grains decomposition of $(\tubes,Y)_\delta$ wrt $\tubes_\rho$.

Then at least one of the following must hold.

\begin{itemize}

\item[(A)] Conclusion (A) of the common setup for Moves \#1, \#2, \#3.

\item[(B)] Conclusion (B) of the common setup for Moves \#1, \#2, \#3. In addition, 
		\begin{itemize}
			\item[(iv)] $\CKT^{\operatorname{loc}}(\mathcal{P}')\leq \delta^{-\boldsymbol{\zeta}}$. 
		\end{itemize}

\item[(C)] Conclusion (C) of the common setup for Moves \#1, \#2, \#3. In addition, 
	\begin{itemize}
		\item[(v)] $\tilde c\geq c$, $\delta^{-\boldsymbol{\zeta}/400}\rho\leq \tilde\rho\leq 1$.
	\end{itemize}

\end{itemize}
\end{lem}

\begin{proof}
\noindent{\bf Step 1.} 
Let $0<\eps_1<\cdots <\eps_4$ be small quantities to be chosen below. We will choose $\eps_i$ very small compared to $\eps_{i+1}$ for each $i=1,\ldots, 3$; we will choose $\eps_4$ very small compared to $\eps$; we will choose $\alpha,\eta$ very small compared to $\eps_1$. 

First, we may suppose that
\begin{equation}\label{aAlmostDeltaStep3}
a\leq \delta^{1-\eps_1},
\end{equation}
or else Conclusion (A) immediately holds, provided we choose $\alpha$ and $\eta$ sufficiently small depending on $\eps_1$. The argument is identical to the argument in Step 1 of the proof of Lemma \ref{squareGrainsGetLonger}.

Next we will regularize the set $\tubes$. By dyadic pigeonholing and replacing $(\tubes,Y)_\delta$ by a $\gtrsim(\log 1/\delta)^{-1/\eps_1}$ refinement $(\tubes_1,Y_1)_\delta$, we can suppose that 

\begin{itemize}
	\item[(a)] 
		For each scale of the form $\tau_i=\delta^{\eps_1 i}$, $i=1,\ldots,\eps_1^{-1}$, there exists a ``density'' $\lambda_i$ so that
		\begin{equation}
		\Big| B(x,\tau_i)\cap\bigcup_{T\in\tubes}Y(T)\Big| \sim \lambda_i |B(x,\tau_i)|\quad\textrm{for every}\ x\in\bigcup_{T\in\tubes_1}Y_1(T).
		\end{equation} 

	\item[(b)] For each $i=1,\ldots,\eps_1^{-1}$, there is a pair $(\tubes_{\tau_i},Y_{\tau_i})_{\tau_i}$ that is $\gtrapprox_\delta \delta^{\eta}$ dense. Furthermore, $\tubes_{\tau_i}$ is a balanced partitioning cover of $\tubes_1$; and for each $T_{\tau_i}$ we have
	\[
		Y_{\tau_i}(T_{\tau_i})\subset T_{\tau_i}\cap \bigcup_{T\in\tubes_1}N_{\tau_i}(Y_1(T)).
	\]
\end{itemize}

From the above items, we have that $\FS(\tubes_{\tau_i})\lesssim\FS(\tubes_1)\lesssim (\log 1/\delta)^{-1/\eps_1}\delta^{-\eta}$, and 
\[
\CKT(\tubes_{\tau_i})\lesssim \CKT(\tubes_1)\frac{\#\tubes_{\tau_i}}{\#\tubes_1}\frac{|T_{\tau_i}|}{|T|}.
\]
If $\eta>0$ is selected sufficiently small depending on $\eps_1$, then we can apply the estimate $\cE(\boldsymbol{\sigma},\boldsymbol{\omega})$ (with $\eps_1$ in place of $\eps$) to conclude that
\begin{equation}\label{tubesVolumeBdDensity}
\begin{split}
\Big|\bigcup_{T\in\tubes}Y(T)\Big| & \gtrsim \delta^{2\eps_1} \lambda_i \tau_i^{\boldsymbol{\omega}} \CKT(\tubes_{\tau_i})^{-1}(\#\tubes_{\tau_i})|T_{\tau_i}|\Big(\CKT(\tubes_{\tau_i})^{-3/2}\FS(\tubes_{\tau_i}) (\#\tubes_{\tau_i})|T_{\tau_i}|^{1/2} \Big)^{-\boldsymbol{\sigma}}\\
& \gtrapprox_{\delta} \delta^{2\eps_1+ O(\eta)}   \Big[\lambda_i\Big(\frac{\tau_i}{\delta}\Big)^{\boldsymbol{\omega}}  \Big(\frac{|T_{\tau_i}|  (\#\tubes_{\tau_i})^{1/2}}{|T| (\#\tubes)^{1/2}} \Big)^{\boldsymbol{\sigma}} \Big] 
\delta^{\boldsymbol{\omega}}(\#\tubes)|T| \Big((\#\tubes)|T|^{1/2}\Big)^{-\boldsymbol{\sigma}}\\
& \gtrapprox \delta^{3\eps_1}  \Big[\lambda_i\Big(\frac{\tau_i}{\delta}\Big)^{\boldsymbol{\omega}} \frac{|T_{\tau_i}|^{\boldsymbol{\sigma}/2}}{|T|^{\boldsymbol{\sigma}/2}}\Big] 
\delta^{\boldsymbol{\omega}}(\#\tubes)|T| \Big((\#\tubes)|T|^{1/2}\Big)^{-\boldsymbol{\sigma}}.
\end{split}
\end{equation}
For the last inequality, we used the fact that $\CKT(\tubes_1)\lessapprox_{\delta} \delta^{-\eta}$, and so $(\#\tubes)/(\#\tubes_{\tau_i}) \lessapprox_{\delta} \delta^{-\eta} |T_{\tau_i}|/|T|$. 
In particular, if there is an index $i$ for which  $\lambda_i\Big(\frac{\tau_i}{\delta}\Big)^{\boldsymbol{\omega}} \frac{|T_{\tau_i}|^{\boldsymbol{\sigma}/2}}{|T|^{\boldsymbol{\sigma}/2}}$ 
is substantially larger than 1, then we will obtain Conclusion (A) of Lemma \ref{WiderGrainsSmallCKT}.

\medskip

\noindent \noindent{\bf Step 2.} 
Let $\mathcal{P}_1=\mathcal{P}$. For each $P\in\mathcal{P}_1$, define $Y_1(P) = Y(P)\cap\bigcup_{T\in\tubes_1[T_\rho]}Y_1(T)$, where $T_\rho\in\tubes_\rho$ is the unique $\rho$ tube with $P\in T_\rho$ and $\angle(\dir(P),\dir(T_\rho))\leq 2\rho$. Since $(\tubes_1,Y_1)_\delta$ is an $\approx_\delta 1$ refinement of $(\tubes,Y)_\delta$, by Definition \ref{twoScaleGrainsDecomp} Items  (ii) and (iii) we have that $(\mathcal{P}_1,Y_1)_{a\times b\times c}$ is a $\approx_\delta 1$ refinement of $(\mathcal{P},Y)_{a\times b\times c}$.

Apply Lemma \ref{everySetHasARegularShadingLem} (every shading has a regular sub-shading) to each shading $Y_1(P),\ P\in\mathcal{P}_1$. This gives us a regular sub-shading $Y_2(P)\subset Y_1(P)$. Let $\mathcal{P}_2\subset\mathcal{P}_1$ be those prisms for which $|Y_2(P)|\geq\delta^{2\eta}|P|$; we have that $(\mathcal{P}_2,Y_2)_{a\times b\times c}$ is a $\gtrsim 1$ refinement of $(\mathcal{P}_1,Y_1)_{a\times b\times c}$.

Let $\mathcal{P}_3=\mathcal{P}_2$. By dyadic pigeonholing, we can select a number $\theta_0\in [\frac{a}{b},1]$ and a $(\log 1/\delta)^{-1}$ refinement $(\mathcal{P}_3,Y_3)_{a\times b\times c}$ of $(\mathcal{P}_2,Y_2)_{a\times b\times c}$ so that for each $x\in\bigcup_{P\in\mathcal{P}_3}Y_3(P)$, we have $\theta(x)\sim\theta_0$, where $\theta(x)$ is as defined in Definition \ref{thetaMinDefn}.

We first consider the case where $\theta_0\geq \delta^{-\eps_1}(a/b)$. Our goal is to show that Conclusion (A) holds. Let $\mathcal{P}_3'\subset\mathcal{P}_3$ be the set of those prisms for which $|Y_3(P_0)|\geq \frac{1}{100}\delta^{2\eta}|P_0|$. Then for each $x\in Y_3(P_0)$, we have

\[
\frac{a}{b} + \sup_{P\in\mathcal{P}_2}\angle(\Pi(P_0),\Pi(P))\gtrsim \theta_0.
\]
We have $|Y_3(P_0)|\geq \frac{1}{100}\delta^{2\eta}|P_0|$; each $P\in\mathcal{P}_2$ satisfies $|Y_2(P)|\geq \frac{1}{100}\delta^{2\eta}|P_0|$; and $Y_2(P)$ is regular. Thus we can apply Lemma \ref{OnePrismWithTangentPlanes} (with $Y_3(P_0)$ in place of $Y_0(P_0)$ and $(\mathcal{P}_2,Y_2)_{a\times b\times c}$ in place of $(\mathcal{P},Y)_{a\times b\times c}$) to conclude that
\begin{equation}\label{thickenedNeighbourhoodDense}
\Big|N_{b \theta_0}(P_0) \cap \bigcup_{P\in\mathcal{P}}Y(P)\Big| \gtrapprox_\delta \delta^{8\eta} |N_{b \theta}(P_0)|.
\end{equation}

Recall that \eqref{thickenedNeighbourhoodDense} holds for each $P_0\in\mathcal{P}_3'$, and $(\mathcal{P}_3', Y_3)_{a\times b\times c}$ is a $\gtrapprox_\delta 1$ refinement of $(\mathcal{P},Y)_{a\times b\times c}$. After replacing $(\mathcal{P}_3', Y_3)_{a\times b\times c}$ by a further $\sim 1$ refinement $(\mathcal{P}_3', Y_3')_{a\times b\times c}$, we can suppose that for each $x\in \bigcup_{P\in\mathcal{P}_3'}Y'_3$, we have
\[
\Big|N_{b \theta_0}(x) \cap \bigcup_{P\in\mathcal{P}}Y(P)\Big| \gtrapprox_\delta \delta^{8\eta} |N_{b \theta_0}(x)|.
\]

Finally, if $(\tubes_1,Y_1')_\delta$ is the refinement of $(\tubes_1,Y_1)_\delta$ induced by $(\mathcal{P}_3',Y_3')_{a\times b\times c}$, then by Definition~\ref{twoScaleGrainsDecomp}, Item (ii), $(\tubes_1,Y_1')_\delta$ is a  $\approx_\delta 1$-refinement of $(\tubes_1,Y_1)_\delta$, and  for each $x\in\bigcup_{T\in\tubes_1'}Y_1'(T)$ we have
\begin{equation}\label{thickenedNeighbourhoodDenseBall}
\Big|N_{b \theta_0}(x) \cap \bigcup_{T\in\tubes_1'}Y(T)\Big| \gtrapprox_\delta \delta^{8\eta} |N_{b \theta_0}(x)|.
\end{equation}

Since $b\theta_0\geq \delta^{-\eps_1}a$, by Corollary \ref{corOfInflateTubesToSlabsLem} we see that Conclusion (A) holds, provided we select $\eta>0$ sufficiently small depending on $\eps_1$, and select  $\alpha\leq \eps_1 \boldsymbol{\omega}/2$.

Henceforth we shall suppose that $\theta_0\leq \delta^{-\eps_1}(a/b)$, i.e.
\begin{equation}
\sup_x \sup_{P,P'\in (\mathcal{P}_3)_{Y_3}(x)}\angle\big(\Pi(P),\ \Pi(P')\big) \leq \delta^{-\eps_1}(a/b).
\end{equation}


\medskip

\noindent{\bf Step 3.}
Recall the discussion following Definition \ref{defnPSquare} (the quantity ``$K$'' in that discussion is $\delta^{-\eps_1}$ in this context); after replacing $(\mathcal{P}_3,Y_3)_{a\times b\times c}$ by a $\sim 1$ refinement, which we will denote by $(\mathcal{P}_4,Y_4)_{a\times b\times c}$ (which in turn induces a $\sim 1$ refinement $(\tubes_4,Y_4)_\delta$ of $(\tubes_1,Y_1)_\delta)$, we can find a set $\mathcal{U}$ of pairwise distinct prisms of dimensions $\delta^{-\eps_1}\frac{ac}{b} \times c\times c$ so that the following holds.
\begin{itemize}
	\item[(a)] The sets $\{\mathcal{P}_4\langle U\rangle\colon U\in \mathcal{U}\}$ are a partition  of $\mathcal{P}_4$.
	\item[(b)] The sets $\big\{\bigcup_{P\in\mathcal{P}_4\langle U\rangle }Y_4(P)\colon  U\in\mathcal{U}\big\}$ are disjoint.
\end{itemize}

Each $U\in\mathcal{U}$ is a prism of dimensions $\delta^{-\eps_1}\frac{ac}{b} \times c\times c=\delta^{-\eps_1}\frac{a}{\rho} \times c\times c$. Thus for each $U\in\mathcal{U}$, there is a set $\{Z\}$ of $\lesssim\delta^{-3\eps_1}$ prisms of dimensions comparable to $\frac{a}{\rho}\times c\times c$ with the property that for each $P\in\mathcal{P}_4\langle U\rangle$, there is a prism $Z$ from this collection with $\square(P)\subset Z$ (recall that $\rho = b/c$, and thus $\square(P)$ is a prism of dimensions comparable to $\frac{a}{\rho}\times c\times c$). 

Let $Z_U$ be a prism of dimensions comparable to $\frac{a}{\rho}\times c\times c$ that maximizes 
\[
\#\{ P\in \mathcal{P}_4\colon \square(P)\subset Z\},
\]
so in particular $\#\mathcal{P}_4\langle Z_U\rangle \gtrsim\delta^{3\eps_1}(\#\mathcal{P}_4\langle U\rangle )$. Let $\mathcal{Z} = \{Z_U\colon U\in\mathcal{U}\}$; let $\mathcal{P}_5=\bigcup_{U\in\mathcal{U}}\mathcal{P}_4\langle Z_U\rangle $; and let $Y_5$ be the restriction of $Y_4$ to $\mathcal{P}_5$. Then $(\mathcal{P}_5,Y_5)_{a\times b\times c}$ is a $\gtrsim\delta^{3\eps_1}$ refinement of $(\mathcal{P}_4,Y_4)_{a\times b\times c}$, and we have the following analogue of Items (a) and (b).

\begin{itemize}
	\item[(a$'$)] The sets $\{\mathcal{P}_4\langle Z\rangle\colon Z\in \mathcal{Z}\}$ become a partition  of $\mathcal{P}_5$.
	\item[(b$'$)] The sets $\big\{\bigcup_{P\in\mathcal{P}_5\langle Z\rangle }Y_5(P)\colon Z\in\mathcal{Z}\big\}$ are disjoint.
\end{itemize}

For each $Z\in\mathcal{Z}$, the sets in $(\mathcal{P}_5\langle Z\rangle)^Z$ are convex sets of dimensions comparable to $\rho\times\rho\times 1$, i.e. the sets are comparable to $\rho$ tubes. To record this useful fact, we will define $(\tilde\tubes_Z, \tilde Y_5)_{\rho} = ((\mathcal{P}_5\langle Z\rangle)^Z,Y_5^Z)_{\rho\times\rho\times 1}$.

After replacing $(\mathcal{P}_5,Y_5)_{a\times b\times c}$ and $\mathcal{Z}$ with $\approx_\delta 1$ refinements, we may suppose that each set $\tilde\tubes_Z$ has approximately the same size (up to a factor of 2) for each $Z\in\mathcal{Z}$, and similarly each set $|\tilde Y_5(\tilde T)|$ has approximately the same size for each $\tilde T\in\tilde\tubes_Z$. Furthermore, we can suppose that each pair  $(\tilde\tubes_Z, \tilde Y_5)_{\rho}$ is $\gtrapprox_\delta \delta^{\eta}$ dense (indeed, recall that $(\mathcal{P},Y)_{a\times b\times c}$ is $\delta^\eta$ dense;  $(\mathcal{P}_4,Y_4)_{a\times b\times c}$  is a $\gtrapprox_{\delta}1$-refinement of $(\mathcal{P},Y)_{a\times b\times c}$; and $(\mathcal{P}_5\langle Z\rangle,Y_5)_{a\times b\times c}$ is a $\approx_\delta 1$ refinement of $(\mathcal{P}_4  \langle Z\rangle,Y_4)_{a\times b\times c}$).

\medskip

\noindent{\bf Step 4.}
For notational convenience, we will fix a prism $Z\in\mathcal{Z}$. In what follows, we will find certain quantities (for example certain scales, multiplicities, etc.), and navigate between different cases depending on the specifics of the arrangement $\tilde\tubes_Z$. However, by pigeonholing the set $\mathcal{Z}$, we may suppose that all quantities described below are the same (up to a factor of 2) for each $Z\in\mathcal{Z}$, and thus the same cases occur for each $Z\in\mathcal{Z}$.

 Apply Proposition \ref{factoringConvexSetsProp} (factoring convex sets) to $\tilde\tubes_Z$. We obtain a number $m\geq 1$; a $\approx_\delta 1$ refinement $\tilde\tubes_Z'$ of $\tilde\tubes_Z$, and a partitioning cover $\mathcal{W}_Z$ of $\tilde\tubes_Z'$ consisting of congruent prisms; we shall denote the dimensions of these prisms by $s\times t\times 1$ (since each prism contains at least one tube, we know that the longest dimension is $\sim 1$). We have that $\mathcal{W}_Z$ factors $\tilde\tubes_Z'$ from below with respect to the Frostman Convex Wolff Axioms, and from above with respect to the Katz-Tao Convex Wolff Axioms, both with error $\lessapprox_\delta 1$. Finally,

\begin{equation}\label{tubesInW}
\CKT(\tilde\tubes'_Z)\leq m,\quad\textrm{and}\quad \#\tilde\tubes_Z'[W]\approx_\delta m \frac{|W|}{|\tilde T|}\quad\ \textrm{for each}\ W\in\mathcal{W}_Z.
\end{equation}

We first consider the case where 
\begin{equation}\label{mSmall}
m\leq\delta^{-\boldsymbol{\zeta}}.
\end{equation}
Our goal is to show that Conclusion (B) of Lemma \ref{WiderGrainsSmallCKT} holds. 

As described at the beginning of Step 4, we can suppose that \eqref{mSmall} is true for at least half the prisms $Z\in\mathcal{Z}$. Let $\mathcal{P}_6\subset\mathcal{P}_5$ be given by $\mathcal{P}_6=\bigcup_Z \phi_Z^{-1}(\tilde\tubes_Z')$, where the union is taken over those prisms in $\mathcal{Z}$ for which \eqref{mSmall} holds (note that $\mathcal{P}_6\subset\mathcal{P}_5$, since each tube in $\tilde \tubes_Z'\subset \tilde \tubes_Z$ is the image of a prism from $\mathcal{P}_5$ under the map $\phi_Z$). Let $Y_6$ be the restriction of $Y_5$ to $\mathcal{P}_6$.

We have that $(\mathcal{P}_6,Y_6)_{a\times b\times c}$ is a $\gtrapprox_\delta 1$ refinement of $(\mathcal{P}_5,Y_5)$. Undoing the scaling $\phi_Z$, 
 we have that for each $P\in\mathcal{P}_6$, 
\[
\CKT\big(\mathcal{P}_6\langle 2\square(P)\rangle \big) \lesssim \sup_{Z\in\mathcal{Z}}\CKT(\mathcal{P}_6\langle Z\rangle)\leq m.
\]
We conclude that 
\[
\CKT^{\operatorname{loc}}(\mathcal{P}_6)\lesssim m.
\]

Applying a final dyadic pigeonholing, we can select a $\approx_\delta 1$ refinement $(\mathcal{P}',Y')_{a\times b\times c}$ of $(\mathcal{P}_6,Y_6)_{a\times b\times c}$ (this in turn induces a   $\gtrapprox_{\delta} \delta^{3\eps_1}$ refinement $(\tubes',Y')_\delta$ of $(\tubes_4,Y_4)_\delta$) and a set $\tubes_\rho'\subset\tubes_\rho$ so that the following holds: $\tubes_\rho'$ is a balanced partitioning cover of $(\tubes',Y')_\delta$, and $(\mathcal{P}',Y')_{a\times b\times c}$ is a robustly $\approx_\delta \delta^{3\eps_1}$-dense two-scale grains decomposition of $(\tubes',Y')_\delta$ wrt $\tubes'_\rho$. 

Since $\CKT^{\operatorname{loc}}(\mathcal{P}')\lesssim m$, Conclusion (B), Item (iv) of Lemma \ref{WiderGrainsSmallCKT} is satisfied. $(\mathcal{P}',Y')_{a\times b\times c}$ satisfies Conclusion (B), Item (iii) by construction. $(\tubes',Y')_\delta$ satisfies Conclusion (B), Item (i), provided $\eps_1$ is chosen sufficiently small depending on $\eps$.  Finally, since 
\[
\#\big(\tubes'[T_\rho]_{Y'}(x)\big)\gtrapprox_\delta\delta^{3\eps_1} \#(\tubes[T_\rho])_Y(x)\quad\textrm{for every}\ T_\rho\in\tubes_\rho'\ \textrm{and every}\ x\in\bigcup_{T\in\tubes'[T_\rho]}Y'(T),
\] 
and since (by hypothesis) $(\tubes,Y)_\delta$ is broad with error $\delta^{-\eta}$ relative to $\tubes_\rho$, we conclude that $(\tubes',Y')_\delta$ is broad with error $\lessapprox_\delta\delta^{-\eta-3 \eps_1}$ relative to $\tubes_\rho'$. We will select $\eta$ and $\eps_1$ sufficiently small so that this quantity is $\leq\delta^{-\eps}$. This verifies Conclusion (B), Item (ii) \footnote{To be precise, we can only ensure that the error is $\leq\delta^{-\eps}$ provided $\delta>0$ is sufficiently small, depending on the implicit constant in the above $\gtrapprox$ notation. However if this fails then Conclusion (A) holds, provided we select $\kappa>0$ sufficiently small.}.

In summary, if \eqref{mSmall} holds, then Conclusion (B) of Lemma \ref{WiderGrainsSmallCKT} is satisfied. Henceforth we will suppose that \eqref{mSmall} fails, i.e.
\begin{equation}\label{mLarge}
m > \delta^{-\boldsymbol{\zeta}}.
\end{equation}


\medskip

\noindent{\bf Step 5.}
In Step 4, we fixed a choice of prism $Z\in\mathcal{Z}$. We will continue to fix this choice of $Z$, and in additional we will fix a choice of $W\in\mathcal{W}_Z$. As in Step 4, we can assume (by dyadic pigeonholing) that all relevant scales, multiplicities, etc.~are approximately the same (up to a factor of 2) for each $Z\in\mathcal{Z}$ and each $W\in\mathcal{W}_Z$). 

In the arguments that follow, we will analyze the pairs $(\tilde\tubes'_Z[W],\tilde Y_5)_\rho$ constructed in Step 4. For notational convenience, we will refer to such a pair as $(\tilde\tubes,\tilde Y)_\rho$. Recall that this pair is $\gtrapprox_\delta \delta^{\eta}$ dense; the cardinality of $\tilde\tubes$ is given by \eqref{tubesInW}; and $m$ satisfies \eqref{mLarge}. In particular, we have
\begin{equation}\label{CFCTildeTubes}
\CFC(\tilde\tubes^W)\lessapprox_\delta 1.
\end{equation}

Apply Corollary \ref{coveringTubesBroadCor} (finding a broad scale) to the set $\tilde\tubes$. Denote the ``output'' scale of this Corollary by $\tau$ (in Corollary \ref{coveringTubesBroadCor}, this output scale is called $\rho$, but that variable is already in use). Abusing notation, we will continue to use $(\tilde\tubes,\tilde Y)_\rho$ to refer to the output of Corollary \ref{coveringTubesBroadCor}. Thus there is a set $\tubes_\tau$ that forms a balanced partitioning cover of $\tilde\tubes$, and $(\tilde\tubes,\tilde Y)_\rho$ is broad with error $\lessapprox_\delta 1$ relative to $\tubes_{\tau}$. Furthermore, 
\begin{equation}\label{boundedOverlapSets}
\textrm{the sets}\ \bigcup_{\tilde T\in\tilde\tubes[T_\tau]}\tilde Y(\tilde T), \quad T_\tau\in\tubes_{\tau}\quad \textrm{are}\ \lesssim\tau^{-\boldsymbol{\beta}}\ \textrm{overlapping}.
\end{equation}

We claim that
\begin{equation}\label{tauNotTooSmall}
\tau\geq \rho^{1-\boldsymbol{\omega}/8},
\end{equation}
or else Conclusion (A) of Lemma \ref{WiderGrainsSmallCKT} holds, and we are done. To verify this claim, note that if \eqref{tauNotTooSmall} failed, then the sets $\{\tilde Y(\tilde T)\colon \tilde T\in\tilde\tubes\}$ are $\lesssim \rho^{-\boldsymbol{\omega}/4}\tau^{-\boldsymbol{\beta}}\leq\delta^{-\boldsymbol{\omega}/2}$ overlapping; but this fact, combined with Item (b$'$) from Step 3, gives Conclusion (A) of Lemma \ref{WiderGrainsSmallCKT}.


\medskip

\noindent{\bf Step 6.}
We would like to apply the estimate $\cE(\boldsymbol{\sigma},\boldsymbol{\omega})$  to each set $(\tilde\tubes^{T_\tau},\tilde Y^{T_\tau})_{\rho/\tau}$. However, we do not currently have a good estimate for $\FS(\tilde\tubes^{T_\tau})$. To fix this problem, apply Proposition \ref{slabWolffFactoring} (factoring convex sets with respect to the Frostman Slab Wolff Axioms) to each set $(\tilde\tubes[T_\tau],\tilde Y)_\rho$, with $\eps_2$ in place of $\eps$. We can do so, provided $\eps_1$ and $\eta$ is selected sufficiently small compared to $\eps_2$. This gives us a $\gtrapprox_\delta \delta^{\eps_2}$ refinement $(\tilde\tubes'[T_\tau], \tilde Y')_\rho$ of $(\tilde\tubes[T_\tau],\tilde Y)_\rho$, and a family of convex subsets of $T_\tau$, which we denote by $\mathcal{V}_{T_\tau}$, that factors $\tilde\tubes'[T_\tau]$ from below with respect to the Frostman Slab Wolff Axioms with error $\delta^{-\eps_2}$. In addition, 
\begin{equation}\label{disjointSetsV}
\textrm{the sets}\ \bigcup_{\tilde T\in\tilde\tubes'[V]}\tilde Y'(\tilde T), \quad V \in\mathcal{V}_{T_\tau}\quad \textrm{are disjoint}.
\end{equation}

After pigeonholing, we may suppose that the sets in $\mathcal{V}_{T_\tau}$ have the same dimensions, and furthermore these dimensions are common across all $T_\tau\in\tubes_\tau$. Denote these dimensions $\theta\times\tau'\times 1$. Since $V\subset T_\tau$, we must have $\tau'\leq\tau$. We claim that this inequality is almost tight, in the sense that
\begin{equation}\label{tauPrimeSize}
\delta^{\eps_2/\boldsymbol{\beta}}\tau\lessapprox_\delta \tau' \leq\tau.
\end{equation}
To verify this claim, note that $(\tilde \tubes,\tilde Y)_\rho$ is broad with error $\lessapprox_\delta 1$ relative to the cover $\tubes_\tau$. Since  $(\tilde\tubes'[T_\tau], \tilde Y')_\rho$ is a $\gtrapprox_\delta \delta^{\eps_2}$ refinement of $(\tilde\tubes[T_\tau], \tilde Y)_\rho$, by pigeonholing there is at least one point $x$ for which the tubes in $\tilde\tubes'_{Y'}(x)$ point in directions that are broad with error $\lessapprox_\delta \delta^{-\eps_2}$ inside a cap of radius $\tau$. This means that there are at least two tubes from this set that make an angle $\gtrapprox_\delta \delta^{\eps_2/\boldsymbol{\beta}}\tau$. On the other hand, by \eqref{disjointSetsV} we have that the pair of tubes described above must be contained in a common $\theta\times\tau'\times 1$ prism. This establishes \eqref{tauPrimeSize}.

Let $\tilde\tubes'$ be the union of the sets $\tilde\tubes'[T_\tau]$, as $T_\tau$ ranges over the elements of $\tubes_\tau$. Let $\tilde Y'$ be the associated shading on $\tilde\tubes'$ coming from the pairs $(\tilde\tubes'[T_\tau], \tilde Y')_\rho$. Abusing notation, we will rename this pair $(\tilde\tubes,\tilde Y)_\rho$. At this point in the argument, this pair is $\gtrapprox_\delta \delta^{\eta+ \eps_2}\geq \delta^{2\eps_2}$ dense.   


\medskip

\noindent{\bf Step 7.}
We first consider the case where the prisms $\mathcal{V}_{T_\tau}$ from Step 6 are almost tubes, in the sense that 
\begin{equation}\label{thetaBig}
\theta\geq \delta^{\boldsymbol{\zeta}/10} \tau'.
\end{equation}
If \eqref{thetaBig} is true, then each set $\tilde\tubes^V$ consists of prisms of dimensions $\frac{\rho}{\tau'}\times\frac{\rho}{\theta}\times 1$. The pair $(\tilde\tubes^V,\tilde Y^V)_{\rho/\tau'\times  \rho/\theta \times 1}$ is $\gtrapprox_\delta\delta^{2\eps_2}$ dense. Let us suppose for the moment that 
\begin{equation}\label{rhoSmallDeltaSqrtEps2}
\rho\leq\delta^{\sqrt \eps_2}.
\end{equation}
If we select $\eps_2$ sufficiently small so that $\eps_2 \leq \frac{1}{2}\sqrt\eps_2 \boldsymbol{\beta}\boldsymbol{\omega}$, and hence $\frac{8\eps_2\boldsymbol{\beta}}{\sqrt\eps_2 \boldsymbol{\beta}\boldsymbol{\omega}-8\eps_2}\leq \frac{16\sqrt\eps_2}{\boldsymbol{\omega}}$, then by \eqref{tauNotTooSmall} and \eqref{tauPrimeSize}, we have $\delta^{2\eps_2}\geq \big(\frac{\rho}{\tau'}\big)^{\frac{32\sqrt{\eps_2}}{\boldsymbol{\omega}}}$.

Thus if $\eps_2$ is chosen sufficiently small compared to $\eps_3$ and $\boldsymbol{\omega}$, and if \eqref{rhoSmallDeltaSqrtEps2} is true, then we can apply the estimate $\cF(\boldsymbol{\sigma},\boldsymbol{\omega})$ (recall Definition \ref{defnF} and Remark \ref{discussionOfDRemark}, Situation 3) to estimate the volume of each (rescaled) set $(\tilde\tubes^V,\tilde Y^V)_{\rho/\tau'\times \rho/\theta \times 1}$, with $\eps_3$ in place of $\eps$ and $32\sqrt{\eps_2}/\boldsymbol{\omega}$ in place of $\eta$. This gives the estimate
\begin{equation}\label{volumeBoundTV}
\Big| \bigcup_{\tilde T\in\tilde\tubes[V]}\tilde Y^V(\tilde T^V)\Big| \gtrsim \delta^{\eps_3}\Big(\frac{\rho}{\theta}\Big)^{\boldsymbol{\omega}} \tilde m^{-1} (\#\tilde\tubes[V])|\tilde T^V|\Big(\tilde m^{-3/2} \tilde \ell (\#\tilde\tubes[V])|\tilde T^V|^{1/2}\Big)^{-\boldsymbol{\sigma}},
\end{equation} 
where we define
\begin{equation}\label{tildeMTildeLDefn}
\tilde m = \CKT(\tilde\tubes[V])\leq \CKT(\tilde\tubes) \leq  m\quad\textrm{and}\quad \tilde\ell = \FS(\tilde\tubes[V])\lesssim \delta^{-\eps_2}.
\end{equation}
(for the first inequality, we used \eqref{tubesInW}). On the other hand, if \eqref{rhoSmallDeltaSqrtEps2} is false, then \eqref{volumeBoundTV} follows from the fact that the LHS of \eqref{volumeBoundTV} is bounded below by the volume of a single shading $|\tilde Y^V(T^V)|$, and we can select a tube with volume 
\begin{equation}\label{volumeOfASingleShading}
|\tilde Y^V(T^V)| \gtrapprox \delta^{2\eps_2}|T^V| = \delta^{2\eps_2}\frac{\rho^2}{\tau'\theta} \geq \delta^{2\eps_2}\rho^2\geq \delta^{3\eps_2}\geq\delta^{\eps_3} .
\end{equation}
Since $\sigma \in (0, 2/3]$, by Remark~\ref{rmk: RHSF} and Remark \ref{discussionOfDRemark}, Situation 3, 
\begin{equation}\label{complicatedBdBy1}
\tilde m^{-1}(\#\tilde\tubes[V])|\tilde T^V|\Big(\tilde m^{-3/2} \tilde \ell (\#\tilde\tubes[V])|\tilde T^V|^{1/2}\Big)^{-\boldsymbol{\sigma}} 
 \lesssim 1.
\end{equation}

Combining \eqref{volumeOfASingleShading} and \eqref{complicatedBdBy1} establishes \eqref{volumeBoundTV} in the case where \eqref{rhoSmallDeltaSqrtEps2} is false. We conclude that \eqref{volumeBoundTV} holds, independently of whether \eqref{rhoSmallDeltaSqrtEps2} is true or false.  

Undoing the scaling $\phi_V$ and substituting the values of $\tilde m$ and $\tilde \ell$ from \eqref{tildeMTildeLDefn}, we have
\begin{equation}\label{boundForTildeTInsideV}
\Big| \bigcup_{\tilde T\in\tilde\tubes[V]}\tilde Y(\tilde T)\Big| 
\gtrsim \delta^{2\eps_3} \Big(\frac{\rho}{\theta}\Big)^{\boldsymbol{\omega}} m^{-1}(\#\tilde\tubes[V])|\tilde T|\Big( m^{-3/2} (\#\tilde\tubes[V])\frac{|\tilde T|^{1/2}}{|V|^{1/2}}\Big)^{-\boldsymbol{\sigma}}.
\end{equation} 

By \eqref{boundedOverlapSets}, \eqref{disjointSetsV} (recall that we have renamed $\tilde Y'$ as $\tilde Y$), and \eqref{boundForTildeTInsideV}, we have
\begin{equation}\label{volumeBdTildeTubes}
\begin{split}
\Big| \bigcup_{\tilde T\in\tilde\tubes}\tilde Y(\tilde T)\Big| 
&\gtrsim \tau^{\boldsymbol{\beta}}\sum_{T_\tau\in\tubes_\tau}\sum_{V\in\mathcal{V}_{T_\tau}}\Big| \bigcup_{\tilde T\in\tilde\tubes[V]}\tilde Y(\tilde T)\Big|\\
& \gtrsim \delta^{2\eps_3+\boldsymbol{\beta}} \Big(\frac{\rho}{\theta}\Big)^{\boldsymbol{\omega}} m^{-1}(\#\tilde\tubes)|\tilde T|\Big( m^{-3/2} (\#\tilde\tubes[V])\frac{|\tilde T|^{1/2}}{|V|^{1/2}}\Big)^{-\boldsymbol{\sigma}}.
\end{split}
\end{equation}
Next,  by \eqref{CFCTildeTubes} and using the fact that $(\tilde \tubes' [T_\tau], Y')_\rho$ is a $\gtrapprox_\delta \delta^{\eps_2}$ refinement of $(\tubes [T_\tau], \tilde Y)_\rho$,  we have that $\CFC(\tilde\tubes^W) \lessapprox_\delta \delta^{-2\eps_2}$.     By \eqref{tubesInW}, $ \#\tilde \tubes \gtrapprox_\delta \delta^{2\eps_2} \Big ( m \frac{|W|}{|\tilde T|} \Big)$.  Since 
\[
\#\tilde\tubes[V]\leq \CFC (\tilde\tubes^W)\frac{|V|}{| W|}(\#\tilde\tubes)\lessapprox_\delta \delta^{-2\eps_2} \frac{|V|}{|W|}\Big( m \frac{| W|}{|\tilde T|}\Big), 
\]
using  
\eqref{mLarge}, \eqref{tauPrimeSize}, \eqref{thetaBig}, and \eqref{volumeBdTildeTubes} we conclude that
\begin{equation}\label{volumeBdTildeTubesMore}
\begin{split}
\Big| \bigcup_{\tilde T\in\tilde\tubes}\tilde Y(\tilde T)\Big| & \gtrapprox_\delta \delta^{2\eps_3+\boldsymbol{\beta}-\boldsymbol{\zeta}\boldsymbol{\sigma}/2} \Big(\frac{\rho}{\theta}\Big)^{\boldsymbol{\omega}} m^{-1}(m \frac{| W|}{|\tilde T|})|\tilde T|\Big( m^{-1} \big(m \frac{|V|}{|\tilde T|}\big)\frac{|\tilde T|^{1/2}}{|V|^{1/2}}\Big)^{-\boldsymbol{\sigma}}\\
&\gtrapprox_\delta \delta^{2\eps_3+\boldsymbol{\beta}-\boldsymbol{\zeta}\boldsymbol{\sigma}/4} \Big(\frac{\rho}{\theta}\Big)^{\boldsymbol{\omega}+\boldsymbol{\sigma}} |W|.
\end{split}
\end{equation}


\medskip

\noindent{\bf Step 8.}
Let us analyze \eqref{volumeBdTildeTubesMore}. First, the set on the LHS of \eqref{volumeBdTildeTubesMore} is contained in $W$, which is a prism of dimensions $s\times t\times 1$. Second, the quantities $\eps_3$ and $\boldsymbol{\beta}$ are chosen after $\boldsymbol{\zeta}$ and $\boldsymbol{\sigma}$, so we can select the former quantities to ensure that $\delta^{-2\eps_3+\boldsymbol{\beta}-\boldsymbol{\zeta}\boldsymbol{\sigma}/4}\geq \delta^{-\boldsymbol{\zeta}\boldsymbol{\sigma}/5}$. 

Thus by pigeonholing, we can select a ball $B$ of radius $\theta$ (recall that $\rho\leq \theta\leq s$) with the property that
\begin{equation}\label{bigBallIntersection}
\Big| B \cap \bigcup_{\tilde T\in\tilde\tubes}\tilde Y(T)\Big| \gtrapprox \delta^{-\boldsymbol{\zeta}\boldsymbol{\sigma}/5}\Big(\frac{\rho}{\theta}\Big)^{\boldsymbol{\omega}+\boldsymbol{\sigma}} |B|.
\end{equation}

Recall that at the beginning of Step 4 we fixed a prism $Z\in\mathcal{Z}$. The set of $\rho$ tubes $\tilde \tubes$ are the images of prisms from $\mathcal{P}_3\langle Z\rangle$ under the linear map $\phi_Z$. Let $B^\dag = \phi_Z^{-1}(B)$, where $B$ is the ball described above. Then $B^\dag$ is an ellipsoid of dimensions $\theta \frac{a}{\rho}\times \theta  c\times \theta  c$. By \eqref{bigBallIntersection}, we have
\begin{equation}\label{bigBDagBallEstimate}
\Big| B^\dag \cap \bigcup_{P \in \mathcal{P}_1}Y_1(P)\Big| \gtrapprox \delta^{-\boldsymbol{\zeta}\boldsymbol{\sigma}/5}\Big(\frac{\rho}{\theta }\Big)^{\boldsymbol{\omega}+\boldsymbol{\sigma}} |B^\dag|.
\end{equation}
Let $1\leq i\leq \eps_1^{-1}$ be the index so that $\delta^{i\eps_1}\leq \theta  \frac{a}{\rho}<\delta^{(i-1)\eps_1}$. Such an index exists since $\theta \in [\rho,1]$ and $\delta\leq \frac{a}{\rho}  = \frac{ac}{b}\leq c \leq 1$.   Recall that we defined $\tau_i = \delta^{i\eps_1}$,  so $ \theta  a /\rho \leq \tau_i \delta^{-\eps_1}$ implies  $\rho/\theta \geq  \delta^{\eps_1}  a /\tau_i\geq \delta^{1+\eps_1}/\tau_i $.   Then there exists a ball $B_{\tau_i}$ of radius $\tau_i$,  so that
\begin{equation}\label{lambdaILowerBd}
\lambda_i  \geq \Big|B_{\tau_i}\cap\bigcup_{T\in\tubes_1} Y_1(T)\Big| |B_{\tau_i}|^{-1}  
\geq \delta^{3\eps_1-\boldsymbol{\zeta}\boldsymbol{\sigma}/5}\Big(\frac{\rho}{\theta}\Big)^{\boldsymbol{\omega}+\boldsymbol{\sigma}}  
\geq\delta^{4\eps_1-\boldsymbol{\zeta}\boldsymbol{\sigma}/5} \Big(\frac{\delta}{\tau_i}\Big)^{\boldsymbol{\omega}}\frac{|T|^{\boldsymbol{\sigma}/2}}{|T_{\tau_i}|^{\boldsymbol{\sigma}/2}}. 
\end{equation}
Combining \eqref{tubesVolumeBdDensity} and \eqref{lambdaILowerBd}, we conclude that Conclusion (A) of Lemma \ref{WiderGrainsSmallCKT} holds, provided we select $\eps_1 \leq  \frac{1}{100}\boldsymbol{\zeta}\boldsymbol{\sigma}$ and select $\alpha$ sufficiently small.

This concludes our analysis of the case where \eqref{thetaBig} holds (the analysis of this case began at the start of Step 7). Henceforth we shall suppose that \eqref{thetaBig} fails.

\medskip

\noindent{\bf Step 9.}
We shall now return to the start of Step 7, except, instead of assuming \eqref{thetaBig}, we will instead suppose that 
\begin{equation}\label{thetaSmall}
\theta< \delta^{\boldsymbol{\zeta}/10} \tau'.
\end{equation}
Informally, \eqref{thetaSmall} says that the prisms $V\in\mathcal{V}_{T_\tau}$ are flat. 

Recall that in Steps 4 and 5, we fixed a prism $Z\in\mathcal{Z}$ and a prism $W\in\mathcal{W}_Z$.  In this step, we will fix a $\tau$ tube $T_{\tau}\in \tubes_{\tau}$ and a $\theta\times\tau'\times 1$ prism $V\in\mathcal{V}_{T_{\tau}}$. Define $\tubes^\dag = \tilde\tubes[V]$ and let $Y^\dag$ be the restriction of $\tilde Y$ to $\tubes^\dag$. Thus $(\tubes^\dag,Y^\dag)_{\rho}$ is a set of $\rho$ tubes contained in $V$, and
\begin{equation}\label{goodFSConstant}
\FS( (\tubes^\dag)^V)\lessapprox_\delta \delta^{-\eps_2}.
\end{equation}
After pigeonholing, we may suppose that $(\tubes^\dag,Y^\dag)_{\rho}$ is $\gtrapprox_\delta\delta^{2\eps_2}$ dense.
For each ``stem'' $T^\dag_0 \in \tubes^\dag$, define the ``hairbrush'' 
\[
\mathcal{H}(T^\dag_0)=\{T^\dag\in\tubes^\dag\colon Y^\dag(T_0^{\dag} )\cap Y^\dag(T^{\dag})\neq\emptyset\}.
\] 
We claim that for each tube $T^\dag_0\in\tubes^\dag$, the set
\begin{equation}\label{shadingOfX}
N_{(\tau'/\theta)\rho }(T^\dag_0) \cap \bigcup_{T^{\dag} \in\mathcal{H}(T^\dag_0)}Y^\dag(T^\dag)
\end{equation}
is contained in a rectangular prism of dimensions comparable to $\rho\times (\tau' /\theta)\rho \times 1$; we will call this rectangular prism $X=X(T^{\dag}_0)$. This claim follows from straightforward geometric considerations --- See Figure \ref{boxControllsHairbrushFigure}. Define $Y(X)$ to be the set \eqref{shadingOfX}, so $Y(X)\subset X$.

\begin{figure}[h!]
\centering
\begin{overpic}[scale=0.5]{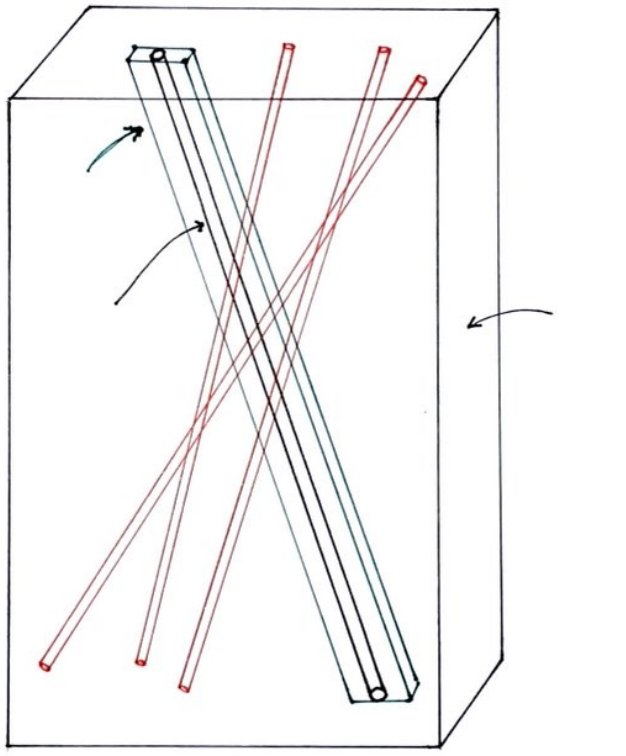}
 \put (3,72) {$X(T^{\dag}_0)$}
 \put (12,54) {$T^{\dag}_0$}
 \put (73,53) {$V$}
\end{overpic}
\caption{All tubes (red) in this figure intersect the stem tube $T^\dag_0$ (black). Since each tube is contained in the $\theta\times\tau'\times 1$ prism $V$, each tube makes angle $\lesssim \theta/\tau'$ with the plane $\Pi(V)$. Since each red tube intersect $T^\dag_0$, the union of these tubes, intersected with the $(\tau'/\theta)\rho$ neighbourhood of $T^\dag_0$, are contained in the prism $X=X(T^{\dag}_0)$ (green) of dimensions $\frac{\theta}{\tau'}\cdot (\frac{\tau'}{\theta})\rho \times   (\frac{\tau'}{\theta})\rho \times 1=\rho\times  (\frac{\tau'}{\theta})\rho \times 1$.
}
\label{boxControllsHairbrushFigure}
\end{figure}


Next, we claim that if $\eps_2$ is chosen sufficiently small compared to $\eps_3$, then for a $\gtrapprox_\delta 1$ fraction of the tubes $T^\dag_0\in \tubes^{\dag}$ we have
\begin{equation}\label{YXBigShading}
| Y(X)| \gtrapprox_\delta \delta^{\eps_3}|X|,\quad \textrm{where}\ X = X(T^\dag_0).
\end{equation}
The estimate \eqref{YXBigShading} says that the shading $Y(X)$ is $\gtrapprox_\delta \delta^{\eps_3}$ dense. This is a standard Cordoba-type $L^2$ argument. In brief, let $Y^\ddag(T^\dag)\subset Y^\dag(T^\dag),\  T^\dag\in\tubes^\dag$ be a regular shading, in the sense of Definition \ref{regularShading}, with $|Y^\ddag(T^\dag)|\geq \frac{1}{2}|Y^\dag(T^\dag)|$. By pigeonholing, a $\gtrapprox_\delta 1$ fraction of the tubes $T^\dag_0\in \tubes^{\dag}$ satisfy
\begin{equation}\label{stillHighMult}
|\{x\in Y^\dag(T^\dag_0)\colon \# \tubes^\dag_{Y^\ddag}(x) \geq \frac{1}{4} \#\tubes^\dag_{Y^\dag}(x)\}| \gtrapprox_\delta\delta^{2\eps_1}|T^\dag_0|.
\end{equation}
For each point $x$ in the set on the LHS of \eqref{stillHighMult}, we can select a tube $T^\dag\in\tubes^\dag$ with $x\in Y^{\ddag}(T)$ and $\angle(\dir(T^\dag_0),\dir(T^\dag))\gtrapprox_\delta \tau$ (recall that in Step 5, we refined our pair $(\tilde\tubes,\tilde Y)_\rho$ to be broad with error $\lessapprox_\delta 1$ relative to $\tubes_\tau$). We now choose a $\delta/\tau$-separated set of points from the LHS of \eqref{stillHighMult}, and consider the corresponding tubes $\{T^\dag\}$. Since each shading $Y^{\ddag}(T^\dag)$ is regular, we have  that the sum of the volumes of these shadings, restricted to $X$, has volume $\gtrapprox_\delta \delta^{2\eps_1}|X|$: 
\[
\sum_{\{T^{\dag}\} } |Y^{\ddag}(T^{\dag})\cap X| \gtrapprox_\delta \delta^{2\eps_1}|X|.
\]
Finally, we use a Cordoba-style $L^2$ argument to show that the corresponding shadings $\{Y^{\ddag}(T)\}$ are almost disjoint inside $X$; this gives \eqref{YXBigShading}.

Abusing notation, we will refine $\tubes^{\dag}$ so that each $T^\dag\in\tubes^\dag$ satisfies \eqref{YXBigShading}. By dyadic pigeonholing and replacing $\tubes^\dag$ by a $\approx_\delta 1$ refinement (abusing notation, we will continue to refer to this set as $\tubes^\dag$), we can select a number $M$ and a set $\mathcal{X}$ of essentially distinct $\rho\times\frac{\tau' \rho}{\theta}\times 1$ prisms of cardinality $\#\mathcal{X}=M^{-1}(\#\tubes^\dag)$, so that 
\begin{itemize}
	\item[(a)] For each $X\in\mathcal{X}$, there are $\sim M$ tubes $T^\dag\in\tubes^\dag$ with $T^\dag\subset X$ and $X(T^\dag)$ comparable to $X$. Denote this latter set by $\tubes^\dag_X$.
	\item[(b)] Each set $\tubes^\dag[X],\ X\in\mathcal{X}$ has the same cardinality (up to a factor of 2).
\end{itemize}
Note that $\tubes^\dag_X\subset\tubes^\dag[X]$, but the two sets need not be equal; it could be the case that $\#\tubes^\dag[X]$ is much larger than $\#\mathcal{X}$. In particular, $\tubes^\dag = \bigsqcup_{X\in\mathcal{X}}\tubes^\dag_X$, but $\mathcal{X}$ might not be a partitioning cover of $\tubes^\dag$. For example, there could exist a tube in $\mathbb{T}^{\dag}[X]$ that is contained in a different prism $V'$; such a tube will be also be contained in the set $\mathbb{T}^{\dag}_{X'},$ where $X'$ a prism with orientation compatible with $V'$ (recall Figure \ref{boxControllsHairbrushFigure}). In particular, it is possible that $X$ and $X'$ intersect transversely. 

\medskip

\noindent{\bf Step 10.}
In Step 9 we fixed a choice of $\tau$ tube $T_{\tau}\in \tubes_{\tau}$ and a $\theta\times\tau'\times 1$ prism $V\in\mathcal{V}_{T_{\tau}}$. We then constructed a pair $(\mathcal{X},Y)_{\rho\times\frac{\tau'}{\theta}\rho\times 1}$. The set $\mathcal{X}$ depended on the choice of prism $V\in\mathcal{V}_{T_{\tau}}$; we will highlight this dependence by writing $\mathcal{X}_V$. In this step we will analyze the interaction between different collections $\mathcal{X}_V$.  

Define $\mathcal{V} = \bigcup_{T_\tau\in\tubes_\tau}\mathcal{V}_{T_\tau}$. Note that $\mathcal{V}$ is a set of $\theta\times\tau'\times 1$ prisms contained in $W$ (the set $W$ was fixed at Step 5). The quantity $M$ from Step 9 depends on the choice of $V$, but after pigeonholing and refining $\mathcal{V}$ we may suppose that this number is the same (up to a factor of 2) for every $V\in\mathcal{V}$.

We will first  consider the case where 
\begin{equation}\label{notTooMuchOverlap}
\# \mathbb{T}^{\dag}_X \sim M \leq \delta^{-\boldsymbol{\zeta}/100}\frac{|X|}{|T^\dag|}.
\end{equation}
We will show that Conclusion (A) of Lemma \ref{WiderGrainsSmallCKT} holds.

Recall from Figure \ref{boxControllsHairbrushFigure} that the prisms $X\in\mathcal{X}_V$ and $V$ have compatible orientations, in the sense that $X^V$ is a prism of dimensions comparable to $\frac{\rho}{\theta}\times \frac{\rho}{\theta}\times 1$. Thus we will refer to the set $(\mathcal{X}_V^V,Y^V)_{\frac{\rho}{\theta}\times \frac{\rho}{\theta}\times 1})$ as $(\tubes_{\frac{\rho}{\theta},V},Y_V)_{\frac{\rho}{\theta}}$. 
 
We claim that 
\begin{equation}\label{minOverlappingVs}
		\textrm{the sets}\ \Big\{\ \bigcup_{ X\in \mathcal{X}_V} Y(  X),\ V\in\mathcal{V}\Big\}\quad\textrm{are}\ \leq\delta^{-\boldsymbol{\beta}}\ \textrm{overlapping}.
	\end{equation}
This claim follows from combining \eqref{boundedOverlapSets} and \eqref{disjointSetsV} (and noting that $\tau^{-\boldsymbol{\beta}}\leq \delta^{-\boldsymbol{\beta}}$).

Observe that for each $V\in\mathcal{V}$, we have
\begin{equation}\label{computeEllRhoTHeta}
\FS(\tubes_{\frac{\rho}{\theta},V})\lessapprox_\delta \delta^{-2\eps_2}.
\end{equation}
This is because $\FS(\tilde\tubes^V)\lessapprox_\delta \delta^{-2\eps_2}$ (recall that $\mathcal{V}$ factors $\tilde\tubes$ from below with small error with respect to the Frostman Slab Wolff Axioms), and by Item (b) from Step 9, each $X\in\mathcal{X}_V$ contains the same number (up to a factor of 2) of  tubes from $\tilde\tubes[V]$. This means that $\FS(\mathcal{X}_V^V)\lessapprox_\delta \delta^{-2\eps_2}$, which is precisely \eqref{computeEllRhoTHeta}.

Define
\begin{equation}\label{defnMRhoTheta}
m_{\frac{\rho}{\theta}}=m M^{-1}\frac{|X|}{|\tilde T|},
\end{equation}
where $m$ is as defined in \eqref{tubesInW}, and $|X| = \rho\times\frac{\tau'}{\theta}\rho\times 1 = \frac{\tau'}{\theta}\rho^2$ is the volume of a prism from $\mathcal{X}_V$. We have 
\begin{equation}\label{computeMRhoTHeta}
\CKT(\tubes_{\frac{\rho}{\theta},V})\lessapprox_\delta m M^{-1}\frac{|X|}{|\tilde T|}\leq m_{\frac{\rho}{\theta}}.
\end{equation}

Finally, we compute
\begin{equation}\label{boundSumTubesInV}
\sum_{V\in\mathcal{V}} (\#\tubes_{\frac{\rho}{\theta},V})\gtrapprox_\delta M^{-1} \#\tilde\tubes[W] \gtrapprox mM^{-1}\frac{|W|}{|\tilde T|},
\end{equation}
where $W$ is the prism fixed at Step 5, and the final inequality used \eqref{tubesInW}.

Fix a choice of $V\in\mathcal{V}$. Applying the estimate $\cE(\boldsymbol{\sigma},\boldsymbol{\omega})$ to $(\tubes_{\frac{\rho}{\theta},V},Y_V)_{\rho/\theta}$ with $\eps_4$ in place of $\eps$ and using \eqref{computeEllRhoTHeta}, we have
\begin{equation}
\begin{split}
\Big| \bigcup_{T_{\frac{\rho}{\theta}}\in\tubes_{\frac{\rho}{\theta},V}}Y_V(T_{\frac{\rho}{\theta}})\Big|
&\gtrapprox_\delta \delta^{\eps_4+2\eps_2}\Big(\frac{\rho}{\theta}\Big)^{\boldsymbol{\omega}} m_{\frac{\rho}{\theta}}^{-1}(\#\tubes_{\frac{\rho}{\theta},V})|T_{\frac{\rho}{\theta}}|
\Big(m_{\frac{\rho}{\theta}}^{-3/2} (\#\tubes_{\frac{\rho}{\theta},V})|T_{\frac{\rho}{\theta}}|^{1/2} \Big)^{-\boldsymbol{\sigma}}.
\end{split}
\end{equation}
In the above inequality, we used \eqref{computeMRhoTHeta} plus the fact that $\boldsymbol{\sigma}\leq 2/3$ (the latter inequality allows us to replace  $\CKT(\tubes_{\frac{\rho}{\theta},V})$ with the potentially larger quantity $m_{\frac{\rho}{\theta}}$). Undoing the scaling $\phi_V$ and using \eqref{minOverlappingVs}, we conclude that
\begin{equation}\label{boundInsideWZ}
\begin{split}
\Big|\bigcup_{\tilde T\in\tilde\tubes}Y(\tilde T)\Big| &\geq \Big|\bigcup_{V\in\mathcal{V}}\bigcup_{X\in\mathcal{X}_V}Y(X)\Big|\\
&\geq\delta^{\boldsymbol{\beta}}\sum_{V\in\mathcal{V}}\Big|\bigcup_{X\in\mathcal{X}_V}Y(X)\Big|\\
&\geq  \delta^{\boldsymbol{\beta}} \frac{|X|}{|T_{\frac{\rho}{\theta}}|}\sum_{V\in\mathcal{V}}\Big| \bigcup_{T_{\frac{\rho}{\theta}}\in\tubes_{\frac{\rho}{\theta},V}}Y_V(T_{\frac{\rho}{\theta}})\Big|\\
&\gtrapprox_\delta  \delta^{\boldsymbol{\beta}} \frac{|X|}{|T_{\frac{\rho}{\theta}}|} \cdot  
\delta^{\eps_4+2\eps_2}\Big(\frac{\rho}{\theta}\Big)^{\boldsymbol{\omega}} m_{\frac{\rho}{\theta}}^{-1}\Big(\sum_{V\in\mathcal{V}}\#\tubes_{\frac{\rho}{\theta},V}\Big)|T_{\frac{\rho}{\theta}}|
\Big(m_{\frac{\rho}{\theta}}^{-3/2} \big(\sup_{V\in\mathcal{V}}\#\tubes_{\frac{\rho}{\theta},V}\big)|T_{\frac{\rho}{\theta}}|^{1/2} \Big)^{-\boldsymbol{\sigma}}\\
&\gtrapprox_\delta  \delta^{\boldsymbol{\beta}+2\eps_4} \frac{|X|}{|T_{\frac{\rho}{\theta}}|} \cdot  
\Big(\frac{\rho}{\theta}\Big)^{\boldsymbol{\omega}} m_{\frac{\rho}{\theta}}^{-1}\Big(mM^{-1}\frac{|W|}{|\tilde T|}\Big)|T_{\frac{\rho}{\theta}}|
\Big(m_{\frac{\rho}{\theta}}^{-3/2} \big(\sup_{V\in\mathcal{V}}\#\tubes_{\frac{\rho}{\theta},V}\big)|T_{\frac{\rho}{\theta}}|^{1/2} \Big)^{-\boldsymbol{\sigma}},
\end{split}
\end{equation}
where the final inequality used \eqref{boundSumTubesInV}. Substituting \eqref{defnMRhoTheta} and simplifying, we obtain
\begin{equation}\label{reduceLHSboundInsideWZ}
\begin{split}
\textrm{LHS}\ \eqref{boundInsideWZ} &\gtrapprox_\delta \delta^{\boldsymbol{\beta}+2\eps_4} \frac{|X|}{|T_{\frac{\rho}{\theta}}|} \cdot  
\Big(\frac{\rho}{\theta}\Big)^{\boldsymbol{\omega}} \Big(m M^{-1}\frac{|X|}{|\tilde T|}\Big)^{-1}\Big(mM^{-1}\frac{|W|}{|\tilde T|}\Big)|T_{\frac{\rho}{\theta}}|\\
&\qquad\qquad\qquad\qquad\cdot \Big(\Big(m M^{-1}\frac{|X|}{|\tilde T|}\Big)^{-3/2} \big(\sup_{V\in\mathcal{V}}\#\tubes_{\frac{\rho}{\theta},V}\big)|T_{\frac{\rho}{\theta}}|^{1/2} \Big)^{-\boldsymbol{\sigma}}\\
&\gtrapprox_\delta \delta^{\boldsymbol{\beta}+2\eps_4} 
\Big(\frac{\rho}{\theta}\Big)^{\boldsymbol{\omega}} |W|  \Big(m M^{-1}\frac{|X|}{|\tilde T|}\Big)^{\boldsymbol{\sigma}/2}     \Big(\Big(m M^{-1}\frac{|X|}{|\tilde T|}\Big)^{-1} \big(M^{-1}(\sup_{V\in \mathcal{V}} \#\tilde\tubes[V])\big)|T_{\frac{\rho}{\theta}}|^{1/2} \Big)^{-\boldsymbol{\sigma}}\\
&\gtrapprox_\delta \delta^{-\boldsymbol{\sigma}\boldsymbol{\zeta}/4 }
\Big(\frac{\rho}{\theta}\Big)^{\boldsymbol{\omega}} |W|\Big(m^{-1} \frac{|\tilde T|}{|X|} \big(\#\tilde\tubes[V]\big)|T_{\frac{\rho}{\theta}}|^{1/2} \Big)^{-\boldsymbol{\sigma}},
\end{split}
\end{equation}
where the second inequality used the fact that $\#\tubes_{\frac{\rho}{\theta},V}\sim M^{-1}(\#\tilde\tubes[V])$  for each $V\in\mathcal{V}=\mathcal{V}$, and the third inequality used \eqref{mLarge}, \eqref{notTooMuchOverlap}, $|T^{\dag}|=|\tilde{T}|$, and the fact that $\boldsymbol{\beta}$ and $\eps_4$ are small compared to $\boldsymbol{\sigma}\boldsymbol{\zeta}$.

Observe that the set on the LHS of \eqref{reduceLHSboundInsideWZ} is contained in $|W|$, while the RHS involves the term $|W|$. Thus \eqref{reduceLHSboundInsideWZ} gives a lower bound for the density of $\bigcup_{\tilde T\in\tilde\tubes}\tilde Y(\tilde T)$ inside $W$. This bound contains the term $\delta^{-\boldsymbol{\sigma}\boldsymbol{\zeta}/4 }$ --- this quantity is much larger than 1, and this will eventually allow us to conclude that Conclusion (A) of Lemma \ref{WiderGrainsSmallCKT} holds.

Let us analyze the final term in brackets $(\cdots)^{-\boldsymbol{\sigma}}$. We have
\begin{equation}\label{reduceInsideBrackets}
\begin{split}
m^{-1} \frac{|\tilde T|}{|X|} \big(\#\tilde\tubes[V]\big)|T_{\frac{\rho}{\theta}}|^{1/2}
&\leq m^{-1} \frac{|\tilde T|}{|X|} \Big(\CFC(\tilde\tubes^W)\frac{|V|}{|W|}(\#\tilde \tubes[W])\Big)|T_{\frac{\rho}{\theta}}|^{1/2}\\
&\lessapprox_\delta  m^{-1} \frac{|\tilde T|}{|X|} \Big(\frac{|V|}{|W|}(m\frac{|W|}{|\tilde T|})\Big)|T_{\frac{\rho}{\theta}}|^{1/2}\\
&= \frac{|V|}{|X|}|T_{\frac{\rho}{\theta}}|^{1/2}\\
& = \frac{\theta}{\rho},
\end{split}
\end{equation}
where the second inequality used \eqref{CFCTildeTubes} and \eqref{tubesInW}.

Combining \eqref{reduceLHSboundInsideWZ} and \eqref{reduceInsideBrackets}, we conclude that
\begin{equation}\label{tildeTubesLargeVolume}
\Big|\bigcup_{\tilde T\in\tilde\tubes}Y(\tilde T)\Big| \gtrapprox_\delta \delta^{-\boldsymbol{\sigma}\boldsymbol{\zeta}/4} \Big(\frac{\rho}{\theta}\Big)^{\boldsymbol{\omega}+\boldsymbol{\sigma}} |W|.
\end{equation}

\medskip

\noindent{\bf Step 11.}
We can now argue similarly to our reasoning in Step 8. The set on the LHS of \eqref{tildeTubesLargeVolume} is contained in $W$, which is a prism of dimensions $s\times t\times 1$, with $\delta\leq \rho\leq\theta\leq s\leq 1$. Thus there exists a ball $B_\theta$ of radius $\theta$ so that
\begin{equation*}
\Big|B_\theta  \cap \bigcup_{\tilde T\in\tilde\tubes}Y(\tilde T)\Big|\gtrapprox_\delta \delta^{-\boldsymbol{\sigma}\boldsymbol{\zeta}/4} \Big(\frac{\rho}{\theta}\Big)^{\boldsymbol{\omega}+\boldsymbol{\sigma}} |B_\theta|.
\end{equation*}
Recall that at the beginning of Step 4 we fixed a prism $Z\in\mathcal{Z}$. The set of $\rho$ tubes $\tilde \tubes$ are the images of prisms from $\mathcal{P}_3\langle Z\rangle$ under the linear map $\phi_Z$. Let $B^\dag = \phi_Z^{-1}(B_\theta)$; $B^\dag$ is an ellipsoid of dimensions $\theta \frac{a}{\rho}\times \theta c\times\theta c$, which satisfies
\[
\Big|B^\dag \cap \bigcup_{T_1\in\tubes_1}Y_1(T)\Big| \gtrapprox_\delta \delta^{-\boldsymbol{\sigma}\boldsymbol{\zeta}/4} \Big(\frac{\rho}{\theta}\Big)^{\boldsymbol{\omega}+\boldsymbol{\sigma}} |B^\dag|.
\]
This is the analogue of \eqref{bigBDagBallEstimate}. An identical argument (in particular, note that 
$a\in [\delta, 1]$) shows that Conclusion (A) of Lemma \ref{WiderGrainsSmallCKT} holds, provided we select $\alpha<\boldsymbol{\sigma}\boldsymbol{\zeta}/10$. This concludes our analysis of the case where \eqref{notTooMuchOverlap} holds. 

\medskip

\noindent{\bf Step 12}.
We shall now return to the start of Step 10, except, instead of assuming \eqref{notTooMuchOverlap}, we will instead suppose that 
\begin{equation}\label{TooMuchOverlap}
M > \delta^{-\boldsymbol{\zeta}/100}\frac{|X|}{|\tilde T|}.
\end{equation}
We summarize the situation thus far: 
\begin{itemize}
	\item We have fixed a choice of $Z\in\mathcal{Z}$ (these are prisms of dimensions $\frac{a}{\rho}\times c\times c$) and $W\in\mathcal{W}_Z$. 
	\item We have a set $\tilde\tubes$ of $\rho$ tubes contained in $W$. 
	\item We have a set $\mathcal{V}$ of $\theta\times\tau'\times 1$ prisms contained in $W$. 
	\item For each $V\in\mathcal{V}$, we have a set $\mathcal{X}_V$ of $\rho\times\frac{\tau'}{\theta}\rho\times 1$ prisms, and a partition $\tilde\tubes[V]=\bigsqcup_{X\in\mathcal{X}_V}(\tilde\tubes[V])_X$. 
	\item Each set $(\tilde\tubes[V])_X$ has cardinality roughly $M$, where $M$ satisfies \eqref{TooMuchOverlap}.
	\item We have a shading $\tilde Y$ on $\tilde \tubes$ so that $(\tilde \tubes, \tilde Y)_\rho$ is $\gtrapprox_\delta \delta^{2\eps_2}$ dense.
\end{itemize}

Let $\mathcal{X} = \bigcup_{V\in\mathcal{V}}\mathcal{X}_V$ (this is a slight abuse of notation, since in Step 9 we defined $\mathcal{X}$ to be a set of the form $\mathcal{X}_V$, where the prism $V$ was fixed in advance). To simplify notation, define $\tilde\tubes_X = (\tilde\tubes[V])_X$, where $V$ is the (unique) prism for which $X\in\mathcal{X}_V$. Since $(\tilde\tubes, \tilde Y)_\rho$ is $\gtrapprox_\delta \delta^{2\eps_2}$ dense, after a harmless refinement of $\mathcal{X}$ we may suppose that each pair $(\tilde\tubes_X,\tilde Y)_\rho$ is $\gtrapprox_\delta \delta^{2\eps_2}$ dense, where $\tilde Y$ is the restriction of the shading on $\tilde\tubes$ to the set $\tilde\tubes_X$. 

Apply Corollary \ref{coveringTubesBroadCor} (finding a broad scale) to each pair $(\tilde\tubes_X,\tilde Y)_\rho$, for each $X\in\mathcal{X}$. This yields a scale $\tilde\rho$ and a $\gtrapprox_\delta 1$ refinement $(\tilde\tubes_X',\tilde Y')_\rho$ that is broad with error $\lessapprox_\delta 1$ with regard to a balanced partitioning cover of $\tilde\rho$ tubes. 

After dyadic pigeonholing and refining the set $\mathcal{X}$, we can suppose that the scale $\tilde\rho$ from Corollary \ref{coveringTubesBroadCor} is the same for each prism $X$. We claim that 
\begin{equation}\label{tildeRhoBig}
\tilde\rho\gtrapprox_\delta \delta^{-\boldsymbol{\zeta}/200}\rho.
\end{equation} 
To verify this claim, observe that for each $X\in\mathcal{X}$, the sets $\{\tilde Y'(\tilde T)\colon \tilde T\in\tilde\tubes_X'\}$ are $\lessapprox_\delta (\tau'/\tilde\rho)^{\boldsymbol{\beta}}(\tilde\rho/\rho)$ overlapping. This is because the tubes from $\tilde\tubes'_X$ whose shadings pass through a common point $x$ must point in directions confined to a $\rho\times \tau'$ region in the unit sphere $S^2\subset \RR^3$ (we identify $S^2$ with the set of directions of tubes in $\RR^3$). Since these tubes are essentially distinct and pass through a common point, they must point in $\rho$ separated directions. Thus at most $\rho/\tilde\rho$ tubes can point in directions confined to a $\rho\times \tilde\rho$ region in $S^2$ (this corresponds to those $\rho$ tubes contained in a single $\tilde\rho$ tube), and at most $(\tau'/\tilde\rho)^{\boldsymbol{\beta}}$ distinct $\tilde\rho$ tubes can contribute to the count. This implies that
\[
\Big|\bigcup_{\tilde T\in\tilde\tubes'_X}\tilde Y'(\tilde T)\Big| 
\gtrapprox_\delta 
\Big(\frac{\tilde\rho}{\tau'}\Big)^{\boldsymbol{\beta}}\Big(\frac{\rho}{\tilde\rho}\Big)\sum_{\tilde T\in\tilde\tubes'_X}|\tilde Y'(\tilde T)|
%
\geq \delta^{-\boldsymbol{\zeta}/200}\Big(\frac{\rho}{\tilde\rho}\Big)|X|,
\]
where the second inequality used \eqref{TooMuchOverlap} and $\boldsymbol{\beta}< \boldsymbol{\zeta}/200$. Since the set on the LHS is contained in $X$, we obtain \eqref{tildeRhoBig}, as claimed.

Next, for each prism $X\in\mathcal{X}$ there exists a set $\{U\}$ of essentially distinct $\rho\times\tilde\rho\times 1$ prisms, each of which are contained in $X$, so that these sets form a partitioning cover of $\tilde\tubes_X'$, and for each such $U$, the pair $(\tilde\tubes_X'[U],\tilde Y')_{\rho}$ is broad with error $\lessapprox_\delta 1$ relative to the cap of diameter $\tilde\rho$ centered at the point $\dir(U)$. See Figure \ref{breakXIntoUFig}.


\begin{figure}[h!]
\centering
\begin{overpic}[ scale=0.5]{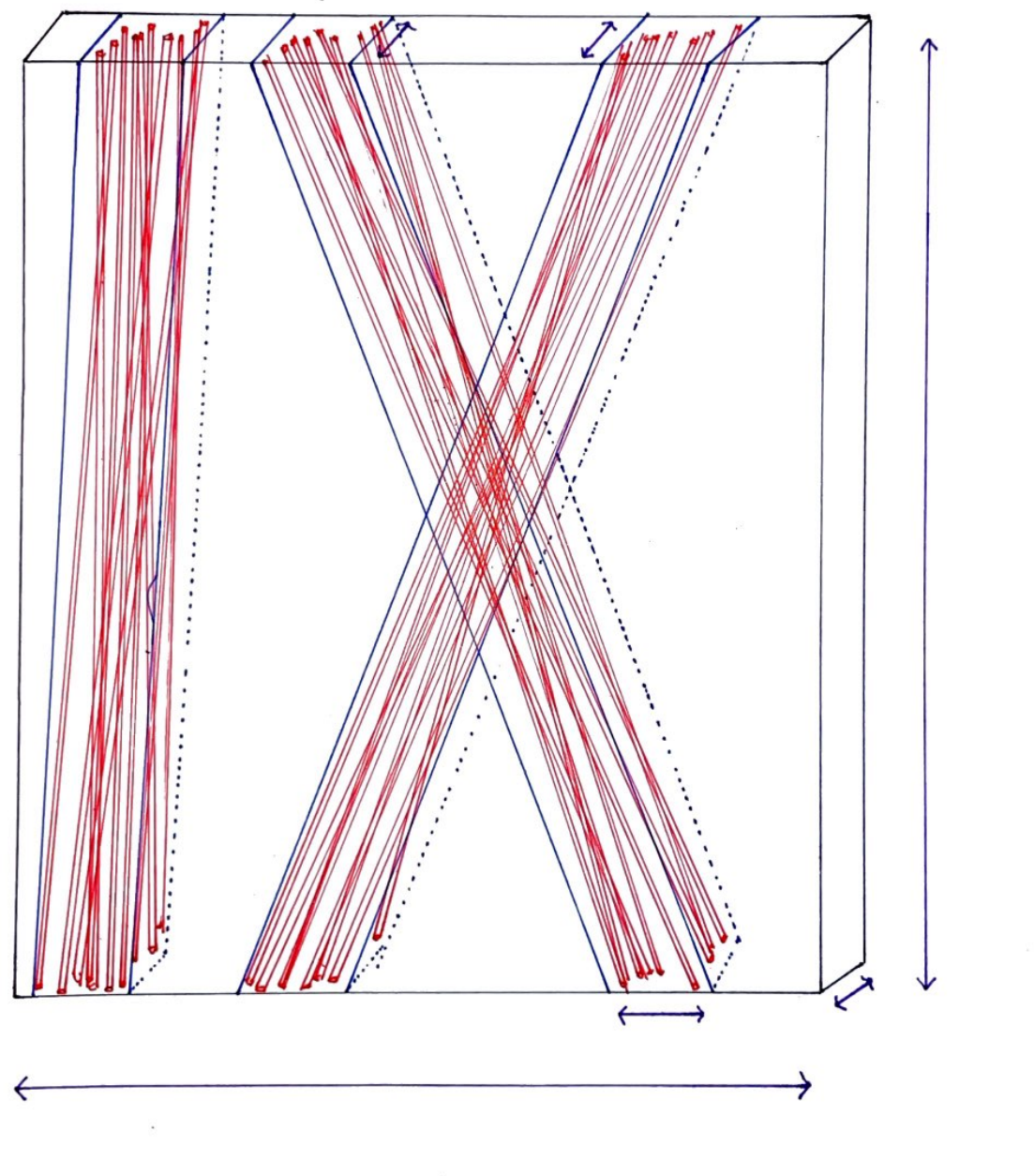}
 \put (35,3) {$\delta^{-\frac{\zeta}{20}}\rho$}
 \put (55,10) {$\tilde\rho$}
 \put (73,12) {$\rho$}
 \put (47,96) {$\rho$}
 \put (81,53) {$1$}
\end{overpic}
\caption{We find a set of $\rho\times\tilde\rho\times 1$ prisms (blue) inside $X$ (black prism), which forms a partitioning cover of  $\tilde\tubes_X'$
 (red lines). A typical pair of tubes from  $\tilde\tubes_X'$ inside a common (blue) $\rho\times\tilde\rho\times 1$ prism intersect at angle roughly $\tilde\rho$.}
\label{breakXIntoUFig}
\end{figure}
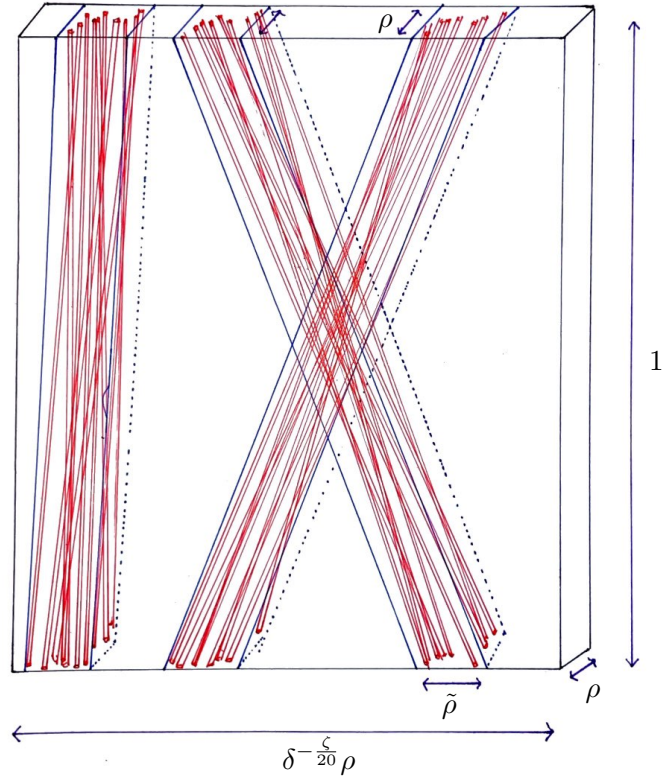

For each such prism $U$, define
\[
Y(U) = \bigcup_{\tilde T\in\tilde\tubes_X'[U]}\tilde Y'(\tilde T).
\]
Then a Cordoba-style $L^2$ argument shows that for each prism $U$ for which $(\tilde\tubes_X'[U],\tilde Y')_\rho$ is $\gtrapprox_\delta \delta^{2\eps_2}$ dense (here $\tilde Y'$ denotes the restriction of the shading on $\tilde\tubes_X'$ to $\tilde\tubes_X'[U]$), we have $|Y(U)|\gtrapprox_\delta \delta^{8\eps_2}|U|$. Let $\mathcal{U}$ denote the set of all prisms $U$ for which this holds, as $X$ ranges over the elements of $\mathcal{X}$. Note, however, that the prisms in $\mathcal{U}$ need not be essentially distinct. 

\medskip

\noindent{\bf Step 13}.
Our task in this step is to unwind the various transformations and rescalings from the previous step, and to understand what the prisms $U$ (and their associated shadings $Y(U)$) correspond to in the original space in which the tubes $\tubes$ and prisms $\mathcal{P}$ reside.

In Step 12, we fixed a choice of $Z\in\mathcal{Z}$ and $W\in\mathcal{W}_Z$.  Recall that $\tilde \tubes$ comes from $(\mathcal{P}_5\langle Z \rangle )^Z$, and is a set of $\rho \times \rho \times 1$-tubes. We obtained a set $\mathcal{U}=\mathcal{U}_{W,Z}$ of $\rho\times\tilde{\rho}\times 1$ prisms (this $\mathcal{U}$ is different from the set $\mathcal{U}$ in Step 3, the latter $\mathcal{U}$ was defined to introduce $\mathcal{Z}$ and only appeared within Step 3),  and a shading $Y(U)$ on these prisms. These prisms are contained inside $W$. Undoing the linear transformation $\phi_Z$, we have a set $\phi_Z^{-1}(\mathcal{U}_{W,Z})$ of prisms.  After pigeonholing, we may assume that these prisms are of dimensions $\tilde a\times \tilde b\times c$, for some $\tilde a\geq a$ and $\tilde b\geq b$. Since the linear transformation $\phi_Z$ distorts volume by a factor of  $(a/\rho) c^2$, we have that  $(\tilde a \tilde b) = a c \tilde{\rho}$. The dimensions $\tilde a, \tilde b$ depend on the choice of $W$ and $Z$, but after pigeonholing, we may suppose that these values are the same (up to a factor of 2) for all $Z\in\mathcal{Z}$ and all $W\in\mathcal{W}_Z$. We claim (provided $\eps_2$ is chosen sufficiently small compared to $\eps_3$)  that either
\begin{equation}\label{aPrimeSmall}
\tilde a\leq \delta^{-\eps_3}a,
\end{equation}
or else Conclusion (A) of Lemma \ref{WiderGrainsSmallCKT} holds (provided $\alpha$ is chosen sufficiently small, depending on $\eps_3$). The argument is identical to the argument in Step 1 of Lemma \ref{squareGrainsGetLonger}; we refer the reader there for details.

Henceforth we shall assume that \eqref{aPrimeSmall} holds. Abusing notation, we will re-define $\tilde\rho = \tilde b/c$; this re-definition might decrease the value of $\tilde\rho$ by as much as $\delta^{\eps_3}$, and \eqref{tildeRhoBig} might be weakened to $\tilde\rho\geq \delta^{-\boldsymbol{\zeta}/400}\rho$. With this re-definition of $\tilde\rho$, we clearly have $\tilde b=\tilde\rho c$.

Let 
\[
\mathcal{Q} = \bigcup_{Z\in\mathcal{Z}}\bigcup_{W\in\mathcal{W}_Z}\phi_Z^{-1}(\mathcal{U}_{Z,W}).
\]
For each $Q\in\mathcal{Q}$ of the form $Q =\phi_Z^{-1}(U)$, define the shading $Y(Q)=\phi_Z^{-1}(Y(U))$  and define the set $\mathcal{P}_Q\subset \mathcal{P}[Q]$ as $\phi_Z^{-1} (\tilde\tubes_X'[U])$, where $X$ is the $\rho \times \frac{\tau'}{\theta}\rho \times 1$-prism associated to $U$ (see Step 12). Summarizing our conclusions thus far, we have the following. 
\begin{itemize}
	\item Each $Q\in\mathcal{Q}$ is a prism of dimensions $\tilde a\times \tilde b\times c$, with $\tilde b =\tilde\rho c$.

	\item For each $Q\in\mathcal{Q}$, there is a set $\mathcal{P}_Q\subset\mathcal{P}[Q]$. Define $\mathcal{P}'=\bigcup_{Q\in\mathcal{Q}}\mathcal{P}_Q$. There is a shading $Y'(P)\subset Y(P)$, so that $(\mathcal{P}',Y')_{a\times b\times c}$ is a $\gtrapprox_\delta \delta^{2\eps_2}$ refinement of $(\mathcal{P},Y)_{a\times b\times c}$.

	\item Each $Q\in\mathcal{Q}$ has a shading $Y(Q)$ given by
	\[
		Y(Q) = \bigcup_{P\in\mathcal{P}_Q}Y'(P).
	\]
	We have $|Y(Q)|\gtrapprox_\delta\delta^{8\eps_2}|Q|$ for each $Q\in\mathcal{Q}$.

	\item For each $Q\in\mathcal{Q}$ and each $x\in Y(Q)$, the prisms $P\in \mathcal{P}_Q$ with $x\in Y'(P)$ point in directions that are broad with error $\lessapprox_\delta 1$ inside a cap of radius $\tilde b/c$ centered at $\operatorname{dir}(Q)$ (these correspond to the tubes $\tilde\tubes'_X[U]$ and their associated shading, which are broad with error $\lessapprox_\delta 1$ inside the $\rho\times\tilde\rho\times 1$ prism $U$). 

	\item The refinement of $(\tubes,Y)_\delta$ induced by the refinement $(\mathcal{P}',Y')_{a\times b\times c}$ of $(\mathcal{P},Y)_{a\times b\times c}$ is $\gtrapprox_\delta \delta^{2\eps_2}$ dense. If we denote this refinement by $(\tubes',Y')_\delta$, then (by the definition of being an induced refinement) we have 
	\[
		\bigcup_{T\in\tubes'}Y'(T)= \bigcup_{Q\in\mathcal{Q}}Y(Q).
	\] 
\end{itemize}

The pair $(\mathcal{Q},Y)_{\tilde a\times \tilde b\times c}$ has some of the desired properties from Conclusion (C) of Lemma \ref{WiderGrainsSmallCKT}. Observe that for each $Q\in\mathcal{Q},$ $\dir(Q)$ is defined up to uncertainty $\tilde\rho$. Thus after a refinement of $\mathcal{Q}$ and $\tubes'$, we can find a set of $\tilde\rho$ tubes $\tubes_{\tilde\rho}$ and a partition $\mathcal{Q}=\bigsqcup_{T_{\tilde\rho}\in\tubes_{\tilde\rho}}\mathcal{Q}_{T_{\tilde\rho}}$, so that each $Q\in\mathcal{Q}_{T_{\tilde\rho}}$ is contained in $T_{\tilde\rho}$ and satisfies $\angle(\dir(Q),\dir(T_{\tilde\rho}))\leq\tilde\rho$. 

Fix a prism $Q\in\mathcal{Q}_{T_{\tilde\rho}}$ and a point $x\in Y(Q)$. The set of prisms $P\in\mathcal{P}_Q$ with $x\in Y'(P)$ point in directions that satisfy $\angle(\dir(P),\dir(Q))\leq\tilde\rho$, and this set of directions is $\rho$-separated and broad with error $\lessapprox_\delta 1$ at scales $\geq\rho$ inside a cap of diameter $\tilde\rho$ centered at $\dir(Q)$. For each such $P$ (contained in a $\rho$ tube $T_{\rho}$), the set of tubes $T\in\tubes[T_\rho]$ with $x\in Y'(T)$ are broad with error $\lessapprox_\delta \delta^{-2\eps_2}$ at scales $\geq\delta$, inside a cap of diameter $\rho$ centered at $\dir(P)$. Thus by Lemma \ref{broadnessAcrossScales} (broadness combines across scales) and Items (ii) and (iii) of Definition~\ref{twoScaleGrainsDecomp}, we have that the set of tubes $T\in\tubes'$ associated to the point $x\in Y(Q)$ point in directions that are broad with error $\lessapprox_{\delta}\delta^{-2\eps_2}$ at scales $\geq\delta$ inside a cap of diameter $\tilde\rho$ centered at $\dir(Q)$. 

Note that even though the conclusion of Lemma~\ref{broadnessAcrossScales} is a statement about multi-sets, here it is also true in the sense of sets.  Indeed,  Item (ii) of Definition~\ref{twoScaleGrainsDecomp} says that for each point $x\in \bigcup_{Q\in \mathcal{Q}_{T_{\tilde \rho}}}Y(Q)$, we have that the sets
$\{ \dir (T):  T\in \tubes[T_\rho] \text{ and } x\in Y'(T)\cap Y'(P)\}$ are disjoint, as $P$ ranges over the elements of $\bigcup_{Q\in \mathcal{Q}_{T_{\tilde \rho}} } \mathcal{P}_Q$ that are contained in $T_\rho$.
 In particular, we can construct a set $\tubes_{\tilde\rho}$ of $\tilde\rho$ tubes that covers $\tubes'$, so that $(\tubes',Y')_\delta$ is broad with error $\lessapprox_\delta \delta^{-2\eps_2}$ relative to  $\tubes_{\tilde{\rho}}$.

\medskip

\noindent{\bf Step 14}.

We would like to show that $(\tubes',Y')_\delta$, $\mathcal{Q}$ and its associated shading $Y$, and $\tilde\tubes$ satisfy Conclusion (C) from Lemma \ref{WiderGrainsSmallCKT}. First, the prisms in $\mathcal{Q}$ might not be essentially distinct. This can be fixed, however, using the same argument as was employed in Step 5 from the proof of Lemma \ref{squareGrainsGetLonger}. In brief, we merge comparable prisms from $\mathcal{Q}$ into a single prism. Denote the resulting set of prisms by $\tilde{\mathcal P}$; these are essentially distinct prisms of dimensions comparable to $\tilde a\times\tilde b\times c$. Let $\tilde Y$ be the shading on the prisms of $\tilde{\mathcal{P}}$ obtained by taking the union of the shadings $Y(Q),\ Q\subset \tilde P$. We define 
\[
\mathcal{P}_{\tilde P}=\bigcup_{\substack{Q\in\mathcal{Q}\\ Q\subset\tilde P}}\mathcal{P}_Q.
\]
 Note that if $P\in\mathcal{P}_{\tilde P}$, then $\angle(\dir(P),\dir(\tilde P))\lesssim \tilde b/c =\tilde\rho$. Thus after enlarging $\tilde a$ and $\tilde b$ by a constant factor if needed, we can ensure that if (i): $T\cap P\neq\emptyset$, (ii): $T$ exists $P$ through its long ends, and (iii): $P\in\mathcal{P}_{\tilde P}$, then $T\cap\tilde P\neq\emptyset$ and $T$ exists $\tilde P$ through its long ends. 
At this point, the pair $(\tubes',Y')_\delta$ and $(\tilde{\mathcal{P}},\tilde Y)_{\tilde a\times\tilde b\times c}$ satisfy some of the requirements of Conclusion (C) from Lemma \ref{WiderGrainsSmallCKT}. The situation matches the setup at the end of Step 5 in the proof of Lemma \ref{squareGrainsGetLonger}.

We now proceed with the same argument that was used in Steps 6 -- 8 from the proof of Lemma \ref{squareGrainsGetLonger}. We conclude that either Conclusion (A) of Lemma \ref{WiderGrainsSmallCKT} holds (this is the same as Conclusion (A) of Lemma \ref{squareGrainsGetLonger}), or else there is a set $\tubes_{\tilde\rho}$ and a further refinement of $(\tubes',Y')_\delta$ and $(\tilde P,\tilde Y)_{\tilde a\times\tilde b\times c}$ that satisfies Items (i), (ii), and (iii) of Conclusion (C) from Lemma \ref{WiderGrainsSmallCKT} (Items (i), (ii), and (iii) of Conclusion (C) from Lemma \ref{squareGrainsGetLonger}). Note that Item (iv) Conclusion (C) is also satisfied, since $\tilde\rho\geq\delta^{-\boldsymbol{\zeta}/400}\rho$. We conclude that Conclusion (C) from Lemma \ref{WiderGrainsSmallCKT} holds.
\end{proof}


\section{A refined induction-on-scales argument}\label{refinedInductionOnScaleSec}
The goal in this section is to show that if $(\tubes,Y)_\delta$ is a set of $\delta$-tubes, then either $\bigcup_{\tubes}Y(T)$ has larger volume than one would expect from the estimate \eqref{defnCEEstimate} from Assertion $\cE(\sigma,\omega)$, or else there exists a scale $\delta<\!\!<\rho<\!\!<1$ and a set of $\rho$ tubes that factors $\tubes$ above and below with respect to the Katz-Tao Convex Wolff Axioms and Frostman Slab Wolff Axioms. The precise statement is as follows. 

\begin{prop}\label{refinedInductionOnScaleProp}
Let $\omega,\zeta>0$ and $\sigma\in(0,2/3]$, and suppose that $\cE(\sigma,\omega)$ is true. Then there exists 
$\alpha,\eta,\kappa>0$ 
so that the following holds for all $\delta>0$.
Let $(\tubes,Y)_\delta$ be $\delta^{\eta}$ dense, and suppose that $\CKT(\tubes)\leq\delta^{-\eta}$ and $\FS(\tubes)\leq\delta^{-\eta}$. Then at least one of the following must hold.

\begin{itemize}
\item[(A)]\itemizeEqnVSpacing
\begin{equation}\label{conclusionArefinedInductionOnScaleProp}
\Big| \bigcup_{T\in\tubes}Y(T) \Big| \geq \kappa \delta^{\omega-\alpha}(\#\tubes)|T|\big((\#\tubes)|T|^{1/2}\big)^{-\sigma}.
\end{equation}
\item[(B)] There exists a refinement $(\tubes',Y')_\delta$ of $(\tubes,Y)_\delta$ that is $\delta^{\zeta}$ dense, a number $\rho\in [\delta^{1-\omega/100},\delta^{\omega/100}]$, and a set $\tubes_\rho$ that factors $\tubes'$ above and below with respect to both the Katz-Tao Convex Wolff Axioms and the Frostman Slab Wolff Axioms, both with error $\leq\delta^{-\zeta}$.
\end{itemize} 
\end{prop}


\begin{proof}
\noindent{\bf Step 1.}
Let $\eps_1,\eps_2,\eps_3$ be small numbers to be chosen below. We will select $\eps_1$ very small compared to $\eps_{2}$ and $\eps_2$ very small compared to $\eps_3$. These numbers depend on $\omega,\sigma,$ and $\zeta$. We will select $\eta$ and $\alpha$ very small compared to $\eps_1$.

Apply Proposition \ref{grainsDecomposition} (two scale grains decomposition) to $(\tubes,Y)$ with $\eps_1$ in place of $\zeta$, and let $\alpha_1 = \alpha_1(\omega,\sigma,\eps_1)$ be the output of that proposition. If Conclusion (A) of Proposition \ref{grainsDecomposition} holds, then \eqref{conclusionArefinedInductionOnScaleProp} is true (provided we select $\alpha\leq\alpha_1$), and we are done. 

Next, suppose Conclusion (B) of Proposition \ref{grainsDecomposition} holds. Let $\rho, c$, $(\tubes_1,Y_1)_\delta,$ $\tubes_\rho,$ and $(\mathcal{G},Y)_{  a \times \rho c \times c}$ be the output from Proposition \ref{grainsDecomposition}, Conclusion (B).

Apply Proposition \ref{factoringConvexSetsProp} (factoring convex sets) to $\tubes_\rho$. We obtain a $\approx_\rho 1$ refinement of $\tubes_\rho$, which in turn induces a $\approx_\rho 1$ refinement of $(\tubes_1,Y_1)_\delta$ (abusing notation, we will continue to refer to these objects as $\tubes_\rho$ and $(\tubes_1,Y_1)_\delta$) and a collection $\mathcal{Z}$ of congruent convex sets that factors $\tubes_\rho$ from above with respect to the Katz-Tao Convex Wolff Axioms and from below with respect to the Frostman Convex Wolff Axioms, both with error $\lessapprox_\rho 1$. Furthermore,
	\begin{equation}\label{lotsOfUinWAppliedToTRho}
		\#\tubes_\rho[Z]\gtrapprox_\rho \CKT(\tubes_\rho)|Z||T_\rho|^{-1}\quad\textrm{for each}\ Z\in\mathcal{Z}.
	\end{equation} 

If $\CKT(\tubes_\rho)\leq\delta^{-\zeta},$ then provided we select $\eps_1\leq\zeta/2$, we have that $(\tubes_1,Y_1)$ and $\tubes_\rho$ satisfy Conclusion (B) of Proposition \ref{refinedInductionOnScaleProp}, and we are done. Indeed; by Remark \ref{FrostmanWolffInheritedUpwardsDownwards}(A) we have $\FS(\tubes_\rho)\lesssim\FS(\tubes_1)\lessapprox_\delta \delta^{-\eta_1-\eps_1}\leq\delta^{-\zeta}$, while by Remark \ref{FrostmanWolffInheritedUpwardsDownwards}(B) we have $\CKT(\tubes_1^{T_\rho})\lesssim \CKT(\tubes)\leq\delta^{-\eta}$ for each $T_\rho\in\tubes_\rho$. 

\medskip

\noindent{\bf Step 2.}
We now consider the case where $\CKT(\tubes_\rho) > \delta^{-\zeta}$, and hence
\begin{equation}\label{lotsOfRhoTubesInsideEachZ}
\#\tubes_\rho[Z]\gtrsim \delta^{-\zeta}\frac{|Z|}{|T_\rho|}\quad\textrm{for each}\ Z\in\mathcal{Z}.
\end{equation}

Our goal is to show that Conclusion (A) of Proposition \ref{refinedInductionOnScaleProp} holds, provided $\alpha>0$ is chosen appropriately. 

First, we claim that either Conclusion (A) of Proposition \ref{refinedInductionOnScaleProp} holds, or else the prisms in $\mathcal{Z}$ are almost tubes. Indeed, let $t\times\theta\times 2$ be the dimensions of the prisms in $\mathcal{Z}$. If $\eps_1$ is chosen sufficiently small depending on $\eps_2, \omega,$ and $\sigma$, then by applying Proposition \ref{PropALLbPropGen} with $\eps_2$ in place of $\eps$, we have
\begin{equation}
\Big|\bigcup_{T\in\tubes_1}Y_1(T)\Big| \gtrsim   \delta^{\omega+\eps_2}\Big(\frac{\theta}{t}\Big)^{\omega} (\#\tubes)|T|\big( (\#\tubes)|T|^{1/2}\big)^{-\sigma},
\end{equation} 
where the implicit constant depends on $\eps_2$. In particular, we may suppose that 
\begin{equation}\label{tAlmostTheta}
t \geq \delta^{\eps_3}\theta,
\end{equation} 
or else Conclusion (A) of Proposition \ref{refinedInductionOnScaleProp} holds, provided $\eps_2\leq \eps_3\omega/2$ and $\alpha\leq\eps_3\omega/2$. Replace each $t\times\theta\times 2$ prism $Z\in\mathcal{Z}$ with its coaxial $\theta$-tube. After dyadic pigeonholing and replacing $(\tubes_1,Y_1)_\delta$ and $\tubes_\rho$ with a $\approx_\delta 1$ refinement, we can find a balanced cover $\tubes_\theta$ of $\tubes_\rho$ that factors $\tubes_\theta$ from below with respect to the Frostman Convex Wolff Axioms and from above with respect to the Katz-Tao Wolff Axioms, both with error $\lessapprox_\delta \delta^{-\eps_3}$.

\eqref{lotsOfRhoTubesInsideEachZ} implies that for each $T_\theta\in\tubes_\theta$ we have $\#\tubes_\rho[T_\theta]\gtrapprox_\delta \delta^{\eps_3-\zeta}(\theta/\rho)^2$, and hence
\begin{equation}\label{manyTubesRho}
\#\tubes_\rho\gtrapprox_\delta \delta^{\eps_3-\zeta}(\theta/\rho)^2(\#\tubes_\theta).
\end{equation}

\medskip

\noindent{\bf Step 3.}
For each $T_\theta\in\tubes_\theta$, define 
\[
\mathcal{G}_{T_\theta}=\bigcup_{T_\rho\in\tubes_\rho[T_\theta]}\mathcal{G}_{T_\rho}.
\] 
We claim that either Conclusion (A) holds (for a suitably chosen value of $\alpha$), or else there is a $\approx_\delta 1$ refinement of $(\mathcal{G},Y)_{a\times\rho c\times c}$ so that the following holds: $a\in [\delta, \delta^{1-\eps_1}]$ and for each $T_\theta\in\tubes_\theta$ and each $G\in\mathcal{G}_{T_\theta}$, the set of grains $G'\in\mathcal{G}_{T_\theta}$ with $Y(G)\cap Y(G')\neq\emptyset$ is contained in a prism of dimensions comparable to $\frac{\theta \delta^{1-\eps_2}}{\rho}\times \theta c \times c$ (compare this with the dimensions of $G$, which are $a\times \rho c \times c$). The argument is identical to the argument in Steps 1 and 2 from Lemma \ref{WiderGrainsSmallCKT}; we refer the reader to those Steps for details.

We shall suppose henceforth that for each $T_\theta\in\tubes_\theta$ and each $G\in\mathcal{G}_{T_\theta}$, the set of grains $G'\in\mathcal{G}_{T_\theta}$ with $Y(G)\cap Y(G')\neq\emptyset$ is contained in a prism of dimensions comparable to $\frac{\theta \delta^{1-\eps_2}}{\rho}\times \theta c \times c$. 

\medskip

\noindent{\bf Step 4.}
By dyadic pigeonholing we can find a number $\mu_{\operatorname{fine}}$ and a $\approx_\delta 1$ refinement of $(\tubes_1,Y_1)_\delta$ and $\tubes_\rho$ so that for each $T_\rho\in\tubes_\rho$ and each $x\in\bigcup_{T\in\tubes_1[T_\rho]}Y_1(T)$, we have 
\[
\#\big((\tubes_1[T_\rho])_{Y_1}(x)\big) \sim \mu_{\operatorname{fine}}.
\] 
We can choose these refinements so it continues to be the case that $(\tubes_1^{T_\rho}, Y_1^{T_\rho})_{\delta/\rho}$ is $\gtrapprox_\delta \delta^{\eps_1}$ dense and $\FS(\tubes_1^{T_\rho})\lessapprox_\delta\delta^{-\eps_1}$ for each $T_\rho\in\tubes_\rho$ (recall that we still have $\CKT(\tubes_1^{T_\rho})\leq\delta^{-\eta}$).

If $\eps_1$ is chosen sufficiently small depending on $\eps_2$, then we can apply the estimate $\cE(\sigma,\omega)$ to conclude that for each $T_\rho\in\tubes_\rho$, we have
\[
\Big|\bigcup_{T\in\tubes_1[T_\rho]}Y_1(T)\Big| \gtrsim   \big(\frac{\delta}{\rho}\big)^{\omega+\eps_2}(\#\tubes_1[T_\rho])  |T| \Big((\#\tubes_1[T_\rho])\big(\frac{|T|}{|T_\rho|}\big)^{1/2}\Big)^{-\sigma},
\]
where the implicit constant depends on $\eps_2$, and hence by \eqref{manyTubesRho},
\begin{equation}\label{upperBdMuFine}
\mu_{\operatorname{fine}} 
\lesssim   \Big(\frac{\delta}{\rho}\Big)^{-\omega-\eps_2}\Big(\frac{\#\tubes_1}{\#\tubes_\rho} \frac{\delta}{\rho}\Big)^{\sigma}
\lessapprox_\delta \delta^{-2\eps_3+\sigma\zeta}\Big(\frac{\delta}{\rho}\Big)^{-\omega}\Big(\frac{\#\tubes_1}{\#\tubes_\theta} \frac{\delta\rho}{\theta^2}\Big)^{\sigma}.
\end{equation}

\medskip

\noindent{\bf Step 5.}
In previous applications of induction on scale, the estimate \eqref{upperBdMuFine} would be paired with a multiplicity estimate on the tubes in $\tubes_\rho$. Our innovation, however, is to pair the estimate \eqref{upperBdMuFine} with a multiplicity estimate on $\mathcal{G}$.

After refining the pair $(\mathcal{G},Y)_{a\times\rho c \times c}$ by a $\approx_\delta 1$ factor (this in turn refines $(\tubes_1,Y_1)_\delta$ by a similar quantity), we can find a number $\mu_{\operatorname{medium}}$ so that for each $T_\theta\in\tubes_\theta$ and each $x\in\bigcup_{G\in\mathcal{G}_{T_\theta}}Y(G)$, we have
\[
\#\{G \in \mathcal{G}_{T_\theta} \colon x\in Y(G)\} \sim \mu_{\operatorname{medium}}.
\]

Our task is to estimate $\mu_{\operatorname{medium}}$. Recalling the conclusion of Step 3, we can cover $T_\theta$ by rectangular prisms $P$ of dimensions comparable to $\frac{\theta \delta^{1-\eps_2}}{\rho} \times \theta c \times c$, so that every pair of grains $G,G'\in\mathcal{G}_{T_\theta}$ with $Y(G)\cap Y(G')\neq\emptyset$ are contained in a common prism. Let $\mathcal{P}$ denote this set of prisms; then for each $P\in\mathcal{P}$, $\mathcal{G}_{T_\theta}^P$ is a set of prisms, each of which has dimensions roughly $\frac{\rho}{\theta}\times \frac{\rho}{\theta}\times 1$ (more precisely, each prism in $\mathcal{G}_{T_\theta}^P$ has dimensions comparable to $s\times t\times 1$, where $s,t\in [\delta^{\eps_2}\frac{\rho}{\theta},\frac{\rho}{\theta}]$; this additional $\delta^{\eps_2}$ factor will be harmless).  After pigeonholing, we may assume that the lengths $s$ and $t$ are the same for every $T_{\theta}\in \mathbb{T}_{\theta}$ and every $P\in \mathcal{P}$.

We have $\CKT(\mathcal{G}_{T_\theta}^P)\lesssim\delta^{-\eps_2} \CKT^{\operatorname{loc}}(\mathcal{G}) \leq\delta^{-\eps_1-\eps_2}$, while 
\[
\FS(\mathcal{G}_{T_\theta}^P)\leq \CFC(\mathcal{G}_{T_\theta}^P)\leq \CKT(\mathcal{G}_{T_\theta}^P)(\rho/\theta)^{-2}(\#\mathcal{G}_{T_\theta}^P)^{-1}\leq \delta^{-\eps_1 -2\eps_2}(\rho/\theta)^{-2}(\#\mathcal{G}_{T_\theta}^P)^{-1}.
\]

If $\eps_2$ is selected sufficiently small depending on $\eps_3$, $\omega$, and $\sigma$, then we can apply Assertion $\cE(\sigma,\omega)$ to conclude that
\begin{equation*}
\begin{split}
\Big|\bigcup_{G^P \in \mathcal{G}_{T_\theta}^P}Y^P(G^P)\Big| & \geq \big(\frac{\rho}{\theta}\big)^{\omega+\eps_3} \delta^{5\eps_2} (\#\mathcal{G}_{T_\theta}^P)|G^P|\Big( \big[(\rho\theta)^{-2}(\#\mathcal{G}_{T_\theta}^P)^{-1}\big]\big[\#\mathcal{G}_{T_\theta}^P\big]\big[|G^P|^{1/2}] \Big)^{-\sigma}\\
&\gtrapprox_\delta \big(\frac{\rho}{\theta}\big)^{\omega+\eps_3}  \delta^{5\eps_2} (\#\mathcal{G}_{T_\theta}^P)|G^P|\big(\frac{\rho}{\theta}\big)^{\sigma},
\end{split}
\end{equation*}
and thus
\begin{equation}\label{upperBdMuMedium}
\mu_{\operatorname{medium}} \lessapprox_\delta \delta^{-2\eps_3} \big(\frac{\rho}{\theta}\big)^{-\omega-\sigma}.
\end{equation}

\medskip

\noindent{\bf Step 6.}
At this point, we have estimated the quantities $\mu_{\operatorname{fine}}$ and $\mu_{\operatorname{medium}}$. The former allows us to control the number of $\delta$ tubes that contribute to (a specific point in) a grain, while the latter allows us to control the number of grains that contribute to (a specific point in) a $\theta$ tube. 

It remains to place a dense shading on $\tubes_\theta$ and obtain a corresponding multiplicity estimate for the number of $\theta$ tubes that contribute to (a specific point in) $\RR^3$.  After dyadic pigeonholing, we can refine $(\tubes_1,Y_1)_\delta$ and $\tubes_\rho$ so that for each $T_\theta\in\tubes_\theta$ and each $x\in\bigcup_{T\in\tubes_1[T_\theta]}Y_1(T)$, we have that $|B(x,\theta)\cap\bigcup_{T\in\tubes_1[T_\theta]}Y(T)|$ has roughly the same volume. Let $Y(T_\theta)=T_\theta\cap N_\theta\big(\bigcup_{T\in\tubes_1[T_\theta]}Y_1(T)\big);$ then $(\tubes_\theta,Y)_\theta$ is $\gtrapprox_\delta \delta^{\eps_1}$ dense. After further pigeonholing we can find a number $\mu_{\operatorname{coarse}}$ so that 
\[
\# (\tubes_\theta)_Y(x) \sim\mu_{\operatorname{coarse}}\quad\textrm{for each}\ x\in\bigcup_{T_\theta\in\tubes_\theta}Y(T_\theta).
\]

By Remark \ref{FrostmanWolffInheritedUpwardsDownwards}(A), we have $\FS(\tubes_\theta)\lesssim\FS(\tubes)\lessapprox_\delta\delta^{-\eps_1}$, and $\CKT(\tubes_\theta)\lesssim \delta^{-\eps_3}\CKT(\mathcal{Z})\lessapprox_\delta \delta^{-\eps_3}$. Thus if $\eps_1$ is chosen sufficiently small compared to $\eps_2$, then we can apply $\cE(\sigma,\omega)$ to conclude that
\begin{equation*}
\begin{split}
\Big|\bigcup_{T_\theta\in\tubes_\theta}Y(T_\theta)\Big|& \gtrsim \theta^{\omega+\eps_2 +\eps_3}(\#\tubes_\theta)|T_\theta|\Big((\#\tubes_\theta)|T_\theta|^{1/2}\Big)^{-\sigma},
\end{split}
\end{equation*}
and hence
\begin{equation}\label{upperBdMuCoarse}
\mu_{\operatorname{coarse}}\lesssim \theta^{-\omega-\eps_2-\eps_3}\Big((\#\tubes_\theta)\theta\Big)^{\sigma}.
\end{equation}

Combining \eqref{upperBdMuFine}, \eqref{upperBdMuMedium}, and \eqref{upperBdMuCoarse}, we conclude that for each $x\in\RR^3$ we have
\begin{equation}
\begin{split}
\#\{T\in\tubes_1\colon x\in Y_1(T)\} & \leq \mu_{\operatorname{fine}}\ \mu_{\operatorname{medium}}\ \mu_{\operatorname{coarse}}\\
& \lessapprox
\delta^{-4\eps_3+\sigma\zeta} \delta^{-\omega}\Big((\#\tubes)\delta\Big)^{\sigma}.
\end{split}
\end{equation}
We conclude that
\[
\Big|\bigcup_{T\in\tubes}Y(T)\Big| \gtrapprox_\delta \delta^{\omega + 4\eps_3 - \sigma\zeta}(\#\tubes_1)|T|\Big((\#\tubes)|T|^{1/2}\Big)^{-\sigma}.
\]
Since $\#\tubes_1\gtrapprox_\delta \delta^{\eps_1}(\#\tubes)$, we have that Conclusion (A) holds, provided we select $\eps_3 < \sigma\zeta/10$ and $\alpha<\sigma\zeta/10$.
\end{proof}


\section{Sticky Kakeya for tubes satisfying the Katz-Tao Convex Wolff Axioms at every Scale}\label{stickyKakeyaEveryScaleSec}
In Section \ref{cEIffcDSec}, we recalled a version of the Sticky Kakeya Theorem that was proved in \cite{WZ23}; this is Theorem \ref{WZThm52}. Theorem \ref{WZThm52} applies to families of tubes that satisfy the Frostman Convex Wolff Axioms at every scale, in the sense of Definition \ref{convexAtEveryScaleFromAssouadPaper}. In this section, we will prove an analogue of Theorem \ref{WZThm52} for sets of tubes that satisfy the \emph{Katz-Tao} Convex Wolff Axioms at every scale.

\begin{defn}\label{KatzTaoConvexWolffAtEveryScaleDefn}
Let $K\geq 1,\delta>0$. We say a set $\tubes$ of essentially distinct $\delta$-tubes satisfies the \emph{Katz-Tao Wolff Axioms at every scale with error $K$} if for every $\rho_0\in [\delta,1]$, there exists $\rho\in [\rho_0,  K \rho_0)$ and a set of $\rho$-tubes $\tubes_\rho$ that satisfies the following properties.
\begin{itemize}
	\item[(i)] $\tubes_\rho$ is a $K$-balanced partitioning cover of $\tubes$.
	\item[(ii)] $\CKT(\tubes_\rho)\leq K$.
\end{itemize}
\end{defn}

\begin{thm}\label{katzTaoEveryScaleStickyKakeyaThm}
For all $\eps>0$, there exists $\eta,\kappa>0$ so that the following holds for all $\delta>0$. Let $\tubes$ be a set of $\delta$-tubes that satisfy the Katz-Tao Convex Wolff Axioms at every scale with error $\delta^{-\eta}$, and let $Y(T)$ be a $\delta^{\eta}$ dense shading. Then
\begin{equation}\label{katzTaoEveryScaleStickyKakeyaIneq}
\Big|\bigcup_{T\in\tubes}Y(T)\Big|\geq \kappa\delta^\eps(\#\tubes)|T|.
\end{equation}
\end{thm}

In the next section, we will combine Theorem \ref{katzTaoEveryScaleStickyKakeyaThm} with Proposition \ref{refinedInductionOnScaleProp} to prove Proposition \ref{improvingProp}. Theorem \ref{katzTaoEveryScaleStickyKakeyaThm} is proved by combining Theorem \ref{WZThm52} with the following Nikishin-Stein-Pisier Factorization type result.
\begin{prop}\label{factorizationProp}
Let $\eps>0$. Then there exists $K_\eps,\eta>0$ so that the following holds for all $\delta>0$. Let $\tubes$ be a non-empty set of $\delta$ tubes inside the unit ball in $\RR^3$ that satisfy the Katz-Tao Convex Wolff axioms at every scale with error $\delta^{-\eta}$. Then there exist rigid transformations $A_1,\ldots,A_N,$ $N\leq K_\eps (\#\tubes)^{-1}|T|^{-1}$ so that each set $A_i(\tubes)$ is contained inside $B(0,2)$, and $\bigcup_{i=1}^N A_i(\tubes)$ contains a subset of essentially distinct tubes that satisfies the Frostman Convex Wolff Axioms at every scale with error $K_\eps\delta^{-\eps}$. 
\end{prop}

\begin{proof}[Proof of Theorem \ref{katzTaoEveryScaleStickyKakeyaThm} using Proposition \ref{factorizationProp}]
Fix $\eps>0$ and let $\eta=\eta(\eps)>0$ be a small quantity to be determined below. Let  $(\tubes,Y)_\delta$ be $\delta^\eta$ dense, and suppose that $\tubes$ satisfies the Katz-Tao Wolff Axioms at every scale with error $\delta^{-\eta}$. After a harmless refinement we may suppose $|Y(T)|\geq\delta^{2\eta}|T|$ for each $T\in\tubes$. 

Apply Proposition \ref{factorizationProp} with a small value $\eps_1$ in place of $\eps$. We may do this, provided $\eta$ is selected sufficiently small depending on $\eps_1$. Let $\tilde\tubes \subset \bigcup_{i=1}^N A_i(\tubes)$ be the output from Proposition \ref{factorizationProp}. Note that each $\tilde T\in\tilde\tubes$ is of the form $\tilde T = A_i(T)$ for some index $i$ and some $T\in\tubes$, and hence we can define the shading $\tilde Y(\tilde T) = A_i(Y(T))$; we have $|\tilde Y(\tilde T)| = |Y(T)|\geq\delta^{2\eta}|T|$, and hence $(\tilde T, \tilde Y)_\delta$ is $\delta^{2\eta}$ dense. 

If $\eps_1$ and $\eta$ are chosen sufficiently small depending on $\eps$, then we can apply Theorem \ref{WZThm52} to conclude that
\[
N\Big|\bigcup_{T\in\tubes}Y(T)\Big| =  \sum_{i=1}^N\Big|\bigcup_{T\in\tubes}A_i(Y(T))\Big|\geq \Big| \bigcup_{i=1}^N \bigcup_{T\in\tubes}A_i(Y(T))\Big| \geq \Big|\bigcup_{\tilde T\in\tilde \tubes}\tilde Y(\tilde T)\Big| \geq \kappa_\eps \delta^{\eps}.
\]
Re-arranging and noting that $N\leq K_{\eps}(\#\tubes)^{-1}|T|^{-1}$, we obtain \eqref{katzTaoEveryScaleStickyKakeyaIneq}, with $\kappa=\kappa_{\eps} K_{\eps}^{-1}$.
\end{proof}

It remains to prove Proposition \ref{factorizationProp}. We will do so below.


\subsection{Nikishin-Stein-Pisier Factorization and the Convex Wolff Axioms}\label{factorizationArgSection}

We begin with a single-scale version of Proposition \ref{factorizationProp}. We first need the following definition. 
\begin{defn}
We say that a set $\tubes$ of $\delta$ tubes is regular with granularity $\tau\in[\delta,1]$ if for every scale $\rho\in[\delta,1]$ of the form $\rho = \delta \tau^{-\ell}$, $\ell\in\mathbb{N}$, we have that $\tubes$ has a balanced partitioning cover by $\rho$ tubes.
\end{defn}

\begin{defn}
For $\rho>0$, we define $\mathfrak{A}_\rho$ to be the set of rigid transformations $A\colon \RR^3\to\RR^3$ that satisfy $|Ax-x|\leq\rho$ for all $x\in B(0,1)$. 
\end{defn}

\begin{lem}\label{findingGoodRotationManyFamilies}
For all $\eps>0$, there exists $\eta>0$ and $K_\eps\geq 1$ so that the following holds for all $0<\delta\leq\rho\leq 1$. Let $K,M\geq 1$ and let $\tubes_1,\ldots,\tubes_K$ be sets of $\delta$ tubes in $B(0,1)\subset\RR^3$, each of cardinality at most $M$. Suppose that the tubes in each set $\tubes_j$ are regular with granularity $\delta^{\eta}$, and furthermore each set $\tubes_j$ is contained in a $\rho$ tube.

Then there exists a set of rigid transformations $\mathcal{A}\subset\mathfrak{A}_\rho$  with $\#\mathcal{A} = \big\lceil \frac{\rho^2}{M \delta^2}\big\rceil$, so that
\begin{equation}\label{favorableCKTBd}
\CKT\Big(\bigsqcup_{A\in\mathcal{A}} A(\tubes_j)\Big) \leq K_\eps\delta^{-\eps}(\log(2+K))\CKT(\tubes_j),\quad j = 1,\ldots,K.
\end{equation}
\end{lem}
\begin{rem}
Note that for distinct $A,A'\in\mathcal{A}$, the sets $A(\tubes_j)$ and $A'(\tubes_j)$ might contain common tubes, and thus the disjoint union on the LHS of \eqref{favorableCKTBd} should be interpreted as a multiset. By Remark \ref{remarksFollowingConvexWolffDefn}(D), the LHS of \eqref{favorableCKTBd} is well-defined.
\end{rem}

\begin{proof}$\phantom{1}$\\
{\bf Step 1.}
Define $\tilde\delta = \delta/\rho$. First, we may suppose that $M<\tilde\delta^{-2}$, or else we can define $\mathcal{A}=\{I\}$ (here $I\colon\RR^3\to\RR^3$ is the identity map) and we are done. Similarly, we may suppose $\rho\geq\delta^{1-\eps/2}$, or else we can define $\mathcal{A}$ to be $\big\lceil \frac{\rho^2}{M \delta^2}\big\rceil$ infinitesimally perturbed copies of $I$, and \eqref{favorableCKTBd} follows from the fact that
\[
\CKT\Big(\bigsqcup_{A\in\mathcal{A}} A(\tubes_j)\Big) \leq (\#\mathcal{A})\CKT(\tubes_j).
\]

Fix an index $j\in[1,\ldots, K]$ and let $\tubes = \tubes_j$. By hypothesis, all of the tubes in $\tubes$ are contained in a common $\rho$ tube, which we will denote by $T_\rho$. Fix numbers $\delta\leq a\leq b\leq 2\rho$, with both $a$ and $b$ of the form $\delta^{\ell \eta}$. Let $\nu\geq 1$ be a power of 2. By hypothesis, $\tubes$ has a balanced partitioning cover $\tubes_a$.

Let $\mathcal{W}_{\nu}$ be a maximal set of essentially distinct $a\times b\times 2$ prisms, each of which satisfy $\#\tubes[W]\in  \big[\nu\frac{\#\tubes\phantom{_a}}{\#\tubes_a},\ 2\nu\frac{\#\tubes\phantom{_a}}{\#\tubes_a}\big)$. Note that each $W\in \mathcal{W}_{\nu}$ is contained in $N_{2\rho}(T_\rho)$. Observe that if $W$ is an $a\times b\times 2$ prism and $T\in\tubes[T_a]$ with $T\subset W$, then $T_a\subset 2W$. In particular, since $\tubes_a$ is a balanced partitioning cover of $\tubes$, we have
\[
\#\tubes_a[2W] \geq \nu/2\quad\textrm{for each}\ W\in\mathcal{W}_\nu.
\]

Each tube $T_a\in\tubes_a$ is contained in $\lesssim \frac{b}{a}$ essentially distinct $2a\times 2b\times 4$ prisms. Thus by double-counting we have
\begin{equation}\label{estimateOnSizeWNu}
\#\mathcal{W}_\nu \lesssim (\#\tubes_a)\frac{b}{a} \nu^{-1}.
\end{equation}
The above estimate is useful when $\nu$ is not too large. When $\nu\frac{\#\tubes\phantom{_a}}{\#\tubes_a} \geq \CKT(\tubes)(ab)|T|^{-1}$, then $\mathcal{W}_\nu=\emptyset$. 

\medskip

\noindent {\bf Step 2.}
We say two rigid motions $A,A'\in\mathfrak{A}_\rho$ are $\delta$-separated if there exists a point $x\in B(0,1)$ with $|A(x) - A'(x)|\geq\delta$. Let $\mathfrak{A}_{\rho}^\delta$ be a maximal $\delta$-separated subset of $\mathfrak{A}_\rho$; we have $\#\mathfrak{A}^\delta_\rho \sim \tilde\delta^{-6}=\delta^6/\rho^6$. 

Let
\begin{equation}\label{defnOfN}
N = 2\Big\lceil \frac{\rho^2}{M\delta^2} \Big\rceil,
\end{equation}
and let $A_1,\ldots,A_N$ be chosen uniformly and independently at random from $\mathfrak{A}_\rho^\delta$. We have
\begin{equation}\label{probabilityA1ANLarge}
\mathbb{P}\big( \#\{A_1,\ldots,A_N\}\geq  N/2\big)\geq 3/4,
\end{equation}
where $\#\{A_1,\dots, A_N\}$ denotes the number of distinct rigid motions in the set $\{A_1,\dots, A_N\}$. 

 Fix a $a\times b\times 2$ prism $W_0$. We would like to estimate the probability that
\begin{equation}\label{badW0EstimateHolds}
\#\Big( \bigsqcup_{i=1}^N A_i(\tubes) \Big)[W_0] \geq K_\eps \delta^{-\eps/2}(\log(2+K)) \CKT(\tubes) |W_0||T|^{-1},
\end{equation}
i.e.~we would like to estimate the probability that
\begin{equation}\label{W0Bad}
\sum_{i=1}^N \# \tubes[(A_i^{-1}(W_0))] \geq K_\eps \delta^{-\eps/2}(\log (2+K)) \CKT(\tubes) |W_0||T|^{-1}.
\end{equation}
Note that if \eqref{W0Bad} occurs, then by pigeonholing, there must be some dyadic $\nu$ so that
\begin{equation}\label{badNuAndW0}
\#\big\{ (W,i)\in  (\mathcal{W}_\nu\times\{1,\ldots,N\})\colon  A_i(W)\ \textrm{is contained in }\ 10W_0\big\} \geq Z_\nu,
\end{equation}
where 
\begin{equation}\label{ZNu}
Z_\nu  = \big(4\log(1/\delta)\big)^{-1} K_\eps \delta^{-\eps/2}(\log(2+K)) \CKT(\tubes) |W_0||T|^{-1}    \Big( \sup_{W\in \mathcal{W}_\nu} \#\tubes[W] \Big)^{-1}. 
\end{equation}
We may suppose that
\begin{equation}\label{boundNyTubesTubesA}
\sup_{W\in \mathcal{W}_\nu}  \#\tubes[W]  \sim \nu \frac{\#\tubes}{\#\tubes_a}  \leq  \CKT(\tubes)(ab)|T|^{-1},
\end{equation}
since otherwise $\mathcal{W}_\nu=\emptyset$, and thus it is impossible for either of Inequality \eqref{badNuAndW0} or \eqref{W0Bad} to be true. Since $|W_0|=2 ab$, we can use \eqref{boundNyTubesTubesA} to bound $Z_\nu$. We conclude that
\[
Z_\nu \geq \big(2\log(1/\delta)\big)^{-1} K_\eps \delta^{-\eps/2}(\log (2+K)).
\]
We will choose the constant $K_\eps$ sufficiently large so that $Z_\nu \geq 2$, and in particular 
\begin{equation}\label{ZnuVsNzM1}
Z_\nu - 1 \sim Z_\nu.
\end{equation} 
This will be relevant in Step 3 when we apply Chernoff's inequality.

\medskip

\noindent {\bf Step 3.}
We will estimate the probability that \eqref{badNuAndW0} occurs for a fixed choice of $\nu$ and $W_0$. For two $a\times b\times 2$ prisms $W,W_0$, both of which are contained inside $N_{2\rho}(T_\rho)$ we have that the number of $A\in\mathfrak{A}_\rho^\delta$ for which $A(W)$ is comparable to $W_0$ (or equivalently, $A^{-1}(W_0)$ is comparable to $W$) 
is $\lesssim \delta^{-6} a^2 b^2\rho \min \{ a/b, \rho \}.$ We will write this as $\tilde\delta^{-6}\tilde{a}^2 \tilde{b}^2 \min \{ \tilde{a}/(\tilde{b} \rho), 1\} \leq \tilde\delta^{-6}\tilde{a}^2 \tilde{b}^2$, where we define $\tilde{a}=a/\rho$ and $\tilde{b}=b/\rho$.

The reason for this numerology is as follows. Without loss of generality, assume $W_0= [0, a]\times [0, b]\times [0,2]$. Then a rigid motion  $A$ is determined  by   $A(v_i)$ with $v_0=(0, 0, 0), v_1=(0, 1, 0), v_2=(0, 0,  1)$.  Since $A\in \mathfrak{A}_{\rho}^{\delta}$, the number of $\delta$-separated choice for $A(v_0)$ is  $\leq \frac{| W\cap B(0, \rho) |}{ \delta^3} \sim \frac{ab\rho}{\delta^3}$. Once $A(v_0)$ is fixed, the number of $\delta$-separated choices for $A(v_2)$ is $\leq \frac{ab}{\delta^2}$.   Once $A(v_0), A(v_2)$ are fixed, the number of $\delta$-separated choices for $A(v_1)$ is $\leq \min \{ \frac{a}{b}, \rho \} \delta^{-1}$.

Thus if we define $X_i$ to be the event that there exists $W\in\mathcal{W}_\nu$ such that $A_i(W)$ is comparable to $W_0$, then
\begin{equation}\label{probXi}
\mathbb{P}(X_i)\lesssim \tilde{\delta}^{-6}\tilde a^2 \tilde b^2 \cdot \#\mathfrak{A}_\rho^\delta \cdot  \#\mathcal{W}_\nu \lesssim  \tilde a^2\tilde b^2 (\#\mathcal{W}_\nu).
\end{equation}
Since the prisms in $\mathcal{W}_\nu$ are essentially distinct, there can exist at most $O(1)$ $W\in\mathcal{W}_\nu$ such that $A_i(W)$ is comparable to $W_0$. Thus by linearity of expectation, we have
\begin{equation*}
\begin{split}
\mathbb{E}&\Big( \#\{ (W,i)\in  (\mathcal{W}_\nu\times\{1,\ldots,N\})\colon  A_i(W)\ \textrm{is comparable to}\ W_0\}\Big)\\
& = \mathbb{E}(X_1+\ldots+X_N)\\
&\lesssim   N  \tilde a^2 \tilde b^2  (\#\mathcal{W}_\nu)\\
&\lesssim 2\Big\lceil \frac{\rho^2}{M\delta^2} \Big\rceil  \tilde a^2 \tilde b^2  (\#\mathcal{W}_\nu) 
\end{split}
\end{equation*}
On the third line we used \eqref{probXi}; on the fourth line we used \eqref{defnOfN} and  \eqref{estimateOnSizeWNu}.

Define $X = X_1+\ldots+X_N$. Recalling \eqref{ZNu} and \eqref{ZnuVsNzM1}, we have 
\begin{equation*}
\begin{split}
\gamma:=  \frac{Z_\nu}{\mathbb{E}(X)}
& \gtrsim \log(1/\delta)^{-1} K_\eps \delta^{-\eps/2}(\log (2+K))  \CKT(\tubes) |W_0||T|^{-1} \Big( 2\Big\lceil \frac{\rho^2}{M\delta^2} \Big\rceil  \tilde a^2 \tilde b^2  \sup_{W\in \mathcal{W}_\nu} \#\tubes [W] \cdot  (\#\mathcal{W}_\nu ) \Big)^{-1} \\
&  \gtrsim \log(1/\delta)^{-1} K_\eps \delta^{-\eps/2}(\log (2+K))  \CKT(\tubes) |W_0||T|^{-1} \Big(  2\Big\lceil \frac{\rho^2}{M\delta^2} \Big\rceil  \tilde a^2 \tilde b^2 \frac{b}{a} (\# \tubes) \Big)^{-1}\\
& \gtrsim \log(1/\delta)^{-1} K_\eps \delta^{-\eps/2}(\log (2+K))
\end{split}
\end{equation*}
On the second line we used \eqref{estimateOnSizeWNu} and the inequality  $\#\tubes[W]\sim \nu \frac{\#\tubes}{\tubes_a}$; on the third line we used the fact that $|T|\sim\delta^2$ and the fact that $\tilde a = a/\rho$; and on the final line we used the fact that $\CKT(\tubes)\geq 1$, $\tilde b^2\leq 4$, $|W_0| = 2ab$, and $\#\tubes\leq M$.

Hence we can apply the multiplicative Chernoff's inequality to conclude that
\begin{equation*}
		\mathbb{P}(X\geq Z_\nu) 
		\leq \Big( \frac{e^{\gamma-1}}{\gamma^\gamma}\Big)^{\mathbb{E}(X)}  
		\lesssim e^{-\gamma \mathbb{E}(X)} =  e^{-Z_\nu}
		\lesssim (2+K)^{- (\log 1/\delta)^{-1} K_\eps \delta^{-\eps/2}} 
		\lesssim K^{-1} \exp[-K_\eps \delta^{-\eps/3}].
\end{equation*}
For the second inequality, we used the fact that $K_{\eps}$ is sufficiently large, $ \frac{ e^{\gamma-1}}{\gamma^{\gamma} }\leq e^{-\gamma}$ for all $\delta>0$ and $K\geq 1$.

\medskip

\noindent {\bf Step 4.}
There are $\log(1/\delta)$ choices of $\nu$; at most $a^{-3}b^{-1}\leq\delta^{-4}$ essentially distinct prisms $W_0$; and $\eta^{-2}$ choices of numbers $(a,b)$. Thus the probability that there exists some prism $W\subset N_{2\rho}(T_\rho)$ of dimensions $a\times b\times 2$ for some pair $a \leq b$ of the form $\delta^{\eta \ell}$ for which \eqref{badW0EstimateHolds} holds is $\lesssim \delta^{-6} K^{-1} \exp[-K_\eps \delta^{-\eps/3}].$ We will select $K_\eps$ sufficiently large so that this quantity is at most $(2K)^{-1}$. If no such prism exists, then provided $\eta\leq \eps/4$ we have
\begin{equation}\label{goodCKTBdOnTubes}
\CKT\Big(\bigsqcup_{i=1}^N A_i(\tubes)\Big) \leq K_\eps \delta^{-\eps}(\log(2+K)) \CKT(\tubes).
\end{equation}
Indeed, for every prism $W\subset\RR^3$ of dimensions $a\times b\times 2,$ we can select a prism $W'\subset W$ of dimensions $a'\times b'\times 2$ with $a',b'$ of the form $\delta^{ \ell \eta}$ and $|W'|\leq\delta^{-2\eta}|W|\leq \delta^{-\eps/2}|W|$.

In particular, we have that \eqref{goodCKTBdOnTubes} holds with probability at least $1-(2K)^{-1}$. Since there are $K$ sets of tubes $\tubes_1,\ldots,\tubes_K$, we conclude that with probability at least $1/2$, we have that \eqref{goodCKTBdOnTubes} is true for every set $\tubes_1,\ldots,\tubes_K$. Finally, by \eqref{probabilityA1ANLarge}, we have that the probability that $\#\mathcal{A}\geq N/2$ \emph{and} \eqref{goodCKTBdOnTubes} is true for every set $\tubes_1,\ldots,\tubes_K$ is at least $1/4$. We conclude that there exists a choice of $A_1,\ldots,A_N$ and a set $\mathcal{A}\subset\{A_1,\ldots A_N\}$ of cardinality $\Big\lceil \frac{\rho^2}{M\delta^2} \Big\rceil$ so that \eqref{favorableCKTBd} holds.
\end{proof}

Lemma \ref{findingGoodRotationManyFamilies} has the following consequence.
\begin{cor}\label{corFactoringSingleScale}
For all $\eps>0$, there exists $\eta>0$ and $K_\eps\geq 1$ so that the following holds for all $0<\delta\leq\rho\leq 1$. Let $\tubes$ be a set of $\delta$-tubes, and let $\tubes_\rho$ be a balanced partitioning cover of $\tubes$. Suppose that for each $T_\rho\in\tubes_\rho$, we have that $\tubes^{T_\rho}$ is regular with granularity $(\delta/\rho)^{\eta}$. 

Then there exists a set of rigid transformations $\mathcal{A}$ with $\#\mathcal{A} =   2\big \lceil \frac{\#\tubes_\rho}{\#\tubes\phantom{_\rho}} \frac{|T_\rho|}{|T|}\big\rceil$ and
\begin{equation}\label{ATRhoSubset3TRho}
A(T_\rho)\subset N_{3\rho}(T_\rho)\quad\textrm{for each}\ A\in\mathcal{A},\ T_\rho\in\tubes_\rho,
\end{equation}
so that

\begin{equation}\label{favorableCKTBdInsideRho}
\CKT\Big(\Big(\bigsqcup_{A\in\mathcal{A}} A(\tubes)\Big)[N_{3\rho}(T_\rho)]\Big)  \leq K_\eps\delta^{-\eps}\CKT(\tubes)\quad\textrm{for each}\ T_\rho\in\tubes_\rho.
\end{equation}
\end{cor} 
\begin{rem}
An immediate application of Lemma \ref{findingGoodRotationManyFamilies} yields the slightly weaker statement
\[
\CKT\Big(\Big(\bigsqcup_{A\in\mathcal{A}} A(\tubes[T_\rho])\Big)\Big)  \leq K_\eps\delta^{-\eps}\CKT(\tubes)\quad\textrm{for each}\ T_\rho\in\tubes_\rho,
\]
but \eqref{favorableCKTBdInsideRho} follows from the hypothesis that the tubes in $\tubes_\rho$ are essentially distinct, and hence there are $\lesssim 1$ tubes $T_\rho'\in\tubes_\rho$ for which $\big(\bigsqcup_{\mathcal{A}} A(\tubes[T_\rho'])\big)[T_\rho]$ is non-empty.
\end{rem}

We are now ready to prove Proposition \ref{factorizationProp}.
\begin{proof}[Proof of Proposition \ref{factorizationProp}]
{\bf Step 1.}
Fix $\eps>0$. Let $\eps_1>0$ be a small quantity to be chosen below. We will choose $\eps_1$ small compared to $\eps$, and $\eta$ small compared to $\eps_1$. Let $\tubes$ be a set of $\delta$ tubes that satisfy the Katz-Tao Convex Wolff axioms at every scale with error $\delta^{-\eta}$.

Provided we select $\eta\leq\eps$, there exists a number $J\sim \eps^{-1}$ and scales $1=\rho_0>\rho_1>\ldots>\rho_J = \delta$ with $\delta^\eps \leq \rho_{j+1}/\rho_{j} < 1/8$ for each index $j$, so that the following holds: for each index $j=0,\ldots,J-1$, there is a $\delta^{-\eta}$-balanced partitioning cover $\tubes_{\rho_j}$ of $\tubes$, with $\CKT(\tubes_{\rho_j})\leq\delta^{-\eta}$. Define $\tubes_{\rho_J} = \tubes$. 

Currently, $\tubes_{\rho_j}$ covers $\tubes$, but unfortunately, it might not be the case that $\tubes_{\rho_{j}}$ covers $\tubes_{\rho_{j+1}}$. We can fix this as follows. We know that for each index $j>0$, we have $\tubes[T_{\rho_{j+1}}]\neq\emptyset$ for each $T_{\rho_{j+1}}\in\tubes_{\rho_{j+1}}$. Since $\tubes_{\rho_j}$ covers $\tubes$, we have that for each $T_{\rho_{j+1}}\in\tubes_{\rho_{j+1}}$, there exists $T_{\rho_j}\in\tubes_{\rho_j}$ so that $\tubes[T_{\rho_{j+1}}]\cap \tubes[T_{\rho_j}]\neq\emptyset$, and hence $T_{\rho_{j+1}}\cap T_{\rho_j}$ contains a unit line segment. But this implies that $T_{\rho_{j+1}}\subset N_{2\rho_{j+1}}(T_{\rho_j})$. Thus, if we replace each set $\tubes_{\rho_j}$ with $\{N_{2\rho_{j+1}}(T_{\rho_j})\colon T_{\rho_j}\in\tubes_{\rho_j}\}$ for $j = 0,\ldots,J-1$ (leaving $\tubes_{\rho_J}$ unchanged), then for each index $j$ we have that $\tubes_{\rho_j}$ covers $\tubes_{\rho_{j+1}}$.   It is still the case that $\tubes_{\rho_j}$ is a $\lesssim\delta^{-\eta}$-balanced cover of $\tubes$, and hence $\tubes_{\rho_j}$ is a $\lesssim\delta^{-2\eta}$-balanced cover of $\tubes_{\rho_{j+1}}$. Abusing notation, we will continue to refer to these sets as $\tubes_{\rho_j}$ (even though this set is technically a collection of $(\rho_j + 2\rho_{j+1})$ tubes; since $\rho_{j+1}\leq \rho_j/4$, this distinction will not be important). Note that it might no longer be the case that the tubes in $\tubes_{\rho_j}$ are essentially distinct, nor that $\tubes_{\rho_j}$ is a partitioning cover of $\tubes$. 

\medskip

\noindent {\bf Step 2.}

 After dyadic pigeonholing, we can replace each set $\tubes_{\rho_j}$ by a subset (abusing notation, we will continue to refer to this subset as $\tubes_{\rho_j}$) so that $\#\tubes_{\rho_J}\gtrapprox \#\tubes$, and the following holds for each $j=0,\ldots,J$. (To simplify notation, we define $\tilde T_{\rho_j} = N_{3\rho_j}(T_{\rho_j})$ and $\tilde{\tubes}_{\rho_{j}}=\{\tilde T_{\rho_j}\colon T_{\rho_j}\in\tubes_{\rho_j}\}$ ).

\begin{itemize}
\item[(i)] The tubes in $\tilde \tubes_{\rho_{j}}$ are essentially distinct.
\item[(ii)] For each $T_{\rho_{j}}\in\tubes_{\rho_{j}}$, we have $\tilde\tubes_{\rho_{j+1}}[\tilde T_{\rho_{j}}]=\tilde\tubes_{\rho_{j+1}}[T_{\rho_{j}}].$
\item[(iii)] $\tubes_{\rho_{j}}$ is a balanced partitioning cover of $\tilde \tubes_{\rho_{j+1}}$.
\item[(iv)] for each $T_{\rho_{j}}\in\tubes_{\rho_{j}}$, we have that $({\tilde \tubes}_{\rho_{j+1}})^{T_{\rho_{j}}}$ is regular with granularity $(3\rho_{j}/\rho_{j+1})^{\eta_1}$.
\end{itemize}

For each index $j = 1,\ldots,J$, apply Corollary \ref{corFactoringSingleScale} with $\eps_1$ in place of $\eps$, $\tilde \tubes_{\rho_j}$ in place of $\tubes$, $\tubes_{\rho_{j-1}}$ in place of $\tubes_\rho$, $3\rho_j$ in place of $\delta$, and $3\rho_j/\rho_{j-1}$ in place of $\rho$. Let $\mathcal{A}_j$ be the output of Corollary \ref{corFactoringSingleScale}. 

By \eqref{ATRhoSubset3TRho} plus Item (ii) above, for each $T_{\rho_{j-1}}\in\tubes_{\rho_{j-1}}$, we have
\begin{equation}
\bigsqcup_{A\in\mathcal{A}_j} A(\tilde{\tubes}_{\rho_{j}}[\tilde T_{\rho_{j-1}}]) = \bigsqcup_{A\in\mathcal{A}_j} A(\tilde{\tubes}_{\rho_{j}}[T_{\rho_{j-1}}]) \subset \tilde T_{\rho_{j-1}},
\end{equation}
and hence by Item (iii), 
\begin{equation}\label{tildeTubesCovers}
\tilde\tubes_{\rho_{j-1}}\ \textrm{covers}\ \bigsqcup_{A\in\mathcal{A}_j} A(\tilde{\tubes}_{\rho_{j}}).
\end{equation} 
Here the disjoint union is in the sense of multi-sets. 

We have
\begin{equation}\label{cardEstimateOverCalAj}
\#\Big( \Big(\bigsqcup_{A\in\mathcal{A}_j} A(\tilde{\tubes}_{\rho_{j}})\Big)[\tilde T_{\rho_{j-1}}]\Big) \geq  (\#\mathcal{A}_j)\Big(\frac{\#\tilde \tubes_{\rho_j}}{2\#\tubes_{\rho_{j-1}}}\Big) \gtrsim \frac{|T_{\rho_{j-1}}|}{|T_{\rho_j}|}\quad \textrm{for each}\ T_{\rho_{j-1}}\in\tubes_{\rho_{j-1}}.
\end{equation}

By \eqref{favorableCKTBdInsideRho}, there exists $C_{\eps_1}\geq 1$ so that 
\begin{equation}\label{goodCKTBoundEachUnion}
\CKT\Big(\Big(\bigsqcup_{A\in\mathcal{A}_j} A(\tilde\tubes_{\rho_j})\Big)[\tilde T_{\rho_{j-1}}]\Big)  \leq C_{\eps_1}\delta^{-\eps_1-2\eta}\quad\textrm{for each}\ \tilde T_{\rho_{j-1}}\in\tilde\tubes_{\rho_{j-1}}.
\end{equation}

\medskip

\noindent {\bf Step 3.} Define $\mathcal{A}^{(0)}=\{I\}$ and for $j=1,\ldots,J$, define 
\[
\mathcal{A}^{(j)} = \{A_1\circ A_2 \circ \ldots \circ A_j\colon A_i\in\mathcal{A}_i\ \textrm{for each}\ i=1,\ldots j\}.
\]
Define $\mathcal{A} = \mathcal{A}^{(J)}$; this set of transformations will be the output of Proposition \ref{factorizationProp}.

For each index $j=0,\ldots,J$, define
\[
\hat\tubes_{\rho_j} = \bigsqcup_{A\in\mathcal{A}^{(j)}}A(\tilde\tubes_{\rho_j}).
\]
By \eqref{tildeTubesCovers}, $\hat\tubes_{\rho_{j-1}}$ covers $\hat\tubes_{\rho_{j}}$ for each $j=1,\ldots, J$.

By \eqref{cardEstimateOverCalAj}, for each $1\leq j\leq J$ and each $T_{\rho_{j-1}}\in\tubes_{\rho_{j-1}}$ we have
\[
\#\hat\tubes_{\rho_{j}}[\tilde T_{\rho_{j-1}}] \gtrsim  \frac{|T_{\rho_{j-1}}|}{|T_{\rho_{j}}|}. 
\]

Thus after dyadic pigeonholing, we can replace each set $\hat\tubes_{\rho_j}$ by a refinement $\hat\tubes_{\rho_j}'$ (in the sense of multi-sets), so that $\hat\tubes_{\rho_{j-1}}'$ is a balanced partitioning cover of $\hat\tubes_{\rho_{j}}'$ for each $j=1,\ldots, J$, and for each $j=1,\ldots,J$ and each $T_{\rho_{j-1}}\in\tubes_{\rho_{j-1}}$ we have 
\[
\#\hat\tubes_{\rho_{j}}'[\tilde T_{\rho_{j-1}}] \gtrsim \delta^\eta \frac{|T_{\rho_{j-1}}|}{|T_{\rho_{j}}|},
\]
and hence
\begin{equation}\label{tubesJBig}
\# \hat\tubes_{\rho_{J}}'[\tilde T_{\rho_j}] \gtrsim \delta^{\eta/\eps} \frac{|T_{\rho_{j}}|}{|T|}\quad\textrm{for each}\ \tilde T_{\rho_j}\in \hat\tubes_{\rho_{j}}'.
\end{equation}
\medskip

\noindent {\bf Step 4.} Using Lemma \ref{KTWSubMultiplicativeTubes} and \eqref{goodCKTBoundEachUnion}, we have 
\[
\CKT(\hat\tubes_{\rho_J}')\leq \CKT(\hat\tubes_{\rho_J})\lesssim \prod_{j=1}^J \big(C_{\eps_1}\delta^{-\eps_1-2\eta}\big) \lesssim \delta^{-2\eps_1 J}.
\]
In particular, for each $T\in \hat\tubes_{\rho_J}'$, there are $\lesssim \delta^{-2\eps_1 J}$ tubes $T'\in\hat\tubes_{\rho_J}'$ comparable to $T$. Recall that $J\sim\eps^{-1}$. We will choose $\eps_1$ sufficiently small so that $2\eps_1 J<\eps/4$. Thus we can select a set  $\tubes^\dag \subset\hat\tubes_{\rho_J}'\subset\bigcup_{A\in\mathcal{A}} A(\tilde \tubes_{\rho_J})$  with
$
\#\tubes^\dag \gtrsim \delta^{\eps/4}(\#\hat\tubes_{\rho_J}') \geq \delta^{\eps}|T|^{-1}
$
that consists of of essentially distinct tubes. The set $\tubes^\dag$ will be the output of Proposition \ref{factorizationProp}.

It remains to verify that $\tubes^\dag$ satisfies the Frostman Convex Wolff Axioms at every scale, with error $\delta^{-\eps}$. To do so, it suffices to note that for each index $j = 0,\ldots,J-1$, we have that $\hat\tubes_{\rho_j}'$ is a $ \delta^{-\eps/4} 2^{J-j}$-balanced partitioning cover of $\tubes^\dag$, and thus if we choose $\eta<\eps^2/2$, then for each $\tilde T_{\rho_j}\in  \hat\tubes_{\rho_j}'$ we can use \eqref{tubesJBig} to conclude that
\[
\CFC\big( (\tubes^\dag)^{\tilde T_{\rho_j}}\big)\leq \frac{\CKT(\tubes^\dag)}{\#\tubes^\dag[\tilde T_{\rho_j}]} \frac{|\tilde T_{\rho_j}|}{|T|}
\lesssim \delta^{-\eta/\eps}\CKT(\tubes^\dag) \leq \delta^{-\eta/\eps  -\eps/4} \CKT(\hat\tubes_{\rho_J}')\leq  \delta^{-\eps}.\qedhere
\]
\end{proof}


\section{Multi-scale analysis and the proof of Proposition \ref{improvingProp}}
Our goal in this section is to prove Proposition \ref{improvingProp} by combining Proposition \ref{refinedInductionOnScaleProp} and Theorem \ref{katzTaoEveryScaleStickyKakeyaThm}.

\begin{lem}\label{bigVolumeOrFineDivisions}
Let $\sigma\in(0,2/3]$; $\omega,\eps>0$ and $N\geq 1$. Suppose that $\cE(\sigma,\omega)$ is true. Then there exists $\eta,\alpha,\kappa>0$ so that the following holds for all $\delta>0$. Let $(\tubes,Y)_\delta$ be $\delta^\eta$ dense with $\CKT(\tubes)\leq\delta^{-\eta}$ and $\FS(\tubes)\leq\delta^{-\eta}$. Then at least one of the following must hold.
\begin{itemize}
	\item[(A)] \itemizeEqnVSpacing
		\[
			\Big|\bigcup_{T\in\tubes}Y(T)\Big| \geq \kappa \delta^{\omega-\alpha}(\#\tubes)|T|\big((\#\tubes)|T|^{1/2}\big)^{-\sigma}.
		\]

	\item[(B)] There is a $\delta^\eps$ refinement $(\tubes',Y')_\delta$ of $(\tubes,Y)_\delta$; a number $J\leq 2^N$; scales $\delta=\rho_J<\rho_{J-1}<\ldots<\rho_0 = 1$; and sets $\tubes_{\rho_j},\ j=0,\ldots,J$, so that the following holds
	\begin{itemize}
		\item[(i)] $\rho_{i+1}/\rho_i\geq \delta^{(1-\omega/100)^N}$ for each $i = 0,\ldots,J-1$.
		\item[(ii)] $\tubes' = \tubes_{\rho_J},$ and $\tubes_{\rho_0}$ consists of a single 1 tube.
		\item[(iii)] For each $i=1,\ldots,J-1$ and each $T_{\rho_{i-1}}\in\tubes_{\rho_{i-1}}$, we have that $\tubes_{\rho_i}^{T_{\rho_{i-1}}}$ factors $\tubes_{\rho_{i+1}}^{T_{\rho_{i-1}}}$ from above and below with regard to the Katz-Tao Convex Wolff Axioms and Frostman Slab Wolff Axioms, both with error $\delta^{-\eps}$. 
	\end{itemize}
\end{itemize}
\end{lem}
\begin{proof}$\phantom{1}$\\
{\bf Step 1.}
We will prove the result by induction on $N$. When $N=0$ there is almost nothing to prove; since the tubes in $\tubes$ are contained in the unit ball, we can select a $\sim 1$ refinement $\tubes'\subset\tubes$ that is contained in a single 1 tube. We have $J = 1$, and Conclusion (B) always holds (note that Item (iii) of Conclusion B is vacuously true when $J=1$.

Suppose now that the result has been proved for $N-1.$ Let $\theta\in(0,2/3]$ and let $\eps,\omega>0$. Let $\eps_1,\eps_2,\eta,\alpha,c>0$ be small quantities to be determined below. We will choose $\eps_1$ depending on $\eps,N,\theta,\omega$; $\eps_2$ small compared to $\eps_1$; and $\eta,\alpha,c$ small compared to $\eps_2$.

Let $(\tubes,Y)_\delta$ be $\delta^\eta$ dense and satisfy $\CKT(\tubes)\leq\delta^{-\eta}$ and $\FS(\tubes)\leq\delta^{-\eta}$. If $\eta$ is chosen sufficiently small depending on $\omega,\sigma$, and $\eps_2$, then we can apply the induction hypothesis with $\eps_2$ in place of $\eps$; let $\alpha_1=\alpha_1(\omega,\sigma,\eps_2,N-1)$ be the corresponding value of $\alpha$.

If Conclusion (A) holds when we apply the induction hypothesis, then Conclusion (A) holds and we are done, provided we select $\alpha\leq\alpha_1$. Suppose instead that Conclusion (B) holds, i.e. there is a $\delta^{\eps_2}$ refinement $(\tubes',Y')_\delta$; scales $\delta = \rho_{J} < \rho_{J-1} < \ldots < \rho_0 = 1$ with $J\leq 2^{N-1}$ and $\rho_{i+1}/\rho_i\geq \delta^{(1-\omega/100)^{N-1}}$; and sets $\tubes_{\rho_j}$ satisfying Conclusion (B) with $\eps_2$ in place of $\eps$.

Replacing $(\tubes',Y')_\delta$ and each set $\tubes_{\rho_i}$ with a $(\log 1/\delta)^{-J}\geq (\log 1/\delta)^{-2^N}$ refinement, we can construct shadings $Y_{\rho_i}$ on $\tubes_{\rho_i}$ so that the following holds.
\begin{itemize}
	\item[(a)] $|Y_{\rho_0}(T_{\rho_0})|\gtrsim (\log 1/\delta)^{-2^N}\delta^{\eta}$, where $T_{\rho_0}$ is the unique tube in $\tubes_{\rho_0}$. 
	\item[(b)] For each $T_{\rho_i}\in\tubes_{\rho_i}$, we have that $(\tubes_{\rho_{i+1}}^{T_{\rho_i}}, Y_{\rho_{i+1}}^{T_{\rho_i}})_{\rho_{i+1}/\rho_i}$ is $\gtrsim (\log 1/\delta)^{-2^N}\delta^{\eta}$ dense,  and $(Y_{\rho_J}, \tubes_{\rho_J} )_{\delta}$ is a $\gtrsim (\log 1/\delta)^{-2^N}$-refinement of $(\tubes', Y')_{\delta}$.
	\item[(c)] Item (iii) of Conclusion (B) remains true, with the error $\delta^{-\eps_2}$ weakened to $\lesssim (\log 1/\delta)^{2^N} \delta^{-\eps_2}$.

	\item[(d)] For each $i=1,\ldots,J-1$ and each $T_{\rho_{i-1}}\in\tubes_{\rho_{i-1}}$, we have that $\tubes_{\rho_i}^{T_{\rho_{i-1}}}$ is a balanced partitioning cover of $\tubes_{\rho_{i+1}}^{T_{\rho_{i-1}}}$.

	\item[(e)] For each $i = 0,\ldots, J-1$, the quantity
	\[
		\Big|\bigcup_{T\in\tubes_{\rho_{i+1}}[T_{\rho_i}]}Y_{\rho_{i+1}}(T_{\rho_{i+1}})\Big|
	\]
	is approximately the same (up to a multiplicative factor of 2) for each $T_{\rho_i}\in \tubes_{\rho_i}$.
	\item[(f)] We have
	\begin{equation}\label{productOfShadingEqn}
		\Big|\bigcup_{T\in\tubes}Y(T)\Big| \gtrsim \prod_{i=1}^{J-1}\Big( \sup_{T_{\rho_i\in\tubes_{\rho_i}}}\frac{1}{|T_{\rho_i}|}\Big|\bigcup_{T_{\rho_{i+1}}\in\tubes_{\rho_{i+1}}[T_{\rho_{i}}]}Y_{\rho_{i+1}}(T_{\rho_{i+1}})\Big|\Big).
	\end{equation}
\end{itemize}
Item (f) is obtained by selecting each shading $Y_{\rho_{i+1}}$ so that the union $\bigcup_{T_{\rho_{i+1}}\in\tubes_{\rho_{i+1}}[T_{\rho_{i}}]}Y_{\rho_{i+1}}(T_{\rho_{i+1}}) \subset Y_{\rho_i}(T_{\rho_i})$ has approximately the same density (up to a multiplicative factor of 2) on balls of radius $\rho_{i}$ contained in $Y_{\rho_i}(T_{\rho_i})$. By the previous item, the RHS of \eqref{productOfShadingEqn} is reduced by at most a factor of $2^J$ if the supremum in \eqref{productOfShadingEqn} is replaced by an infimum. Items (a) - (e) are obtained by dyadic pigeonholing; we omit the details.

If $\eps_2$ is chosen sufficiently small depending on $\eps_1$, then for each index $i$ and each $T_{\rho_i}\in\tubes_{\rho_i}$ we have
\begin{equation}\label{averageDensityRhoIRhoIp1}
\frac{1}{|T_{\rho_i}|}\Big|\bigcup_{T_{\rho_{i+1}}\in\tubes_{\rho_{i+1}}[T_{\rho_{i}}]}Y_{\rho_{i+1}}(T_{\rho_{i+1}})\Big| 
\gtrsim \delta^{\eps_1}\big(\frac{\rho_{i+1}}{\rho_i}\big)^{\omega}\big(\frac{\#\tubes_{\rho_{i+1}}}{\#\tubes_{\rho_i}}\big)|T_{\rho_{i+1}}|\Big( \big(\frac{\#\tubes_{\rho_{i+1}}}{\#\tubes_{\rho_i}}\big)\big(\frac{|T_{\rho_{i+1}}}{|T_{\rho_i}|}\big)^{1/2} \Big)^{-\sigma}
\end{equation}
Indeed, \eqref{averageDensityRhoIRhoIp1} follows from the estimate $\cE(\sigma,\omega)$; this estimate can be applied, for example, when $\frac{\rho_{i+1}}{\rho_i} < \delta^{\eps_1}$ and $\eps_2$ is selected sufficiently small depending on $\omega,\sigma$, and $\eps_1$. If $\frac{\rho_{i+1}}{\rho_i} \geq \delta^{\eps_1}$, then \eqref{averageDensityRhoIRhoIp1} follows from the elementary fact that the volume of the LHS of \eqref{averageDensityRhoIRhoIp1} is bounded below by the volume of $Y_{\rho_{i+1}}(T_{\rho_{i+1}})$ for every $T_{\rho_{i+1}}\in\tubes_{\rho_{i+1}}[T_{\rho_i}]$. 

\medskip

\noindent {\bf Step 2.}

We say an index $i$ is of Type C if $\rho_{i+1}/\rho_i \geq 2^{-N}\delta^{(1-\omega/100)^{N}}$; in this case no further splitting is required. If instead $\rho_{i+1}/\rho_i < 2^{-N}\delta^{(1-\omega/100)^{N}}$, then apply Proposition \ref{refinedInductionOnScaleProp} (refined induction on scales) with $\omega$ and $\sigma$ as above, and $\zeta=\eps$ to each arrangement $(\tubes_{\rho_{i+1}}^{T_{\rho_i}},Y_{\rho_{i+1}}^{T_{\rho_i}})_{\rho_{i+1}/\rho_i}$; we can do this provided $\eps_2$ is chosen sufficiently small depending on $\omega,\sigma,N,$ and $\eps$. If Conclusion (A) holds for at least one $T_{\rho_i}\in\tubes_{\rho_i}$, then we say the index $i$ is of Type A. Otherwise we say that the index $i$ is of Type B.

Suppose there exists at least one index of Type A; call this index $i_0$. Let $\alpha_2 = \alpha_2(\omega,\sigma,\eps)$ be the output of Proposition \ref{refinedInductionOnScaleProp}. Then applying \eqref{averageDensityRhoIRhoIp1} to each index $i\neq i_0$; using Conclusion (A) of Proposition \ref{refinedInductionOnScaleProp} to estimate the contribution from index $i_0$; and using \eqref{productOfShadingEqn} to combine these estimates, we have
\begin{equation}
\begin{split}
\Big|\bigcup_{T\in\tubes}Y(T)\Big| & 
\gtrsim \prod_{\substack{i=1,\ldots,J-1 \\ i \neq i_0}}\Big(\delta^{\eps_1}\big(\frac{\rho_{i+1}}{\rho_i}\big)^{\omega}\big(\frac{\#\tubes_{\rho_{i+1}}}{\#\tubes_{\rho_i}}\big)|T_{\rho_{i+1}}\Big( \big(\frac{\#\tubes_{\rho_{i+1}}}{\#\tubes_{\rho_i}}\big)\big(\frac{|T_{\rho_{i+1}}}{|T_{\rho_i}|}\big)^{1/2} \Big)^{-\sigma}\Big)\cdot \\
& \qquad \cdot \Big(\delta^{\eps_1}\big(\frac{\rho_{i_0+1}}{\rho_{i_0}}\big)^{\omega-\alpha_2}\big(\frac{\#\tubes_{\rho_{i_0+1}}}{\#\tubes_{\rho_{i_0}}}\big)|T_{\rho_{i_0+1}}\Big( \big(\frac{\#\tubes_{\rho_{i_0+1}}}{\#\tubes_{\rho_{i_0}}}\big)\big(\frac{|T_{\rho_{i_0+1}}}{|T_{\rho_{i_0}}|}\big)^{1/2} \Big)^{-\sigma}\Big)\\
& \geq \delta^{\eps_1 J}\big(\frac{\rho_{i_0+1}}{\rho_{i_0}}\big)^{-\alpha_2}(\#\tubes)|T|\big((\#\tubes)|T|^{1/2}\big)^{-\sigma}\\
& \gtrsim \delta^{\eps_1 2^N -\alpha_2 (1-\omega/100)^{N}}(\#\tubes)|T|\big((\#\tubes)|T|^{1/2}\big)^{-\sigma}
\end{split}
\end{equation}
Thus Conclusion (A) of Lemma \ref{bigVolumeOrFineDivisions} holds, provided we select $\alpha \leq \frac{1}{2}\alpha_2 (1-\omega/100)^{N}$ and $\eps_1<2^{-N-1}\alpha_2 (1-\omega/100)^{N}$.

\medskip

\noindent {\bf Step 3.}
Suppose instead that every index is of Type B or Type C. We will show that Conclusion (B) of Lemma \ref{bigVolumeOrFineDivisions} holds. For each index $i$ of Type B, starting with the lowest, we will perform the following procedure. After dyadic pigeonholing, there is a number $s\in \big[ \big(\frac{\rho_{i+1}}{\rho_i}\big)^{1-\frac{\omega}{100}},\ \big(\frac{\rho_{i+1}}{\rho_i}\big)^{\frac{\omega}{100}}\big]$ and a set $\tubes_{\rho_i}'\subset\tubes_{\rho_i}$ with $\#\tubes_{\rho_i}'\gtrapprox_\delta \#\tubes_{\rho_i}$ such that the output ``$\rho$'' from Proposition \ref{refinedInductionOnScaleProp} is between $s$ and $2s$ for each $T_{\rho_i}\in\tubes_{\rho_i}'$. Define $\rho_{i+1/2}=s\rho_i$. Then
\begin{equation}
\frac{\rho_{i+1/2}}{\rho_i}\geq s \geq \big(\frac{\rho_{i+1}}{\rho_i}\big)^{1-\frac{\omega}{100}}\geq (2^{-N+1}\delta^{(1-\omega/100)^{N-1}})^{1-\frac{\omega}{100}}\geq 2^{-N}\delta^{(1-\omega/100)^N}.
\end{equation}
Similarly, we have
\begin{equation}
\frac{\rho_{i+1}}{\rho_{i+1/2}}=\frac{\rho_{i+1}}{\rho_{i}}\frac{\rho_i}{\rho_{i+1/2}}\geq \frac{\rho_{i+1}}{\rho_{i}}\frac{1}{2s}\geq \frac{1}{2}\big(\frac{\rho_{i+1}}{\rho_{i}}\big)^{1-\frac{\omega}{100}}\geq 2^{-N}\delta^{(1-\omega/100)^N}.
\end{equation}
This verifies Item (i) from Conclusion (B) of Lemma \ref{bigVolumeOrFineDivisions}.

 Let $\tubes_{\rho_{i+1/2}}$ be the union of the corresponding sets from  Proposition \ref{refinedInductionOnScaleProp}, and let 
\[
\tubes_{\rho_{i+1}}'=\bigcup_{T_{\rho_{i+1/2}}\in\tubes_{\rho_{i+1/2}}}\tubes_{\rho_{i+1}}[T_{\rho_{i+1/2}}].
\]
Then for each $T_{\rho_i}\in\tubes_{\rho_i}'$, we have that $\tubes_{\rho_{i+1/2}}^{T_{\rho_i}}$ factors $(\tubes'_{\rho_{i+1}})^{T_{\rho_i}}$ above and below with respect to the Katz-Tao Convex Wolff Axioms and Frostman Slab Wolff Axioms, both with error $\delta^{-\eps}$.  This verifies Item (iii) from Conclusion (B) of Lemma \ref{bigVolumeOrFineDivisions}.

To conclude the proof, we re-index our sets $\tubes_{\rho_i}'$ and $\tubes_{\rho_{i+1/2}}$ using consecutive integers $1,2,3\ldots$, and let $\tubes'=\tubes_{\rho_{J'}}$, where $J'$ is the final index after re-indexing. Note that $J'\leq 2J-1\leq 2^N$. After re-indexing, we still have $\rho_0=1$ and $\rho_{J'}=\delta$. This concludes the induction step.
\end{proof}


\begin{lem}\label{bigVolumeOrKatzTaoAllScales}
Let $\sigma\in(0,2/3]$ and $\omega,\eps>0$. Suppose that $\cE(\sigma,\omega)$ is true. Then there exists $\eta,\alpha,\kappa>0$ so that the following holds for all $\delta>0$. Let $(\tubes,Y)_\delta$ be $\delta^\eta$ dense with $\CKT(\tubes)\leq\delta^{-\eta}$ and $\FS(\tubes)\leq\delta^{-\eta}$. Then at least one of the following must hold.
\begin{itemize}
	\item[(A)] \itemizeEqnVSpacing
		\[
			\Big|\bigcup_{T\in\tubes}Y(T)\Big| \geq \kappa \delta^{\omega-\alpha}(\#\tubes)|T|\big((\#\tubes)|T|^{1/2}\big)^{-\sigma}.
		\]

	\item[(B)] There is a $\delta^\eps$ refinement $(\tubes',Y')_\delta$ of $(\tubes,Y)_\delta$ that satisfies the Katz-Tao Convex Wolff Axioms at every scale with error $\delta^{-\eps}$ (recall Definition \ref{KatzTaoConvexWolffAtEveryScaleDefn}). 
\end{itemize}
\end{lem}
\begin{proof}
Let $\eps_1>0$ be a small quantity to be chosen below. Select $N$ sufficiently large so that $(1-\omega/100)^N<\eps$. Apply Lemma \ref{bigVolumeOrFineDivisions} to $(\tubes,Y)_\delta$ with this choice of $N$, with $\eps_1$ in place of $\eps$, and with $\omega,\sigma$ as above. If Conclusion (A) of Lemma \ref{bigVolumeOrFineDivisions} holds, then Conclusion (A) of Lemma \ref{bigVolumeOrKatzTaoAllScales} holds and we are done. 

Suppose instead that Conclusion (B) of Lemma \ref{bigVolumeOrFineDivisions} holds. After a harmless refinement we can suppose that each set $\tubes_{\rho_i}$ is a balanced partitioning cover of $\tubes_{\rho_{i+1}}$, and hence $\tubes_{\rho_i}$ is a $2^J\leq\delta^{-\eps}$ balanced partitioning cover of $\tubes'$.

We claim that if $\eps_1$ is chosen sufficiently small (depending on $N$ and $\eps$), then the output $(\tubes',Y')_\delta$ of Lemma \ref{bigVolumeOrFineDivisions} satisfies the Katz-Tao Convex Wolff Axioms at every scale with error $\delta^{-\eps}$. We verify this as follows. Since $(1-\omega/100)^N<\eps$, for each $\rho\in[\delta,1]$ the interval $[\rho,\delta^{-\eps}\rho]$ contains at least one $\rho_i$. We have already noted that $\tubes_{\rho_i}$ satisfies Item (i) from Definition \ref{KatzTaoConvexWolffAtEveryScaleDefn}. It remains to verify Item (ii). Applying Lemma \ref{KTWSubMultiplicativeTubes}, we have that for each index $i$,
\[
\CKT(\tubes_i)\lesssim \prod_{j=i+1}^J \delta^{-2\eps_1}\lesssim \delta^{-2^{N+1}\eps_1}.
\]
We will select $\eps_1$ sufficiently small so that $\CKT(\tubes_i)\leq \delta^{-\eps}$ for all sufficiently small $\delta>0$. If $\delta>0$ is not sufficiently small, then Conclusion (A) always holds, provided we select $\kappa>0$ sufficiently small.
\end{proof}


\begin{proof}[Proof of Proposition \ref{improvingProp}]
Fix $\omega,\sigma>0$ and suppose that $\cE(\sigma,\omega)$ is true. Let $\eta_1$ be the output of Theorem \ref{katzTaoEveryScaleStickyKakeyaThm} with $\omega/2$ in place of $\eps$. Let $\alpha_2,\eta_2,c_2$ be the output of Lemma \ref{bigVolumeOrKatzTaoAllScales} with $\eta_1$ in place of $\eps$. Then for every pair $(\tubes,Y)_\delta$ that is $\delta^{\eta_2}$ dense and satisfies $\CKT(\tubes)\leq\delta^{-\eta_2}$, $\FS(\tubes)\leq\delta^{-\eta_2}$, we have 
\begin{equation}
\Big|\bigcup_{T\in\tubes}Y(T)\Big| \gtrsim \delta^{\omega - \min(\alpha_2,\omega/2)}(\#\tubes)|T|\Big((\#\tubes)|T|\Big)^{-\sigma}.
\end{equation}
In particular, we have that $\cD(\sigma,\omega-g(\sigma,\omega))$ is true, where $g(\sigma,\omega) = \min(\alpha_2,\omega/2)$. 
\end{proof}


\section{Tube Doubling}\label{extensionAndKeletiSec}
In this section we will prove Theorem \ref{tubeDoublingTheoremR3}. We will prove the following slightly stronger statement. 
\begin{tubeDoublingTheoremR3Again}
For all $\eps>0$, there exists $\eta>0$ so that the following is true for all $\delta>0$ sufficiently small. Let $\tubes$ be a set of $\delta$ tubes in $\RR^3$. For each $T\in\tubes$, let $Y(T)\subset T$ with $|Y(T)|\geq\delta^\eta|T|$. Let $R\geq 1$ and for each $T\in\tubes$, let $ T_R$ denote the $R$-fold dilate of $T$. Then
\begin{equation}\label{lineSegmentExtensionInsideThm}
\Big|\bigcup_{T\in\tubes} T_R\Big| \leq \delta^{-\eps} R^3\Big|\bigcup_{T\in\tubes}Y(T)\Big|.
\end{equation}
\end{tubeDoublingTheoremR3Again}
\noindent Theorem \ref{tubeDoublingTheoremR3} is the special case where $Y(T) = T$ and $R = 2$.  
\begin{proof}
{\bf Step 1.} 
First, we may suppose that the tubes in $\tubes$ are contained in $B(0,1)$ (indeed, we can cover $\RR^3$ by a set of boundedly overlapping unit balls, so that each $T\in\tubes$ is contained in at least one ball). Second, we may suppose that the tubes in $\tubes$ are essentially distinct, and in particular $\#\tubes\lesssim \delta^{-4}$ (indeed, we can replace $\tubes$ by a maximal, essentially distinct sub-collection, and this affects the RHS of Inequality \eqref{lineSegmentExtensionInsideThm} by a constant factor).

Fix $\eps>0$.  For each $i=1,2,\ldots,$ we will define sets $\tubes_i,\tubes_i'$, and $\mathcal{W}_i$ as follows. We begin by defining $\tubes_0=\tubes$. Define $\tubes_i'$ and $\mathcal{W}_i$ to be the output of Proposition \ref{factoringConvexSetsProp} applied to $\tubes_{i-1}$. Define $\tubes_i = \tubes_{i-1}\backslash\tubes_i'$. We continue this process until $\tubes_N$ is empty, at which point we halt. Observe that $\tubes = \bigsqcup_{i=1}^N \tubes_i'$. For each index $i$ we have 
\[
\#\tubes_i' \geq 100^{-3} e^{-100 \sqrt{\log(\delta^{-5})}}(\#\tubes_{i-1}),
\]
As a consequence, if $j-i \geq 2\cdot 100^{3} e^{100 \sqrt{\log(\delta^{-5})}},$ then $\#\tubes_j\leq \frac{1}{2}(\#\tubes_i)$. Since $\#\tubes_0\lesssim\delta^{-4}$, we conclude that the process described above must halt after $\lesssim (\log 1/\delta)( 2\cdot 100^{3} e^{100 \sqrt{\log(\delta^{-5})}}) \lessapprox_\delta 1$ steps, i.e.~$N \lessapprox_\delta 1$.

\medskip

\noindent {\bf Step 2.} 
Fix an index $j$. To simplify notation, we write $\tubes' = \tubes'_j$ and $\mathcal{W}' = \mathcal{W}_j$.

We have that $\mathcal{W}'$ factors $\tubes'$ from above with respect to the Katz-Tao convex Wolff Axioms and from below with respect to the Frostman Convex Wolff Axioms, both with error $K\leq 100^3 e^{100 \sqrt{\log(\delta^{-5})}}$. Recall that the prisms in $\mathcal{W}'$ all have the same dimensions --- denote these dimensions by $a\times b\times 1$. 

Let $\eps_1>0$ be a small quantity to be chosen below. Our next task is to estimate the volume of the tubes inside each prism from $\mathcal{W}'$, i.e.~we wish to obtain a bound of the form
\begin{equation}\label{desiredBoundInsideW}
\Big|\bigcup_{T\in\tubes'[W]}Y(T)\Big| \geq\kappa_0\delta^{\eps_1}|W|,
\end{equation}
for some $\kappa_0 = \kappa_0(\eps_1)>0$. To do this, we will estimate the volume of the set $\bigcup_{\tubes'[W]}Y^W(T)^W$; the latter is a union of prisms of dimensions $\delta/b\times\delta/a\times 1$. By Theorem \ref{cDAndcEAreTrue} we have that Assertion $\cE(\eps_1/10,\eps_1/10)$ is true, and thus by Proposition \ref{EiffF}, we have that Assertion $\cF(\eps_1/10,\eps_1/10)$ (recall Definition \ref{defnF}) is true. The latter assertion is helpful, since it allows us to bound the volume of unions of prisms with three distinct dimensions (i.e. prisms that are not tubes).

For each $W\in\mathcal{W}'$, we will apply the estimate $\cF(\eps_1/10,\eps_1/10)$ to the pair $((\tubes')^W,Y^W)_{\delta/b\times \delta/a\times 1}$.  Recall that the estimate $\cF(\eps_1/10,\eps_1/10)$ contains an additional term ``$D$.'' However, as discussed in Remark \ref{discussionOfDRemark}, Situation 3, we have $D\sim 1$ in this setting. If $\eta>0$ is selected sufficiently small depending on $\eps_1$, we conclude that for each $W\in\mathcal{W}'$, we have
\begin{equation}\label{rescaledBdInsideW}
\begin{split}
\Big|\bigcup_{T\in\tubes'[W]}Y^W(T)^W\Big| &\geq \kappa_0 \delta^{\eps_1/10} m^{-1}(\#\tubes'[W])|T^W|\Big(m^{-3/2}\ell (\#\tubes'[W])|T^W|^{1/2}\Big)^{-\eps_1/10}\\
&\gtrapprox_\delta \kappa_0 \delta^{\eps_1/2},
\end{split}
\end{equation}
where $m = \CKT( (\tubes')^W)\lessapprox_\delta (\#\tubes'[W])\frac{|T|}{|W|} = (\#\tubes'[W])|T^W|$, and $\ell = \FS((\tubes')^W)$. For the second inequality, the term $\Big(m^{-3/2}\ell (\#\tubes'[W])|T^W|^{1/2}\Big)$ can be trivially bounded by $\delta^{-4}$. \eqref{desiredBoundInsideW} now follows from \eqref{rescaledBdInsideW} by rescaling and adjusting the constant $\kappa_0$ if necessary.

For each $W\in\mathcal{W}'$, define $Y(W) = \bigcup_{T\in\tubes'[W]}T$; we have $Y(W)\subset W$ and $|Y(W)|\geq\kappa_0\delta^{\eps_1}$ for each $W\in\mathcal{W}'$; we will suppose that $\delta>0$ is chosen sufficiently small so that this quantity is $\geq\delta^{2\eps_1}$.

If $\eps_1>0$ is chosen sufficiently small depending on $\eps$, then we can apply $\cF(\eps/10,\eps/10)$ to $(\mathcal{W}',Y)_{a\times b\times 1}$ (again, the term $D$ has size at most $\CKT(\mathcal{W}')\lessapprox_\delta 1$, as discussed in Remark \ref{discussionOfDRemark}, Situation 2) to conclude that
\begin{equation}\label{boundScaleW}
\Big|\bigcup_{W\in\mathcal{W}'} Y(W)\Big| \geq \kappa_1\delta^{\eps/2}(\#\mathcal{W}')|W|.
\end{equation}

Combining \eqref{desiredBoundInsideW} and \eqref{boundScaleW} and returning to our original notation, we have
\begin{equation}\label{boundForOneJ}
\Big|\bigcup_{T\in\tubes'_j}Y(T)\Big|\geq \kappa_1\delta^{\eps/2} \sum_{W\in\mathcal{W}_j}|W|.
\end{equation}

\medskip

\noindent {\bf Step 3.}
We will now combine the estimates \eqref{boundForOneJ} for different values of $j$. Let $\mathcal{W} = \bigcup_{j=1}^N \mathcal{W}_j$. Then $\mathcal{W}$ covers $\tubes$, and
\begin{equation}\label{lowerBoundTubes}
\Big|\bigcup_{T\in\tubes}Y(T)\Big| \geq N^{-1}\max_{1\leq j\leq N} \Big|\bigcup_{T\in\tubes'_j}T\Big| \gtrapprox_\delta  \kappa_1\delta^{\eps/2} \sum_{W\in\mathcal{W}}|W|.
\end{equation}

Finally, since $\mathcal{W}$ covers $T$, we have that the set of $R$-fold dilates $\{RW\colon W\in\mathcal{W}\}$ covers $\{T_R\colon T\in\tubes\}$, and thus
\begin{equation}\label{upperBdTildeTubes}
\Big|\bigcup_{T\in\tubes} T_R\Big| \leq \Big|\bigcup_{W\in\mathcal{W}}RW\Big|\leq\sum_{W\in\mathcal{W}}|RW| = R^3 \sum_{W\in\mathcal{W}}|W|.
\end{equation}
Inequality \eqref{lineSegmentExtensionInsideThm} now follows by comparing \eqref{lowerBoundTubes} and \eqref{upperBdTildeTubes}. 
\end{proof}


\appendix

\section{A grains decomposition for tubes in $\RR^3$}\label{guthGrainsDecompAppendix}
Our goal in this section is to prove Proposition \ref{GuthGrainsProp}. For the reader's convenience, we reproduce it here. In what follows, recall that broadness is defined in Definition \ref{broadnessDefnShading}. This definition assumes a small parameter $\beta>0,$ which was defined in Section \ref{broadnessSection}. In Section \ref{broadnessSection}, an explicit value of $\beta$ was chosen, which depends on $\omega$ and $\zeta$. In the results below, however, we will allow $\beta$ to be any positive number (though in $\mathbb{R}^n$, we will always have $\beta\in (0, n-1]$), and subsequent quantities (such as $\eta$ and $K_1$) may depend on $\beta$. In particular, the results proved in this section are applicable in Section \ref{twoScaleGrainsSec}.

\begin{GuthGrainsPropRepeat}
Let $\eps>0$. Then there exists $\eta,\kappa>0$ so that the following holds for all $\delta>0$. Let $(\tubes,Y)_\delta$ be $\delta^\eta$ dense and be broad with error $\delta^{-\eta}$. Suppose that the tubes in $\tubes$ are contained in a common 1 tube $T_1$.

Then there is a $\delta^\eps$ refinement $(\tubes',Y')$ that is $\delta^\eps$ dense and is broad with error $\leq\kappa^{-1}\delta^{-\eps}$, and a number $\mu\geq 1$ so that $\mu\sum \#\tubes'_{Y'}(x)$ for each $x\in \bigcup_{\tubes'} Y'(T)$. In addition, there is a number $c\geq  \kappa \mu  \delta^\eps (\delta\#\tubes)^{-1}$ and a pair $(\mathcal{G},Y)_{\delta\times c\times c}$, so that $(\mathcal{G},Y)_{\delta\times c\times c}$ is a robustly $\delta^{\eps}$-dense two-scale grains decomposition of $(\tubes',Y')_\delta$ wrt $\{T_1\}$ (the latter is a set consisting of a single 1 tube).
\end{GuthGrainsPropRepeat}

To prove Proposition \ref{GuthGrainsProp}, we shall repeatedly apply the Guth-Katz polynomial partitioning theorem \cite{GuKa15} to decompose a (large portion of) $\bigcup_{\tubes} Y(T)$ into a union of thin neighbourhoods of semi-algebraic sets. This idea was first used by Guth \cite{Gut16b, Gut18} in the context of the Fourier restriction problem, and later by Guth and the second author \cite{GuZa18} in the context of the Kakeya problem. 

Before stating this result precisely, we will need several definitions. First, recall that a \emph{semi-algebraic set} is a set $S\subset\RR^n$ that is defined by a Boolean combination of polynomial equalities and inequalities. 
\begin{defn}
Let $\rho>0$. A \emph{semi-algebraic grain} of thickness $\rho$ is a set of the form
\[
G = N_{\rho}(S),
\]
where $S\subset\RR^3$ is a semi-algebraic set with $\dim(S)\leq 2$. If there exists a non-zero polynomial $Q\colon\RR^3\to\RR$ with $S\subset Z(Q)$, then we say that $G$ is defined by a polynomial of degree at most $\deg(Q)$.
\end{defn}
Next, we recall the following result of Wongkew \cite{Won93}, which allows us to bound the volume of a semi-algebraic grain.
\begin{thm}\label{WonkewThm}
Let $Z=Z(Q)\subset\RR^n$ be an algebraic variety of dimension $d$. Let $B\subset\RR^n$ be a ball of radius $r$. Then there exists a constant $K$ depending only on $n$ so that for all $0<\rho\leq r$,
\[
|N_{\rho}(Z\cap B)|\leq K (\deg Q)^n \rho^{n-d}r^d.
\]
\end{thm}
\begin{cor}\label{wonkewCor}
Let $G\subset\RR^3$ be a semi-algebraic grain of thickness $\rho$ that is defined by a polynomial of degree at most $D$. Then
\[
|G| \lesssim D^3\ \rho\ \operatorname{diam}(G)^2.
\]
\end{cor}

Next, we will need the following grains decomposition result, which is similar to Proposition 3.2 from \cite{GuZa18}.
\begin{prop}\label{grainsDecompProp}
Let $\eps>0$. Then there exists $K_\eps>0$ so that the following holds. Let $E\subset B(0,1)\subset\RR^3$ be open. Then there exists 
\begin{itemize} 
	\item A set $\mathcal{G}$ of semi-algebraic grains of thickness $100\delta$, each of which has diameter $\leq K_\eps (\#\mathcal{G})^{-1/3}$ and is defined by a polynomial of degree at most $K_\eps$.

	\item For each $G \in\mathcal{G}$, a set $E_G \subset E\cap G$. The sets $\{  E_G \colon G\in\mathcal{G}\}$ are disjoint.
\end{itemize}
Furthermore, we have
	\[
		\sum_{G\in\mathcal{G}}|E_G| \geq K_\eps^{-1} \delta^\eps |E|,
	\]
	and for each $G\in\mathcal{G}$ we have
	\[
		K_\eps^{-1} \delta^\eps\frac{|E|}{\#\mathcal{G}}\leq |E_G| \leq K_\eps \delta^{-\eps}\frac{|E|}{\#\mathcal{G}}.
	\]
Finally, for every $\delta$ tube $T$, we have
\begin{equation}\label{countNumberOfEGintersectT}
\#\{G\in\mathcal{G}\colon T\cap E_G\neq\emptyset\} \leq K_\eps \delta^{-\eps}(\#\mathcal{G})^{1/3}.
\end{equation}
\end{prop}
The above result is essentially Proposition 3.2 from \cite{GuZa18}, except that the result in \cite{GuZa18} does not require that the grains in $\mathcal{G}$ have diameter at most $K_\eps(\#\mathcal{G})^{-1/3}$. Adding this latter property is straightforward --- we simply add additional partitioning planes in the $e_1,e_2,$ and $e_3$ directions at each partitioning step.

Using Proposition \ref{grainsDecompProp}, we have the following structure theorem about arrangements of tubes. 
\begin{lem}\label{grainsDecompBroad}
Let $\eps,\beta>0$. Then there exists $K,\eta>0$ so that the following holds.  Let $(\tubes,Y)_\delta$ be $\delta^\eta$ dense and broad with error $\delta^{-\eta}$ (and exponent $\beta$). Suppose there is a number $\mu$ so that $\#\tubes_Y(x)\in[\mu,2\mu)$ for all $x\in\bigcup_{T\in\tubes}Y(T)$. Then there exists the following:
\begin{itemize}
	\item A $\delta^\eps$ refinement $(\tubes',Y')_\delta$ of $(\tubes,Y)_\delta$.
	\item A partition $E = E_1\sqcup \ldots E_N$ of $\bigcup_{\tubes}Y'(T)$.
\end{itemize}
These objects have the following properties:
\begin{itemize}
	\item[(a)]
	\begin{itemize}
		\item[(i)]
	\begin{equation}\label{boundOnN}
		1\leq N\leq K\delta^{-\eps} \big(\delta \mu^{-1}\ \#\tubes\big)^3.
	\end{equation} 
	\item [(ii)] $r\geq K^{-1}\delta^{\eps}N^{-1/3}$.
\end{itemize}

	\item[(b)] Each set $E_i$ has diameter at most $2r$, and volume 
	\begin{equation}\label{upperLowerBdVolumeEi}
		K^{-1} N^{-1} \delta^{\eps} \Big| \bigcup_{T\in\tubes}Y(T)\Big| \leq |E_i| \leq K \delta \big(\operatorname{diam}(E_i)\big)^2.
	\end{equation}
	\item[(c)]  $(\tubes',Y')_\delta$ is broad with error $K\delta^{-\eps}$ (and exponent $\beta$), and $\#\tubes'_{Y'}(x)\geq K^{-1}\delta^\eps\mu$ for each $x\in\bigcup_{T\in\tubes'}Y'(T)$.
	\item[(d)] For each $T\in\tubes'$,  there exists a set $\{T^{(j)}\}$ of pairwise $100r$-separated tube segments $T^{(j)}\subset T$, each of length $r$. This set has the following properties:
	\begin{itemize}

	\item[(i)] Each segment $T^{(j)}$ intersects $Y'(T)$, and $Y'(T)\subset \bigsqcup T^{(j)}$.

	\item[(ii)] If $T^{(j)}$ is a tube-segment, then $T^{(j)}\cap Y'(T)$ intersects exactly one set $E_i$.
\end{itemize}

\end{itemize}
\end{lem}
\begin{proof}$\phantom{1}$\\
\noindent {\bf Step 1.}
Let $\eps_1$ be a small quantity to be chosen below.  Let $\tubes_1 = \{T\in\tubes\colon |Y(T)|\geq \frac{1}{2}\delta^{\eta}|T|\}$, and let $Y_1$ be the restriction of $Y$ to $\tubes_1$. Then $(\tubes_1,Y_1)_\delta$ is a $(1/2)$-refinement of $(\tubes,Y)_\delta$. Apply Proposition \ref{grainsDecompProp} to $E = \bigcup_{T\in\tubes_1}Y_1(T)$ with $\eps_1$ in place of $\eps$, and let $\mathcal{G}$ and $\{E_G\colon G\in\mathcal{G}\}$ be the output of that Proposition. 

For each $T\in\tubes_1$, there are at most $K_{\eps_1}\delta^{-\eps_1}(\#\mathcal{G})^{1/3}$ grains $G\in\mathcal{G}$ for which $T\cap E_G\neq\emptyset$. Define the set of ``significant'' grains 
\[
\mathcal{G}_{T, \operatorname{sig}}=\big\{G\in\mathcal{G}\colon |T\cap E_G| \geq (\frac{1}{2}|Y_1(T)|)\big(K_{\eps_1}\delta^{-\eps_1}(\#\mathcal{G})^{1/3}\big)^{-1}\big\}.
\]
Then  
\[
\Big| Y_1(T) \cap \bigsqcup_{G\in \mathcal{G}_{T, \operatorname{sig}}} E_G\Big| \geq \frac{1}{2}|Y_1(T)|.
\]

Each set $Y_1(T)\cap E_G$ is contained in a sub-tube of length $\diam(T\cap E_G)$, and we have
\[
\delta^{\eta+\eps_1}K_{\eps_1}^{-1}(\#\mathcal{G})^{-1/3} \lesssim \diam(T\cap E_G)\leq K_{\eps_1}(\#\mathcal{G})^{-1/3}.
\]
The first inequality follows from the volume bound $|T\cap E_G|\gtrsim \delta^{\eta+\eps_1}|T|K_{\eps_1}^{-1} (\#\mathcal{G})^{-1/3}$, while the second follows from the diameter bound on $G$ coming from Proposition \ref{grainsDecompProp}.

Define \begin{equation} \label{eq: rA5}
	 r = \kappa \delta^{\eta+\eps_1}K_{\eps_1}^{-1}(\#\mathcal{G})^{-1/3},
	\end{equation}  where $\kappa\sim 1$ is chosen so that each set $T\cap E_G$ has diameter at least $r$. Then for each grain $G\in\mathcal{G}_{T, \operatorname{sig}}$, we can choose a sub-tube $T_G\subset T$ of length $r$, so that 
\begin{equation}\label{bigVolumeBdInsideTG}
|T_G\cap E_G| \gtrsim \delta^{\eta+\eps_1}K_{\eps_1}^{-2}|T\cap E_G| \gtrsim  
\delta^{2\eta+2\eps_1}K_{\eps_1}^{-3}|T|(\#\mathcal{G})^{-1/3} \gtrsim \delta^{\eta + \eps_1}K_{\eps_1}^{-2} r|T|.
\end{equation}
Furthermore, we can choose these sub-tubes so that every pair of sub-tubes in the set $\{T_G\colon G\in \mathcal{G}_{T, \operatorname{sig}}\}$ either coincide or are interior-disjoint. 

Since the sets $\{E_G\}$ are disjoint, the volume bound \eqref{bigVolumeBdInsideTG} says that for each tube $T$, at most $\delta^{-\eta - \eps_1}K_{\eps_1}^{2}$ tubes from the set $\{T_G  \colon G\in \mathcal{G}_{T, \operatorname{sig}}\}$ can pairwise coincide (the sets $T_G$ are sub-tubes of $T$, as described previously). Thus for each $T\in\tubes_1$, we can select a set $Y_2(T)\subset Y_1(T)$ with the following properties:
\begin{itemize}
	\item 
	\[ 
	|Y_2(T)|\gtrsim \delta^{2\eta+2\eps_1}K_{\eps_1}^{-4}|Y_1(T)|.
	\]
	(This follows from the first inequality in \eqref{bigVolumeBdInsideTG}, plus the fact that at most $\delta^{\eta + \eps_1}K_{\eps_1}^{-2}$ can coincide).

	\item There is a set $\{T^{(j)}\}$ of sub-tubes of $T$, each of length $r$. These tubes are $100r$ separated. Each of these sub-tubes $T^{(j)}$ intersects $Y_2(T)$, and for each sub-tube $T^{(j)}$ there is exactly one $G\in\mathcal{G}$ for which $Y_2(T)\cap T^{(j)}\cap E_G\neq\emptyset$.
\end{itemize}

\medskip

\noindent {\bf Step 2.}
In Step 1 we processed the tubes. In this step, we will process the grains in $\mathcal{G}$. Fix a grain $G$. Cover $G$ by boundedly overlapping balls $\{B_i\}$ of radius $2r$ (recall that $r$ was defined in Step 1). 
We will select a $100r$-separated subset of $\{B_i\}$ (after re-indexing, we will denote this set by $\{B_i\}_{i=1}^M$ for some $M=M(G)\geq 1$) so that 
\begin{equation}\label{sufficientMassPreservedInsideSmallerBalls}
\sum_{i=1}^M \sum_{T\in\tubes_1}|Y_2(T)\cap E_G\cap  B_i| \gtrsim \sum_{T\in\tubes_1}|Y_2(T)\cap E_G|.
\end{equation}
Since the balls are $100r$-separated, each sub-tube $T^{(j)}$ of length $r$ intersects at most one ball. Thus for each $T\in\tubes_1$ and each sub-tube $T^{(j)}$, there is at most one ball $B_i$ from the above collection for which $Y_2(T)\cap T^{(j)}\cap E_G \cap B_i\neq\emptyset$. 

We perform the above operation for each $G\in\mathcal{G}$. Let $\mathcal{G}'$ be the union of sets $\{B_i\cap G\colon i = 1,\ldots, M(G)\}$, as $G$ ranges over the grains in $\mathcal{G}$. For each $G'\in\mathcal{G}'$, define $E_{G'} = G'\cap E_G$, where $G\in\mathcal{G}$ is the (unique) grain that gave rise to $G'$. We have that the sets $\{E_{G'}\colon G'\in\mathcal{G}\}$ are disjoint.

Define $Y_3(T) = Y_2(T)\cap \bigcup_{G'\in\mathcal{G}'}E_{G'}$. By \eqref{sufficientMassPreservedInsideSmallerBalls} we have that $(\tubes_1,Y_3)_\delta$ is a $\gtrsim 1$ refinement of $(\tubes_1,Y_2)_{\delta}$, which in turn is a $\gtrsim  \delta^{2\eta+2\eps_1}K_{\eps_1}^{-4}$ refinement of $(\tubes,Y)_\delta$. Furthermore, for each $T\in\tubes_1$ it is still the case that there is a set of $100r$-separated sub-tubes $\{T^{(j)}\}$, each of which has length $r$ and intersects $Y_3(T)$  (indeed, any sub-tubes $T^{(j)}$ from the previous collection that fail to intersect $Y_3(T)$ can be discarded), so that for each sub-tube $T^{(j)}$, there is exactly one $G'\in\mathcal{G}'$ with $T^{(j)}\cap E_{G'}\neq\emptyset$.

\medskip

\noindent {\bf Step 3.}
Let $\mu'\leq\mu$ and let $(\tubes_1,Y_4)_\delta$ be a $\gtrapprox_\delta1$ refinement  (we use $\gtrapprox_\delta$ to absorb the $K_{\eps_1}^{-4}$ factor) of $(\tubes_1,Y_3)_\delta$ with the property that $\#(\tubes_1)_{Y_4}(x)\sim\mu'$ for each $x\in\bigcup_{T\in\tubes_1}Y_4(T)$. Since $(\tubes_1,Y_4)_\delta$ is a $\approx_\delta \delta^{2\eta+2\eps_1}$ refinement of $(\tubes,Y)_\delta$, we have \begin{equation} \label{eq: mu'A5}
	\mu'\gtrapprox_\delta \delta^{2\eta+2\eps_1} \mu,
	\end{equation}  and hence $(\tubes_1,Y_4)_\delta$ is broad with error $\lessapprox_\delta \delta^{-3\eta-2\eps_1}$ and exponent $\beta$. We will select $\eta$ and $\eps_1$  sufficiently small, and $K$ sufficiently large so that $(\tubes_1,Y_4)_\delta$ is broad with error $\leq K\delta^{-\eps}$, and $\mu' \geq K^{-1}\delta^{\eps}\mu$ (c.f.~Conclusion (c) from Proposition \ref{grainsDecompBroad}). 

For each $G'\in\mathcal{G}'$, define $E'_{G'}=E_{G'}\cap\bigcup_{T\in\tubes_1}Y_4(T)$. After dyadic pigeonholing, we can select a set $\mathcal{G}''\subset\mathcal{G}'$ so that $|E'_{G'}|$ is the same (up to a factor of 2) for each $G'\in\mathcal{G}''$. For each $T\in\tubes_1$, define $Y_5(T) = Y_4(T)\cap\bigcup_{G''\in\mathcal{G}''}E'_{G''}$. 

Define $N = \#\mathcal{G}''$. Abusing notation, we will enumerate the elements of $\{E'_{G''}\colon G''\in\mathcal{G}''\}$ as $E_1,\ldots,E_N$.  Since $\#\tubes_Y(x)\sim \mu$ and   $K_{\eps_1}^{-1} \delta^{\eps_1} \frac{|E|}{\#\mathcal{G}} \leq |E_G|\leq K_{\eps_1} \delta^{-\eps_1} \frac{|E|}{\#\mathcal{G}}$ for each $G\in \mathcal{G}$ (the latter is true since $\mathcal{G}$ was the output of Proposition~\ref{grainsDecompProp}), we have 
\begin{equation}\label{eq: lowerbdN}
\#\mathcal{G}\geq 	N=\#\mathcal{G}''\gtrapprox_{\delta} \delta^{2\eps_1} \#\mathcal{G}.
	\end{equation}

Define $\tubes'=\tubes_1$ and $Y'(T) = Y_5(T)$. This choice of $(\tubes',Y')_\delta$ satisfies Conclusion (c) from Proposition \ref{grainsDecompBroad}.

We have 
\[
\bigcup_{T\in\tubes'}Y'(T) = \bigsqcup_{i=1}^N E_i,
\]
and if the constant $K=K(\eps_1)$ (recall that $\eps_1$ depends on $\eps$) is chosen sufficiently large, then for for each index $i=1,\ldots,N$ we have
\[
N^{-1}\Big|\bigcup_{T\in\tubes}Y(T)\Big|\lessapprox_\delta N^{-1}\Big|\bigcup_{T\in\tubes'}Y'(T)\Big| \lesssim |E_i| \leq K \delta (\operatorname{diam}(E_i))^2,
\]
where the final inequality is Corollary \ref{wonkewCor}. Thus we have established Conclusion (b) from Proposition \ref{grainsDecompBroad}. Recall that for each $T\in\tubes'$, we have a set $\{T^{(j)}\}$ of tube-segments, each of which has length $r$. These segments are $100r$-separated. Discarding segments if necessary, we may suppose that each tube segment $T^{(j)}$ intersects $Y'(T)$. It is still the case that for each surviving tube segments, there is exactly one index $i$ so that $Y'(T)\cap T^{(j)}\cap E_i\neq\emptyset$. This establishes Conclusions (d.i) and (d.ii).

\medskip

\noindent {\bf Step 4.}
It remains to establish Conclusion (a) from Proposition \ref{grainsDecompBroad}. We begin with Conclusion (a.i), i.e.~we must show that $N$ satisfies Inequality \eqref{boundOnN}.  In Step 3 we established Conclusion (d.ii). Combining this with \eqref{eq: rA5} and  \eqref{eq: lowerbdN}, we conclude that for each $T\in\tubes$, there are  $\lessapprox_{\delta}  \delta^{-2\eps_1}N^{1/3}$ indices $i$ for which $Y'(T)\cap E_i\neq\emptyset$. Thus by pigeonholing, there is an index $i$ so that
\[
\{\#T\in\tubes\colon Y'(T)\cap E_i\neq\emptyset\} \lessapprox_{\delta}  \delta^{-2\eps_1}N^{-2/3}(\#\tubes).
\]
Fix this index $i$, and denote the above set by $\tubes_i$.  We have 
\begin{equation}\label{estimateNuMeasure}
\nu\big( \{(x,T,T')\in E_i \times \tubes_i^2\colon x \in Y'(T)\cap Y'(T'),\ \angle(\operatorname{dir}(T),\operatorname{dir}(T))\gtrapprox_\delta \delta^{ ( 3\eta+2\eps_1)/\beta}\big\} \lessapprox_{\delta}  \delta^{3-(3\eta+2\eps_1)/\beta}(\#\tubes_i)^2,
\end{equation}
where $\nu$ denotes the product of Lebesgue measure on $E_i$ with counting measure on $\tubes_i^2$. The above inequality is justified by the observation that for each pair $T,T'\in\tubes_i$ with $\angle(\operatorname{dir}(T),\operatorname{dir}(T'))\geq \delta^{( 3\eta+2\eps_1)/\beta}$, we have $|T\cap T'|\lesssim \delta^{3-( 3\eta+2\eps_1)/\beta}$.

On the other hand, we have $|E_i| \gtrapprox_\delta  \delta^{\eps_1}N^{-1}\big(\mu^{-1}\delta^2(\#\tubes)\big)$, and since $(\tubes',Y')_\delta$ is broad with error $\lessapprox_\delta \delta^{-( 3\eta+2\eps_1)}$, we have that for each $x\in E_i$ there are $\gtrsim (\mu')^2$ pairs $T,T'\in\tubes_i$ so that $(x,T,T')$ is in the above set. We conclude that by \eqref{eq: mu'A5}, 
\begin{equation*}
\begin{split}
\mu\delta^{2+ 6\eta+4\eps_1}N^{-1}(\#\tubes) &\lessapprox_\delta  (\mu')^2 \cdot\Big( \delta^{4\eta+2\eps_1}N^{-1}\mu'^{-1}\delta^2(\#\tubes)\Big) \lesssim \textrm{LHS}\ \eqref{estimateNuMeasure}\\
& \lesssim \textrm{RHS}\ \eqref{estimateNuMeasure} = \delta^{-(3\eta+2\eps_1)/\beta+3} (\#\tubes_i)^2 \lesssim K_{\eps_1}^2 \delta^{-(3\eta+2\eps_1)/\beta -4\eps_1+3}N^{-4/3}(\#\tubes)^2.
\end{split}
\end{equation*}
Re-arranging, we have
\[
N^{1/3} \lessapprox_\delta K_{\eps_1}^2 \mu^{-1}\delta^{-(3\eta+2\eps_1)/\beta - 6\eta- 8\eps_1+1}(\#\tubes).
\]
The claimed bound on $N$ now follows by selecting $3\eta+2\eps_1<\eps \beta/20$, and choosing $K$ appropriately.

Finally,  Conclusion (a.ii) follows from the definition of $r$ and \eqref{eq: lowerbdN}. 
\end{proof}

Lemma \ref{grainsDecompBroad} says nothing about the shape of the sets $\{E_i\}$. The next result will help us show that a substantial portion of each set $E_i$ must be contained in the $\delta$ neighbourhood of a plane.   The idea is that since the volume of $E_{i}$ is small, each point $x\in E_i$ is broad, and $E_i\cap T$ is dense for each tube $T\in \tubes_i$,  we can find many ``triangles'' formed by three tubes from $\tubes_i$ that pairwise intersect at distinct points. Since every triple of lines that pairwise intersect at distinct points must be coplanar, this in turn forces a significant subset of  $E_i$ to be planar.

\begin{lem}\label{EiContainedInPlane}
Let $\eps,\beta>0$. Then there exists $\kappa,\eta>0$ so that the following holds.  Let $(\tubes,Y)_\delta$ be $\delta^\eta$ dense and broad with error $\leq \delta^{-\eta}$ (and exponent $\beta$). Suppose that 
\[
\Big|\bigcup_{T\in\tubes}Y(T)\Big| \leq \delta^{1-\eta}.
\]
Then there is a plane $\Pi$ so that if we define $\tubes'=\{T\in\tubes\colon T\subset N_{2\delta}(\Pi)\}$, then there is a shading $\{Y'(T)\colon T\in\tubes'\}$ and a number $\mu$ so that 
\begin{itemize}
	\item[(a)] $\#\tubes'_{Y'}(x)\sim\mu$ for each $x\in\bigcup_{T\in\tubes'}Y'(T)$.
\item[(b)]
	\begin{equation}\label{mostMassInPlane}
	\sum_{T\in\tubes'}|Y'(T)|  \geq \kappa \delta^{\eps}\sum_{T\in\tubes}|Y(T)|.
	\end{equation}
\item[(c)]
\begin{equation}\label{planeMostlyFull}
\Big|\bigcup_{T\in\tubes'}Y'(T)\Big| \geq \kappa\delta^{1+\eps},
\end{equation}
\end{itemize}
\end{lem}
\begin{proof}
First, after dyadic pigeonholing and replacing $(\tubes,Y)_\delta$ by a refinement, we may suppose that there is a number $\mu$ so that $\#\tubes_Y(x)\sim \mu$ for all $x\in\bigcup_{T\in\tubes}Y(T)$; we still have that $(\tubes,Y)_\delta$ is $\gtrapprox_\delta\delta^\eta$ dense, and $(\tubes,Y)_\delta$ is broad with error $\lessapprox_\delta \delta^{-\eta}$. Observe that 
\[
\delta^{2+\eta}(\#\tubes) \lessapprox_\delta \sum_{T\in\tubes}|Y(T)| \lesssim \mu\Big|\bigcup_{T\in\tubes}Y(T)\Big| \leq \delta^{1-\eta}\mu,
\]
and thus
\begin{equation}\label{numberOfTubesBound}
\#\tubes \lessapprox_\delta \delta^{-1-2\eta}\mu.
\end{equation}
Define 
\begin{equation*}
\begin{split}
\tubes'&=\{T\in\tubes\colon |Y(T)|\geq \delta^{2\eta}|T|\},\\
E & = \{x\in\RR^3\colon \#\tubes'_Y(x) \geq \delta^{2\eta}\mu\},\\
\tubes''&=\{T\in\tubes\colon |Y(T)\cap E| \geq \delta^{4\eta}|T|\}.
\end{split}
\end{equation*}
By double-counting we have $\#\tubes''\geq \delta^{5\eta}(\#\tubes)$. 

Note that for each $T_0\in\tubes''$ and each $x\in Y(T_0)\cap E$, there are $\geq \frac{1}{2} \delta^{2\eta}\mu$ tubes $T\in\tubes'$ with $Y(T)\cap Y(T_0)\neq\emptyset$ and $\angle(\dir(T),\dir(T_0))\geq\delta^{3\eta/\beta}$; this is because there are $\geq \delta^{2\eta}\mu$ tubes $T\in\tubes'$ with $Y(T)\cap Y(T_0)\neq\emptyset$, and by broadness at most $\geq \frac{1}{2} \delta^{3\eta}\mu$ of the tubes passing through a point $x\in\RR^3$ can satisfy $\angle(\dir(T),v)\leq \delta^{3\eta/\beta}$ for any fixed vector $v$. In particular, for each tube $T_0$ of this type, there are at least $\big(\delta^{3\eta}\mu\big)\big(\delta^{3\eta/\beta} \delta^{-1}\big)$ distinct tubes $T\in\tubes'$ with $T\cap T_0\neq\emptyset$ and $\angle(\dir(T),\dir(T_0))\geq\delta^{3\eta/\beta}$. Since each such $T\in\tubes'$ satisfies $Y(T)\geq  \delta^{2\eta}|T|$ and $|T\cap N_{\delta^{10\eta/\beta}}(T_0)|\leq \frac{1}{2}\delta^{2\eta}|T|$, we conclude that for each tube $T_0$ of the type described above, we have 
\[
\sum_{\substack{T\in\tubes'\\ Y(T)\cap Y(T_0)\neq\emptyset\\ \angle(\dir(T),\dir(T_0))\geq\delta^{3\eta/\beta} }}|Y(T)\backslash N_{\delta^{10\eta/\beta}}(T_0)| \gtrsim \big(\delta^{3\eta+3\eta/\beta }\mu\delta^{-1}\big)\big( \delta^{2\eta}|T|\big) \gtrsim \delta^{10\eta/\beta}\delta\mu.
\]
On the other hand,
\[
\Big| \bigcup_{\substack{T\in\tubes'\\ Y(T)\cap Y(T_0)\neq\emptyset\\ \angle(\dir(T),\dir(T_0))\geq\delta^{3\eta/\beta} }} Y(T)\backslash N_{\delta^{10\eta/\beta}}(T_0)\Big| \leq \Big|\bigcup_{T\in\tubes}Y(T)\Big| \leq \delta^{1-\eta},
\]
Thus if we define $F_{T_0}$ to be the set of points $x\in\RR^3\backslash N_{\delta^{10\eta/\beta}}(T_0)$ satisfying
\[
\sum_{\substack{T\in\tubes'\\ Y(T)\cap Y(T_0)\neq\emptyset\\ \angle(\dir(T),\dir(T_0))\geq\delta^{3\eta/\beta} }}\chi_{Y(T)}(x) \geq \delta^{11\eta/\beta}\mu,
\]
then $|F_{T_0}|\geq\delta^{11\eta/\beta+1}$ because $\#\tubes_Y(x)\lesssim \mu$. Since $(\tubes,Y)_\delta$ is broad with error $\delta^{-\eta}$, for each point $x\in F_0$, there are $\gtrsim (\delta^{11\eta/\beta}\mu)^2$ pairs $T,T'\in\tubes'$ satisfying
\begin{equation}\label{goodPairTTPrime}
\begin{split}
& Y(T)\cap Y(T_0)\neq\emptyset,\ Y(T')\cap Y(T_0)\neq\emptyset,\ Y(T)\cap Y(T')\neq\emptyset, \\
& \angle(\dir(T_0),\dir(T))\geq \delta^{3\eta/\beta},\  \angle(\dir(T_0),\dir(T'))\geq \delta^{3\eta/\beta},\ \angle(\dir(T),\dir(T'))\geq \delta^{12\eta/\beta^2},\\
& (T\cap T') \backslash N_{\delta^{10\eta/\beta}}(T_0)\neq\emptyset.
\end{split}
\end{equation}

For each such pair, we have $|Y(T)\cap Y(T')|\leq |T\cap T'|\lesssim \delta^{3-12\eta/\beta^2}$. We conclude that for $T_0$ fixed, there are $\gtrsim (\delta^{11\eta/\beta}\mu)^2(|F_{T_0}| / \delta^{3-12\eta/\beta^2}) \gtrsim \delta^{23\eta/\beta^2}\mu^2\delta^{-2}$ distinct pairs $(T,T')$ satisfying \eqref{goodPairTTPrime}. Thus  by \eqref{numberOfTubesBound} there are $\gtrsim \delta^{23\eta/\beta^2}\mu^2\delta^{-2}(\#\tubes'')\gtrapprox_\delta \delta^{24\eta/\beta^2}(\#\tubes')^2(\#\tubes)$ triples $(T_0,T_1,T_2)\in \tubes''\times \tubes'\times \tubes'$ that satisfy \eqref{goodPairTTPrime}. 

By pigeonholing, we can select a pair $(T,T')$ which is a member of $\gtrapprox_\delta \delta^{24\eta/\beta^2}(\#\tubes)$ such triples. Fix this pair $T,T'$, and let $\tubes_0$ denote the set of tubes $T_0$ so that $(T_0,T',T'')$ satisfies \eqref{goodPairTTPrime}. We have the following:
\begin{itemize}
\item $T$ and $T'$ intersect, and make angle $\gtrsim \delta^{12\eta/\beta^2}$. 
\item For each $T_0\in\tubes_0,$ we have that $Y(T_0)$ intersects both $T$ and $T'$, and makes angle $\geq \delta^{3\eta/\beta}$ with $T$ and $T'$. 
\item $(T\cap T') \backslash N_{\delta^{10\eta/\beta}}(T_0)\neq\emptyset$.
\end{itemize}
The above items imply that each $T_0\in\tubes_0$ is contained in the $\delta^{1-100\eta/\beta^2}$ neighbourhood of a plane --- this is the plane spanned by the tubes $T$ and $T$ (the latter plane is technically not well defined, but is instead defined up to uncertainty $\frac{\delta}{\angle(\dir(T),\dir(T'))}\leq \delta^{1-11\eta/\beta^2}$). In particular, there is a set of $\lesssim \delta^{200\eta/\beta^2}$ planes $\{\Pi_i\}$ so that each $T_0\in\tubes_0$ is contained in the $2\delta$ neighbourhood of some plane from this collection. By pigeonholing, we can select a plane $\Pi$ so that
\[
\sum_{\substack{T\in\tubes_0 \\ T \subset N_{2\delta}(\Pi)}}|Y(T)|  \gtrapprox_\delta \big(\delta^{200\eta/\beta^2}\big) \big(\delta^{2+2\eta}\big)\big(\delta^{24\eta/\beta^2}(\#\tubes)\big) \gtrapprox_\delta \delta^{300\eta/\beta^2}\sum_{T\in\tubes}|Y(T)|.
\]

Define $\tubes'=\{T\in\tubes\colon T\subset N_{2\delta}(\Pi)\}$. By pigeonholing, we can select a number $\mu'$ so that if we define the shading $Y'(T)=\{x\in Y(T)\colon \#\tubes'_Y(x)\sim\mu'\}$, then $(\tubes',Y')_\delta$ is an $\approx_\delta 1$ refinement of $(\tubes',Y)_\delta$. This gives Conclusions (a) with $\mu'$ in place of $\mu$ and (b). Finally, Conclusion (c) follows by a standard Cordoba-type $L^2$ argument (making use of the fact that $(\tubes',Y')_\delta$ is broad with error $\delta^{-400\eta/\beta^2}$ and exponent $\beta$). 
\end{proof}

Next, we record the following re-scaled variant of Lemma \ref{EiContainedInPlane} 
\begin{cor}\label{LemEiContainedInPlaneRescaled}
Let $\eps,\beta>0$. Then there exists $\kappa,\eta>0$ so that the following holds. Let $0<\delta<r\leq 1$ and let $B\subset \RR^3$ be a ball of radius $r$.  Let $(\tubes,Y)_\delta$ be broad with error $\delta^{-\eta}$  (and exponent $\beta$). Suppose that $Y(T)\subset B$ for each $T\in\tubes$, and $\sum_{T\in\tubes}|Y(T)|\geq\delta^\eta \delta^2 r(\#\tubes)$. Suppose furthermore that the tubes in $\tubes$ are contained in a 1 tube $T_1$, and that 
\[
\Big|\bigcup_{T\in\tubes}Y(T)\Big| \leq r^2 \delta^{1-\eta}.
\]
Then there is a set $\tubes'\subset\tubes$; a number $\mu$; a shading $Y'(T)\subset Y(T)$, $T\in\tubes'$; a prism $P$ of dimensions comparable to $\delta\times r\times r$ so that each $T\in\tubes'$ satisfies $T\cap P\neq\emptyset$, and $T$ exists $P$ through its long sides, in the sense of Remark \ref{remarkLongEndsExitSquareGrain}. Furthermore, we have
\begin{itemize}
	\item[(a)] $\#\tubes'_{Y'}(x)\sim\mu$ for each $x\in\bigcup_{T\in\tubes'}Y'(T)$.
\item[(b)]
\begin{equation}\label{mostMassInPlaneInsideGrain}
\sum_{T\in\tubes'}|Y'(T)|\geq \kappa \delta^\eps\sum_{T\in\tubes}|Y(T)|.
\end{equation}
\item[(c)]
\begin{equation}\label{planeMostlyFullInsideGrain}
\Big|\bigcup_{T\in\tubes'}Y'(T)\Big|\geq \kappa \delta^{1+\eps}r^2.
\end{equation}
\end{itemize}
\end{cor}

We are now ready to prove Proposition \ref{GuthGrainsProp}. In brief, we use Lemma \ref{grainsDecompBroad} to cover a substantial portion of $(\tubes,Y)_\delta$ by a union of disjoint sets $\{E_i\}$, and then we use Corollary \ref{LemEiContainedInPlaneRescaled} to trap a substantial portion of each set $E_i$ in a prism $P_i$ of dimensions  comparable to $\delta\times r\times r$, for an appropriately chosen diameter $r$.

\begin{proof}[Proof of Proposition \ref{GuthGrainsProp}]$\phantom{1}$\\
{\bf Step 1}.
Let $\delta_0, \eps_1,\eps_2$ be small numbers to be chosen below. We will select $\delta_0$ very small compared to $\eps_1$, $\eps_1$ very small compared to $\eps_{2}$, and $\eta,\kappa$ very small compared to $\eps_1$. These numbers depend on $\eps$ and $\beta$. We may suppose that $\delta\leq\delta_0$, or else the result is immediate provided we choose $\kappa$ sufficiently small so that $\kappa\delta^\eps(\delta\tubes)^{-1}\leq\delta$. In this case the set $\mathcal{G}$ consists of $\delta\times\delta\times\delta$ balls, and there is nothing to prove. Henceforth we shall assume that $\delta\leq\delta_0$.

After dyadic pigeonholing and replacing $(\tubes,Y)_\delta$ by a $\sim(\log 1/\delta)^{-1}$ refinement, we may suppose there exists a number $\mu\geq 1$ so that $\#\tubes_Y(x)\in[\mu,2\mu)$ for every $x\in\bigcup_{T\in\tubes}Y(T)$; we still have that $(\tubes,Y)_\delta$ is broad with error $\delta^{-\eta}$.

Apply Lemma \ref{grainsDecompBroad} to $(\tubes,Y)_\delta$ with $\eps_1$ in place of $\eps$ and $\beta$ as above; we can do this, provided we select $\eta>0$ sufficiently small depending on $\eps_1$ and $\beta$. Let $K_1, N$, $(\tubes_1,Y_1)_\delta$, and $E_1,\ldots,E_N$ be the output of that lemma. In particular, $(\tubes_1,Y_1)$ is broad with error $K_1 \delta^{-\eps_1}$. Note that by Item (a), we have 
\begin{equation}\label{lowerBdOnR}
r \geq K_1^{-2}\delta^{2\eps_1}(\delta \#\tubes)^{-1}  \mu.
\end{equation}

Let $I_1$ denote the set of indices in $\{1,\ldots,N\}$ for which
\[
\#\{T\in\tubes_1 \colon Y_1(T)\cap E_i\neq\emptyset\} \leq 2  N^{-1}r^{-1}(\#\tubes),
\]
and let $I_2$ denote the remaining indices. For each $T\in\tubes_1$, there are at most $(100 r)^{-1}$ sets $E_i$ for which $Y_1(T)\cap E_i\neq\emptyset$. Thus

\[
\#I_2 \leq \Big(  2N^{-1}r^{-1}(\#\tubes)\Big)^{-1}\sum_{T\in\tubes}\#\big(i\colon Y_1(T)\cap E_i\neq\emptyset \big)\leq N/2.
\]
We conclude that $\#I_1\geq N/2$. For each $i\in I_1$, define 
\[
\tubes^{(i)} = \{T\in\tubes\colon Y_1(T)\cap E_i\neq\emptyset\}.
\]
We have
\begin{equation*}
\begin{split}
\sum_{T\in\tubes^{(i)}}|Y_1(T)\cap E_i| 
& \geq K_1^{-1}\delta^{\eps_1} \mu|E_i|\\
&\geq \mu K_1^{-2}\delta^{2\eps_1} N^{-1} \Big| \bigcup_{T\in\tubes}Y(T)\Big|\\
&\geq K_1^{-1}N^{-1}\delta^{2\eps_1+\eta} (|T| \#\tubes)\\
&\geq K_1^{-2}\delta^{2\eps_1+\eta}N^{-1}(N  r/2)(|T| \# \tubes^{(i)})\\
& \geq K_1^{-3}\delta^{3\eps_1+\eta}(\delta^2 r)(\#\tubes^{(i)}).
\end{split}
\end{equation*}
In the above, the first inequality used  Conclusion (c) of Lemma \ref{grainsDecompBroad}; the second inequality used \eqref{upperLowerBdVolumeEi}; the third inequality used the fact that $(\tubes,Y)_\delta$ is $\delta^\eta$ dense, and $\#\tubes_Y(x)\sim\mu$ for each $x\in\bigcup Y(T)$; and the fourth inequality used the fact that $i\in I_1$.

\medskip

\noindent {\bf Step 2}.
By \eqref{upperLowerBdVolumeEi} we have $|E_i|\leq K \delta \big(\operatorname{diam}(E_i)\big)^2$, and by Conclusion (b) of Lemma \ref{grainsDecompBroad}, $\diam(E_i)\leq 2r$. Let $B_i$ be a ball of radius $r$ that contains $E_i$. For each $T\in\tubes^{(i)}$ define $Y^{(i)}(T) = Y_1(T) \cap E_i \subset B_i$. If $\delta_0$ and $\eps_1$ are chosen sufficiently small depending on $\eps_2$ and $\beta$, then   by Step 1,  $(\tubes^{(i)},Y^{(i)})_\delta$ satisfies the hypotheses of Corollary~\ref{LemEiContainedInPlaneRescaled}. Applying this Corollary, we obtain a set $G_i$ comparable to a $\delta\times r\times r$ prism; a number $\mu_i$; a set $\tubes^{(i)\prime}\subset\tubes^{(i)}$; and a sub-shading $Y^{(i)\prime}(T)\subset Y^{(i)}(T),\ T\in\tubes^{(i)\prime}$, so that the following holds:
\begin{itemize}
	\item Each $T\in\tubes^{(i)\prime}$ intersects $G_i$ and exits $G_i$ through its long ends, in the sense of Remark \ref{remarkLongEndsExitSquareGrain}.
	\item 
\[
\sum_{T\in \tubes^{(i)\prime}}|Y^{(i)\prime}(T)| \geq \kappa_2 \delta^{\eps_2}\sum_{T\in\tubes^{(i)}}|Y^{(i)}(T)|.
\]
\item 
\[
\Big| \bigcup_{T\in \tubes^{(i)\prime}}Y^{(i)\prime}(T)\Big| \geq \kappa_2 \delta^{1+\eps_2}r^2\gtrsim \kappa_2\delta^{\eps_2}|G_i|.
\]
\item For each $x\in \bigcup_{T\in \tubes^{(i)\prime}}Y^{(i)\prime}(T)$ we have $\#\tubes^{(i)\prime}_{Y^{(i)\prime}}(x)\sim\mu_i.$
\end{itemize}

After dyadic pigeonholing, we can select a number $\mu'$ and a set of indices $I_1'\subset I_1$ so that $\mu_i\sim\mu'$ for each $i\in I_1'$. We will choose $\mu'$ in such a way that if we define  $\mathcal{G} = \{G_i\colon i\in I_1'\}$, define the shading
\[
Y(G_i) = \bigcup_{T\in \tubes^{(i)\prime}} Y^{(i)\prime}(T),
\]
and define the shading 
\[
Y_2(T)=Y_1(T) \cap \bigcup_G Y(G),
\]
where the union is taken over those $G\in\mathcal{G}$ for which $T\in \tubes^{(i)\prime}$, then $(\tubes_1,Y_2)_\delta$ is a $\gtrapprox_\delta\delta^{\eps_2}$ refinement of $(\tubes_1,Y_1)_\delta$, and thus a $\gtrapprox_{\delta}\delta^{\eps_1+\eps_2}$ refinement of $(\tubes,Y)_\delta$.

Note that for each $G_i\in\mathcal{G}$, we have $Y(G_i)\subset G_i\cap E_i$, and thus the sets $\{Y(G)\colon G\in\mathcal{G}\}$ are disjoint. Furthermore, $|Y_1(G)|\geq\kappa_2\delta^{\eps_2}|G_i|$, and thus $(\mathcal{G},Y)_{\delta\times r\times r}$ is $\kappa_2\delta^{\eps_2}$ dense.

We have the following:
\begin{itemize}
\item[(i)] $\bigcup_{\tubes}Y_2(T) = \bigsqcup_{G\in\mathcal{G}}Y(G)$.
\item[(ii)] If $Y_2(T)\cap Y(G)\neq\emptyset$, then $T$ exits $G$ through its long ends, in the sense of in the sense of Remark \ref{remarkLongEndsExitSquareGrain}.
\item[(iii)] If $Y_2(T)\cap Y(G)\neq\emptyset$, then $Y_2(T)\cap G \subset Y(G)$.
\end{itemize}
Item (iii) follows from the fact that if  $Y_2(T)\cap Y(G_i)\neq\emptyset$  for some $G_i\in\mathcal{G}$ (recall that $G_i$ is a prism associated to a set $E_i$), then there is a length-$r$ sub-tube $T^{(j)}\subset T$ so that $Y_2(T)\cap T^{(j)}\subset E_i\subset Y(G_i)$. Since the sub-tubes $\{T^{(j)}\}$ associated to $T$ are $100r$ separated, and $\operatorname{diam}(G)\leq 10r$, we conclude that $T^{(j)}$ is the only sub-tube from the above collection that intersects $G$.

We conclude that $(\mathcal{G}, Y)_{\delta\times r\times r}$ is a robustly $\kappa_2\delta^{\eps_2}$-dense two-scale grains decomposition of $(\tubes_1,Y_2)_\delta$ wrt the single 1-tube $\{T_1\}$. Finally, the desired bound on $r$ is given by \eqref{lowerBdOnR}.
\end{proof}


\section{Wolff's hairbrush argument and the proof of Proposition \ref{WolffHairbrush}}\label{WolffHairbrushSec}
The goal of this section is to prove Proposition \ref{WolffHairbrush}. We will restate it here in an expanded form.
\begin{PropWolffHairbrushExpanded}
For all $\eps>0$, there exists $\kappa,\eta>0$ so that the following holds for all $\delta>0$. Let $(\tubes,Y)_\delta$ be $\delta^{\eta}$ dense, with $\CKT(\tubes)\leq\delta^{-\eta}$ and $\FS(\tubes)\leq\delta^{-\eta}$. Then
\begin{equation}\label{desiredTubesBd}
\Big|\bigcup_{T\in\tubes}Y(T)\Big|\geq \kappa \delta^{3/2+\eps}(\#\tubes)^{1/2}.
\end{equation}
\end{PropWolffHairbrushExpanded}
\begin{proof}
The proof uses a standard ``bilinear'' or ``robust transversality'' argument to reduce to the case where a typical pair of intersecting tubes make large angle (i.e.~the unit vectors $\operatorname{dir}(T)$ and $\operatorname{dir}(T')$ are far from parallel), followed by Wolff's hairbrush argument. Since these arguments are covered in detail elsewhere (see e.g.~\cite[\S 2.4]{KaZa21}), we will just provide a brief sketch. 

Fix $\eps>0$ and let $\eta>0$ be a small quantity to be chosen below. Let $(\tubes,Y)_\delta$ be $\delta^{\eta}$ dense, with $\CKT(\tubes)\leq\delta^{-\eta}$ and $\FS(\tubes)\leq\delta^{-\eta}$. Applying standard reductions, we may replace $(\tubes,Y)_\delta$ by a $\geq\delta^{\eta}$ dense refinement so that the following holds: There exists a number $\theta\in[\delta,1]$ so that for each $x\in \bigcup_{T\in\tubes}Y(T)$, there is a vector $v = v(x)$ so that $\angle(v, \operatorname{dir}(T))\leq\theta$ for each $T\in\tubes$ with $x\in Y(T)$, and for each unit vector $w\in\RR^3$ and each $r\in[\delta,\theta]$, we have
\[
\#\{T\in\tubes_Y(x),\ \angle(w, \operatorname{dir}(T))\leq r \} \leq (r/\theta)^\eta \#\tubes_Y(x).
\]
Furthermore, there exists a balanced partitioning cover $\tubes_\theta$ of $\tubes$, so that
\begin{equation}\label{broadAtScaleTheta}
\Big|\bigcup_{T\in\tubes}Y(T)\Big| = \sum_{T_\theta\in\tubes_\theta}\Big|\bigcup_{T\in\tubes[T_\theta]}Y(T)\Big|.
\end{equation}
After a further refinement, we may suppose that each set $\tubes^{T_\theta}$ is $\delta^{3\eta}$-dense. Note that $\tubes^{T_\theta}$ is a set of $\delta/\theta$ tubes that satisfies the broadness condition
\[
\#\{T^{T_\theta}\in\tubes^{T_\theta}_{Y^{T_\theta}}(x),\ \angle(w, \operatorname{dir}(T^{T_\theta}))\leq r \} \leq r^\eta \big(\#\tubes^{T_\theta}_{Y^{T_\theta}}(x)\big),
\]
for all unit vectors $w$ and all $r\in[\delta/\theta,1]$. Thus a standard application of Wolff's hairbrush argument \cite{Wol95} for 2-broad tubes shows that
\[
\Big|\bigcup_{T^\theta\in\tubes^{T_\theta}}Y^{T_\theta}(T^{T_\theta})\Big| \gtrsim \delta^{(5/3)(3\eta)}(\delta/\theta)^{1/2}\Big((\delta/\theta)(\#\tubes^{T_\theta})\Big)^{1/2}.
\]
By \eqref{broadAtScaleTheta}, we conclude that
\[
\Big|\bigcup_{T\in\tubes}Y(T)\Big| \gtrsim \delta^{5\eta}\delta^{3/2}\theta^{1/2}\sum_{T_\theta\in\tubes_\theta}(\#\tubes^{T_\theta})^{1/2}\gtrsim \delta^{5\eta}\delta^{3/2}\theta^{1/2}(\#\tubes_\theta)^{1/2}(\#\tubes^{1/2}).
\]
\eqref{desiredTubesBd} now follows from the observation that $\#\tubes[T_\theta]\leq \theta \FS(\tubes)(\#\tubes)\lessapprox_\delta \delta^{-2\eta}\theta (\#\tubes)$, and thus $\#\tubes_\theta\gtrapprox_\delta \delta^{2\eta}\theta^{-1}$. 
\end{proof}

\end{document}